\documentclass[11pt]{amsart}

\title{Drinfeld center of planar algebra}

\author[ \large
P\lowercase{aramita} D\lowercase{as},
S\lowercase{hamindra} G\lowercase{hosh and}
V\lowercase{ed} G\lowercase{upta}]
{\bf \large P\lowercase{aramita} D\lowercase{as},
S\lowercase{hamindra} K\lowercase{umar} G\lowercase{hosh and}
V\lowercase{ed} P\lowercase{rakash} G\lowercase{upta}}

\thanks{Ved Gupta was partially supported by Shiv Nadar University (SNU) during his affiliation with SNU}

\date{}

\address{Stat-Math Unit, Indian Statistical Institute, Kolkata, INDIA}
\email{paramita.das@isical.ac.in, shami@isical.ac.in}
\address{School of Physical Sciences, Jawaharlal Nehru University, New Delhi, INDIA}
\email{vedgupta@mail.jnu.ac.in}

\usepackage{amsmath}
\usepackage{amsfonts}
\usepackage{latexsym}
\usepackage{amssymb}
\usepackage{mathrsfs}
\usepackage{amscd}
\usepackage{hyperref}
\usepackage{soul}
\usepackage{psfrag}
\usepackage{graphicx}
\usepackage{fullpage}
\usepackage[all]{xy}
\usepackage{rotating}
\usepackage{url}
\usepackage{color}

\numberwithin{equation}{section}
\numberwithin{figure}{section}
\theoremstyle{plain}
\newtheorem{thm}{Theorem}[section]
\theoremstyle{plain}
\newtheorem{lem}[thm]{Lemma}
\theoremstyle{remark}
\newtheorem{rem}[thm]{Remark}
\theoremstyle{plain}
\newtheorem{cor}[thm]{Corollary}
\theoremstyle{definition}
\newtheorem{defn}[thm]{Definition}
\theoremstyle{plain}
\newtheorem{prop}[thm]{Proposition}
\theoremstyle{plain}
\newtheorem{fact}[thm]{Fact}
\newtheorem{conjecture}[thm]{Conjecture}
  
\newcommand{\comments}[1]{}
\newcommand{\ra}{\rightarrow}
\newcommand{\rab}{\rangle}
\newcommand{\lra}{\longrightarrow}

\newcommand{\lab}{\langle}
\newcommand{\mbb}{\mathbb}
\newcommand{\mcal}{\mathcal}

\newcommand{\N}{\mathbb N}

\newcommand{\C}{\mathbb{C}}
\newcommand{\R}{\mathbb{R}}
\newcommand{\mscr}{\mathscr}
\newcommand{\vlon}{\varepsilon}

\newcommand{\vphi}{\varphi}

\newcommand{\btimes}{\boxtimes}


\begin{document}
\global\long\def\vlon{\varepsilon}
\global\long\def\bt{\bowtie}
\global\long\def\ul#1{\underline{#1}}
\global\long\def\ol#1{\overline{#1}}
\global\long\def\norm#1{\left\|{#1}\right\|}
\global\long\def\os#1#2{\overset{#1}{#2}}
\global\long\def\us#1#2{\underset{#1}{#2}}
\global\long\def\ous#1#2#3{\overset{#1}{\underset{#3}{#2}}}
\global\long\def\t#1{\text{#1}}
\global\long\def\lrsuf#1#2#3{\vphantom{#2}_{#1}^{\vphantom{#3}}#2^{#3}}
\global\long\def\tr{\triangleright}
\global\long\def\tl{\triangleleft}
\global\long\def\cc90#1{\begin{sideways}#1\end{sideways}}
\global\long\def\turnne#1{\begin{turn}{45}{#1}\end{turn}}
\global\long\def\turnnw#1{\begin{turn}{135}{#1}\end{turn}}
\global\long\def\turnse#1{\begin{turn}{-45}{#1}\end{turn}}
\global\long\def\turnsw#1{\begin{turn}{-135}{#1}\end{turn}}
\global\long\def\fusion#1#2#3{#1 \os{\textstyle{#2}}{\otimes} #3}

\maketitle
\begin{abstract}
We introduce fusion, contragradient and braiding of Hilbert affine representations of a subfactor planar algebra $P$ (not necessarily having finite depth).
We prove that if $N \subset M$ is a subfactor realization of $P$, then the Drinfeld center of the $N$-$N$-bimodule category generated by $_N L^2 (M)_M$, is equivalent to the category of Hilbert affine representations of $P$ satisfying certain finiteness criterion.
As a consequence, we prove Kevin Walker's conjecture for planar algebras.
\end{abstract}
\section{Introduction}
Vaughan Jones, in \cite{Jon01}, introduced {\em modules over a planar algebra} or {\em annular representations of a planar algebra} to construct subfactors with principal graphs $E_6$ and $E_8$; in the same paper, he explicitly worked out these representations in the case of Temperley-Lieb planar algebra (TL).
Roughly speaking, from a planar algebra $P$, one first creates a $\C$-linear category $AnnP$ (called {\em annular category over $P$}) using annular tangles. 
Annular representations of $P$ are $\C$-linear functors from $AnnP$ to the category of vector spaces.
Although annular representations first appeared in \cite{Jon01}, $AnnTL$ had been studied extensively even before that in \cite{Jon94} and \cite{GL98}.
In \cite{Gho06}, the second named author investigated the annular representations of the group planar algebra (that is, the planar algebra associated to the fixed point subfactor arising from an outer action of a finite group) and showed that the category of these representations is equivalent to the category of representations of a nontrivial quotient of the quantum double of the corresponding group.
Let us elaborate on the reason for the appearance of this nontrivial quotient.
First of all, any element in the planar algebra $P$ can be visualized on a disc / rectangle in $\R^2$; in the same token, a morphism in $AnnP$ (also called {\em annular morphism}) can be thought of as an annulus in $\R^2$ with some additional data.
In either case, these pictorial representations are far from being unique; two pictures that can be obtained from one another via {\em planar isotopy} are considered equivalent.
Such an isotopy for the annulus might not fix its two boundaries; for instance, if the internal disc is rotated by 360 degrees using planar isotopy, then the annular morphism remains unchanged and this is precisely the reason of the quotient in \cite{Gho06} being nontrivial.
Following a suggestion of Jones (in 2003), the second named author noticed that everything fell into place if one uses a more restrictive isotopy (called {\em affine isotopy}) which fixed the boundaries of the annulus at all times.
More precisely, on replacing $AnnP$ by the {\em affine category over $P$} (denoted by $AffP$), it turned out that the category of {\em affine representations of $P$} ($= \t{Rep} (AffP)$) is indeed equivalent to the representation category of the quantum double in the group planar algebra case.

In all likeliness, Jones' hunch regarding the connection between affine representations and quantum double in the group planar algebra case, stemmed from one of Walker's conjecture in the world of TQFT's.
The conjecture (as stated in \cite{FNWW}) is the following:
\begin{conjecture}\label{wconj}
If $\Lambda$ is a $\C$-linear category, and $\Lambda^{\mcal R}$ (resp., $\Lambda^{\mcal A}$) is the rectangular (resp., annular) version of the corresponding locally defined picture / relation category, then $\mcal D (Rep(\Lambda^{\mcal R})) = Rep(\Lambda^{\mcal A})$, that is, the Drinfeld center or quantum double of the representation category of the rectangular picture category is isomorphic to the representation category of the corresponding annular category.
\end{conjecture}
\comments{
\noindent We illustrate the relevance of the above conjecture in the context of group planar algebra. Let $N\subset M$ be the fixed point subfactor corresponding to an outer action of a finite group $G$ on a $II_1$-factor $M$, and $P$ be its associated planar algebra.
The $N$-$N$-bimodule category $\mcal C$ generated by $_N L^2(M)_M$, is known to be equivalent to the representation category of $G$, whose Drinfeld center is equivalent to the representation category of the quantum double of $G$.
Then, by \cite{Gho06}, the Drinfeld center of $\mcal C$ is equivalent to the category of affine representations of $P$.
}
\noindent We illustrate the relevance of the above conjecture in the context of group planar algebra. Let $N\subset M$ be the fixed point subfactor corresponding to an outer action of a finite group $G$ on a $II_1$-factor $M$, and $P$ be its associated planar algebra. The $N$-$N$-bimodule category $\mcal C$ generated by $_N L^2(M)_M$, is known to be equivalent to the representation category of $G$, whose Drinfeld center is equivalent to the representation category of the quantum double of $G$. Then, by \cite{Gho06}, it follows that the Drinfeld center of $\mcal C$ is equivalent to the category of locally finite Hilbert affine representations of $P$. Here, $\mcal C$ (which can also be recovered from $P$ as described at the beginning of Section \ref{conj}) corresponds to the representation category of the rectangular category $\Lambda^{\mcal R}$ in Walker's conjecture and $AffP$ relates to the annular version $\Lambda^{\mcal A}$ of $\Lambda$.
Using this analogy, Jones reformulated Walker's conjecture for planar algebra in the following way.
\begin{conjecture}\label{jconj}
If $N\subset M$ is a finite depth subfactor and $P$ is the associated subfactor planar algebra, then the category of locally finite Hilbert affine representations of $P$ is equivalent to the Drinfeld center of the $N$-$N$-bimodule category generated by $_N L^2(M)_M$.
\end{conjecture}

\noindent Here, {\em locally finite} means that the corresponding Hilbert spaces are finite dimensional. Note that one could make sense of Conjecture \ref{jconj} only for additive categories since tensor product of affine representations was unknown at that time.

Affine representations first appeared in the work of Jones and Reznikoff (see \cite{JR06}) where they considered the Temperley-Lieb planar algebra. The second named author in \cite{Gho} established some finiteness results regarding affine representations of finite depth subfactor planar algebras and also affirmatively answered the question whether the radius of convergence of the dimension function of any affine representation, is at least as big as the inverse-square of the modulus of the planar algebra. In a recent article \cite{DGG}, we showed that the space of affine morphisms at level zero is isomorphic to the fusion algebra of the bicategory of bimodules of the corresponding subfactor (possibly having infinite depth); this was then used to investigate the affine representations with weight zero and prove that the Conjecture \ref{jconj} is true for irreducible depth two planar algebras (which, by Ocneanu and Szymanski (see \cite{Sz}, \cite{KLS}) corresponds to the subfactors arising from {\em minimal} actions of finite dimensional Kac algebras).

We now discuss the main results of this article.
\begin{enumerate}
\item In this article, we prove Conjecture \ref{jconj}.
On a side note, we must mention here that in \cite{FNWW}, right after the statement of Conjecture \ref{wconj} (above), it is stated that the ``conjecture and its higher category generalizations are proved in \cite{Wal}"; however, we (probably due to our inadequate understanding of TQFTs) were unable to find the proof.
In any case, we do not claim that we have proved (or re-proved) Conjecture \ref{wconj} in this article.
Our focus is restricted to Walker's conjecture for planar algebras (as formulated by Jones).
\item\label{introgenconj} Our main theorem (Theorem \ref{genconj}) is a general version of Conjecture \ref{jconj}, where we do not have any restriction on the depth of the subfactor. For this, we had to impose a new finiteness criterion (which we call {\em finite $P$-support}) on the affine representations, a condition that is automatic in the case of finite depth.
The Drinfeld center of the $N$-$N$-bimodule category, turned out to be equivalent to the category of locally finite Hilbert affine representations with finite $P$-support.
\item We introduce {\em fusion}, {\em contragradient} and {\em braiding} of affine representations (possibly not having finite $P$-support or local finiteness property) using a new inner product taking values in the space of {\em commutativity constraints}, {\em antipode functor on the affine category}) and some specific affine tangles respectively.
We show that the equivalence in part (\ref{introgenconj}) is indeed a tensor equivalence.
Here, we must mention our {\em fusion} or the {\em tensor structure} is similar to the concept of {\em pair of pants} for TQFTs. The nontrivial parts however, are (a) equipping it with an inner product to give a Hilbert affine module structure and (b) making it work even in the absence of any extra constraints on the modules, such as, local-finiteness, semisimplicity, finite $P$-support or finiteness of depth.
Moreover, for braiding, the very intuitive idea of {\em twisting the pair of pants} does not provide a unitary braiding and could possibly be an unbounded operator; it needs some tweaking to make it unitary.
\end{enumerate}
The above results were also announced in Chennai at Sunder's 60th birthday conference in April 2012 (see {\cite{Gho3}}). The single most important thing which helped us in achieving the above is the faithful tracial state on the affine endomorphism spaces (in Section \ref{straffmor}). This was motivated by \cite[Lemma 4.1]{DGG} where we gave a pictorial reformulation of the trace on the space of affine morphisms at level zero, which is induced by the obvious trace on the fusion algebra of the bimodule categories (that is, evaluation at the trivial bimodule) via the isomorphism in \cite[Theorem 3.1]{DGG}.

Once we move beyond the world of locally finite Hilbert affine representations, things become increasingly analytical.
This is relevant only for infinite depth planar algebras.
The category $\t{Rep}(AffP)$ might not be semisimple any more.
We will treat this analytical aspect of affine representations in great generality in an upcoming paper where we also work out some examples, like diagonal subfactors, Temperley-Lieb, and some more.
On the categorical front, in the finite depth case, affine representations could come handy in finding the modular invariants, such as, the $S$- and $T$-matrices (which, for some examples had been analyzed by Izumi using sectors, see \cite{Izu2}) or even determining the maximal atlas (since the Drinfeld center forms a complete invariant with respect to Morita equivalence (see the survey article \cite{Nik})).

Another important direction to look at is the bimodule category of Popa's symmetric enveloping algebra inclusion (SEAI) $M\vee M^{op} \subset M \us{e_N} \btimes M^{op}$ of a finite index extremal subfactor $N\subset M$ (see \cite{Pop94, Pop99}), which he showed was essentially the same as Ocneanu's asymptotic inclusion in the finite depth case. From \cite{Ocn, EK98}, the SEAI could be seen as an analogue of the Drinfeld center construction in the context of subfactors.
To be more precise, one needs to look at the composed inclusion $N\vee M^{op} \subset M \vee M^{op} \subset M \us{e_N} \btimes M^{op}$ (and {\em not} just the SEAI) and then its $(N\vee M^{op})$-$(N\vee M^{op})$ bimodule category turns out to be tensor equivalent to the Drinfeld center of the $N$-$N$ bimodule category (see  \cite{Izu1}). Recently, in the finite depth case, the planar algebra of the SEAI was obtained in the work of Stephen Curran \cite{Cur} where he used affine representations.  In the infinite depth case, the SEAI is an infinite index subfactor
for which the Jones¡¯ planar algebra does not exist. However, there has been attempts of developing
a planar calculus for infinite index bimodules over $II_1$-factors (see [Pen]). On the other hand,
the SEAI has also been described using planar algebra techniques in [CJS].
We hope
that our treatment of affine representations will also shed some light on understanding the bimodule
category of the SEAI in this case as well.

\comments{
We hope that our treatment of affine representations will shed some light on understanding the bimodule category of the SEAI in the infinite depth case as well; here, the SEAI is an infinite index subfactor for which the Jones' planar algebra does not exist. However, there has been attempts of developing a planar calculus for infinite index bimodules over $II_1$-factors (see \cite{Pen}). On the other hand, the SEAI has also been described using planar algebra techniques in \cite{CJS}. 
{\color{blue} Unifying the above, {\bf PENDING} ...}}

\section{Preliminaries}\label{prelim}
\subsection{Some useful notations}\label{usenot}
As per our requirements in this article, instead of giving a detailed description of planar algebras, we recall few conventions and notations that we will be following throughout this article.
Since the definition of planar algebras has evolved over the time, we will follow the terminologies as described in \cite{Jon03, Gho}, which is also referred by some authors as {\em shaded planar algebras}.  We shall freely borrow some useful terminologies and conventions from \cite{DGG1}, which we briefly recall below:

\begin{enumerate}
\item We will consider the natural binary operation on $\{-,+\}$ given by $++ := +$, $+-:=-$, $-+:=-$ and $--:=+$. Notations such as $(-)^l$ have to be understood in this context.
\item We will denote the set of all possible colors of discs in tangles by $\t{Col}:= \left\{ \vlon k: \vlon \in \{+,-\},\, k \in \mbb{N}_0\right\}$ where $\N_0 := \N \cup \{ 0 \}$.
\item In a tangle, we will replace (isotopically) parallel strings by a single strand labelled by the number of strings, and an internal disc with color $\vlon k$ will be replaced by a bold dot with the sign $\vlon$ placed at the angle corresponding to the distinguished boundary components of the disc.
For example,
\psfrag{e}{$\vlon$}
\includegraphics[scale=0.25]{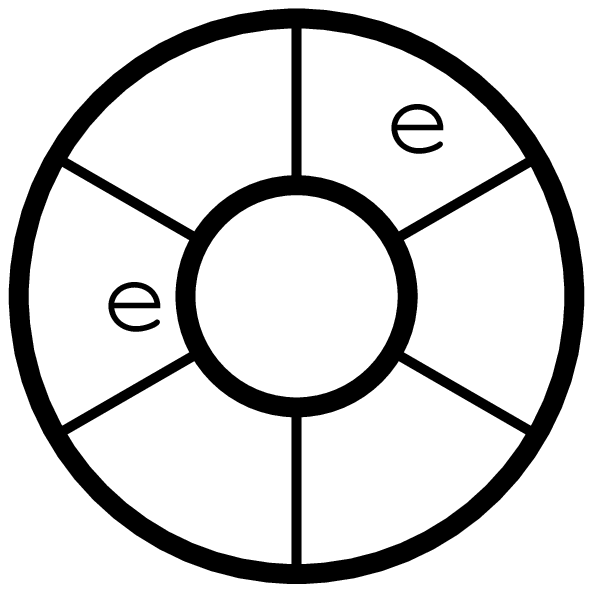}
will be replaced by
\psfrag{2}{$2$}
\psfrag{4}{$4$}
\psfrag{e}{$\vlon$}
\includegraphics[scale=0.25]{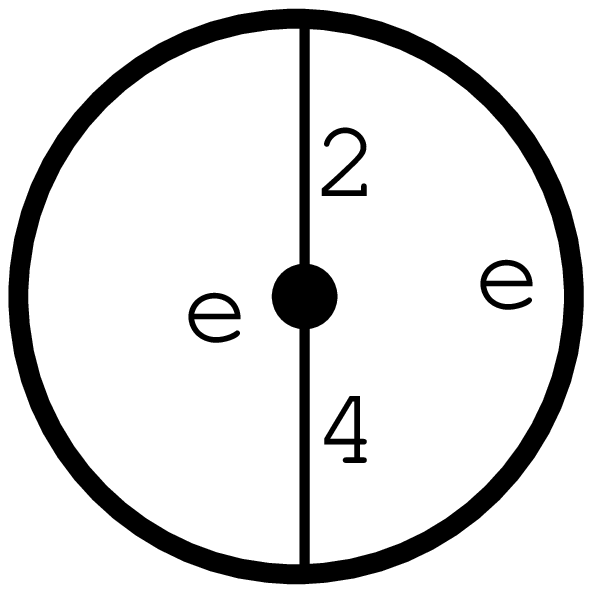}.
In a similar token, if $P$ is a planar algebra, we will replace a $P$-labelled internal disc by a bold dot with the label being placed at the angle corresponding to the distinguished boundary component of the disc; for instance,
\psfrag{x}{$x$}
\psfrag{e}{$\vlon$}
\includegraphics[scale=0.25]{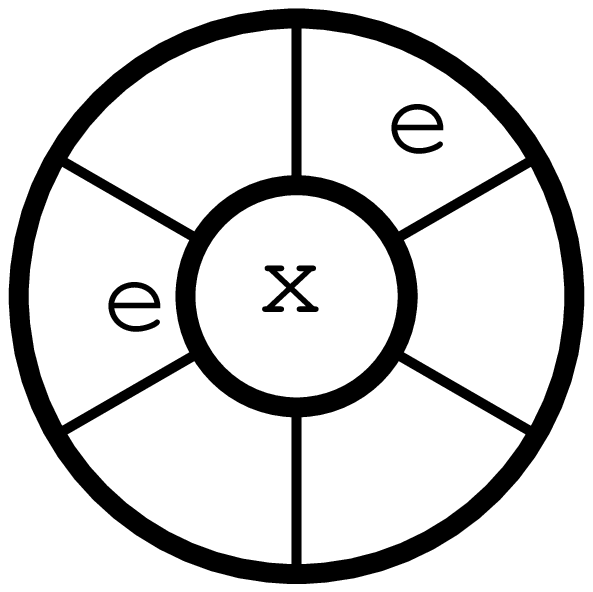}
will be replaced by
\psfrag{x}{$x$}
\psfrag{2}{$2$}
\psfrag{4}{$4$}
\psfrag{e}{$\vlon$}
\includegraphics[scale=0.25]{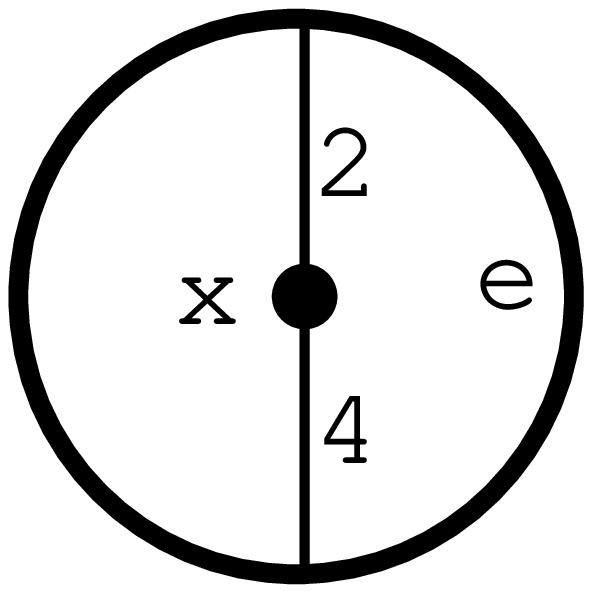}
where $x \in P_{\vlon 3}$.
We will reserve alphabets like $x, y, z$ to denote elements of $P$, $\vlon, \eta, \nu$ to denote a sign, and $k, l, m$ to denote a natural number to avoid confusion.
It should be clear from the context what a bold dot or a string in a picture is labelled by.
\item We set some notation for a set of `generating tangles' (that is, tangles which generate all tangles by composition) in Figure \ref{tangles}.
\begin{figure}[h]
\psfrag{mult}{Multiplication tangle}
\psfrag{unit}{Unit tangle}
\psfrag{identity}{Identity tangle}
\psfrag{jonesproj}{Jones projection tangle}
\psfrag{rightinclusion}{Right inclusion tangle}
\psfrag{leftinclusion}{Left inclusion tangle}
\psfrag{rightcondexp}{Right conditional expectation tangle}
\psfrag{leftcondexp}{Left conditional expectation tangle}
\psfrag{k}{$k$}
\psfrag{2k}{$2k$}
\psfrag{e}{$\vlon$}
\psfrag{-e}{$-\vlon$}
\psfrag{M}{$ M_{\vlon k} =$}
\psfrag{m}{$: (\vlon k,\vlon k)\rightarrow \vlon k$}
\psfrag{1ek}{$1_{\vlon k} =$}
\psfrag{1}{$: \emptyset \rightarrow \vlon k$}
\psfrag{RI}{$RI_{\vlon k} =$}
\psfrag{ri}{$: \vlon k \rightarrow \vlon (k+1)$}
\psfrag{LI}{$LI_{\vlon k}=$}
\psfrag{li}{$: \vlon k \rightarrow -\vlon (k+1)$}
\psfrag{Id}{$I_{\vlon k}=$}
\psfrag{id}{$: \vlon k \rightarrow \vlon k$}
\psfrag{E}{$E_{\vlon (k+1)}=$}
\psfrag{jp}{$: \emptyset \rightarrow \vlon (k+2)$}
\psfrag{RE}{$RE_{\vlon (k+1)} =$}
\psfrag{re}{$: \vlon(k+1) \rightarrow \vlon k$}
\psfrag{LE}{$LE_{\vlon (k+1)} =$}
\psfrag{le}{$: \vlon(k+1) \rightarrow -\vlon k$}
\includegraphics[scale=0.25]{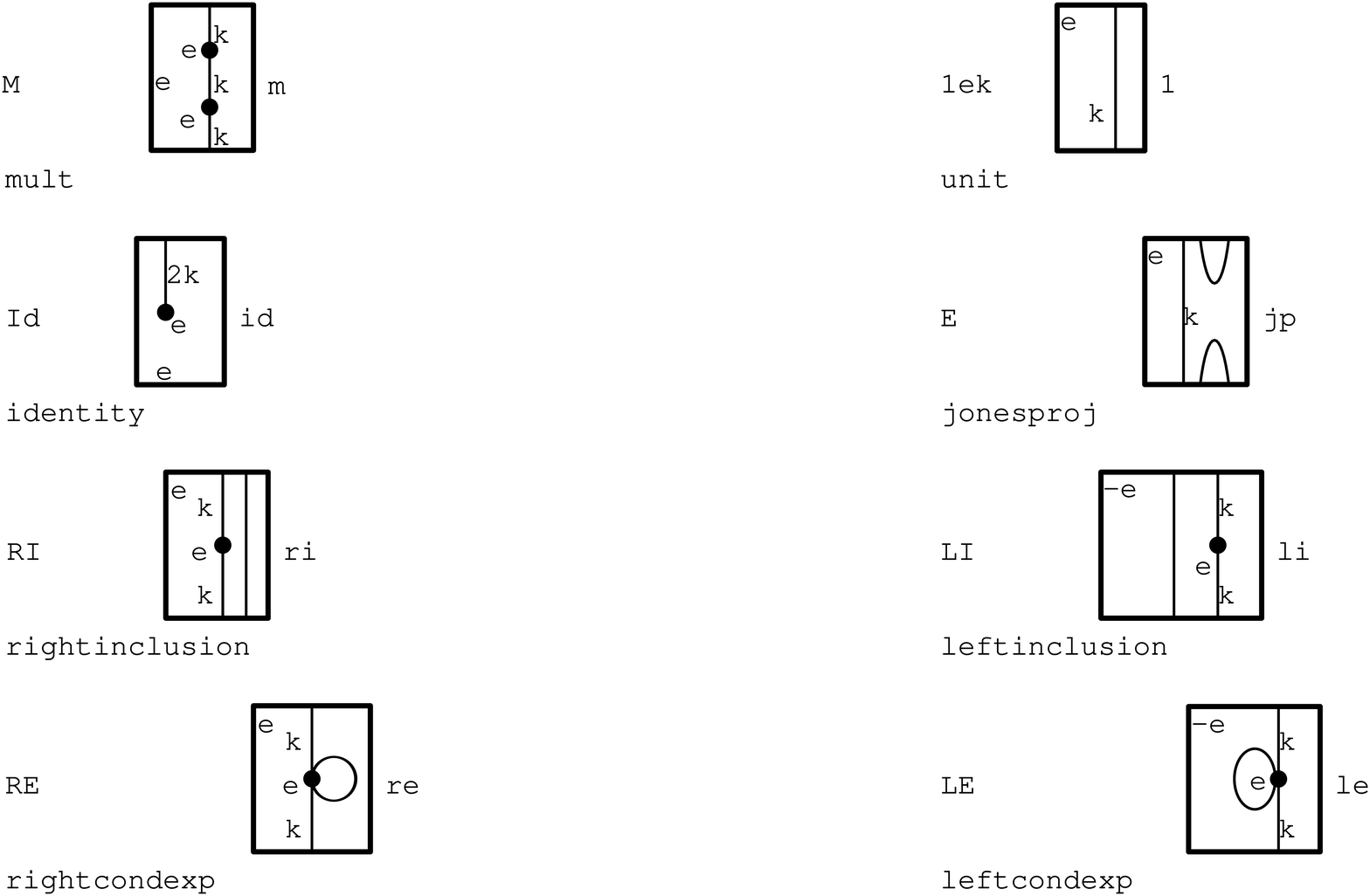}
\caption{Generating tangles.}
\label{tangles}
\end{figure}
\item ${\mcal T}_{\vlon k}$ (resp., ${\mcal T}_{\vlon k} (P)$) will denote the set of tangles (resp., $P$-labelled tangles) which has $\vlon k$ as the color of the external disc; ${\mcal P}_{\vlon k} (P)$ will be the vector space with ${\mcal T}_{\vlon k} (P)$ as a basis.
The action of $P$ induces a linear map ${\mcal P}_{\vlon k} (P) \ni T \os{P}{\longmapsto} P_T \in P_{\vlon k}$.
\end{enumerate}
By a subfactor planar algebra, we will mean the planar algebra associated to a finite index extremal subfactor (as in \cite[Theorem 4.2.1]{Jon}).

\subsection{Affine Category over a  Planar Algebra}\label{affrep}
After the notion of `annular category' over a planar algebra coined by
Jones (\cite{Jon01}), motivated by some natural categorical requirements visible in
some concrete examples, the notion of `affine category' over a planar
algebra evolved in \cite{JR06, Gho}. For our purpose, we will mainly
follow \cite{Gho, DGG} for the terminologies and conventions of
\textit{affine category over a planar algebra} and the corresponding
{\em affine morphisms}. Again, for the sake of convenience, we include
some salient features of this category.

Time and again, to avoid being pedantic, we will abuse terminology by
referring an affine tangular picture as an affine tangle
(corresponding to its affine isotopy class).

In Figure \ref{aff-tangles}, we draw a specific affine tangle called $\Psi_{\varepsilon k, \eta l}^m$ for $\vlon k, \eta l \in \t{Col}$ and $m \in \N_{\vlon , \eta} := 2 \N_0 +\delta_{\vlon, -\eta}$, where a number beside a string has the same meaning as in tangles.
This affine tangle plays a crucial role in the affine category over a planar algebra.
\begin{figure}[h]
\psfrag{AR}{$AR_{\vlon k}$}
\psfrag{y}{$\eta$}
\psfrag{2k-1}{$2k-1$}
\psfrag{2k}{$2k$}
\psfrag{2l}{$2l$}
\psfrag{m}{$m$}
\psfrag{e}{$\varepsilon$}
\psfrag{-e}{$-\varepsilon$}
\psfrag{psi}{$\Psi^m_{\varepsilon k, \eta l}$}
\psfrag{1ek}{$A1_{\varepsilon k}$}
\includegraphics[scale=0.30]{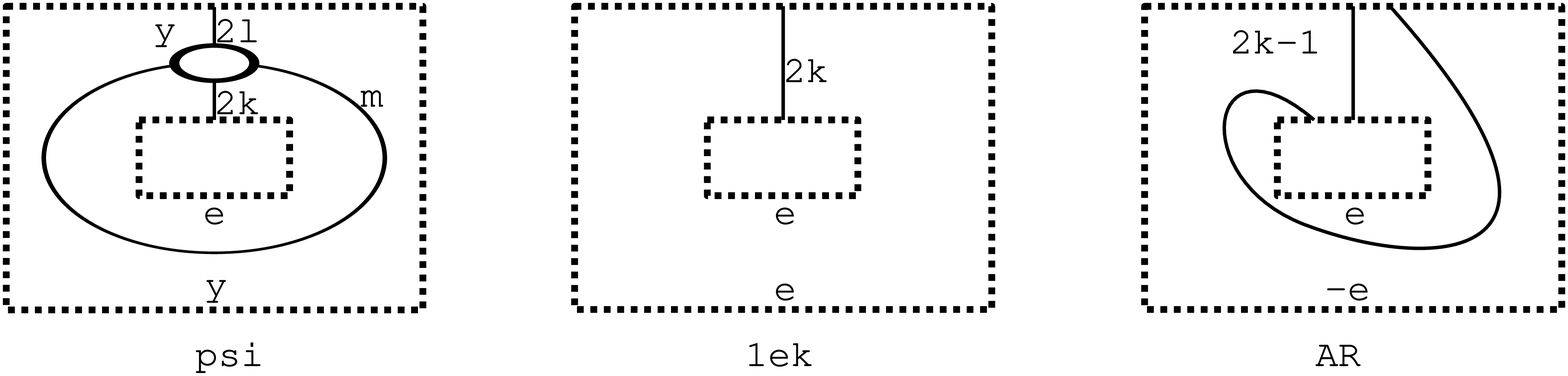}
\caption{Some useful affine tangles. ($\varepsilon , \eta \in \{+, - \}, k, l \in \N_0, m \in \N_{\vlon , \eta}$) \label{aff-tangles}}
\end{figure}

\noindent {\bf Notations:} For each $ \varepsilon k , \eta l\in$ Col, let $\mathcal{AT}_{\varepsilon k, \eta l}$ denote the set of all $(\varepsilon k, \eta l)$-affine tangles, and let $\mcal{A}_{\varepsilon k, \eta l}$ denote the complex vector space with $\mcal{AT}_{\varepsilon k, \eta l}$ as a basis.
The composition of affine tangles $T \in \mathcal{AT}_{\varepsilon k, \eta l}$ and $S
\in \mathcal{AT}_{\xi m, \varepsilon k}$ is given by $T \circ S :=
\frac{1}{2}(2T \cup S) \in \mcal{AT}_{\xi m , \eta l}$ (diagrammatically, it just amounts to plugging in $S$ in the distinguished internal rectangle of $T$ and erasing the boundary); this
composition is linearly extended to the level of the vector spaces
$\mcal{A}_{\vlon k, \eta l}$'s.

The affine tangles $\Psi_{\varepsilon k, \eta l}^m$ in Figure \ref{aff-tangles} have a very useful diagrammatic implication as noted in the Remark below. For details, see \cite[Section 5]{Gho}. 
\begin{rem} \label{Psi-remark}
For each $A \in \mathcal{AT}_{\varepsilon k, \eta l}$, there is an $m
\in \N_{\vlon , \eta}$ and a $T \in
\mcal{T}_{\eta(k + l +m)}$ such that $A = \Psi_{\varepsilon k, \eta
  l}^m (T)$ where $\Psi_{\varepsilon k, \eta l}^m (T)$ is the isotopy
class of the affine tangular picture obtained by inserting $T$ in the
disc of $\Psi_{\varepsilon k, \eta l}^m$.
\end{rem}

Let $P$ be a planar algebra.  An $( \varepsilon k, \eta l)$-affine tangle is said to be $P$-{\em labelled} if each disc is labelled by an element of $P_{\nu m}$ where $\nu m$ is the color of the disc.  Let $\mathcal{AT}_{\varepsilon k, \eta l}(P)$ denote the collection of all $P$-{\em labelled} $(\varepsilon k, \eta l)$-affine tangles, and let $\mathcal{A}_{\varepsilon k, \eta l}(P)$ be the vector space with $\mathcal{AT}_{\varepsilon k, \eta l}(P)$ as a basis.  Composition of $P$-{\em labelled} affine tangles also makes sense as above and extends to their complex span. 

Note that $\Psi_{\varepsilon k, \eta l}^m$ also induces a linear map from $\mathcal{P}_{\eta (k+l+m)}(P)$ into $\mathcal{A}_{\varepsilon k, \eta l}(P)$.
And, from Remark \ref{Psi-remark}, we may also conclude that, for each $A \in \mathcal{A}_{\varepsilon k, \eta l}(P)$, there is an $m \in \N_{\vlon , \eta}$ and an $X \in \mcal{P}_{\eta(k + l +m)}(P)$ such that $A = \Psi_{\varepsilon k, \eta l}^m (X)$.
It can easily be verified that the set $\mathcal{W}_{\varepsilon k,\eta l} := \us{m \in \N_0}{\cup} \left\{ \Psi _{\varepsilon k, \eta l }^{m}(X) : X \in \mathcal{P}_{\eta (k+l+m)}(P) \text{ s.t. } P_X = 0 \right\}$ forms a vector subspace of $\mathcal{A}_{\varepsilon k, \eta l}(P)$.

The category $AP$ is defined by:
\begin{itemize}
\item $\text{ob}(AP) := \{\varepsilon k : \varepsilon \in \{+,-\}, k\in
  \mathbb{N}_{0}\}$,
\item $\text{Mor}_{AP}(\varepsilon k, \eta l) :=
  \frac{\mathcal{A}_{\varepsilon k,\eta l}
    (P)}{\mathcal{W}_{\varepsilon k, \eta l}}$ (also denoted by
  $AP_{\varepsilon k,\eta l}$),
\item composition of morphisms is induced by the composition of
  $P$-{\em labelled} affine tangles (see \cite{Gho}),
\item the identity morphism of $\varepsilon k$ is given by the class
  $[A1_{\varepsilon k}]$ - Figure \ref{aff-tangles}.
\end{itemize}
$AP$ is a $\mbb C$-linear category and is called the \textit{affine
  category over $P$} and the morphisms in this category are called
    {\em affine morphisms}.  (We must add that in the
      Introduction and in \cite{Gho}, the above category was referred
      as $AffP$. We have rechristened it to $AP$ just for the sake of
      convenience.)

For $\vlon k, \eta l \in \t{Col}$ and $m \in \N_{\vlon , \eta}$, consider the composition map
\[
\psi_{\vlon k , \eta l}^{m} \; : \; P_{\eta (k+l+m)} \stackrel{
  I_{\eta (k+l+m)}}{\lra} \mcal{P}_{\eta (k+l+m)}(P) \stackrel{
  \Psi^m_{\vlon k, \eta l}}{\lra} \mcal{A}_{\vlon k, \eta l}(P)
\stackrel{q_{\vlon k, \eta l}}{\ra} AP_{\vlon k, \eta l}
\]
where the map $I_{\eta (k+l+m)} : P_{\eta (k+l+m)} \ra \mcal{P}_{\eta
  (k+l+m)}(P)$ is obtained by labelling the internal disc of the
identity tangle $I_{\eta (k+l+m)}$ (defined in Figure \ref{tangles})
by a vector in $P_{\eta (k+l+m)}$, and $q_{\vlon k , \eta l}:
\mcal{A}_{\vlon k, \eta l}(P) \ra AP_{\vlon k, \eta l}$ is the
quotient map. Note that $\psi^m_{\vlon k, \eta l}$ is indeed linear,
although $I_{\eta (k+l+m)}$ is not.

From the preceding discussion it follows that:
\begin{rem}\label{psi-remark}
For each $a \in AP_{\vlon k, \eta l}$, there exists $m \in \N_{\vlon , \eta}$ and $x \in P_{\eta(k + l +m)}$ such that
$a = \psi_{\vlon k, \eta l}^m (x)$.
\end{rem}

Note that if $P$ is a $\ast $-planar algebra, then each
$\mathcal{P}_{\varepsilon k }(P)$ becomes a $\ast $-algebra where
$\ast $ of a {labelled} tangle is given by $\ast $ of the unlabelled
tangle whose internal discs are labelled with $\ast$ of the labels.
Further, one can define $\ast$ of an affine tangular picture by
reflecting it inside out such that the reflection of the distinguished
boundary segment of any disc becomes the same for the disc in the
reflected picture; this also extends to the $P$-{\em labelled} affine
tangles as in the case of $P$-{\em labelled} tangles.  Clearly, $\ast$
is an involution on the space of $P$-affine tangles, which can be
extended to a conjugate linear isomorphism $\ast :
\mcal{A}_{\varepsilon k, \eta l} (P) \rightarrow \mcal{A}_{\eta l ,
  \varepsilon k} (P)$ for all colours $\varepsilon k$, $\eta l$.
Moreover, it is readily seen that $\ast \left( \mathcal{W}_{\varepsilon k, \eta l}\right) =\mathcal{W}_{\eta l,\varepsilon k}$; so the category $AP$ inherits a $\ast$-category structure.
\begin{defn}
Let $P$ be a planar algebra.
\begin{enumerate}
\item A $\C$-linear functor $V$ from $AP$ to $\mcal{V}ec$ (the category of complex vector spaces), is said to be an `affine $P$-module', that is, there exists a vector space $V_{\vlon k}$ for each $\vlon k \in \text{Col}$ and a linear map $AP_{\vlon k , \eta l} \ni a \os{V}{\longmapsto} V_a \in \text{Mor}_{\mcal{V}ec} (V_{\vlon k} , V_{\eta l}) $ for every $\vlon k , \eta l \in \text{Col}$ such that compositions and identities are preserved. ($V_a$ will be referred as the action of the affine morphism $a$.)
\item Further, for a $*$-planar algebra $P$, 
a `$*$-affine $P$-module' is an additive $*$-functor $V$ from $AP$ to the category of inner product spaces. 
\item A $*$-affine $P$-module $V$ will be called `Hilbert affine $P$-module' if $V_{\vlon k}$'s are Hilbert spaces.
\end{enumerate}
\end{defn}
\comments{
A $\C$-linear functor $V$ from $AP$ to $\mcal{V}ec$ (the category of complex vector spaces), is said to be an {\em affine $P$-module}, that is, there exists vector spaces $V_{\vlon k}$'s and linear maps sending an affine morphism $a \in AP_{\vlon k , \eta l}$ to its action $ V_a
\in \text{Mor}_{\mcal{V}ec} (V_{\vlon k} , V_{\eta l}) $ for $\vlon k
, \eta l \in \text{Col}$ such that compositions and identities are
preserved.  For a $*$-planar algebra $P$, we will be interested in
{\em $*$-affine $P$-modules} preserving the $*$-structure of $AP$,
that is, they will be given by $\C$-linear functors $V$ from $AP$ to
the category of inner product spaces such that $\langle \xi , a \eta
\rangle = \langle a^* \xi , \eta \rangle$ for all affine morphisms
$a$, and $\xi$ and $\eta$ in appropriate $V_{\vlon k}$'s; $V$ will be
called {\em Hilbert affine $P$-module} if $V_{\vlon k}$'s are Hilbert
spaces.  Moreover, a}
Affine modules are called {\em bounded} (resp.,
{\em locally finite}) if the actions of the affine morphisms are
bounded operators with respect to the norm coming from the inner
product (resp., $V_{\vlon k}$'s are finite dimensional).
By closed graph theorem, a Hilbert affine $P$-module is automatically bounded; conversely, every bounded $*$-affine $P$-module can be completed to a Hilbert affine $P$-module.

Below, we give a list of some standard facts on Hilbert affine
$P$-modules for a $*$-planar algebra $P$ with modulus - details can be found in \cite[Section 5]{Gho}. If $V$ is a
Hilbert affine $P$-module and $S_{\vlon k}\subset V_{\vlon k}$ for
$\vlon k\in\mbox{Col}$, then one can consider the `submodule of $V$
generated by $S=\us{\vlon k\in\mbox{Col}}{\coprod}S_{\vlon k}$' given
by $\left\{ \left[S\right]_{\eta l}:=\ol{\mbox{span}\left\{
  \us{\varepsilon k\in\mbox{Col}}{\cup}V_{AP_{\vlon k,\eta l}}\left(S_{\varepsilon k}\right)\right\}
}^{\parallel\cdot\parallel}\right\} _{\eta l}$ which is also the
smallest submodule of $V$ containing $S$.
\begin{rem}\label{irrmod}
Let $V$ be a Hilbert affine $P$-module and $W$ be an $AP_{\vlon k,\vlon k}$-submodule
of $V_{\vlon k}$ for some $\vlon k\in\mbox{Col}$. Then,
\begin{enumerate}
\item $V$ is irreducible if and only if $V_{\vlon k}$ is irreducible $AP_{\vlon k,\vlon k}$-module
for all $\vlon k\in\mbox{Col}$ if and only if $\left[v\right]=V$
for all $0\neq v\in V$.
\item $W$ is irreducible $\Leftrightarrow$ $\left[W\right]$ is an irreducible
submodule of $V$.\end{enumerate}
\end{rem}
\begin{rem}\label{affmodmor}
If $V$ and $W$ are Hilbert affine $P$-modules such that there exists an $\vlon k \in \mbox{Col}$ and an $AP_{\vlon k,\vlon k}$-linear isometry $\theta : V_{\vlon k} \ra W_{\vlon k}$, and $V = \left[ V_{\vlon k} \right]$, then $\theta$ extends uniquely to an 
$AP$-linear isometry $\tilde{\theta}: V \ra W$.
\end{rem}
For $\vlon \in \{ +, - \}$, we will also consider \emph{Hilbert $\vlon$-affine
$P$-module} $V$ consisting of the Hilbert spaces $V_{\pm0},V_{1},V_{2},\ldots$ equipped with a $*$-preserving action of affine morphisms as follows: 
\[
\left.
\begin{tabular}{l}
$AP_{\vlon k,\vlon l} \times V_k \ra V_l$\\
$AP_{\vlon k,\eta 0} \times V_k \ra V_{\eta 0}$\\
$AP_{\eta 0, \vlon l} \times V_{\eta 0} \ra V_l$\\
$AP_{\eta 0, \nu 0} \times V_{\eta 0} \ra V_{\nu 0}$
\end{tabular}
\right\} \t{ for all } k, l \in \N,\; \eta, \nu \in \{\pm\}.
\]
\begin{rem}\label{eaff}
The restriction map from the set of isomorphism classes of Hilbert
affine $P$-modules to that of the Hilbert $\vlon$-affine $P$-modules,
is a bijection.
\end{rem}
 To see this, consider an irreducible Hilbert $+$-affine $P$-module $V$.
Define $(\t{Ind} \, V)_{\vlon k}:=V_{k}$ and $(\t{Ind} \, V)_{\vlon0}:=V_{\vlon0}$
(as Hilbert spaces) and the action of affine morphisms by  $AP_{\vlon k,\eta l}\times (\t{Ind} \, V)_{\vlon k} \ni (a,v) \longmapsto \left(r_{\eta l}\circ a\circ r_{\vlon k}^{-1}\right)\cdot v\in (\t{Ind} \, V)_{\eta l}$
where $r_{\nu s}=\left\{ \begin{array}{ll}
A1{}_{\nu s}, & \mbox{ if }s=0\mbox{ or }\nu=+,\\
AR_{\nu s}, & \mbox{ otherwise;}\end{array}\right.$\\
$A1{}_{\nu s}$ and $AR_{\nu s}$ being the affine tangles mentioned in Figure \ref{aff-tangles}.

For every affine $P$-module $V$, $\mbox{dim}\left(V_{+k}\right)=\mbox{dim}\left(V_{-k}\right)$ for all $k\geq1$ and it increases as $k$ increases. \emph{Weight of $V$} is defined as the smallest non-negative integer $k$ such that $V_{+k}$ or $V_{-k}$ is nonzero.
$V$ is said to have {\em finite support} if there exists $k \in \N_0$ such that $[V_{\pm k}] = V$ and in this case, the smallest such $k$ is called the {\em support of $V$}. Thus, the support of any Hilbert affine $P$-module is the maximum of the weights of the submodules of $V$; this follows from the fact that any Hilbert affine $P$-module $V$ has a unique decomposition (upto isomorphism) $V \cong \us{k \in \N_0} \bigoplus V^k$ where $V^k$ is either the zero module or a submodule (of $V$) with both weight and support being $k$.

\section{Structure of the space of affine morphisms}\label{straffmor}
In this section, we will analyze the space of affine morphisms and prove some of their basic properties which will be very useful in the later sections. Throughout this section, $P$ will denote a (spherical) subfactor planar algebra (possibly having infinite depth).
The analysis of affine morphisms will also show that they are intimately linked with the commutativity constraints appearing in Drinfeld's center construction. Some of the materials in this section are generalization of our analysis of affine morphisms at zero level as discussed in \cite{DGG}.

For $\vlon k , \eta l \in \t{Col}$, consider the set $\mcal S_{\vlon k, \eta l} := \us{m \in \N_{\vlon,\eta} }{\sqcup} \mcal T_{\eta (k+l+m)} (P)$.
Define an equivalence $\sim$ on $\mcal S_{\vlon k , \eta l}$ generated by
\[
\psfrag{eta}{$\eta$}
\psfrag{t}{$T$}
\psfrag{s}{$S$}
\psfrag{~}{$\sim$}
\psfrag{2k}{$2k$}
\psfrag{2l}{$2l$}
\psfrag{m}{$m$}
\psfrag{n}{$n$}
\includegraphics[scale=0.15]{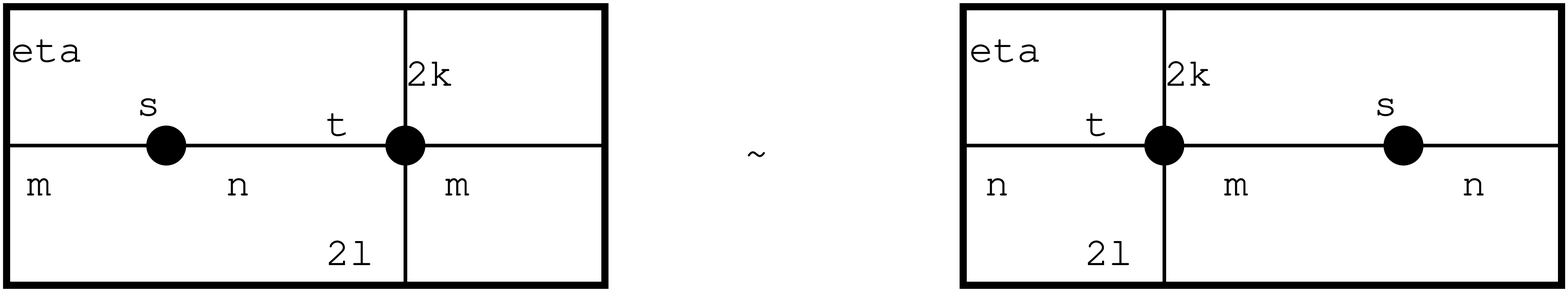}
\]
for all $m, n \in \N_{\vlon,\eta}$, $T \in \mcal T_{\eta (k+l+\frac{m+n}{2})} (P)$ and $S \in \mcal T_{\eta (\frac{m+n}{2})} (P)$. The following lemma is topological in nature and is basically some sort of generalization of what was already established first in \cite{Gho06} and then in \cite{DGG}.
\begin{lem}\label{psi-eq}
The map $\us{m \in \N_{\vlon,\eta} }{\sqcup} \Psi^m_{\vlon k , \eta l} : \mcal S_{\vlon k,\eta l} \longrightarrow \mcal A \mcal T _{\vlon k , \eta l} (P)$ induces a bijection from $\left. \mcal S_{\vlon k , \eta l} \right/ \sim$ to $\mcal A \mcal T _{\vlon k , \eta l} (P)$.
\end{lem}
\begin{proof}
This lemma is already proved in \cite[Lemma 3.7]{DGG} for a specific case, namely, $k=l=0$.
In fact, the same proof actually works even when $k$ or $l$ is nonzero.
\end{proof}
\begin{rem}\label{apmap}
An immediate consequence of Lemma \ref{psi-eq} is the following:

For any vector space $V$, if $f : \mcal S _{\vlon k , \eta l} \ra V$ is a map which is (i) invariant under $\sim$ and (ii) preserves the kernel of the action of $P$ (that is, $\sum_i \alpha_i f(T_i) = 0$ whenever $\sum_i\alpha_i P_{T_i} = 0$), then $f$ induces a unique linear map
$AP_{\vlon k, \eta l} \ni \psi^{m}_{\vlon k , \eta l} (x) \longmapsto f (I_{\eta (k + l + m)} (x)) \in V$ where $I$ denotes the identity tangle described in Figure \ref{tangles}.
\end{rem}
Next, we consider the map
\[
AP_{\vlon k,\vlon k} \ni a \os{\gamma_{\vlon k}}{\longmapsto} \gamma_{\vlon k} (a) := \delta^{-l} \us{\alpha}{\sum} P_{
\psfrag{eta}{$\eta$}
\psfrag{e}{$\vlon$}
\psfrag{u}{$u_\alpha$}
\psfrag{u*}{$u^*_\alpha$}
\psfrag{2k}{$2k$}
\psfrag{2l}{$2l$}
\psfrag{x}{$x$}
\includegraphics[scale=0.15]{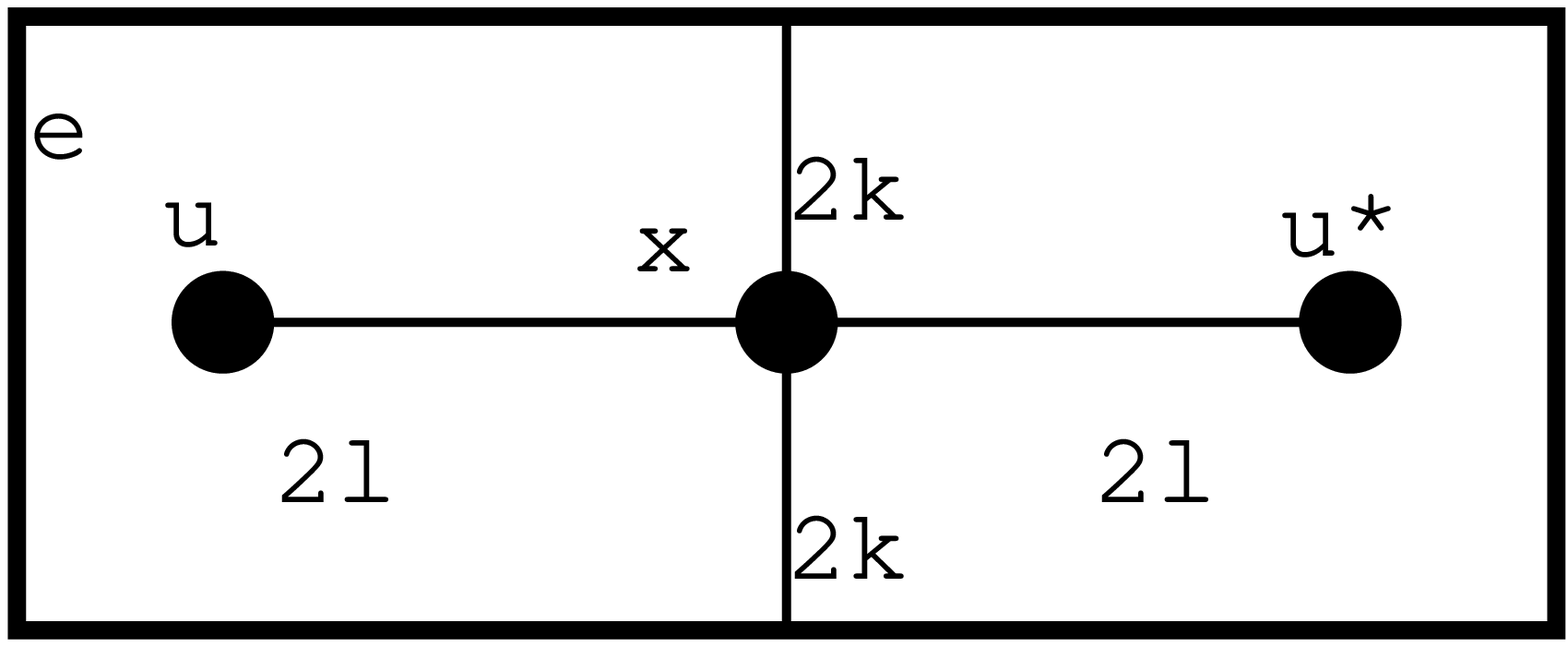}
} \in P_{\vlon 2k}
\]
where $x \in P_{\vlon (2(k+l)}$ such that $a = \psi^{2l}_{\vlon k, \vlon k} (x)$ and $\{ u_\alpha \}_\alpha$ is  an (in fact, any) orthonormal basis of $P_{\vlon l}$ with respect to the canonical trace (that is, the normalized picture trace).
\begin{lem}\label{gamma}
$\gamma_{\vlon k} : AP_{\vlon k, \vlon k} \rightarrow P_{\vlon 2k}$ is well-defined.
\end{lem}
\begin{proof}
Define $\mcal S_{\vlon k , \vlon k} \supset \mcal T_{\vlon 2(k+l)} (P) \ni T \os{\ol{\gamma}_{\vlon k}}{\longmapsto} \delta^{-l} \us{\alpha}{\sum} P_{
\psfrag{eta}{$\eta$}
\psfrag{e}{$\vlon$}
\psfrag{u}{$u_\alpha$}
\psfrag{u*}{$u^*_\alpha$}
\psfrag{2k}{$2k$}
\psfrag{2l}{$2l$}
\psfrag{x}{$P_T$}
\includegraphics[scale=0.15]{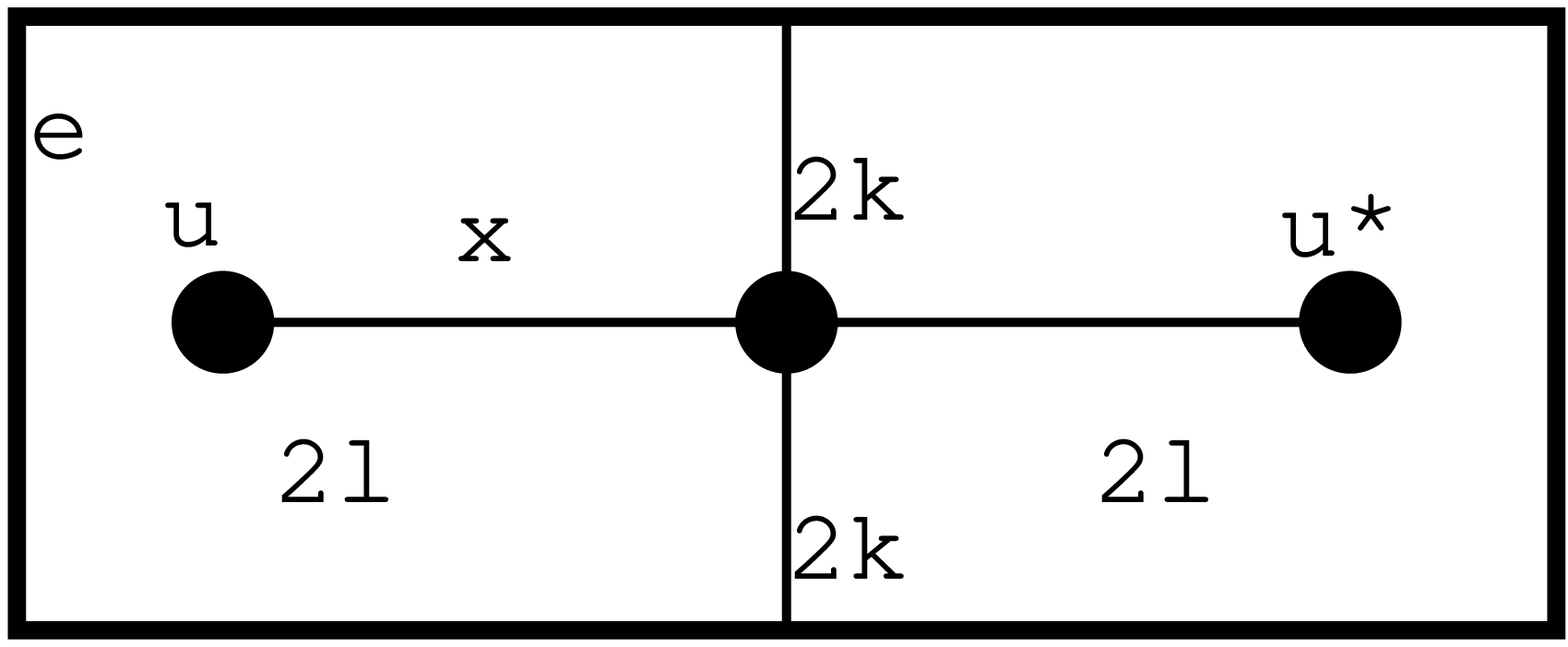}
} \in P_{\vlon 2k}$.
Using simple linear algebra, one can verify that $\ol{\gamma}_{\vlon k}$ is invariant under $\sim$, and hence, by Remark \ref{apmap}, induces a map $\tilde \gamma_{\vlon k} : \mcal A \mcal T_{\vlon k,\vlon k} (P) \rightarrow P_{\vlon 2k}$, that is,
\[
\xymatrix{
\mcal{AT}_{\vlon k, \vlon k}(P) \ar[rr]^{\normalsize {\widetilde \gamma_{\vlon k}} }  & & P_{\vlon 2k}\\
& \mcal{S}_{\vlon k, \vlon k} \ar[lu]^{\us{l \in \N_0}{\sqcup} \Psi^{2l}_{\vlon k, \vlon k}} \ar[ru]_{\bar{\gamma}_{\vlon k}}&} \t{ commutes.}
\]
Extend $\tilde \gamma_{\vlon k}$ linearly to $\tilde \gamma_{\vlon k} : \mcal A_{\vlon k, \vlon k}(P) \rightarrow P_{\vlon k}$.
Note that $\tilde \gamma_{\vlon k} \left( \Psi^{2l}_{\vlon k , \vlon k} (X) \right) = \delta^{-l} \us{\alpha}{\sum} P_{
\psfrag{eta}{$\eta$}
\psfrag{e}{$\vlon$}
\psfrag{u}{$u_\alpha$}
\psfrag{u*}{$u^*_\alpha$}
\psfrag{2k}{$2k$}
\psfrag{2l}{$2l$}
\psfrag{x}{$P_X$}
\includegraphics[scale=0.15]{figures/straffmor/gek2.eps}
}$ for all $l \in \N_0$ and $X \in \mcal P_{\vlon 2(k+l)} (P)$; this also implies $\tilde \gamma ( \mcal W_{\vlon k , \vlon k} ) = \{ 0 \}$.
Hence, $\tilde \gamma_{\vlon k}$ induces a well-defined map from the quotient $AP_{\vlon k,\vlon k}$ to $P_{\vlon 2k}$, namely, $\gamma_{\vlon k}$.
\end{proof}
Define $\omega_{\vlon k} := \t{tr} \circ \gamma_{\vlon k} : AP_{\vlon k,\vlon k} \rightarrow \C$ where $\t{tr}$ is the canonical trace on $P_{\vlon k}$ (that is, the normalized picture trace).

\begin{rem}\label{tromega}
Using (a) the invariance of the action under spherical isotopy and (b) unitarity of the action of the rotation tangles, one can easily show that $\delta^{2k} \omega_{\vlon k} (a \circ b) = \delta^{2l} \omega_{\eta l} (b \circ a)$ for all $a \in AP_{\eta l , \vlon k}$, $b \in AP_{\vlon k , \eta l}$; thus, $\omega_{\vlon k}$ is a trace on $AP_{\vlon k,\vlon k}$.
\end{rem}
We now proceed towards obtaining a stronger version of Remark \ref{psi-remark}. For this, we will use the following fact:
\begin{fact}\label{factz}
If $A$ is a finite dimensional $C^*$-algebra with a faithful tracial state $\tau$ and $\{w_\alpha\}_\alpha$ is an orthonormal basis of $A$ with respect to $\tau$, then $z:= \us{\alpha}{\sum} w_\alpha w^*_\alpha$ is positive, central, invertible and independent of the choice $\{w_\alpha\}_\alpha$ of the orthonormal basis.
\end{fact}
Let $z_{\eta m} := \us{\alpha}{\sum} u_\alpha u^*_\alpha$ where $\{u_\alpha\}_\alpha$ is any orthonormal basis of $P_{\eta m}$ with respect to the trace $\t{tr}$.
\begin{lem}\label{wlogx}
For all $\vlon k, \eta l \in \t{Col}$, $m \in \N_{\vlon,\eta}$ and $x \in P_{\eta (k+l+m)}$, there exists $y \in P_{\eta (k+l+m)}$ such that (i) $\psi^m_{\vlon k,\eta l} (x) = \psi^m_{\vlon k,\eta l} (y)$ and (ii) $P_{
\psfrag{eta}{$\eta$}
\psfrag{u*}{$u^*_\alpha$}
\psfrag{2k}{$2k$}
\psfrag{2l}{$2l$}
\psfrag{m}{$m$}
\psfrag{y}{$y$}
\psfrag{w}{$w$}
\includegraphics[scale=0.15]{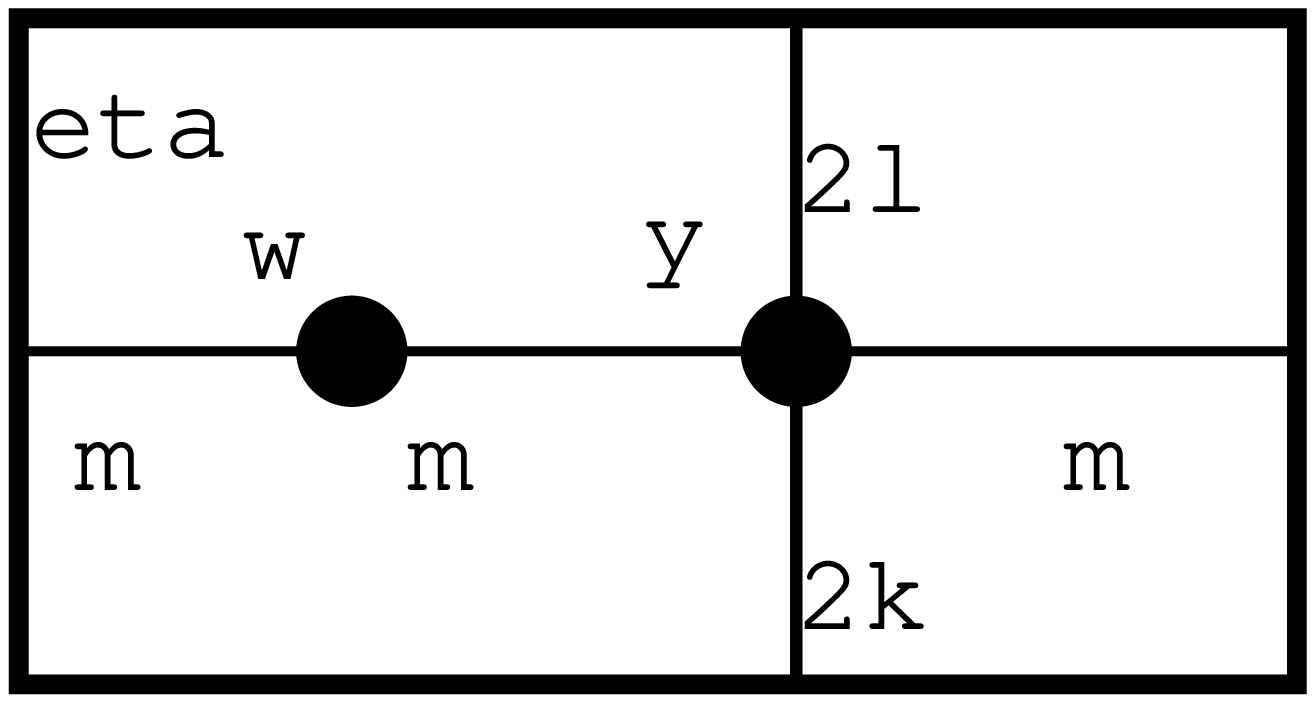}
} = P_{
\psfrag{eta}{$\eta$}
\psfrag{u*}{$u^*_\alpha$}
\psfrag{2k}{$2k$}
\psfrag{2l}{$2l$}
\psfrag{m}{$m$}
\psfrag{y}{$y$}
\psfrag{w}{$w$}
\includegraphics[scale=0.15]{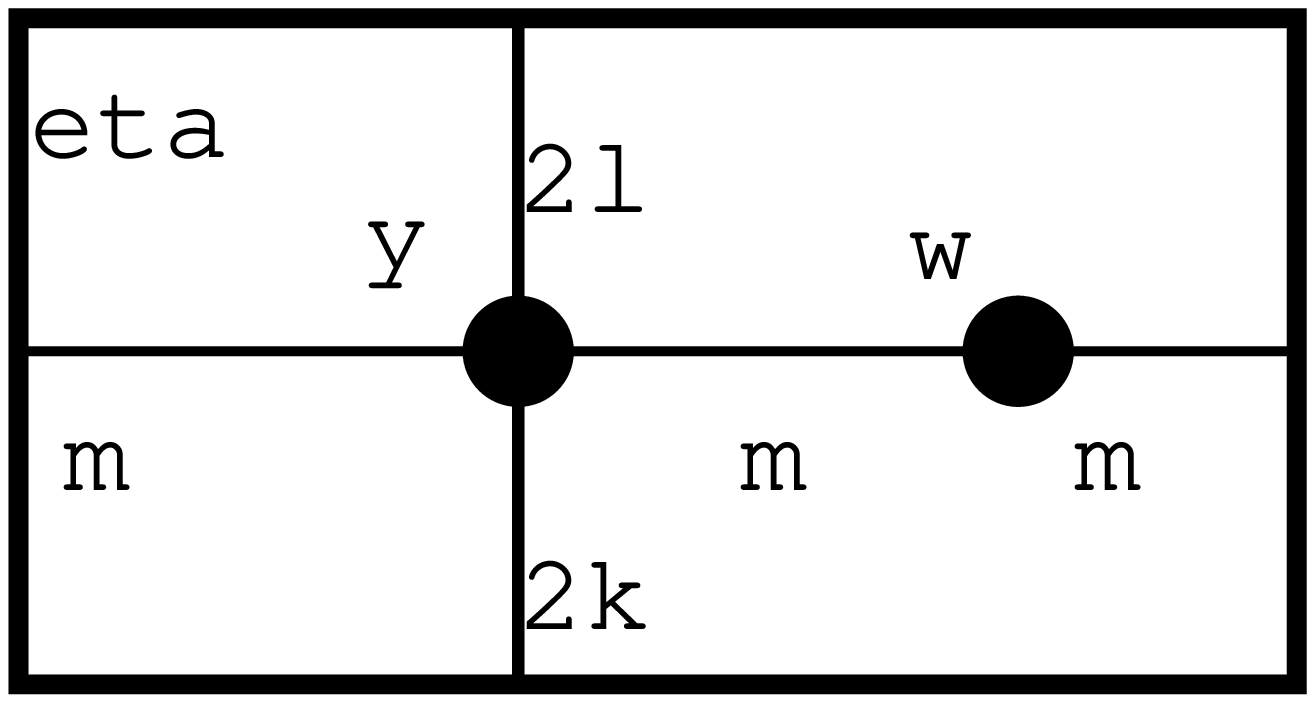}
}$ for all $w \in P_{\eta m}$.
\end{lem}
\begin{proof}
We set $y := \us{\alpha}{\sum} P_{
\psfrag{eta}{$\eta$}
\psfrag{u}{$u_\alpha$}
\psfrag{u*}{$u^*_\alpha$}
\psfrag{2k}{$2k$}
\psfrag{2l}{$2l$}
\psfrag{m}{$m$}
\psfrag{x}{$x$}
\psfrag{z}{$z^{-1}_{\eta m}$}
\includegraphics[scale=0.15]{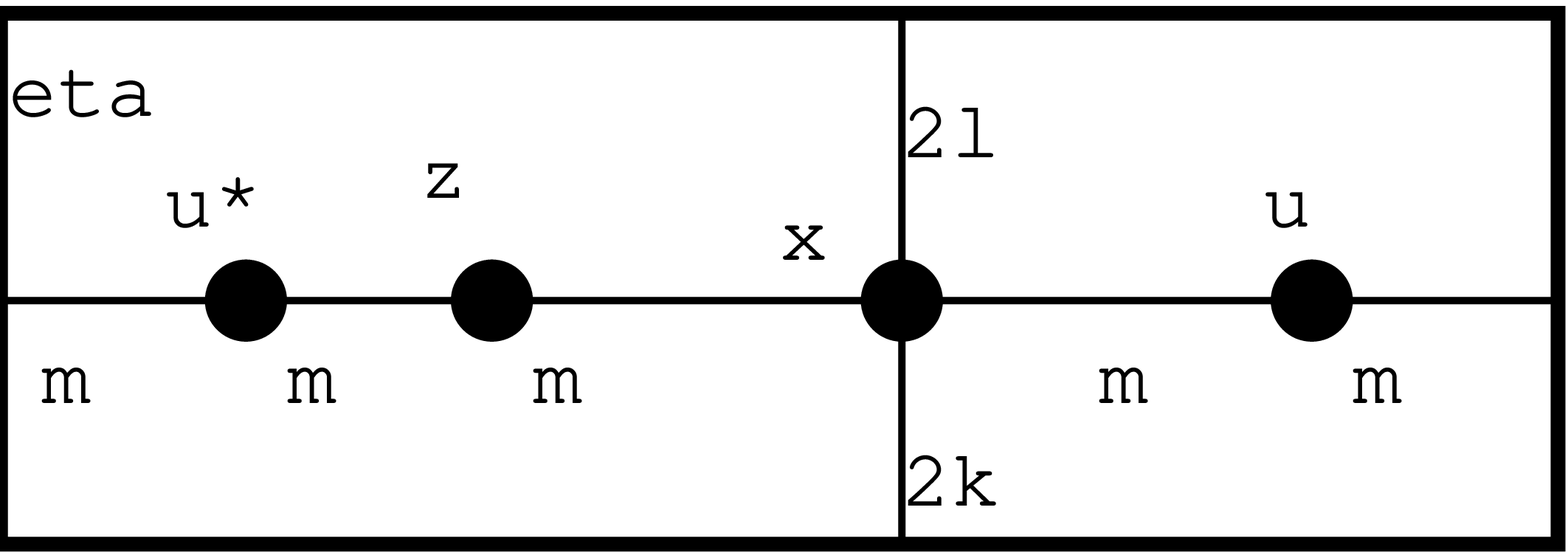}
}$.
(i) immediately follows from Fact \ref{factz}.
For obtaining (ii), we use the Fact \ref{factz} again and also the equation $w u_\alpha = \us{\alpha'}{\sum} \t{tr} (u^*_{\alpha'} w u_\alpha) u_{\alpha'}$ for all $\alpha$.
\end{proof}
We now define a sesquilinear form (conjugate-linear in the first variable) $\lab \cdot , \cdot \rab_{\vlon k}$ on $AP_{\vlon k, \eta l}$ given by
\[
AP_{\vlon k, \eta l} \times AP_{\vlon k, \eta l} \ni (a,b) \os{\lab \cdot , \cdot \rab_{\vlon k}}{\longmapsto} \lab a , b \rab_{\vlon k} := \omega_{\vlon k} (a^* \circ b) \in \C.
\]
\begin{prop}\label{pdip}
$\lab \cdot , \cdot \rab_{\vlon k}$ is positive definite.
\end{prop}
\begin{proof}
Let $0 \neq a \in AP_{\vlon k, \eta l}$.
By Remark \ref{psi-remark} and Lemma \ref{wlogx}, there exists $m \in \N_{\vlon,\eta}$ and $0 \neq y \in P_{\eta (k+l+m)}$ such that $a = \psi^m_{\vlon k, \eta l} (y)$ and condition (ii) of Lemma \ref{wlogx} holds.
Note that if $\{u_\alpha\}_\alpha$ is an orthonormal basis of $P_{\vlon m}$ with respect to $\t{tr}$
\[
\lab a , a \rab_{\vlon k} = \delta^{-(2k+m)} \us{\alpha}{\sum} P_{
\psfrag{u}{$u_\alpha$}
\psfrag{u*}{$u^*_\alpha$}
\psfrag{2k}{$2k$}
\psfrag{2l}{$2l$}
\psfrag{m}{$m$}
\psfrag{x}{$y$}
\psfrag{x*}{$y^*$}
\includegraphics[scale=0.15]{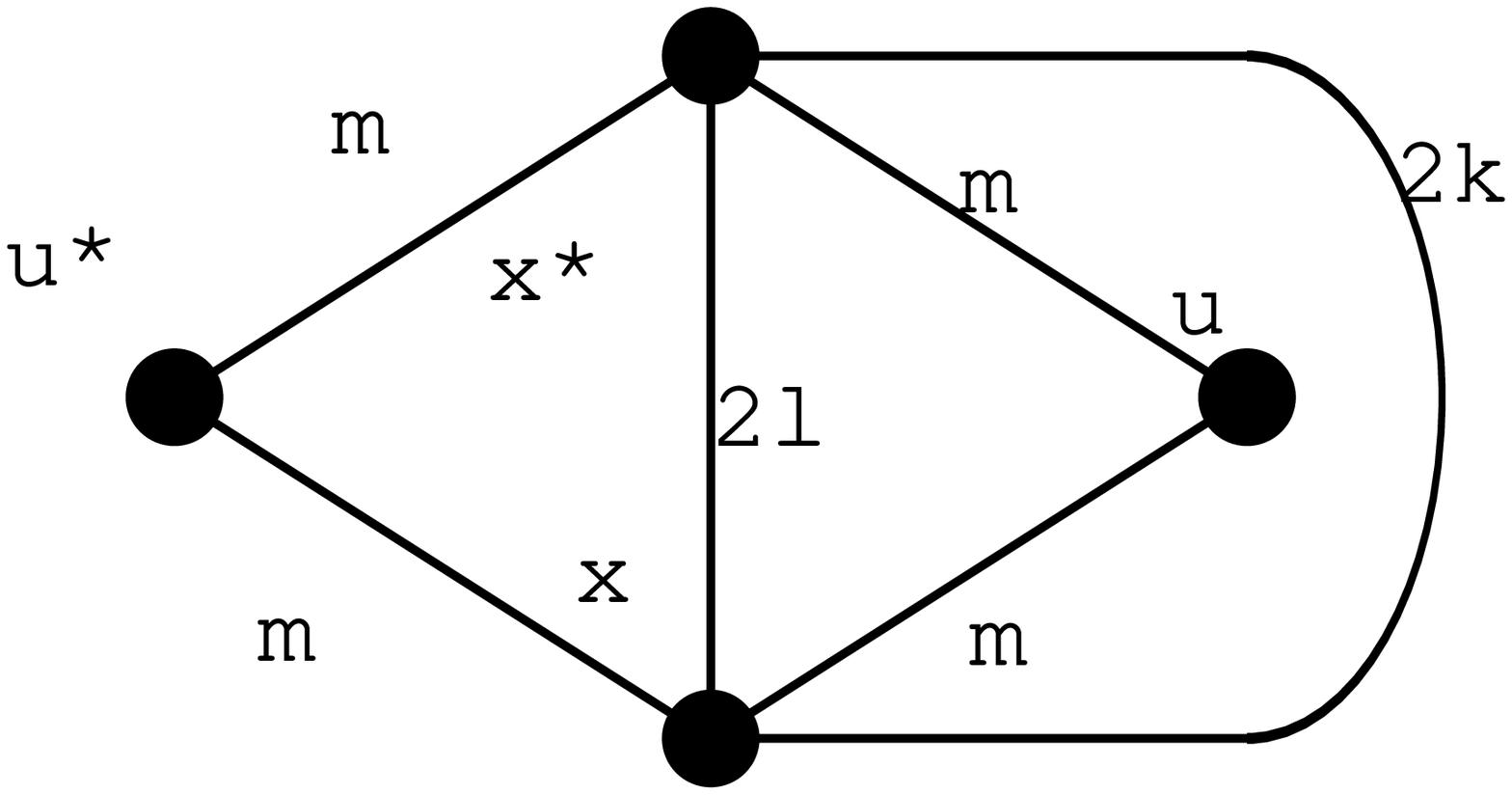}
} \; =  \delta^{-(2k+m)} P_{
\psfrag{z}{$z_{\vlon m}$}
\psfrag{2l}{$2(k+l)+m$}
\psfrag{m}{$m$}
\psfrag{x}{$y$}
\psfrag{x*}{$y^*$}
\includegraphics[scale=0.15]{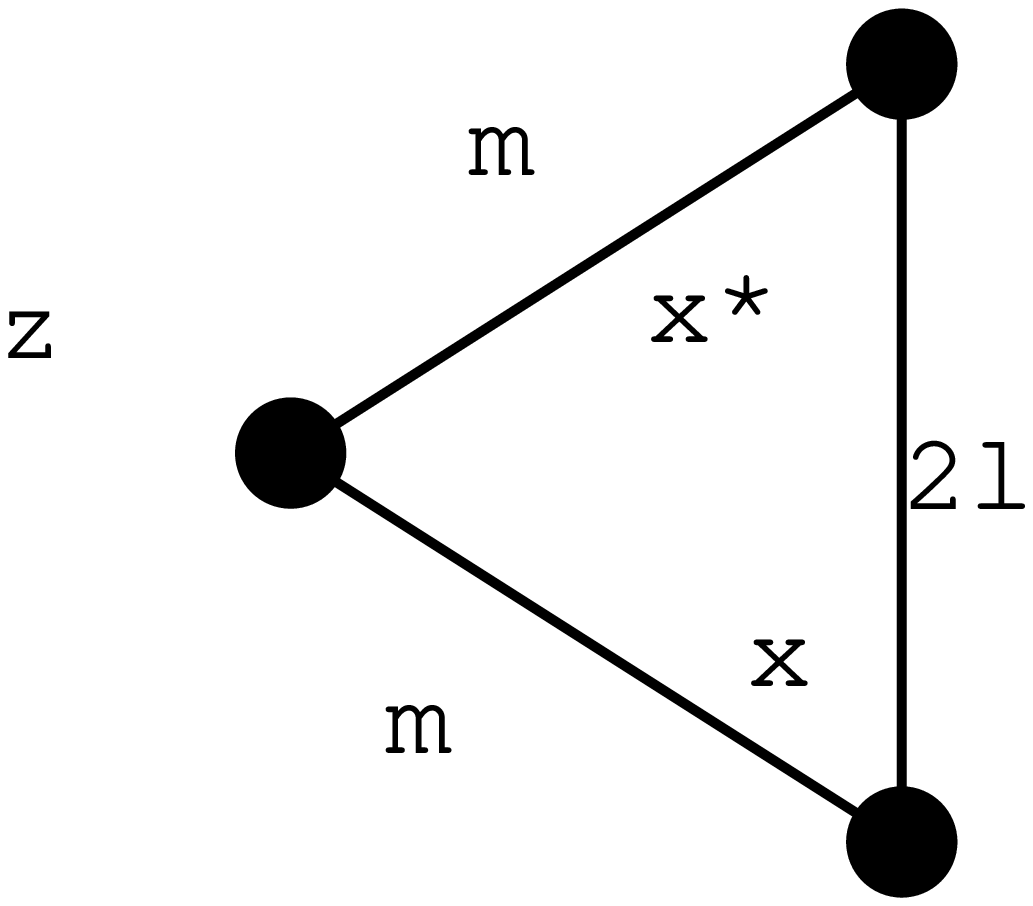}
} \;\;\;\;\;\;\;\;\;\;\;\;\;\;\;\;\;\;\;\; > 0
\]
where the second equality follows from condition (ii) of Lemma \ref{wlogx} and the last inequality can be obtained using (a) positivity of the action of the trace tangle, (b) $y \neq 0$ and (c) $z_{\vlon m}$ being a positive and invertible element of $P_{\vlon m}$.
\end{proof}
Set $Q^m_{\vlon k, \eta l} := \{ y \in P_{\eta (k+l+m)} :  P_{
\psfrag{eta}{$\eta$}
\psfrag{u*}{$u^*_\alpha$}
\psfrag{2k}{$2k$}
\psfrag{2l}{$2l$}
\psfrag{m}{$m$}
\psfrag{y}{$y$}
\psfrag{w}{$w$}
\includegraphics[scale=0.15]{figures/straffmor/wlog.eps}
} = P_{
\psfrag{eta}{$\eta$}
\psfrag{u*}{$u^*_\alpha$}
\psfrag{2k}{$2k$}
\psfrag{2l}{$2l$}
\psfrag{m}{$m$}
\psfrag{y}{$y$}
\psfrag{w}{$w$}
\includegraphics[scale=0.15]{figures/straffmor/wlog2.eps}
} \t{ for all } w \in P_{\eta m} \}$ for all $\vlon k, \eta l \in \t{Col}$ and $m\in \N_{\vlon,\eta}$.
Clearly, $Q^m_{\vlon k,\eta l}$ is a subspace of $P_{\eta (k+l+m)}$.
\begin{cor}\label{wlogx11}
The map $Q^m_{\vlon k, \eta l} \ni x \os{\psi^m_{\vlon k, \eta l} }{\longmapsto} \psi^m_{\vlon k, \eta l}(x) \in AP_{\vlon k , \eta l}$ is injective linear map.
\end{cor}
\begin{proof}
The proof of this fact already appeared in the proof of Proposition \ref{pdip}.
\end{proof}
\begin{cor}
(i) $\omega_{\vlon k}$ is a faithful tracial state on $AP_{\vlon k,\vlon k}$.

(ii) The map $P_{\vlon 2k} \ni x \os{\psi^0_{\vlon k, \vlon k} }{\longmapsto} \psi^0_{\vlon k, \vlon k} (x) \in AP_{\vlon k, \vlon k}$ is a trace preserving ($\omega_{\vlon k}$ on $AP_{\vlon k, \vlon k}$ and $\t{tr}$ on $P_{\vlon 2k}$) inclusion of unital $*$-algebras.

(iii) $\gamma_{\vlon k} : AP_{\vlon k, \vlon k} \ra P_{\vlon 2k}$ is the unique trace preserving conditional expectation.
\end{cor}
\begin{proof}
(i) Faithfulness and positivity can be concluded from Proposition \ref{pdip} by considering $\eta l = \vlon k$.
Indeed, $\omega_{\vlon k} (1_{AP_{\vlon k , \vlon k} }) = 1$.  And traciality follows readily from pictures and a judicious choice of an orthonormal basis for $P_{\vlon (m +n)}$ combining those of $P_{\vlon m}$ and $P_{\vlon n}$.

(ii) It is obvious that the map is a homomorphism of unital $*$-algebras.
Corollary \ref{wlogx11} implies its injectivity.
Note that $\gamma_{\vlon k} \circ \psi^0_{\vlon k, \vlon k} = \t{id}_{P_{\vlon 2k}}$; this implies that the traces are preserved.

(iii) From the definition of $\gamma_{\vlon k}$, we easily get the relation $y \gamma_{\vlon k} (\psi^{2l}_{\vlon k , \vlon k} (x)) = \gamma_{\vlon k} (\psi^0_{\vlon k,\vlon k} (y) \circ \psi^{2l}_{\vlon k,\vlon k} (x))$ for all $l \in \N_0$, $x \in P_{\vlon 2(k+l)}$ and $y \in P_{\vlon 2k}$, and then, apply $\t{tr}$ on both sides to get the required result.
\end{proof}
The linear map in Corollary \ref{wlogx11} is surjective when $P$ has finite depth and $m$ is sufficiently large ($> \frac{1}{2} \t{depth}(P)$); this easily follows from the proof of \cite[Proposition 6.8]{Gho}; however, this is not the case in general.
In order to understand the general case, we introduce a few definitions.

For $\vlon k, \eta l \in \t{Col}$ and $m \in \N_{\vlon , \eta}$, we consider the subspace $AP^{\leq m}_{\vlon k, \eta l} := \t{Range } \psi^m_{\vlon k,\eta l} \subset AP_{\vlon k, \eta l}$.
Note that $AP^{\leq m}_{\vlon k, \eta l} \subset AP^{\leq m+2}_{\vlon k, \eta l}$ and $AP_{\vlon k, \eta l} = \us{m \in \N_{\vlon , \eta} } \cup AP^{\leq m}_{\vlon k, \eta l} $.
We further consider the subspace $AP^{=m}_{\vlon k, \eta l} := AP^{\leq m}_{\vlon k, \eta l} \ominus AP^{\leq m-2}_{\vlon k, \eta l}$ by taking the orthogonal complement with respect to the inner product $ \lab \cdot , \cdot  \rab_{\vlon k}$ and setting $AP^{\leq m-2}_{\vlon k, \eta l} = \{0\}$ when $m<2$.
Clearly, for $m \in \N_{\vlon ,\eta}$, we have $
Q^m_{\vlon k , \eta l} \os{\psi^m_{\vlon k , \eta l}} \cong AP^{ \leq m}_{\vlon k, \eta l} = \us{m \geq n \in \N_{\vlon , \eta} } \oplus AP^{=n}_{\vlon k, \eta l}$; in particular, $AP_{\vlon k, \eta l} = \us{m \in \N_{\vlon , \eta} } \oplus AP^{=m}_{\vlon k, \eta l}$ ($\oplus$ denotes   vector space (and not Hilbert space)  direct sum.) Also, $\left[ AP^{\leq m}_{\vlon k , \eta l} \right]^* = AP^{\leq m}_{ \eta l , \vlon k}$ and $\left[ AP^{= m}_{\vlon k , \eta l} \right]^* = AP^{= m}_{ \eta l , \vlon k}$, and these spaces are finite dimensional (where $*$ denotes the contravariant involution and not the dual).

 Motivated by the commutativity constraints (also referred as {\em half braidings}) in the
  Drinfeld center of the pictorial version of the category of $N$-$N$ bimodules, we collect some interesting
  pictorial properties, which we call  the {\em space of commutativity constraints},  as follows: 
\[
CC_{\vlon k, \eta l} := \left\{ \left. c \in \us{m \in \N_{\vlon , \eta} } \prod  P_{\eta (k+l+m)} \right| P_{
\psfrag{eta}{$\eta$}
\psfrag{w}{$w$}
\psfrag{2k}{$2k$}
\psfrag{2l}{$2l$}
\psfrag{m}{$m$}
\psfrag{n}{$n$}
\psfrag{cm}{}
\psfrag{cn}{$c_n$}
\includegraphics[scale=0.15]{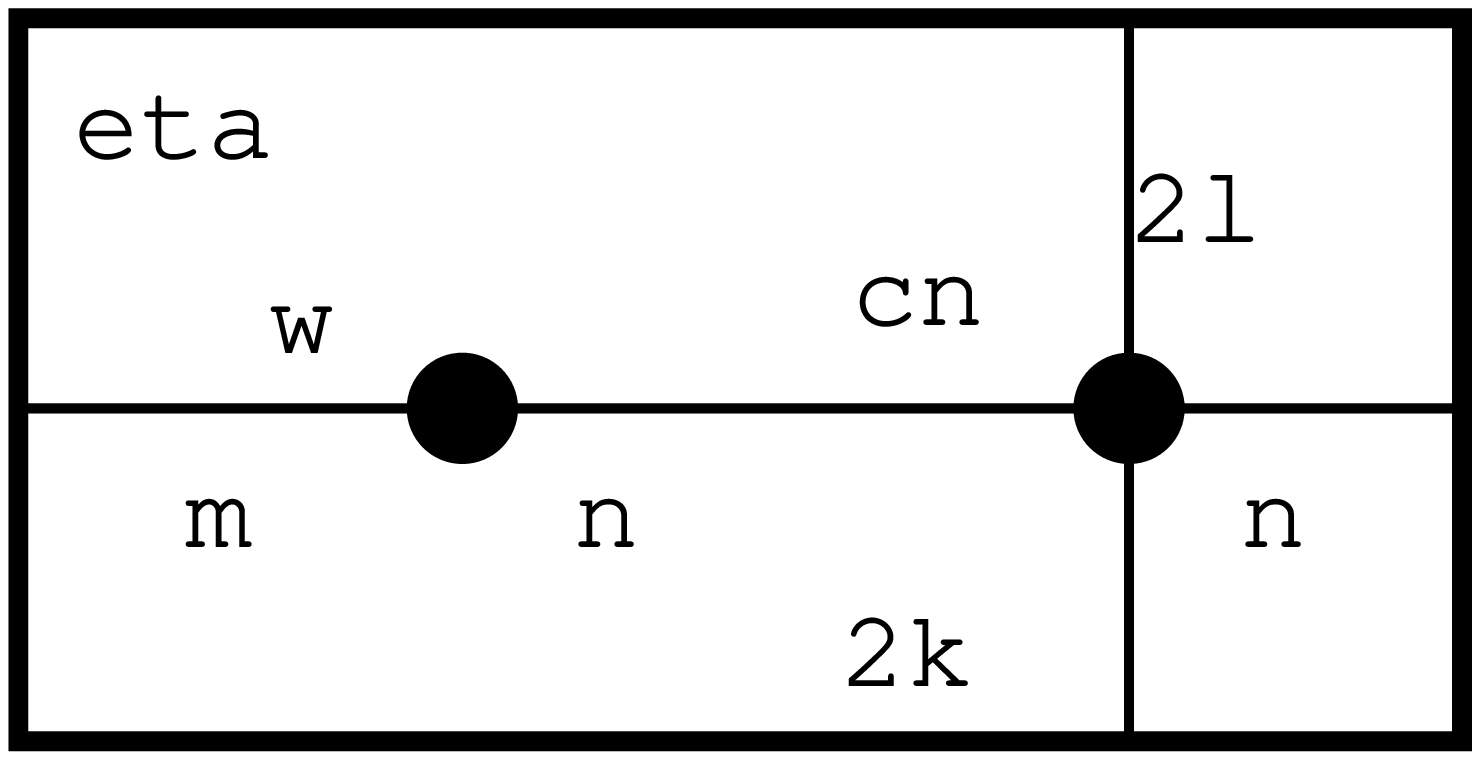}
} = P_{
\psfrag{eta}{$\eta$}
\psfrag{w}{$w$}
\psfrag{2k}{$2k$}
\psfrag{2l}{$2l$}
\psfrag{m}{$m$}
\psfrag{n}{$n$}
\psfrag{cm}{$c_m$}
\psfrag{cn}{}
\includegraphics[scale=0.15]{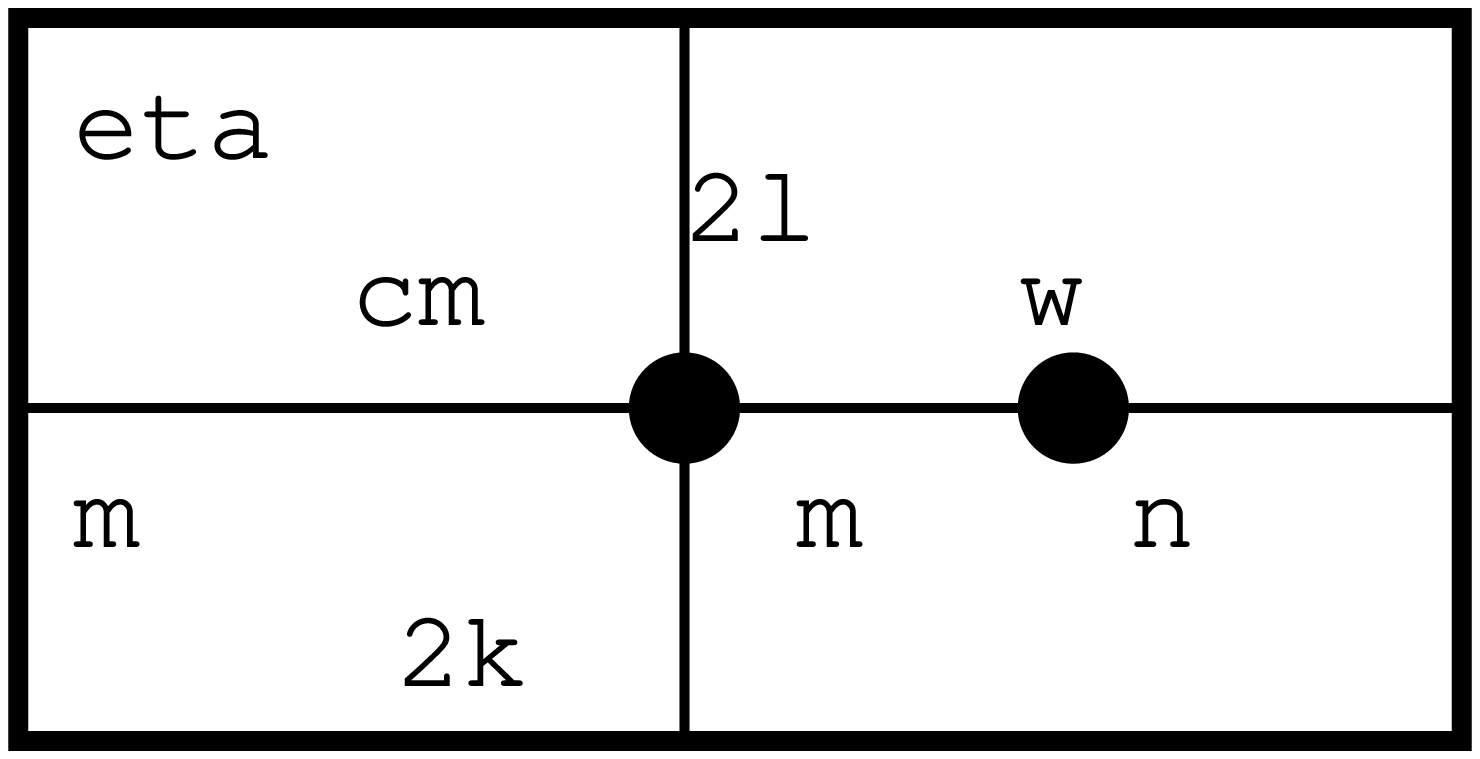}
} \t{ for all } m , n \in \N_{\vlon , \eta}, w \in P_{\eta \frac{m+n}{2} } \right\}.
\]

\comments{
\begin{prop}\label{cma}
For all $\vlon k,\eta l \in \mbox{Col}$, (i) 
the dual vector space $(AP_{\eta l , \vlon k })^\#$, (ii) the product vector space $\us{m \in \N_{\vlon , \eta} } \prod AP^{=m}_{\vlon k , \eta l}$ and (iii) the space of commutativity constraints
\[
CC_{\vlon k, \eta l} := \{ \left. c \in \us{m \in \N_{\vlon , \eta} } \prod  P_{\eta (k+l+m)} \right| P_{
\psfrag{eta}{$\eta$}
\psfrag{w}{$w$}
\psfrag{2k}{$2k$}
\psfrag{2l}{$2l$}
\psfrag{m}{$m$}
\psfrag{n}{$n$}
\psfrag{cm}{}
\psfrag{cn}{$c_n$}
\includegraphics[scale=0.15]{figures/straffmor/cmaiil.eps}
} = P_{
\psfrag{eta}{$\eta$}
\psfrag{w}{$w$}
\psfrag{2k}{$2k$}
\psfrag{2l}{$2l$}
\psfrag{m}{$m$}
\psfrag{n}{$n$}
\psfrag{cm}{$c_m$}
\psfrag{cn}{}
\includegraphics[scale=0.15]{figures/straffmor/cmaiir.eps}
} \t{ for all } m , n \in \N_{\vlon , \eta}, w \in P_{\eta \frac{m+n}{2} } \}
\]
are canonically isomorphic.
\end{prop}
}

\begin{prop}\label{cma}
For all $\vlon k,\eta l \in \mbox{Col}$, (i) 
the dual vector space $(AP_{\eta l , \vlon k })^\#$, (ii) the product vector space $\us{m \in \N_{\vlon , \eta} } \prod AP^{=m}_{\vlon k , \eta l}$ and (iii) the space of commutativity constraints $CC_{\vlon k, \eta l}$ are canonically isomorphic.
\end{prop}

\begin{proof}
Let $s \in \us{m \in \N_{\vlon , \eta} } \prod AP^{=m}_{\vlon k , \eta l}$.
For each $m \in \N_{\vlon , \eta}$, set $\bar s_m := \us{m\geq n \in \N_{\vlon , \eta} } \sum s_n \in AP^{\leq m}_{\vlon k , \eta l}$.
Define $K_s : AP_{\eta l, \vlon k} \ra \C$ by
\[
AP_{\eta l, \vlon k } \supset AP^{\leq m}_{\eta l, \vlon k } \ni a \os{K_s}{\longmapsto} \delta^{2k} \omega_{\vlon k} (a \circ \bar s_m) = \delta^{2l} \omega_{\eta l} (\bar s_m \circ a) \in \C .
\]
For well-definedness of $K_s$, note that for all $a \in AP^{\leq m}_{\eta l , \vlon k}$ and  $m \leq t \in \N_{\vlon,\eta}$, we have $\omega_{\vlon k} (a \circ \bar s_m) = \omega_{\vlon k} (a \circ \bar s_t)$ (using  	 orthogonality of $AP^{\leq m}_{\vlon k , \eta l}$ and $s_n$ when $n \geq m$).
It is easy to see that $K_s$ is linear.
Now, $K_s = 0$ implies $\bar s_m = 0$ for all $m \in \N_{\vlon , \eta}$ using positivity of $\lab \cdot , \cdot \rab_{\vlon k}$ (see Lemma \ref{pdip}); thus, $K : \us{m \in \N_{\vlon , \eta} }{\prod} AP^{=m}_{\vlon k, \eta l} \ra (AP_{\eta l, \vlon k})^{\#}$ is an injective linear map.
For surjectivity, consider $\vphi$ in the dual space.
Since $AP^{\leq m}_{\eta l , \vlon k}$ is finite dimensional, therefore, by positivity of the inner product, there exists unique $\bar s_m \in AP^{\leq m}_{ \vlon k , \eta l }$ such that $\left. \vphi \right|_{AP^{\leq m}_{\eta l , \vlon k}} = \delta^{2l} \left. \omega_{\eta l} (\bar s_m \circ \cdot) \right|_{AP^{\leq m}_{\eta l , \vlon k}} = \delta^{2k} \left. \omega_{\vlon k} (\cdot \circ \bar s_m) \right|_{AP^{\leq m}_{\eta l , \vlon k}}$.
Set $s_m := \bar s_m - \bar s_{m-2}$ for all $m \in \N_{\vlon , \eta}$ where $s_{m-2} = 0$ when $m < 2$.
Now, $s_m \in AP^{\leq m}_{\vlon k,\eta l} \perp AP^{= t}_{\vlon k,\eta l}$ for all $m < t \in \N_{\vlon,\eta}$.
If $a \in AP^{\leq m-2}_{\eta l , \vlon k} \subset AP^{\leq m}_{\eta l , \vlon k}$, then $\omega_{\vlon k} (a \circ \bar s_{m-2}) = \delta^{-2k} \vphi (a) = \omega_{\vlon k} (a \circ \bar s_m)$.
Therefore, $s_m \in AP^{= m}_{\vlon k,\eta l}$ and thereby, $\vphi = K_s$.

Let $c \in CC_{\vlon k , \eta l}$.
Consider the map $\mcal S_{\eta l , \vlon k } \supset \mcal T_{\vlon (k+l+m)} (P) \ni T \mapsto P_{
\psfrag{c}{$c_m$}
\psfrag{x}{$P_T$}
\psfrag{2k}{$2k$}
\psfrag{2l}{$2l$}
\psfrag{2m}{$m$}
\includegraphics[scale=0.15]{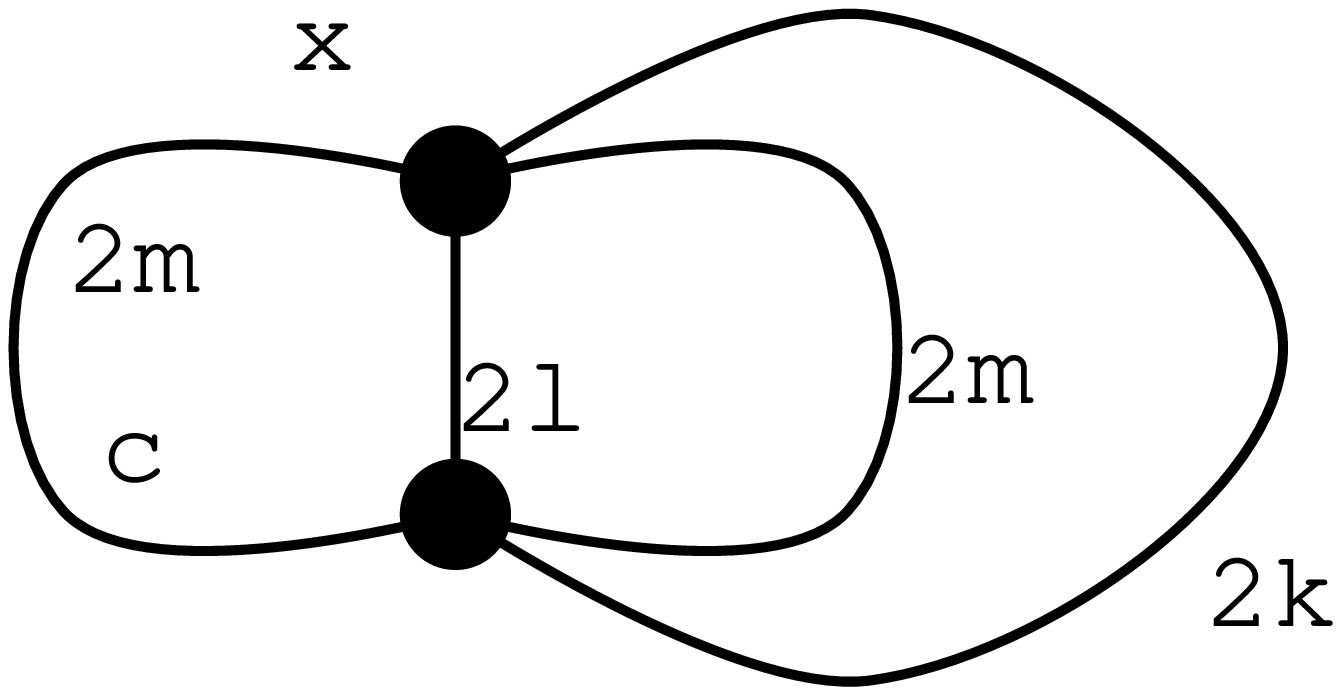}
} \in \C$; by the defining equation of the commutativity constraint $c$, this map becomes invariant under $\sim$.
Thus, by Remark \ref{apmap}, we get a linear functional $ AP_{\eta l , \vlon k} \ni a = \psi^m_{\eta l , \vlon k } (x) \os{L_c}{\longmapsto} P_{
\psfrag{c}{$c_m$}
\psfrag{x}{$x$}
\psfrag{2k}{$2k$}
\psfrag{2l}{$2l$}
\psfrag{2m}{$m$}
\includegraphics[scale=0.15]{figures/straffmor/phic.eps}
} \in \C$.
If $L_c = 0$, then by non-degeneracy of the action of trace tangles, $c_m$ must become zero.
This implies that $L: CC_{\vlon k , \eta l} \ra (AP_{\eta l, \vlon k})^{\#}$ is an injective linear map.
Next, we will prove that $L$ is surjective.
Let $\vphi \in (AP_{\eta l ,\vlon k} )^\#$ and $s$ be its associated element in $\us{m \in \N_{\vlon , \eta} } \prod AP^{=m}_{\vlon k , \eta l}$.
We consider $\bar s_m \in AP^{\leq m}_{\vlon k , \eta l}$ for $m\in \N_{\vlon , \eta}$ as defined above.
By Lemma \ref{wlogx11}, there exists a unique element $x_m \in Q^m_{\vlon k , \eta l}$ such that $\psi^m_{\vlon k,\eta l} (x_m) = \bar s_m$.
Set $c_m := \delta^{-m} P_{
\psfrag{eta}{$\eta$}
\psfrag{z}{$z_{\eta m}$}
\psfrag{2k}{$2k$}
\psfrag{2l}{$2l$}
\psfrag{m}{$m$}
\psfrag{c}{$x_m$}
\includegraphics[scale=0.15]{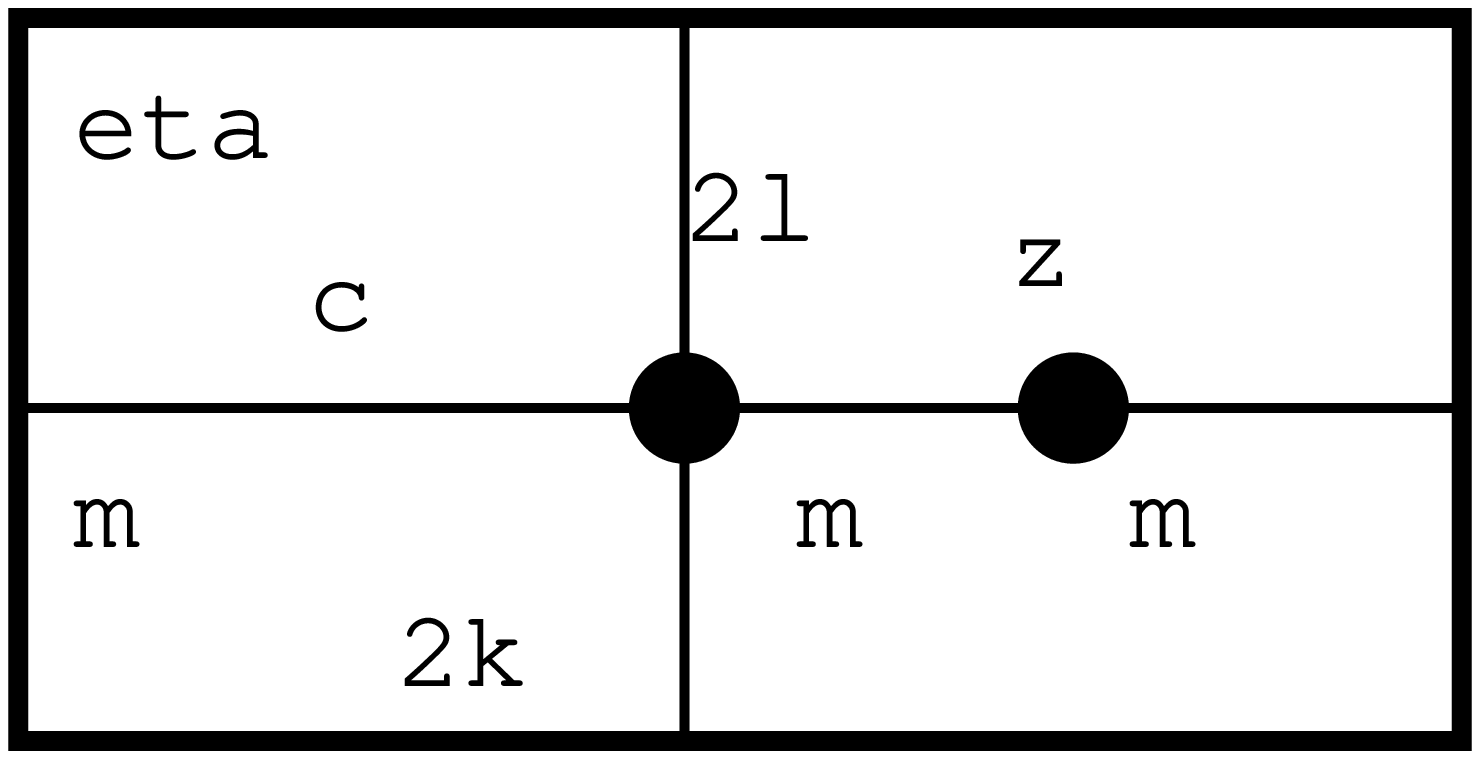}
}$.
So, for all $m\in \N_{\vlon , \eta}$, $x \in P_{\vlon (k+l+m)}$, we get
\begin{equation}\label{phiLc}
\vphi ( \psi^m_{\eta l , \vlon k} (x)) = \delta^{2k} \omega_{\vlon k} (\psi^m_{\eta l , \vlon k} (x) \circ \bar s_m) = \delta^{2k} \omega_{\vlon k} (\psi^m_{\eta l , \vlon k} (x) \circ \psi^m_{\vlon k , \eta l } (x_m)) =  P_{
\psfrag{c}{$c_m$}
\psfrag{x}{$x$}
\psfrag{2k}{$2k$}
\psfrag{2l}{$2l$}
\psfrag{2m}{$m$}
\includegraphics[scale=0.15]{figures/straffmor/phic.eps}
}
\end{equation}
where the last equality follows from the definitions of $\omega_{\vlon k}$ and $z_{\eta m}$, and from the fact that $x_m \in Q^m_{\vlon k, \eta l}$.
We now apply the above equation and obtain
\[
P_{
\psfrag{c}{$c_n$}
\psfrag{x}{$x$}
\psfrag{2k}{$2k$}
\psfrag{2l}{$2l$}
\psfrag{m}{$m$}
\psfrag{n}{$n$}
\psfrag{w}{$w$}
\;\;\;\; \includegraphics[scale=0.15]{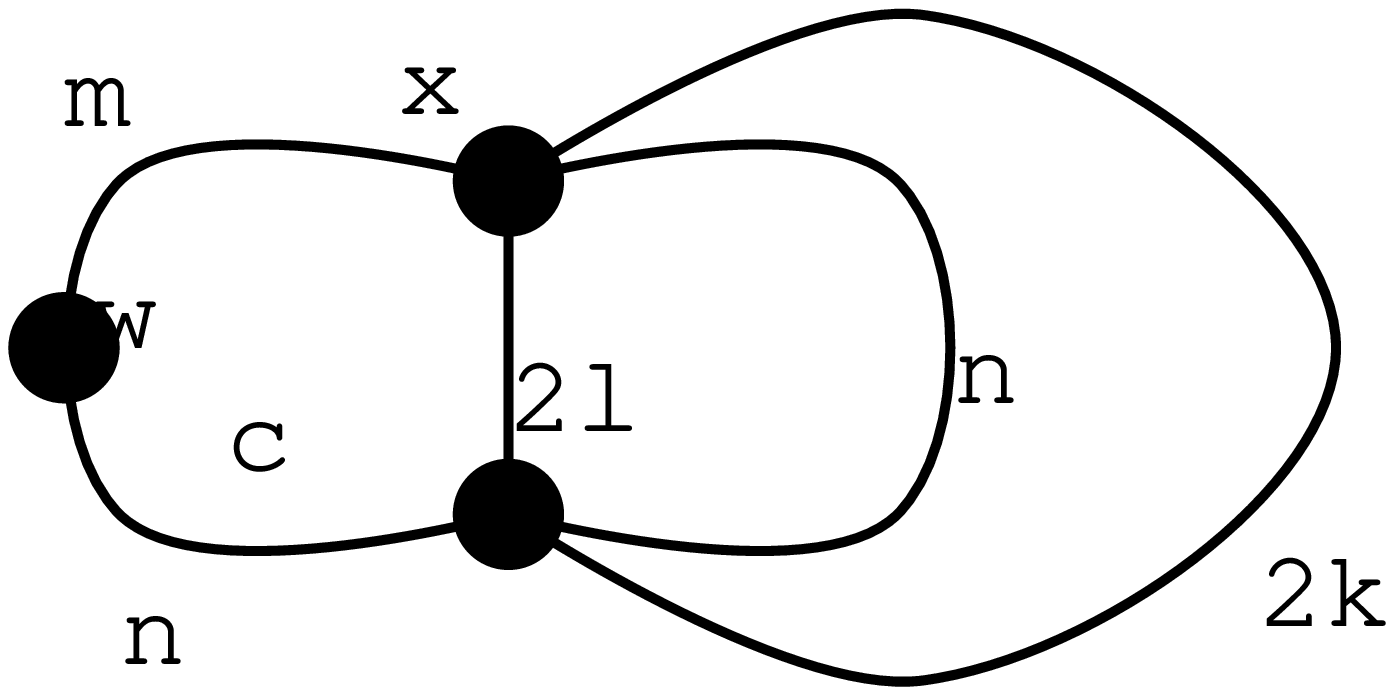}
} = \vphi ( \psi^n_{\eta l, \vlon k} ( P_{
\psfrag{eta}{$\vlon$}
\psfrag{w}{$w$}
\psfrag{2l}{$2k$}
\psfrag{2k}{$2l$}
\psfrag{m}{$m$}
\psfrag{n}{$n$}
\psfrag{cm}{}
\psfrag{cn}{$x$}
\includegraphics[scale=0.15]{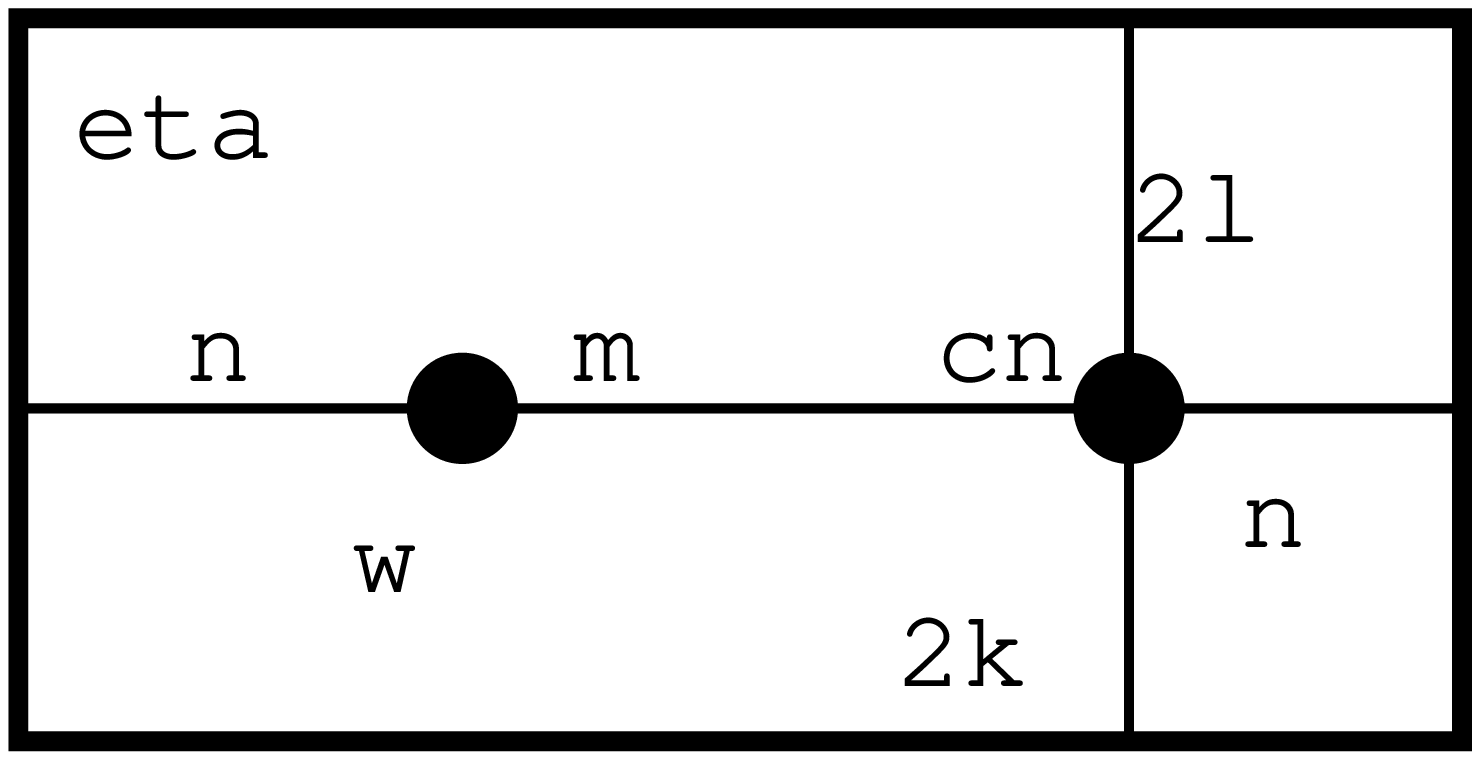}
})) = \vphi ( \psi^m_{\eta l, \vlon k} ( P_{
\psfrag{eta}{$\vlon$}
\psfrag{w}{$w$}
\psfrag{2l}{$2k$}
\psfrag{2k}{$2l$}
\psfrag{m}{$m$}
\psfrag{n}{$n$}
\psfrag{cm}{$x$}
\includegraphics[scale=0.15]{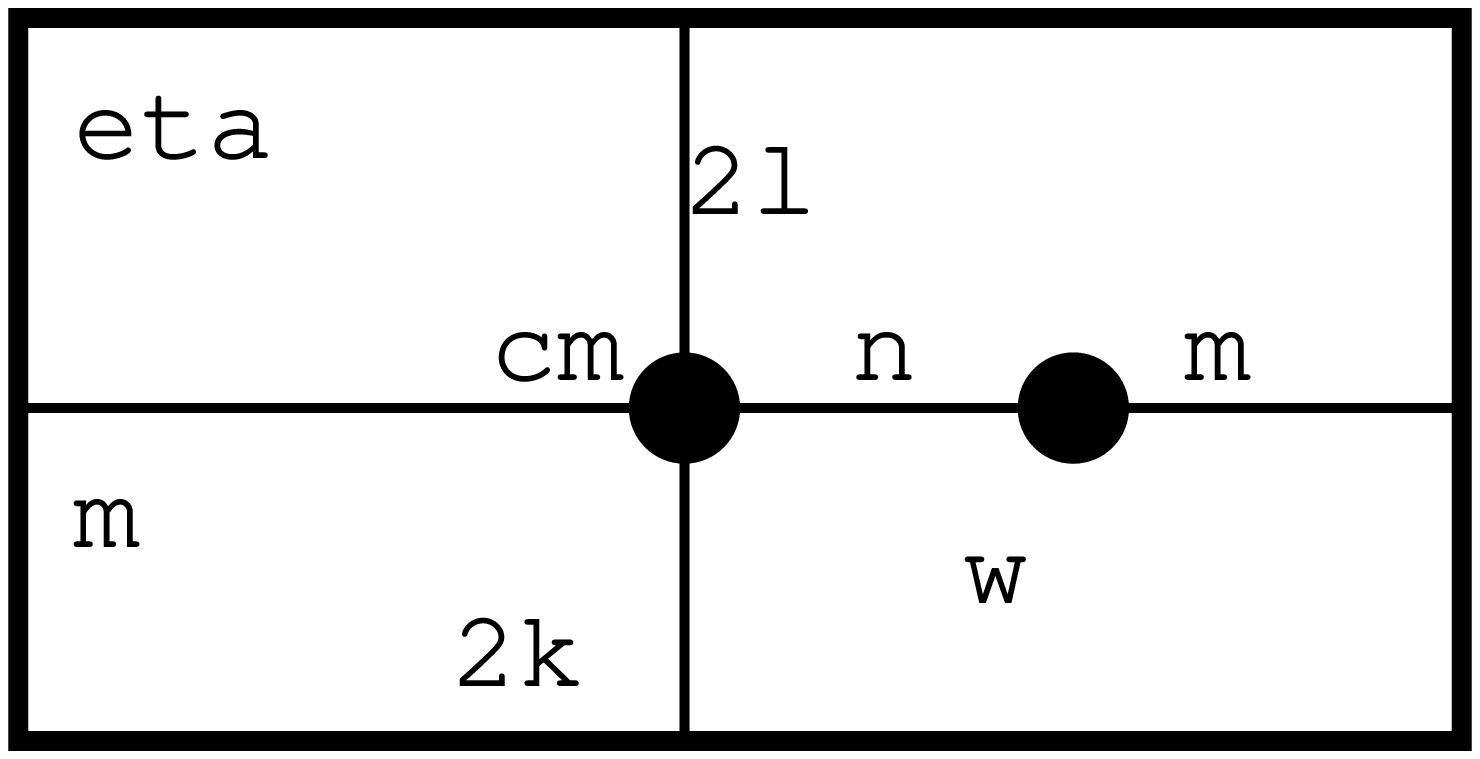}
})) = P_{
\psfrag{c}{$c_m$}
\psfrag{x}{$x$}
\psfrag{2k}{$2k$}
\psfrag{2l}{$2l$}
\psfrag{m}{$m$}
\psfrag{n}{$n$}
\psfrag{w}{$w$}
\includegraphics[scale=0.15]{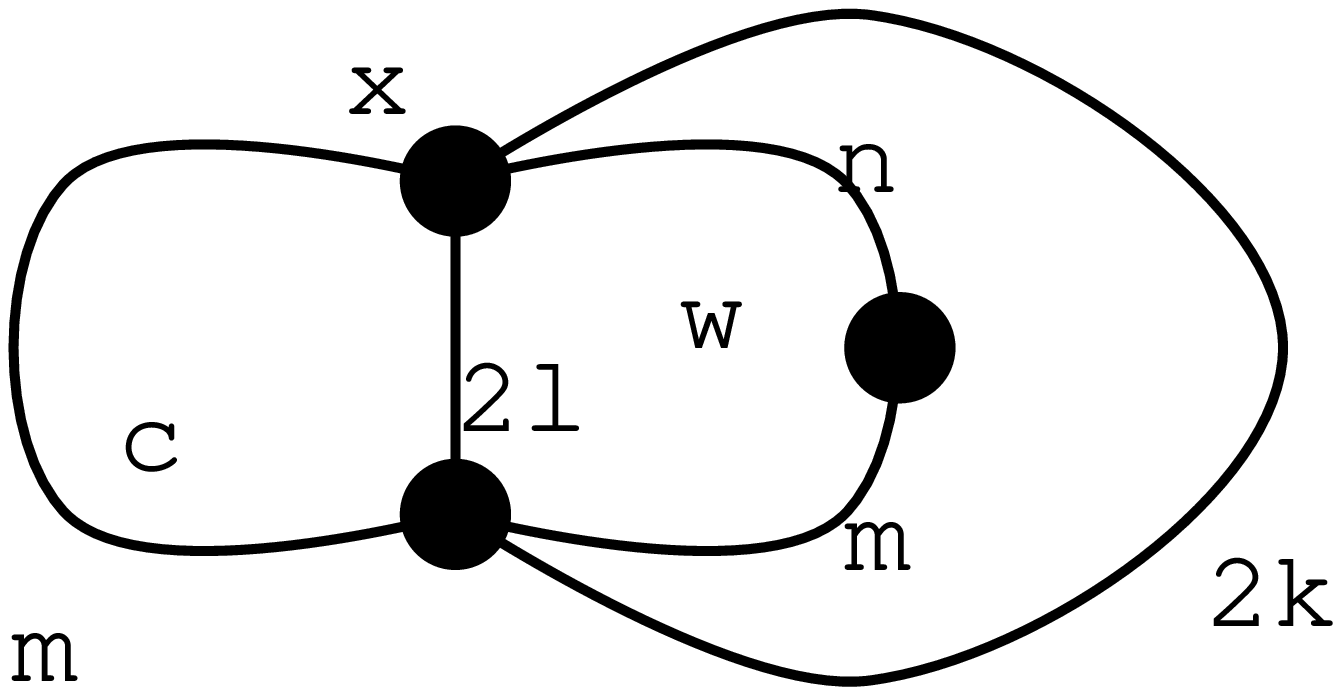}
}
\]
for all $m,n \in \N_{\vlon , \eta}$, $w \in P_{\eta \frac{m+n}{2} }$ and $x \in P_{\vlon (k+l+ \frac{m+n}{2}) }$.
Again, by non-degeneracy of the action of trace tangle, we may conclude that $c_m$'s indeed satisfy the equation of commutativity constraint.
Now, Equation \ref{phiLc} tells us that $\vphi = L_c$.
\end{proof}
The above proposition indicates the possibility of a deeper connection between the space of affine morphisms and the commutativity constraints that appear in the center construction on the bimodule category associated to the planar algebra $P$; this will be more apparent in the latter sections.

\section{Fusion of affine modules}\label{fusion}
Most of this section is devoted towards giving a monoidal structure on the category of Hilbert affine $P$-modules; the {\em commutativity constraint valued inner product} plays a key role in this.  We then introduce a contravariant, involutive, $\C$-linear endofunctor on $AP$ (referred as the {\em antipode}) which helps us in defining the {\em contragradient} of a Hilbert affine $P$-module. And finally the section concludes with a discussion on the braiding on $AP$.

Given a $*$-affine $P$-module $V$, we may consider the {\em commutativity constraint valued} (henceforth, abbreviated as $CC$-valued) inner product $c: V_{\vlon k} \times V_{\eta l} \ra CC_{\vlon k, \eta l}$ defined in the following way:
\[
\begin{array}{cccccl}
V_{\vlon k} \times V_{\eta l} & \lra & (AP_{\eta l, \vlon k})^{\#} & \os{L^{-1}}{\lra} & CC_{\vlon k , \eta l} & \\
\cc90{$\in$} & & \cc90{$\in$} & & \cc90{$\in$}\\
(\xi, \zeta) & \longmapsto & \lab \xi , \bullet \zeta \rab & \longmapsto & c (\xi , \zeta) & = \{ c_m (\xi , \zeta) \}_{m \in \N_{\vlon , \eta}}
\end{array}
\]
where $L$ is the isomorphism obtained in the proof of Proposition \ref{cma}.
Alternately, $c : V_{\vlon k} \times V_{\eta l} \ra CC_{\vlon k , \eta l}$ can be uniquely defined by the relation:
\[
\lab \xi , V_{\psi^m_{\eta l , \vlon k} (x)} \zeta \rab = P_{\!\!\!\!\!
\psfrag{c}{$c_m (\xi , \zeta)$}
\psfrag{x}{$x$}
\psfrag{+}{$+$}
\psfrag{2k}{$2k$}
\psfrag{2l}{$2l$}
\psfrag{m}{$m$}
\includegraphics[scale=0.15]{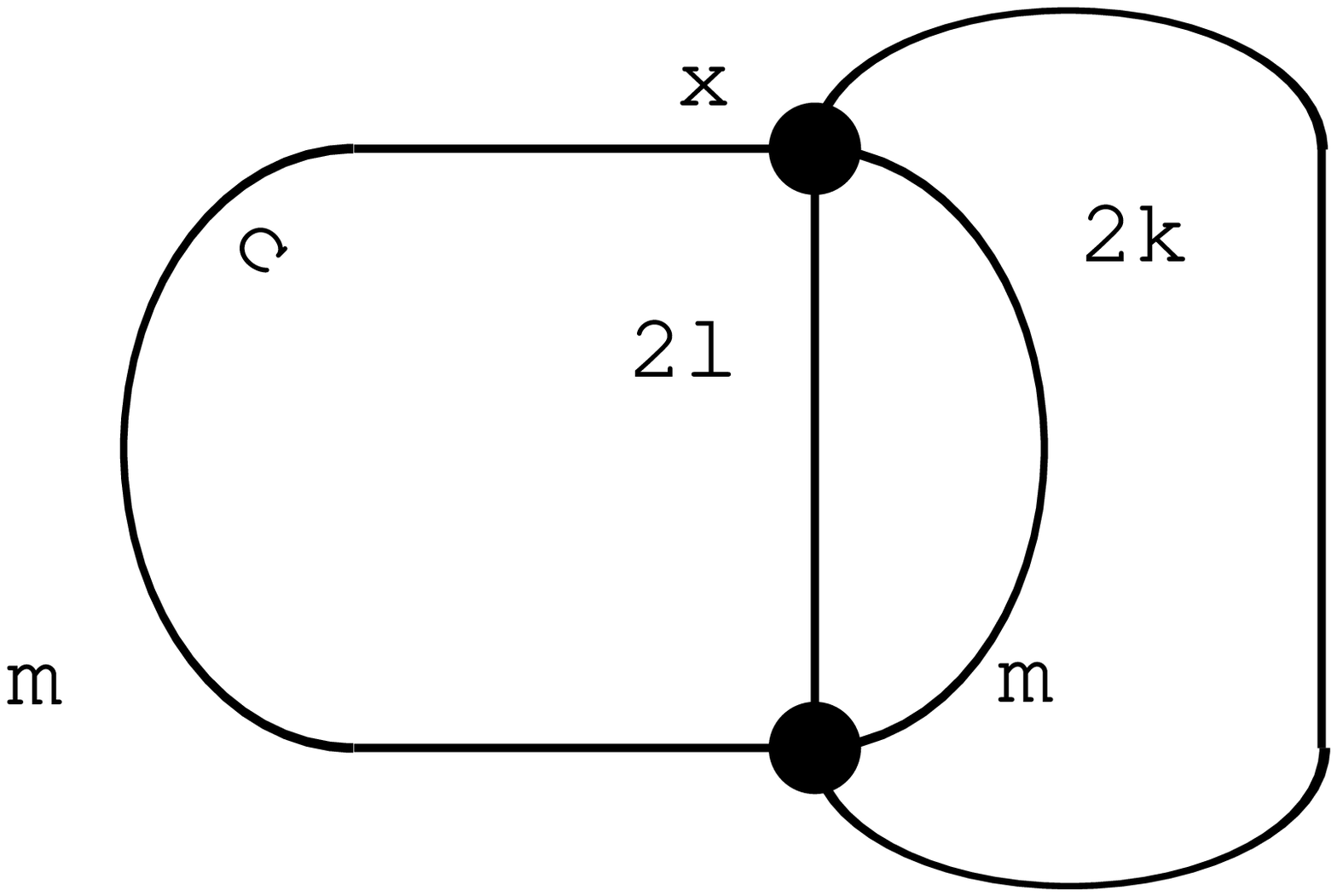}
} \t{ for all } \xi \in V_{\vlon k}, \, \zeta \in V_{\eta l}, \, m \in \N_{\vlon , \eta},  \, x \in P_{\vlon (k+l+m)}.
\]
In the next lemma, we give some basic properties of the CC-valued inner product.
\begin{lem}\label{ccprop}
For a Hilbert affine $P$-module $V$, we have the following.

(i) $c : V_{\vlon k} \times V_{\eta l} \ra CC_{\vlon k , \eta l}$ is linear (resp., conjugate linear) in second (resp., first) variable. 

(ii) $P_{\,
\psfrag{c}{$c_m (\xi , \zeta)$}
\psfrag{c*}{}
\psfrag{x}{$x$}
\psfrag{eta}{$\eta$}
\psfrag{2k}{$2k$}
\psfrag{2l}{$2l$}
\psfrag{m}{$m$}
\includegraphics[scale=0.15]{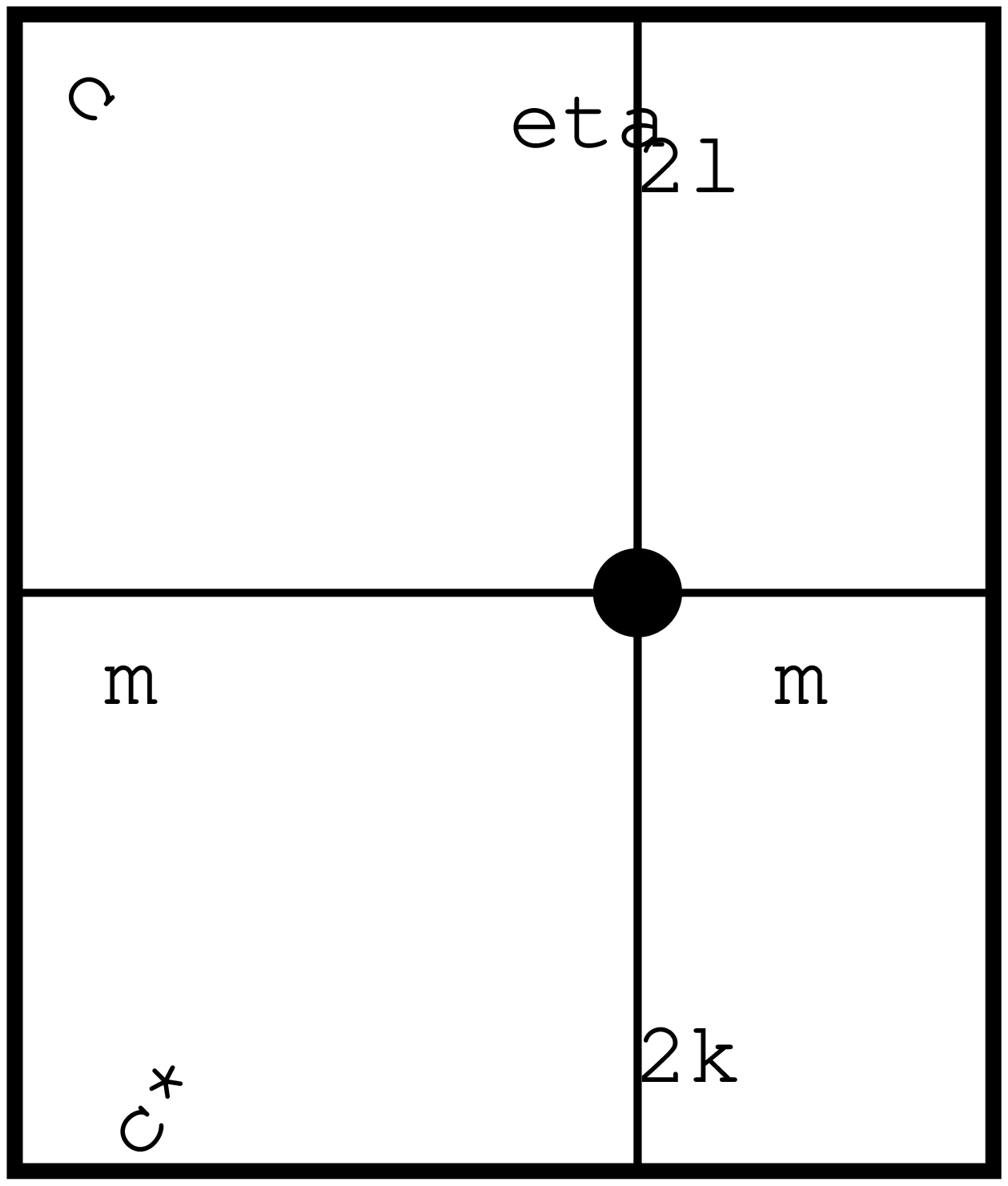}
} = c_m (\xi , \zeta) = P_{\,
\psfrag{c}{}
\psfrag{c*}{$c^*_m (\zeta , \xi)$}
\psfrag{x}{$x$}
\psfrag{eta}{$\eta$}
\psfrag{2k}{$2k$}
\psfrag{2l}{$2l$}
\psfrag{m}{$m$}
\includegraphics[scale=0.15]{figures/fusion/cflip.eps}
}$ for all $\xi \in V_{\vlon k}$, $\zeta \in V_{\eta l}$, $m\in \N_{\vlon , \eta}$.

(iii) $c_s (\xi , V_{\psi^n_{\eta l ,\nu m} (x)} \zeta) = P_{\,
\psfrag{c}{$c_{s+n} (\xi , \zeta)$}
\psfrag{x}{$x$}
\psfrag{+}{$+$}
\psfrag{2k}{$2k$}
\psfrag{2l}{$2l$}
\psfrag{2n}{$2m$}
\psfrag{m}{$n$}
\psfrag{s}{$s$}
\psfrag{nu}{$\nu$}
\includegraphics[scale=0.15]{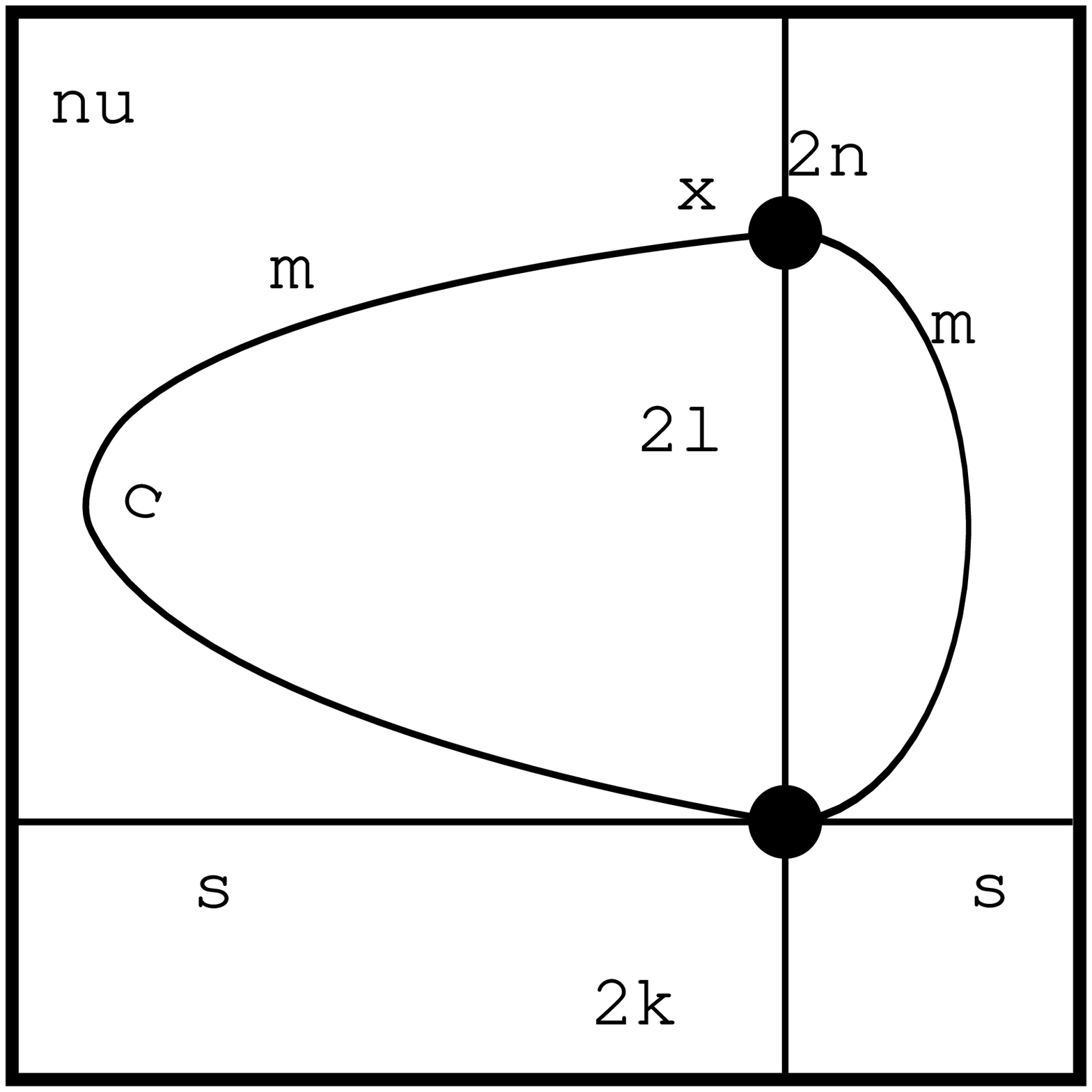}
}$ for all $\xi \in V_{\vlon k}$, $\zeta \in V_{\eta l}$, $n \in \N_{\nu , \eta}$, $x \in P_{\nu (l+m+n)}$, $s \in \N_{\vlon , \nu}$.

(iv) $c_s (V_{\psi^n_{\vlon k, \nu m} (x)} \xi , \zeta) = P_{\,
\psfrag{c}{$c_{s+m} (\xi , \zeta)$}
\psfrag{x}{$x^*$}
\psfrag{+}{$+$}
\psfrag{2k}{$2k$}
\psfrag{2l}{$2l$}
\psfrag{2n}{$2m$}
\psfrag{m}{$n$}
\psfrag{s}{$s$}
\psfrag{nu}{$\eta$}
\includegraphics[scale=0.15]{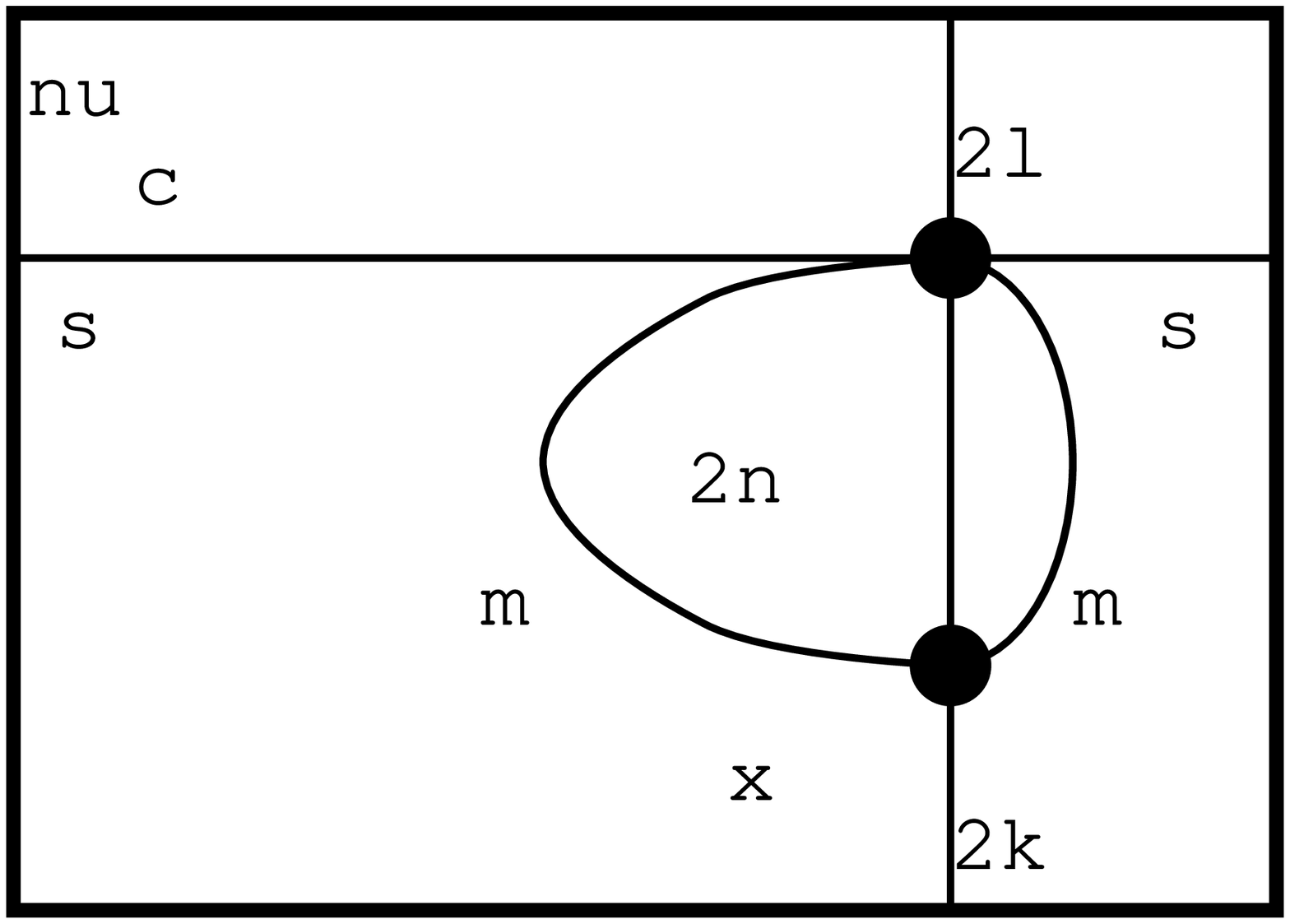}
}$ for all $\xi \in V_{\vlon k}$, $\zeta \in V_{\eta l}$, $n \in \N_{\vlon , \nu}$, $x \in P_{\nu (k+m+n)}$, $s \in \N_{\nu , \eta}$.

(v) If $t_{i,j} = P_{\,
\psfrag{2m}{$2m$}
\psfrag{2ki}{$2k_i$}
\psfrag{2kj}{$2k_j$}
\psfrag{ni}{$n_i$}
\psfrag{nj}{$n_j$}
\psfrag{c}{$c_{n_i + n_j} (\xi_j , \xi_i)$}
\psfrag{xi}{$x_i$}
\psfrag{xj}{$x^*_j$}
\psfrag{eta}{$\eta$}
\includegraphics[scale=0.15]{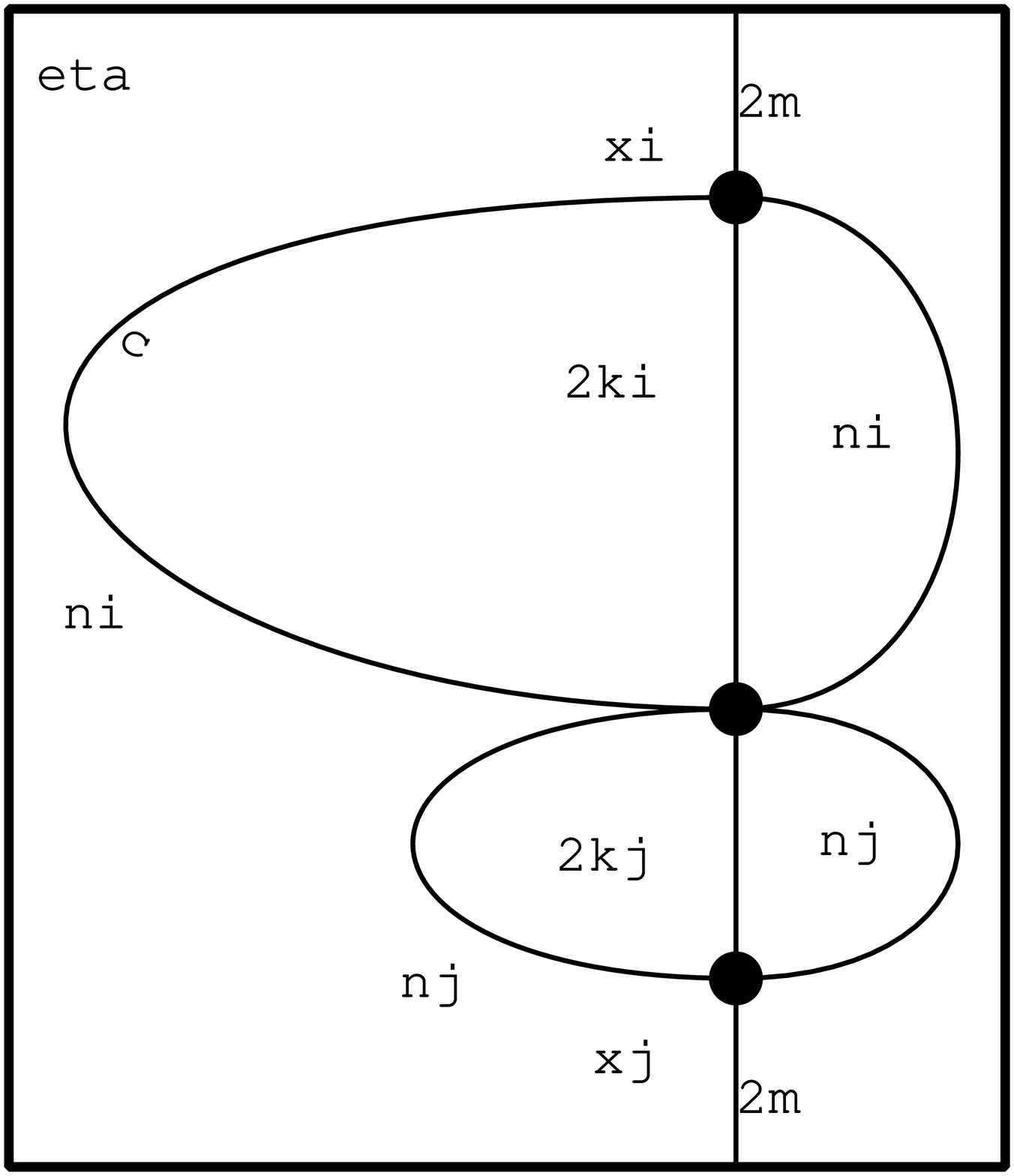}
} \in P_{\eta 2m}$ for all $\xi_i \in V_{\vlon_i k_i}$, $x_i \in P_{\eta (k_i + m + n_i)}$, $k_i \in \N_0$, $n_i \in \N_{\vlon_i , \eta}$ where $i$ lies in a finite set $I$, then $t := \us{i,j \in I}{\sum} E_{i,j} \otimes t_{i,j}$ is positive in the finite-dimensional $C^*$-algebra $M_I \otimes P_{\eta 2m}$ (where $M_I$ denotes the space of complex matrices whose rows and columns are indexed by $I$).

(vi) If $f : V \ra W$ is a morphism of Hilbert affine $P$-modules, then $c_s (\xi , f(\zeta)) = c_s (f^* (\xi) , \zeta )$ for all $\xi \in V_{\vlon k}$, $\zeta \in V_{\eta l}$, $s \in \N_{\vlon , \eta}$.
\end{lem}
\begin{proof}
Statements (i) - (iv) and (vi) can easily be deduced from the definition of the CC-valued inner product and faithfulness of the action of trace tangles.
We will now prove (v).
For this, consider the faithful embedding $M_I \otimes P_{\eta 2m} \hookrightarrow \mcal L \left(l^2 (I) \otimes L^2(P_{\eta 2m} , \tau) \right)$ where $\tau = P_{TR^r_{\eta 2m}} : P_{\eta 2m} \ra \C$.
For any $y = \us{i \in I}{\sum} \hat{e}_i \otimes \hat{y}_i$ where $\{ \hat e_i \}_{i \in I}$ is the standard orthonormal basis of $l^2 (I)$ and $y_i$'s lie in $P_{\eta 2m}$, we have $\lab y , t y \rab = \us{i,j \in I}{\sum} \tau (y^*_i t_{i,j} y_j) = \us{i,j \in I}{\sum} \lab V_{a_i} \xi_i , V_{a_j} \xi_j \rab = \norm{ \us{i \in I}{\sum} V_{a_i} \xi_i}^2 \geq 0$ where $a_i := \psi^{n_i}_{\vlon_i k_i, \eta m} (P_{\,
\psfrag{2m}{$2m$}
\psfrag{2ki}{$2k_i$}
\psfrag{ni}{$n_i$}
\psfrag{xi}{$x_i$}
\psfrag{yi}{$y^*_i$}
\psfrag{eta}{$\eta$}
\includegraphics[scale=0.15]{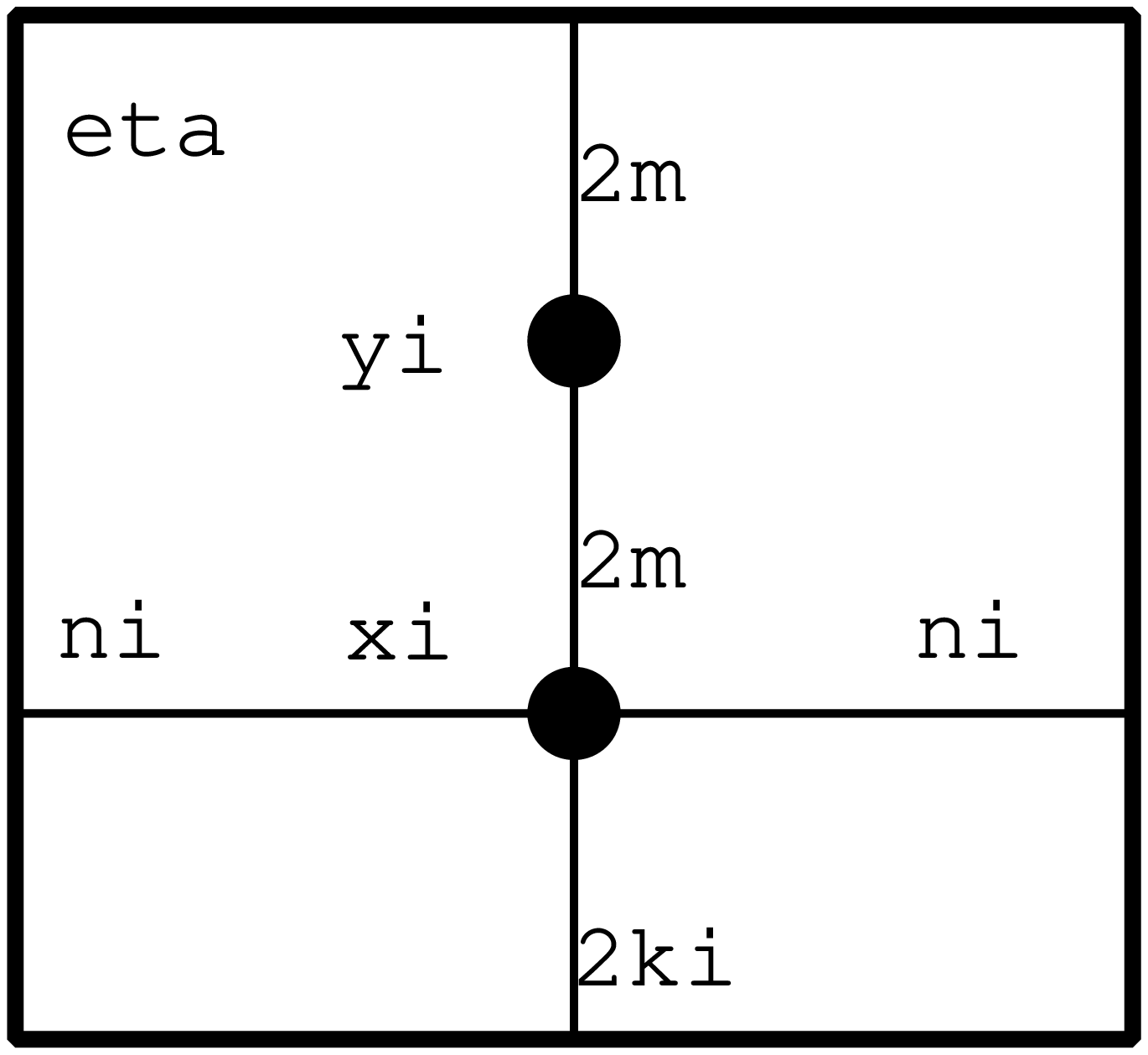}
})$ and where the second equality follows from parts (iii) and (iv).
\end{proof}
Let $V$ and $W$ be two Hilbert affine $P$-modules.
For every $\eta m \in \t{Col}$, consider the space 
\[
U_{\eta m} := \us{\vlon = \pm}{\us{k , l \in \N_0}{\oplus} } \left(V_{\vlon k} \otimes AP_{\vlon (k+l) , \eta m} \otimes W_{\vlon l}\right)
\]

whose element $\xi \otimes a \otimes \zeta$ will be denoted by $\fusion{\xi}{a}{\zeta}$.
We will define a sesquilinear form on $U_{\eta m}$ in the following way:
\[
\left\lab
\fusion{\xi_1}{a_1}{\zeta_1} , \fusion{\xi_2}{a_2}{\zeta_2}
\right\rab \; \; \; := \; \; \; P_{ \!\!\!\!\!\!
\psfrag{n1}{$n_1$}
\psfrag{n2}{$n_2$}
\psfrag{n1+2}{$n_1 + n_2$}
\psfrag{2k1}{$2k_1$}
\psfrag{2k2}{$2k_2$}
\psfrag{2l1}{$2l_1$}
\psfrag{2l2}{$2l_2$}
\psfrag{x2}{$x_2$}
\psfrag{x1*}{$x^*_1$}
\psfrag{c1}{$c_{n_1 + n_2} (\xi_1 , \xi_2)$}
\psfrag{c2}{$c_{n_1 + n_2} (\zeta_1 , \zeta_2\!)$}
\psfrag{2m}{$2m$}
\includegraphics[scale=0.15]{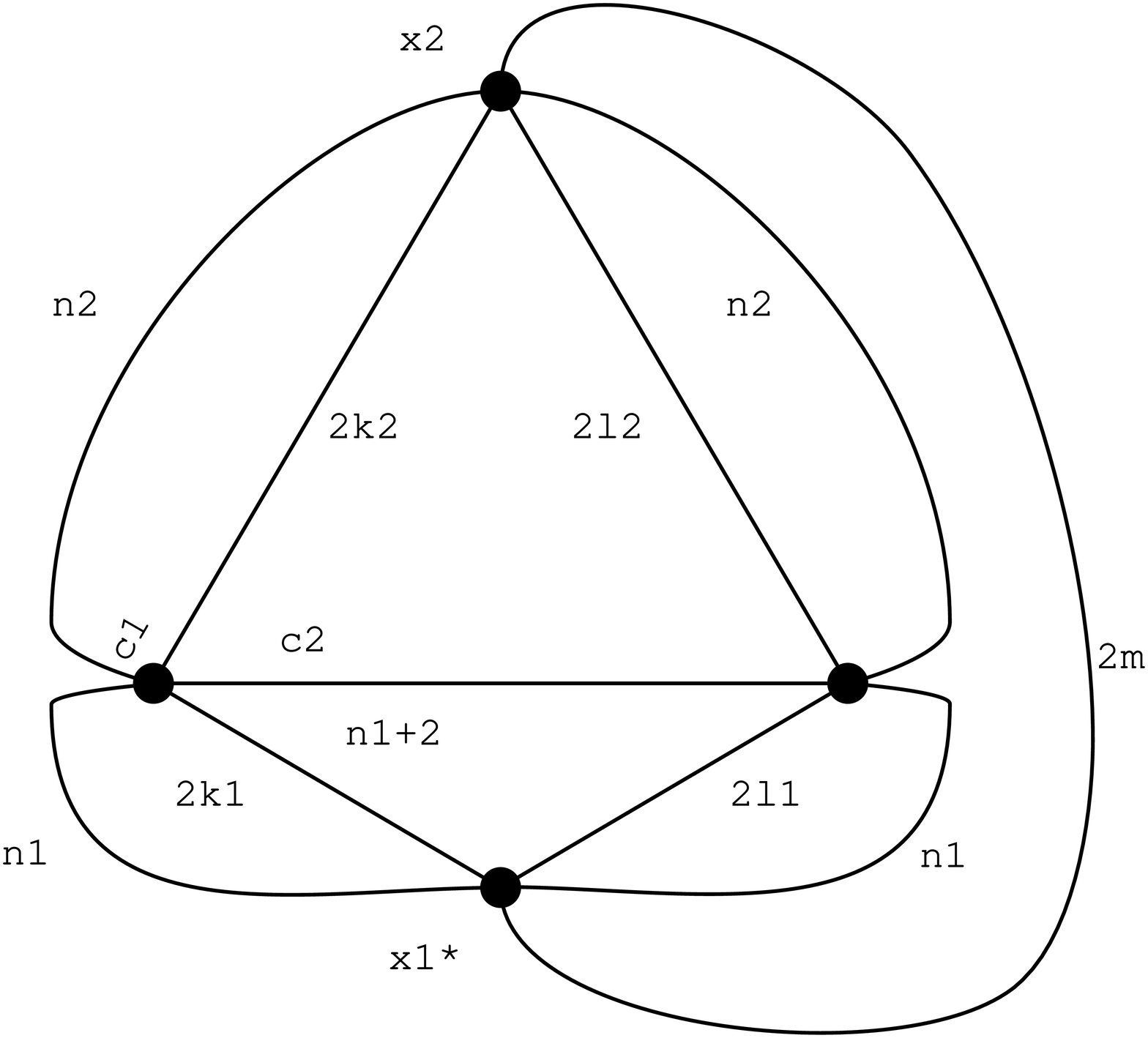}
}\]
for all $\xi_i \in V_{\vlon_i k_i}$, $\zeta_i \in W_{\vlon_i l_i}$, $a_i = \psi^{n_i}_{\vlon_i (k_i + l_i) , \eta m} (x_i)$ where $x_i \in P_{\eta (k_i +l_i + m + n_i)}$ and $i = 1,2$.
Well-definedness follows from Remark \ref{apmap} while linearity (resp., conjugate-linearity) in the second (resp., first) variable comes from Lemma \ref{ccprop} (i). Also, Lemma \ref{ccprop} (ii) implies that $\lab \cdot , \cdot \rab$ is Hermitian; however,  positive semi-definiteness of $\lab \cdot , \cdot \rab$ requires  further understanding of the structure of $U_{\eta m}$'s.
\begin{lem}\label{fuswloglem}
Given $\xi \in V_{\vlon k}$, $\zeta \in W_{\vlon l}$ and $a \in AP_{\vlon (k+l) , \eta m}$, there exist $k' , l' \in \N_0$, $\xi' \in V_{\eta k'}$, $\zeta' \in W_{\eta l'}$, $a' \in AP^{=0}_{\eta (k'+l') , \eta m}$ ($= \t{Range } \psi^0_{\eta (k' + l'), \eta m}$) such that $\lab \cdot , \xi \os{a}{\otimes} \zeta \rab = \lab \cdot , \xi' \os{a'}{\otimes} \zeta' \rab$ and $\lab \xi \os{a}{\otimes} \zeta , \cdot \rab = \lab \xi' \os{a'}{\otimes} \zeta' , \cdot \rab$.
\end{lem}
\begin{proof}
Since $\lab \cdot , \cdot \rab$ is Hermitian, it is enough to show only one equation, namely the first one.
Suppose $x \in P_{\eta (k+l+m+n)}$ such that $a = \psi^n_{\vlon (k+l) , \eta m} (x)$.
Set\\
$k' = k+n$ and $\xi' = V_{u_{\vlon k , n} } \xi \in V_{\eta k'}$,\\
 $l' = l+n$ and $\zeta' = W_{ u_{\vlon l , n} } \zeta \in W_{\eta l'}$, and\\
$a' = \psi^0_{\eta (k' + l'), \eta m} ( x' ) \in AP^{=0}_{\eta (k' + l'), \eta m}$\\
where $x' := P_{
\psfrag{eta}{$\eta$}
\psfrag{2m}{$2m$}
\psfrag{2k}{$2k$}
\psfrag{2l}{$2l$}
\psfrag{n}{$n$}
\psfrag{x}{$x$}
\includegraphics[scale=0.15]{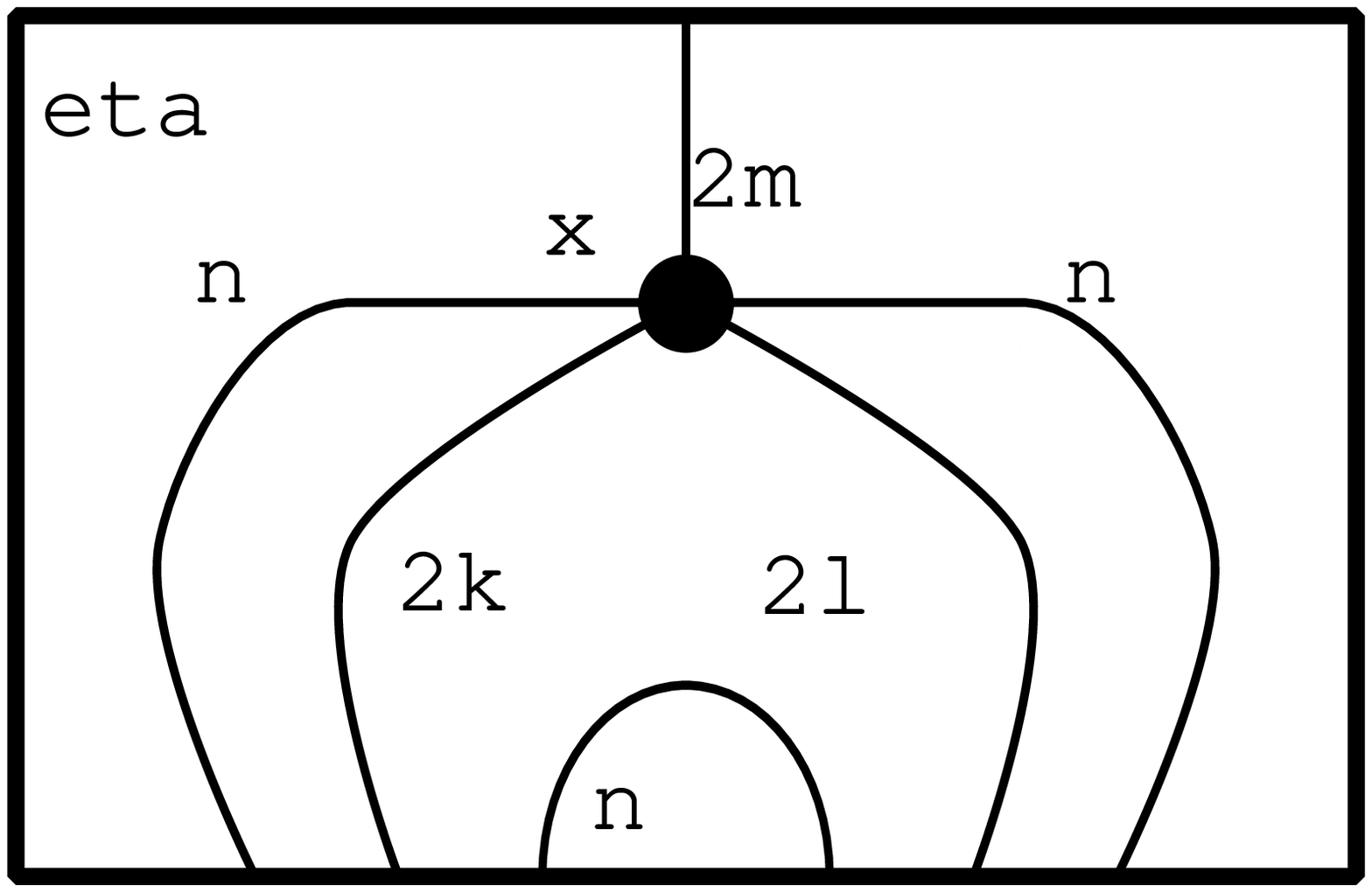}
}$ and $u_{\vlon k , n} = 
\psfrag{2k}{$2k$}
\psfrag{n}{$n$}
\psfrag{e}{$\vlon$}
\psfrag{eta}{$\eta$}
\includegraphics[scale=0.15]{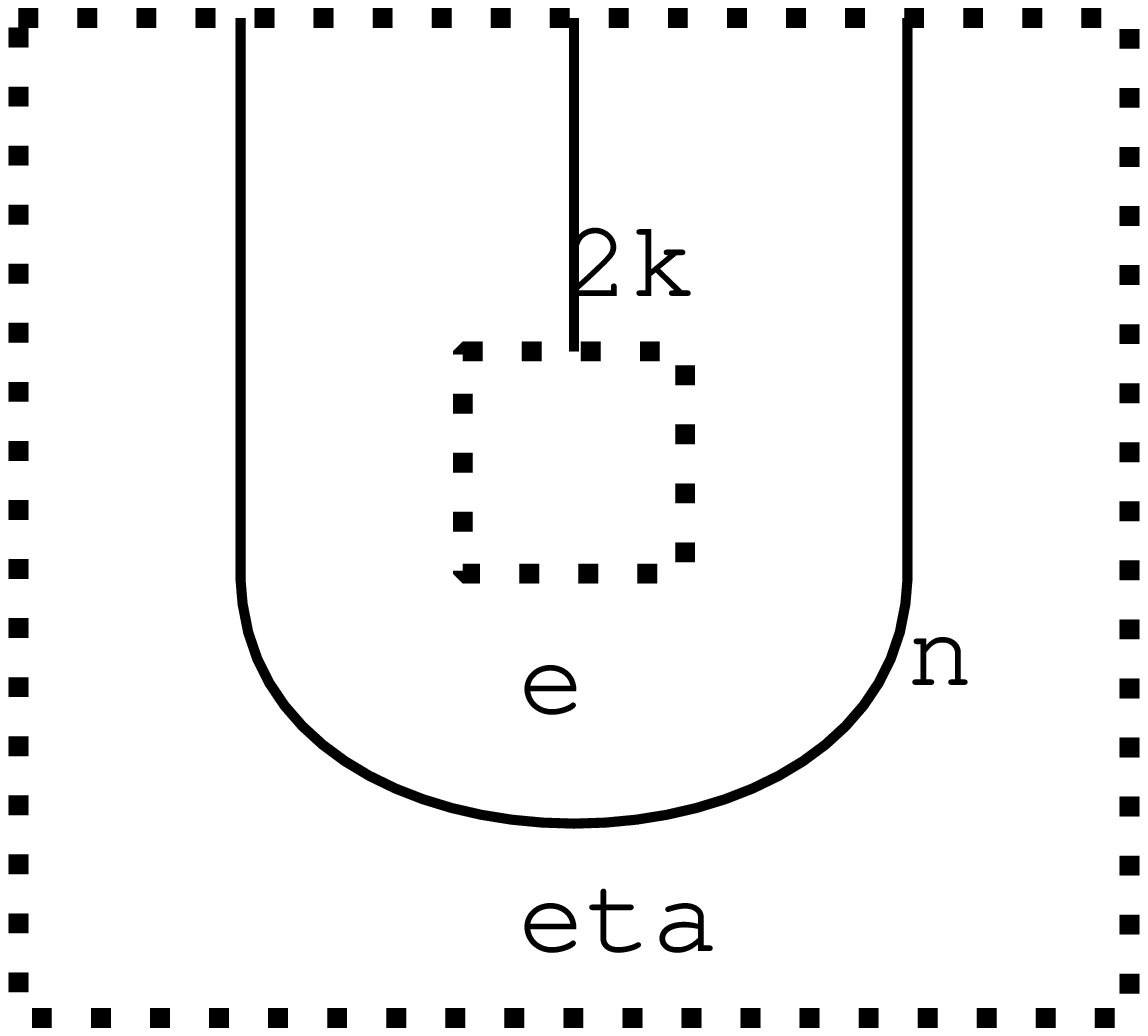}
\in AP_{\vlon k,\eta (k+n)}$
\comments{
$u_{\vlon k , n} = \psi^n_{\vlon k, \eta (k+n)} (1_{P_{\eta 2(k+n)}} ) = \psi^n_{\vlon k, \eta (k+n)} ( P_{
\psfrag{2k}{$2k$}
\psfrag{n}{$n$}
\psfrag{eta}{$\eta$}
\includegraphics[scale=0.15]{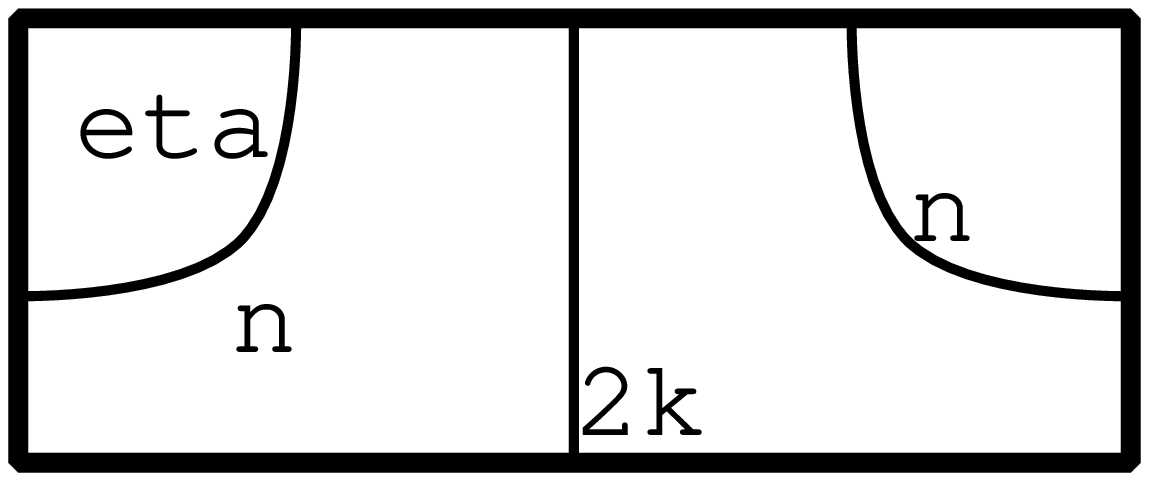}
})$}
.\\
Applying Lemma \ref{ccprop} (iii), one can easily obtain
\[
P_{\,
\psfrag{x}{$x$}
\psfrag{x'}{$x'$}
\psfrag{eta}{$\eta$}
\psfrag{n}{$n$}
\psfrag{n1}{$n_1$}
\psfrag{n+n1}{$n+n_1$}
\psfrag{2k1}{$2k_1$}
\psfrag{2k}{$2k$}
\psfrag{2k'}{$2k'$}
\psfrag{2l1}{$2l_1$}
\psfrag{2l}{$2l$}
\psfrag{2l'}{$2l'$}
\psfrag{x1}{$x_1$}
\psfrag{c1}{$c_{n_1} (\xi_1 , \xi')$}
\psfrag{c2}{$c_{n_1} (\zeta_1 , \zeta')$}
\psfrag{2m}{$2m$}
\includegraphics[scale=0.15]{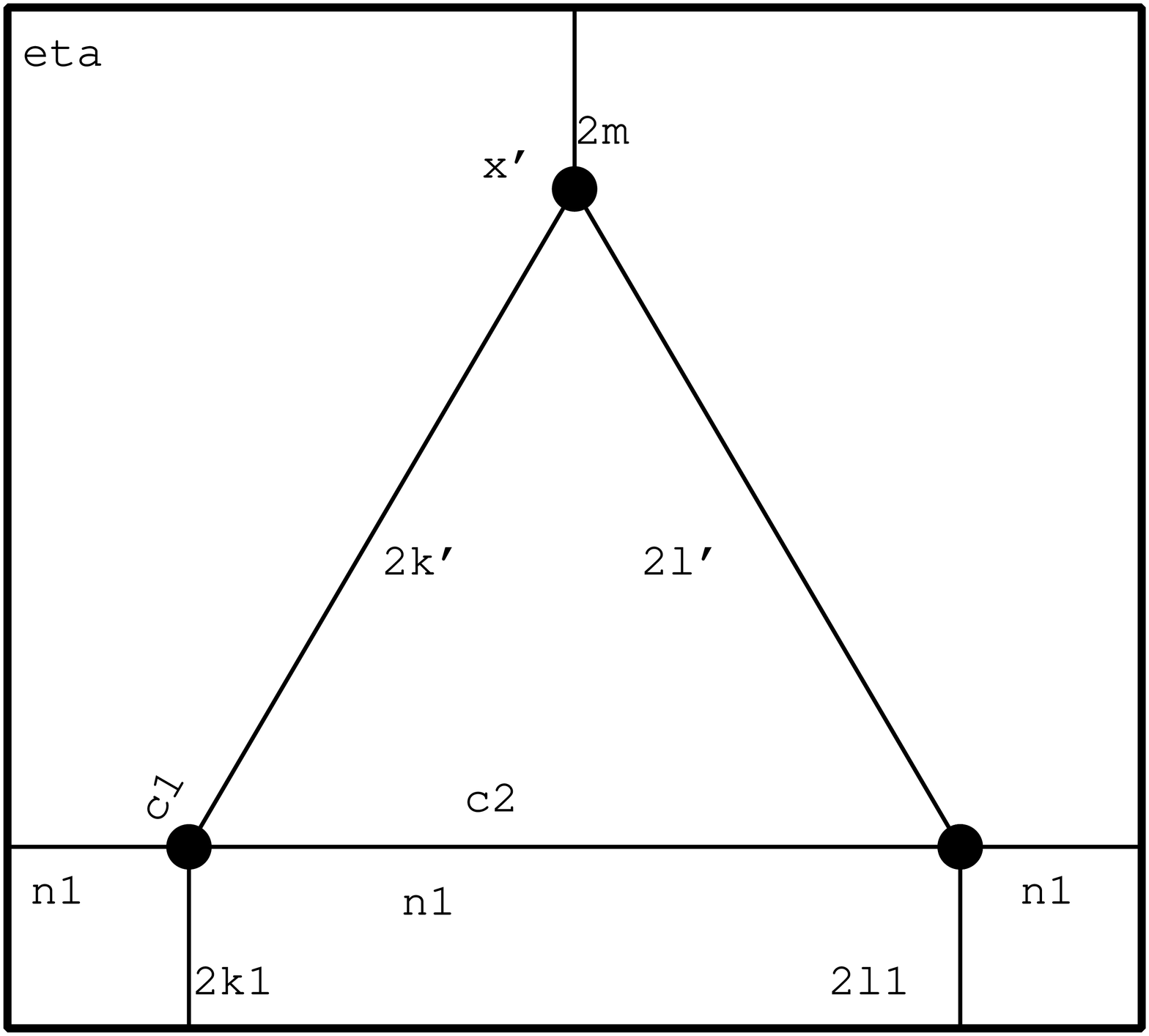}
} = P_{\,
\psfrag{x}{$x$}
\psfrag{x'}{$x'$}
\psfrag{eta}{$\eta$}
\psfrag{n}{$n$}
\psfrag{n1}{$n_1$}
\psfrag{n+n1}{$n+n_1$}
\psfrag{2k1}{$2k_1$}
\psfrag{2k}{$2k$}
\psfrag{2k'}{$2k'$}
\psfrag{2l1}{$2l_1$}
\psfrag{2l}{$2l$}
\psfrag{2l'}{$2l'$}
\psfrag{x1}{$x_1$}
\psfrag{c1}{$c_{n+n_1} (\xi_1 , \xi)$}
\psfrag{c2}{$c_{n+n_1} (\zeta_1 , \zeta)$}
\psfrag{2m}{$2m$}
\includegraphics[scale=0.15]{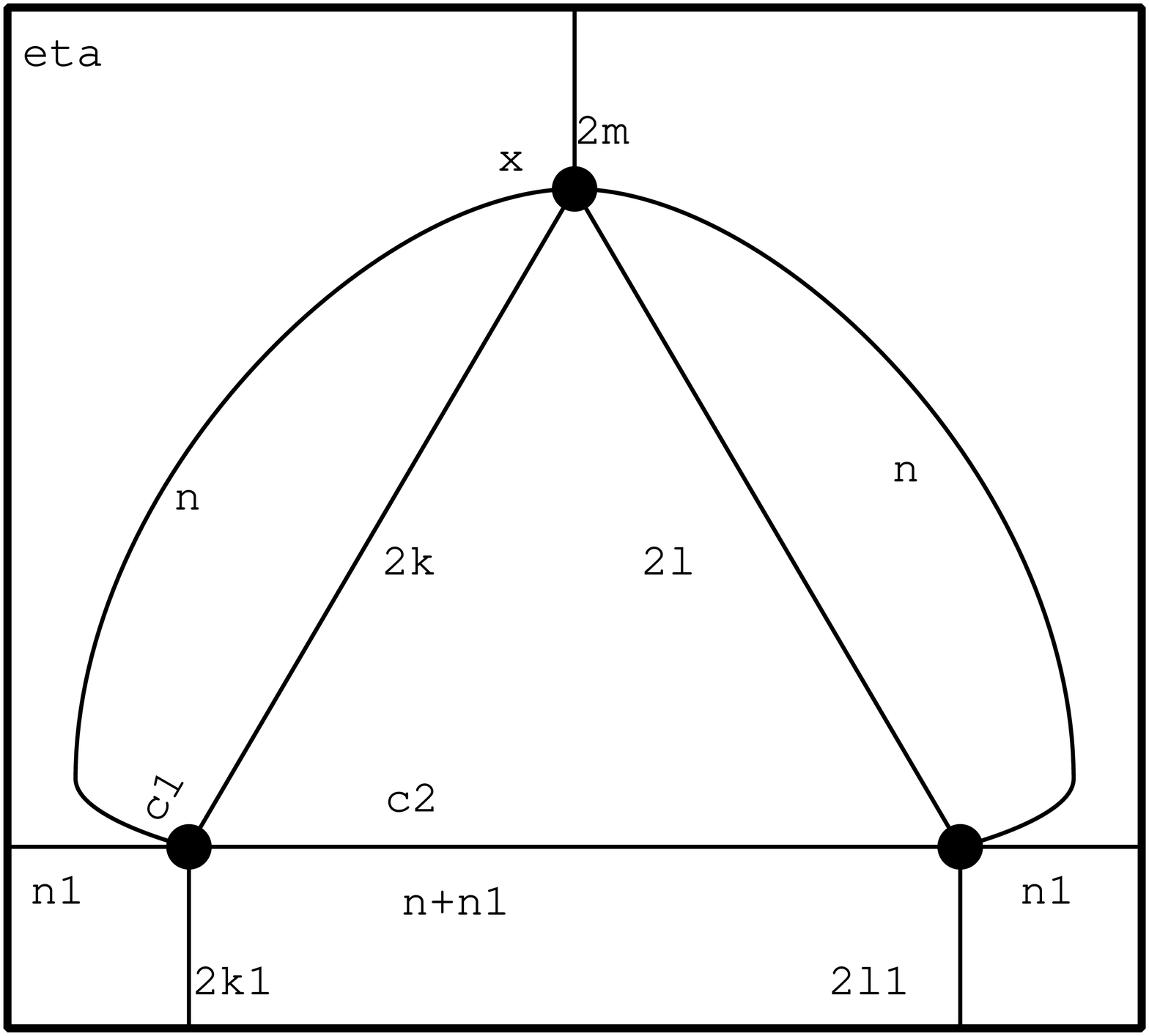}
} \t{ for all } \xi_1 \in V_{\vlon_1 k_1} , \zeta_1 \in W_{\vlon_1 l_1}, n_1 \in \N_{\vlon_1 , \eta}.
\]
From the above equation and the definition of the $\lab \cdot , \cdot \rab$, we get $\lab \xi_1 \os{a_1}{\otimes} \zeta_1 , \xi' \os{a'}{\otimes} \zeta' \rab = \lab \xi_1 \os{a_1}{\otimes} \zeta_1 , \xi \os{a}{\otimes} \zeta \rab$ for all $\xi_1 \in V_{\vlon_1 k_1} , \zeta_1 \in W_{\vlon_1 l_1}, a_1 = \psi^{n_1}_{\vlon_1 (k_1 +l_1) , \eta m} (x_1) \in AP_{\vlon_1 (k_1 +l_1) , \eta m}$ where $n_1 \in \N_{\vlon_1 , \eta}$ and $x_1 \in P_{\eta (k_1 + l_1 + m + n_1)}$.
\end{proof}
\begin{lem}\label{wlogsamekl}
Given $u \in U_{\eta m}$, there exist $k, l \in \N$, a finite set $I$, $\tilde \xi_i \in V_{\eta k}$, $\tilde \zeta_i \in W_{\eta l}$ and $\tilde a_i \in AP^{=0}_{\eta (k+l), \eta m}$ for $i \in I$ such that $\lab \cdot , u \rab = \lab \cdot , \tilde u \rab$ and $\lab u , \cdot  \rab = \lab \tilde u , \cdot \rab$ where $\tilde{u} = \us{i \in I}{\sum} \tilde \xi_i \os{\tilde a_i}{\otimes} \tilde \zeta_i$.
\end{lem}
\begin{proof}
By Lemma \ref{fuswloglem}, we may assume $u = \us{i \in I}{\sum} \xi_i \os{a_i}{\otimes} \zeta_i$ for $\xi_i \in V_{\eta k_i}$, $\zeta_i \in W_{\eta l_i}$, $a_i = \psi^0_{\eta (k_i + l_i) , \eta m} (x_i) \in AP^{=0}_{\eta (k_i + l_i) , \eta m}$ where $i$ runs over a finite set $I$. Choose $k \geq \t{max} \{k_i : i \in I\}$ and $l \geq \t{max} \{l_i : i \in I\}$, and set\\
$\tilde \xi_i := V_{ \psi^0_{\eta k_i , \eta k} (1_{P_{\eta (k + k_i)}}) } \xi_i = V_{ \psi^0_{\eta k_i , \eta k} (P_{\,
\psfrag{2ki}{$2k_i$}
\psfrag{2li}{$$}
\psfrag{k-ki}{$k-k_i$}
\psfrag{eta}{$\eta$}
\includegraphics[scale=0.15]{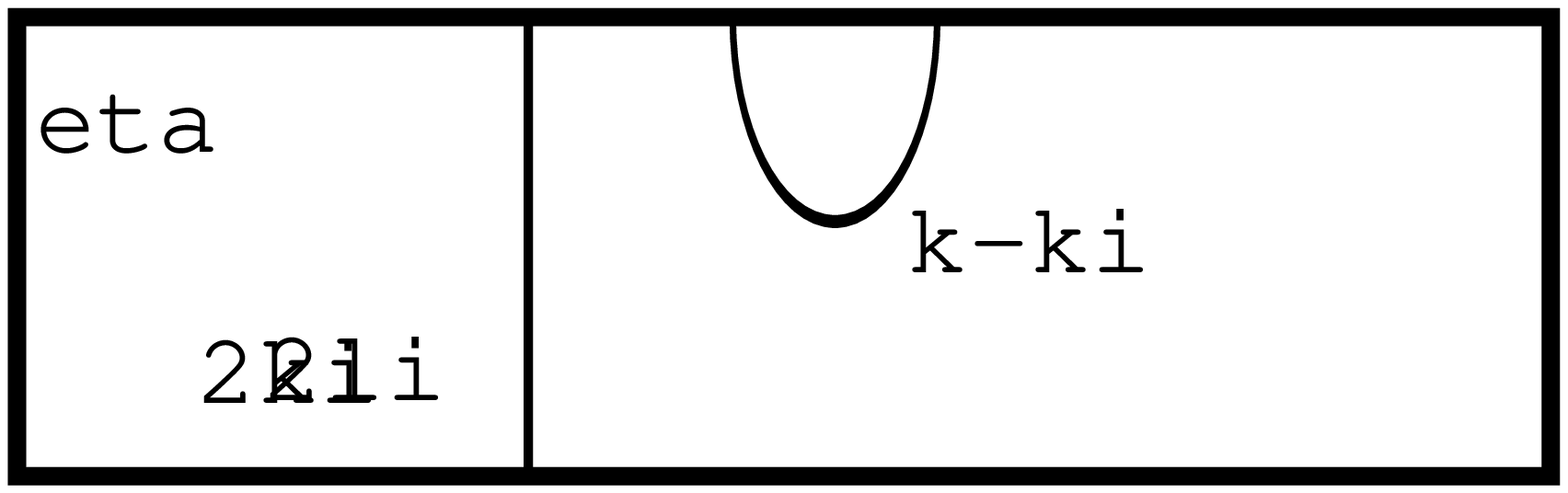}
}) } \xi_i \in V_{\eta k}$,\\
$\tilde \zeta_i := W_{ \psi^0_{\eta l_i , \eta l} (1_{P_{\eta (l + l_i)}}) } \zeta_i = W_{ \psi^0_{\eta l_i , \eta l} ( P_{\,
\psfrag{2ki}{$$}
\psfrag{2li}{$2l_i$}
\psfrag{k-ki}{$l-l_i$}
\psfrag{eta}{$\eta$}
\includegraphics[scale=0.15]{figures/fusion/1etak+ki.eps}
}) } \zeta_i \in W_{\eta l}$, and\\
\begin{tabular}{rl}
$\tilde a_i :=$ & $\delta^{(k_i+l_i)-(k+l)} a_i \circ \psi^0_{\eta (k+l) , \eta (k_i + l_i)} \; ( P_{\,
\psfrag{2ki}{$2k_i$}
\psfrag{2li}{$2l_i$}
\psfrag{k-ki}{$k \! - \! k_i$}
\psfrag{l-li}{$l \! - \! l_i$}
\psfrag{eta}{$\eta$}
\includegraphics[scale=0.15]{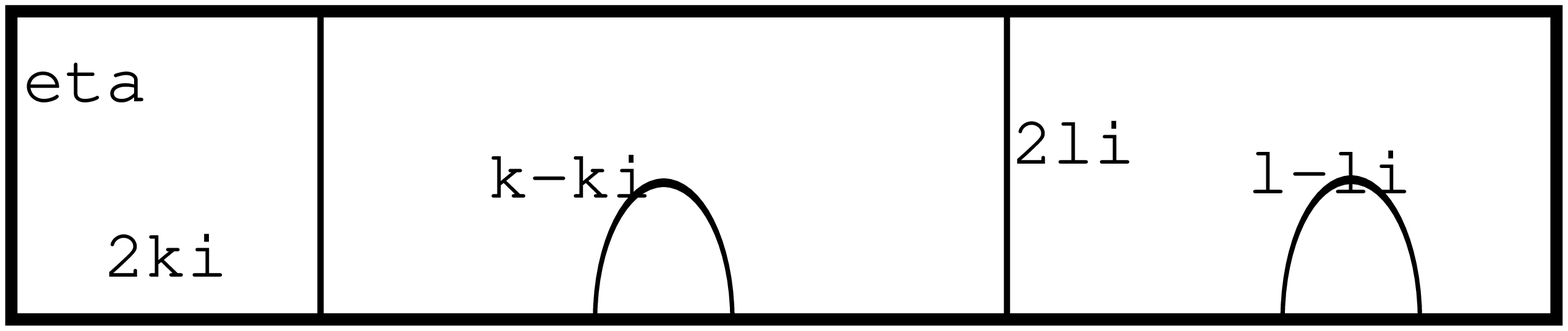}
})$\\
$=$ & $\delta^{(k_i+l_i)-(k+l)} \psi^0_{\eta (k+l) , \eta m} ( P_{\,
\psfrag{xi}{$x_i$}
\psfrag{2m}{$2m$}
\psfrag{2ki}{$2k_i$}
\psfrag{2li}{$2l_i$}
\psfrag{k-ki}{$k \! - \! k_i$}
\psfrag{l-li}{$l \! - \! l_i$}
\psfrag{eta}{$\eta$}
\includegraphics[scale=0.15]{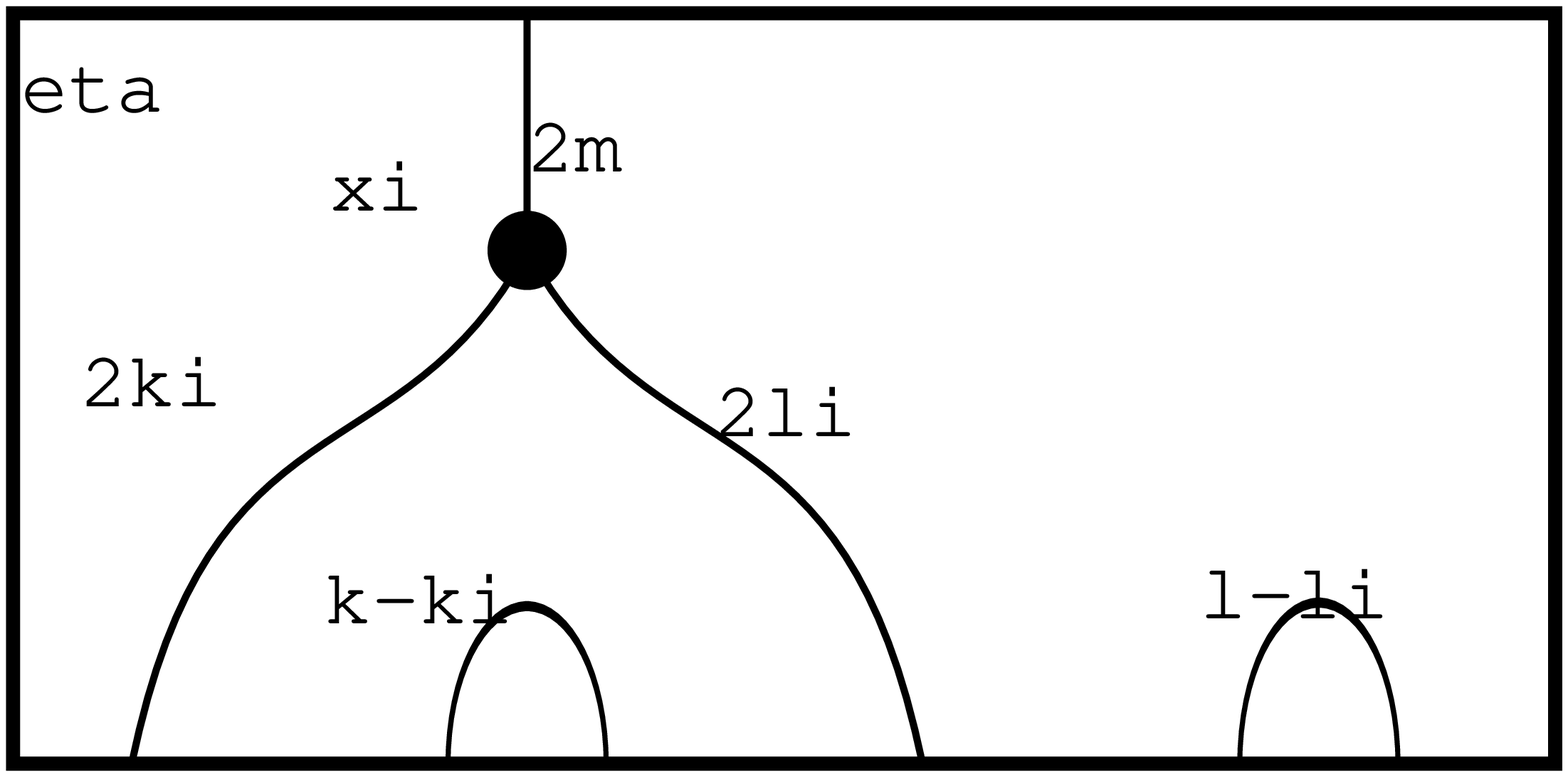}
})\in AP_{\eta (k+l) , \eta m}$.
\end{tabular}\\
Applying Lemma \ref{ccprop} (iii), we get $\lab \cdot , \xi_i \os{a_i}{\otimes} \zeta_i \rab = \lab \cdot , \tilde \xi_i \os{\tilde a_i}{\otimes} \tilde \zeta_i \rab$, and thereby, $\lab \xi_i \os{a_i}{\otimes} \zeta_i , \cdot \rab = \lab \tilde \xi_i \os{\tilde a_i}{\otimes} \tilde \zeta _i, \cdot \rab$.
This gives the desired result.
\end{proof}
\begin{prop}\label{ipsd}
$\lab \cdot , \cdot \rab$ is positive semi-definite.
\end{prop}
\begin{proof}
Let $u \in U_{\eta m}$.
By Lemma \ref{wlogsamekl}, it is enough to show $\lab u , u \rab \geq 0$ for $u = \us{i \in I}{\sum} \xi_i \os{a_i}{\otimes} \zeta_i$ for $\xi_i \in V_{\eta k}$, $\zeta_i \in W_{\eta l}$, $a_i = \psi^0_{\eta (k + l) , \eta m} (x_i) \in AP^{=0}_{\eta (k + l) , \eta m}$ where $i$ runs over a finite set $I$.
Set $u_i := \xi_i \os{a_i}{\otimes} \zeta_i$.
Note that 
$\lab u_j , u_i \rab = P_{\!\!\!\!\!\!\!\!\!\!\!\!\!
\psfrag{xi}{$x_i$}
\psfrag{x*j}{$x^*_j$}
\psfrag{2m}{$2m$}
\psfrag{2k}{$2k$}
\psfrag{2l}{$2l$}
\psfrag{c1}{$c_0 (\xi_j , \xi_i)$}
\psfrag{c2}{$c_0 (\zeta_j , \zeta_i)$}
\psfrag{eta}{$\eta$}
\includegraphics[scale=0.15]{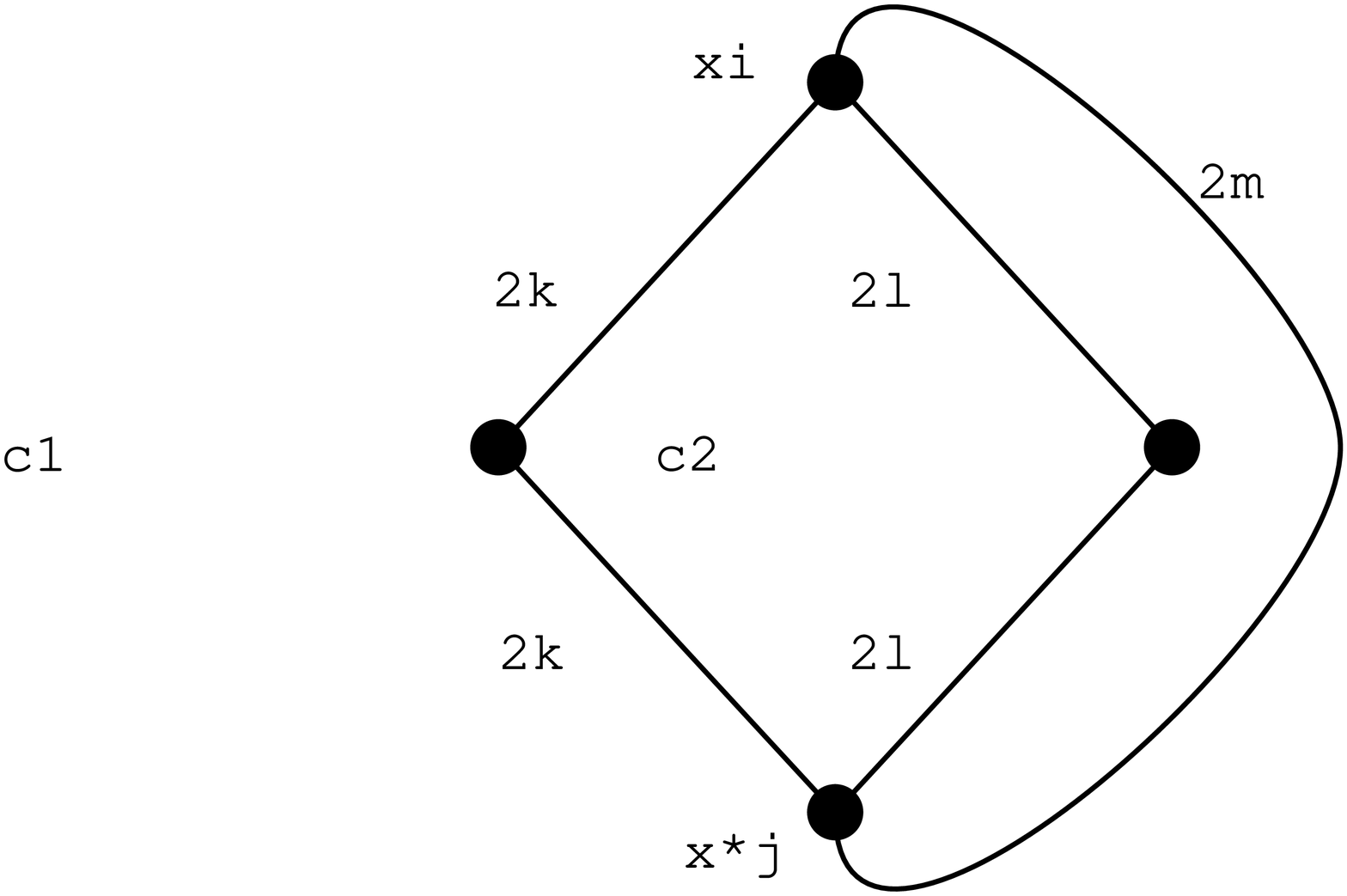}
} = P_{\!\!\!\!
\psfrag{2k}{$2k$}
\psfrag{c1}{$c_0 (\xi_j , \xi_i)$}
\psfrag{tij}{$t_{i,j}$}
\includegraphics[scale=0.15]{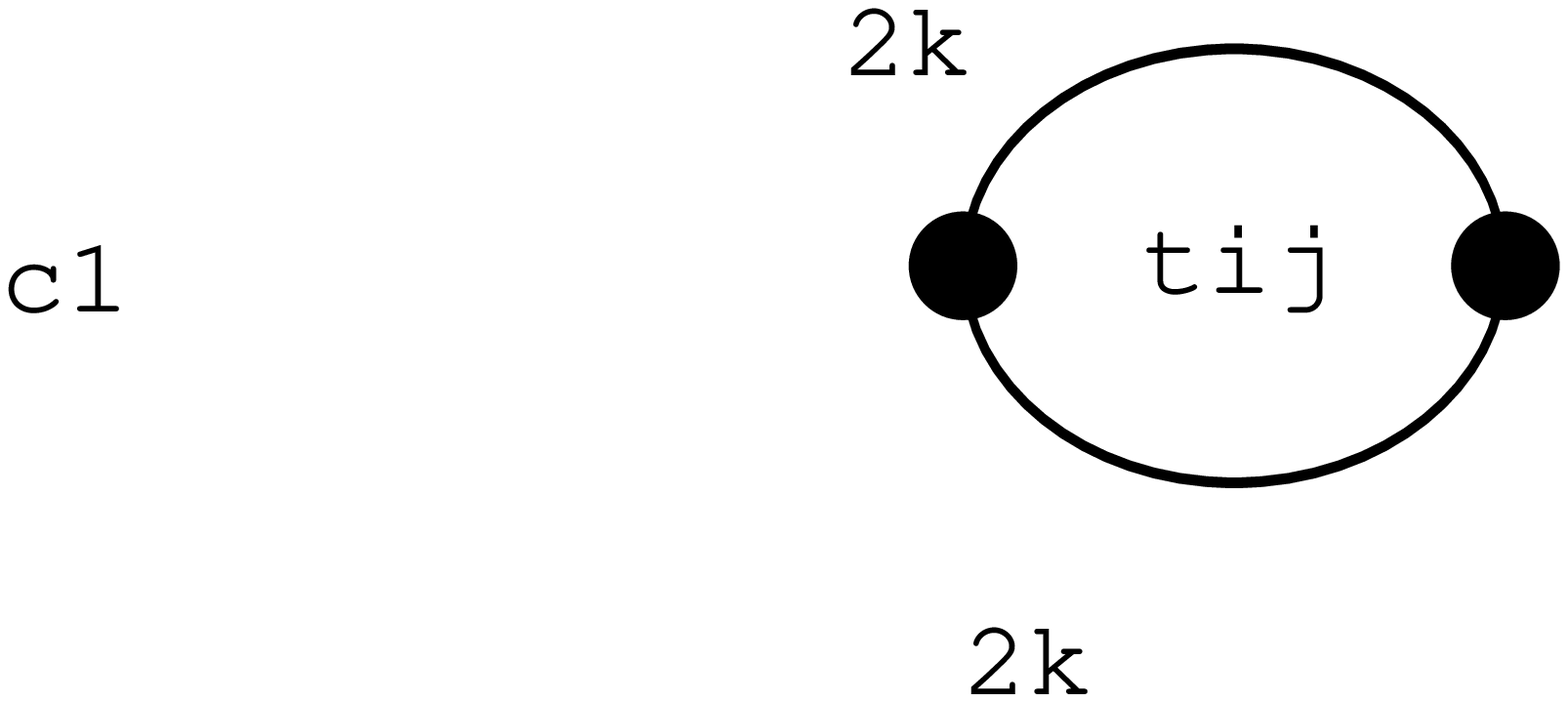}
}$ where $t_{i,j} := P_{
\psfrag{xi}{$x_i$}
\psfrag{x*j}{$x^*_j$}
\psfrag{2m}{$2m$}
\psfrag{2k}{$2k$}
\psfrag{2l}{$2l$}
\psfrag{c1}{$c_0 (\xi_j , \xi_i)$}
\psfrag{c2}{$c_0 (\zeta_j , \zeta_i)$}
\psfrag{eta}{$\eta$}
\includegraphics[scale=0.15]{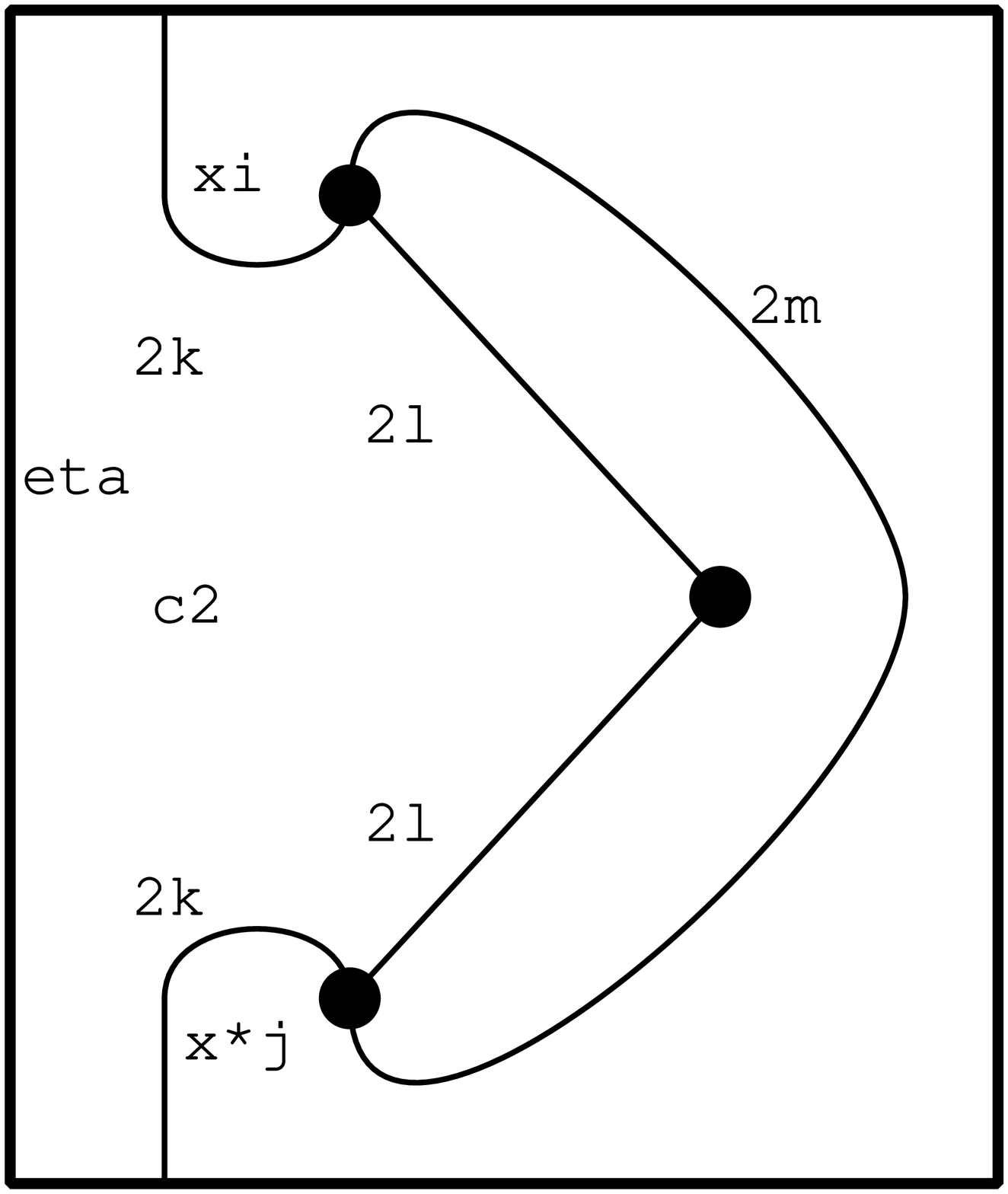}
}\in P_{\eta 2k}$ 
Now, by Lemma \ref{ccprop} (iv) and complete positivity of the conditional expectation from $P_{\eta 2(k+m)}$ to $P_{\eta 2k}$, we may conclude that $t := \us{i,j \in I}{\sum} E_{i,j} \otimes t_{i,j}$ is positive in $M_I \otimes P_{\eta 2k}$.
Let $s := \us{i,j \in I}{\sum} E_{i,j} \otimes s_{i,j}$ be the positive square root of $t$.
Thus, $\lab u , u \rab = \us{i,j \in I}{\sum} \lab u_j , u_i \rab = \us{j' \in I}{\sum} \; \us{i,j \in I}{\sum} P_{\!\!\!\!\!\!\!\!
\psfrag{s}{$s_{i,j'}$}
\psfrag{s*}{$s^*_{j,j'}$}
\psfrag{2k}{$2k$}
\psfrag{c1}{$c_0 (\xi_j , \xi_i)$}
\includegraphics[scale=0.15]{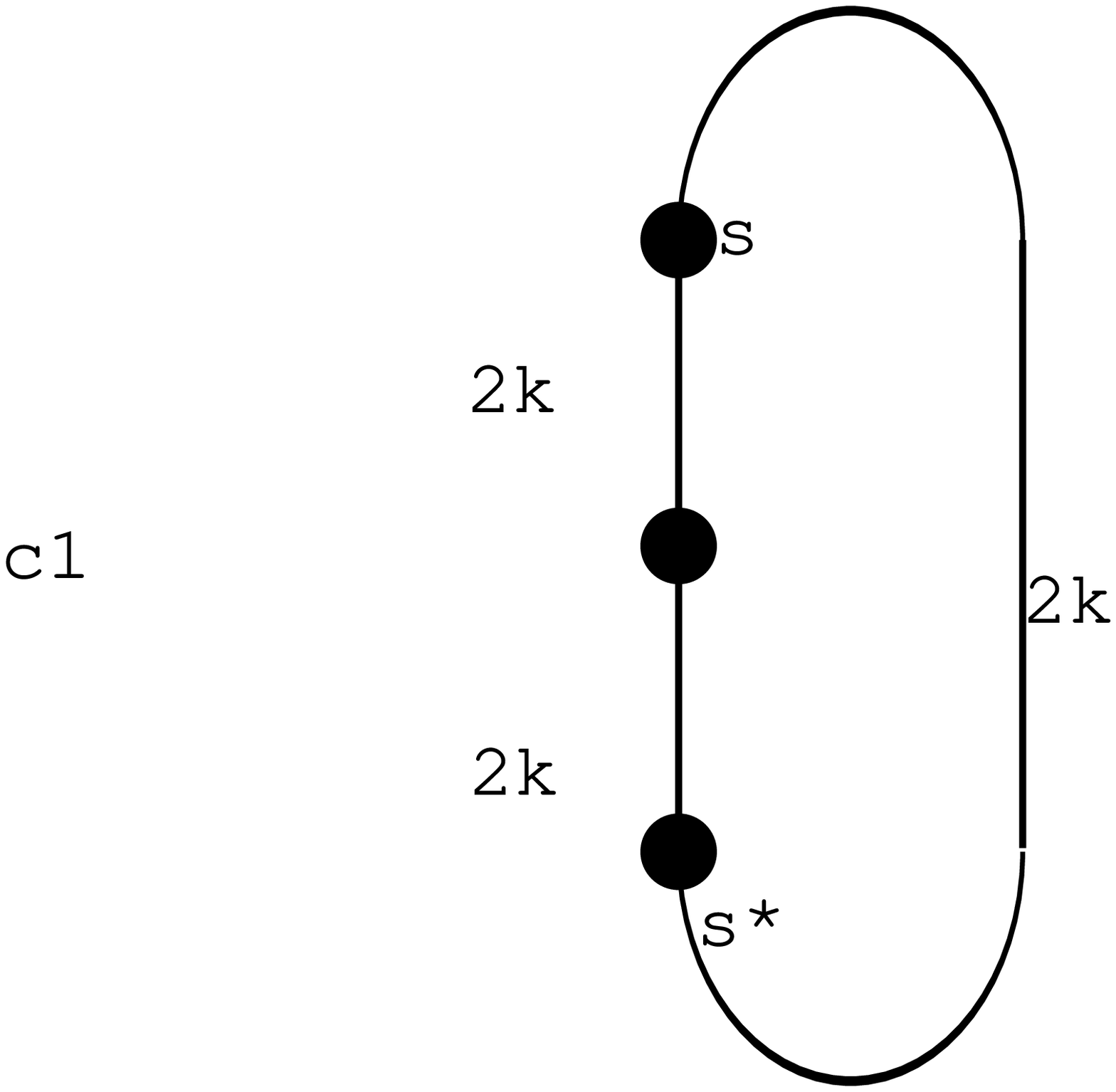}
}$ where the second equality follows from $s^2 = t$ and self-adjointness of $s$.
Now, for all $j' \in I$, applying Lemma \ref{ccprop} (iv) and positive of the action of the trace tangle, we get $\us{i,j \in I}{\sum} P_{\!\!\!\!\!\!\!\!
\psfrag{s}{$s_{i,j'}$}
\psfrag{s*}{$s^*_{j,j'}$}
\psfrag{2k}{$2k$}
\psfrag{c1}{$c_0 (\xi_j , \xi_i)$}
\includegraphics[scale=0.15]{figures/fusion/ujuis.eps}
} \geq 0$.
Hence, $\lab u , u \rab \geq 0$.
\end{proof}
We consider the null subspace with respect to $\lab \cdot , \cdot \rab$ in $U_{\eta m}$ and let $\tilde U_{\eta m}$ be the quotient where the quotient map is written as $U_{\eta m} \ni \xi \os{a}{\otimes} \zeta \mapsto \xi \os{a}{\btimes} \zeta \in \tilde U_{\eta m}$.
\begin{rem}
From Lemma \ref{wlogsamekl}, we get
\[
\tilde U_{\eta m} = \t{span} \left\{ \left( \xi \os{\psi^0_{\eta(k+l) , \eta m} (x)}{\btimes} \zeta \right) :  k,l \in \N_0 , \xi\in V_{\eta k}, \zeta \in W_{\eta l}, x\in P_{\eta (k+l+m)} \right\}.
\]
In fact, a closer look at the proof of Lemma \ref{wlogsamekl} implies that we could take $k = l$ in the spanning set. We will simplify the notation $\left( \xi \os{\psi^0_{\eta(k+l) , \eta m} (x)}{\btimes} \zeta \right)$ and just write $\xi \os{x}{\btimes} \zeta$ to denote it.
\end{rem}
We now proceed towards defining the action of affine morphisms on $\tilde U_{\eta m}$. The collection $\{U_{\eta m}\}_{\eta m \in \t{Col}}$ has an obvious action of affine morphisms, namely, $AP_{\eta m , \nu n} \times U_{\eta m} \ni (b , \xi \os{a}{\otimes} \zeta \mapsto b \cdot (\xi \os{a}{\otimes} \zeta) := (\xi \os{b\circ a}{\otimes} \zeta) \in U_{\nu n}$.
In order to induce this action on $\{\tilde U_{\eta m}\}_{\eta m \in \t{Col}}$, we need to check that the null subspace of $\lab \cdot , \cdot \rab$ is preserved under the action.
We will prove this as well as the boundedness of the action by proving the following proposition.
\begin{prop}
For all $b \in AP_{\eta m , \nu n}$, there exists $M \in (0 , \infty)$ satisfying $\lab b \cdot u , b \cdot u \rab \leq M \lab u, u \rab$ for every $u \in U_{\eta m}$.
\end{prop}
\begin{proof}
There exists $n' \in \N_{\eta , \nu}$ and $y \in P_{\nu (m+n+n')}$ such that $b = \psi^{n'}_{\eta m , \nu n} (y)$.
Without loss of generality, we may assume $u = \us{i \in I}{\sum} \left( \xi_i \os{\psi^0_{\eta (k+l) , \eta m} (x_i)}{\otimes} \zeta_i \right)$ for some finite set $I$, $k , l \in \N_0$, $\xi_i \in V_{\eta k}$, $\zeta_i \in W_{\eta l}$, $x_i \in P_{\eta (k+l+m)}$.
To see this, one needs to first go through the arguments in the proofs of Lemma \ref{fuswloglem} and \ref{wlogsamekl}, and show that for all $\xi \in V_{\vlon k}$, $\zeta \in W_{\vlon k}$ and $a \in AP_{\vlon (k+l) , \eta m}$, there exists large enough $k', l' \in \N_0$, $\xi' \in V_{\eta k'}$, $\zeta' \in W_{\eta l'}$ and $a' \in AP^{=0}_{\eta (k'+l') , \eta m}$ such that $\lab \cdot , \xi \os{b \circ a}{\otimes} \zeta \rab = \lab \cdot , \xi' \os{b \circ a'}{\otimes} \zeta' \rab$ and $\lab \xi \os{b \circ a}{\otimes} \zeta , \cdot \rab = \lab \xi' \os{b \circ a'}{\otimes} \zeta' , \cdot \rab$.
Now, 
\[
\lab b \cdot u , b \cdot u \rab = \us{i,j \in I}{\sum} \lab \xi_j \os{b \circ a_j}{\otimes} \zeta_j , \xi_i \os{b \circ a_i}{\otimes} \zeta_i \rab = \us{i,j \in I}{\sum} P_{ \!\!\!\!\!\!
\psfrag{2n}{$2n$}
\psfrag{2n'}{$2n'$}
\psfrag{n'}{$n'$}
\psfrag{2k}{$2k$}
\psfrag{2l}{$2l$}
\psfrag{xi}{$x_i$}
\psfrag{xj*}{$x^*_j$}
\psfrag{y}{$y$}
\psfrag{y*}{$y^*$}
\psfrag{c1}{$c_{2n'} (\xi_j , \xi_i)$}
\psfrag{c2}{$c_{2n'} (\zeta_j , \zeta_i)$}
\psfrag{2m}{$2m$}
\includegraphics[scale=0.15]{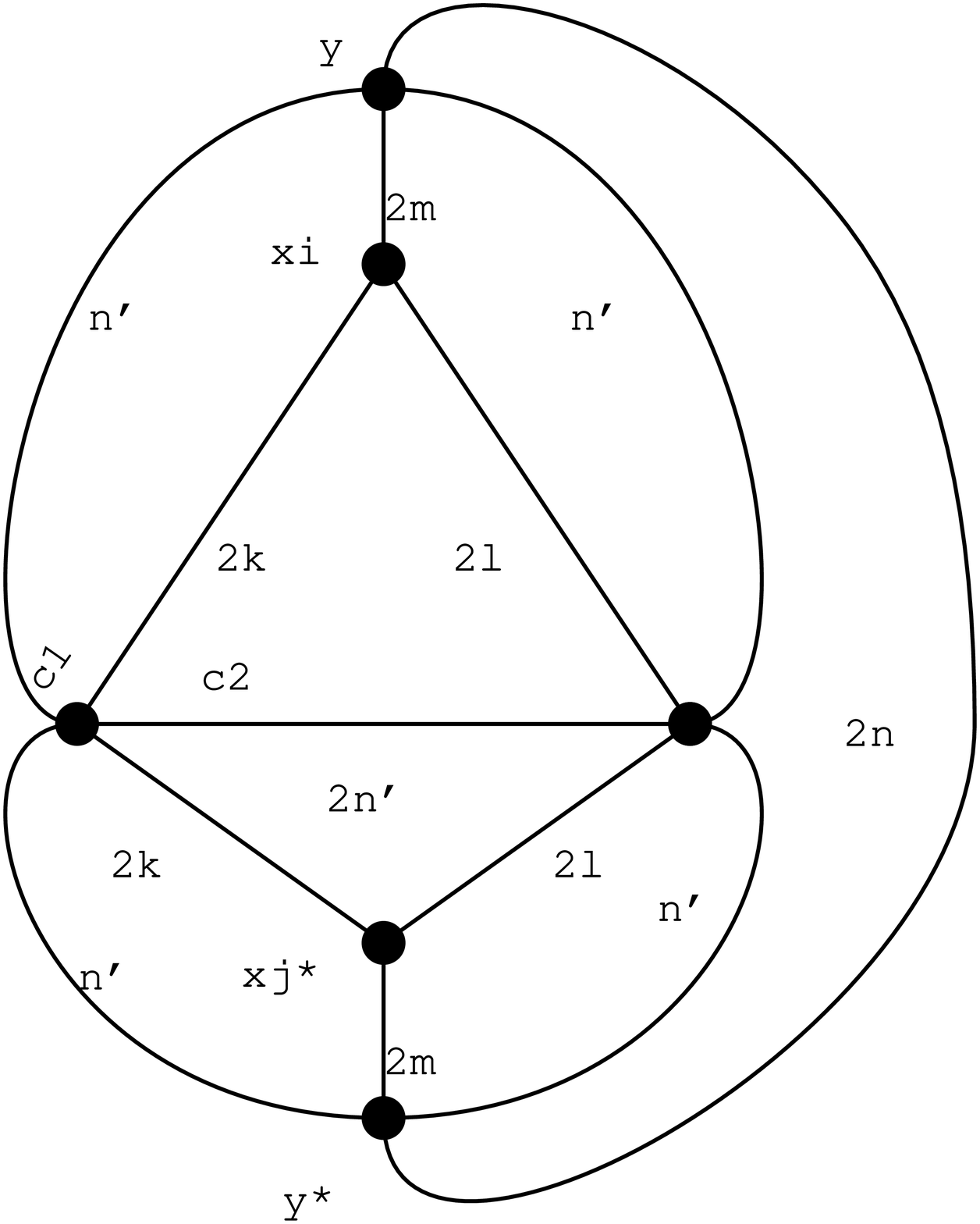}
} = f(\tilde y)
\]
where $\tilde y = P_{\,
\psfrag{2n}{$2n$}
\psfrag{2n'}{$2n'$}
\psfrag{n'}{$n'$}
\psfrag{2k}{$2k$}
\psfrag{2l}{$2l$}
\psfrag{nu}{$\nu$}
\psfrag{xj*}{$x^*_j$}
\psfrag{y}{$y$}
\psfrag{y*}{$y^*$}
\psfrag{c1}{$c_{2n'} (\xi_j , \xi_i)$}
\psfrag{c2}{$c_{2n'} (\zeta_j , \zeta_i)$}
\psfrag{2m}{$2m$}
\includegraphics[scale=0.15]{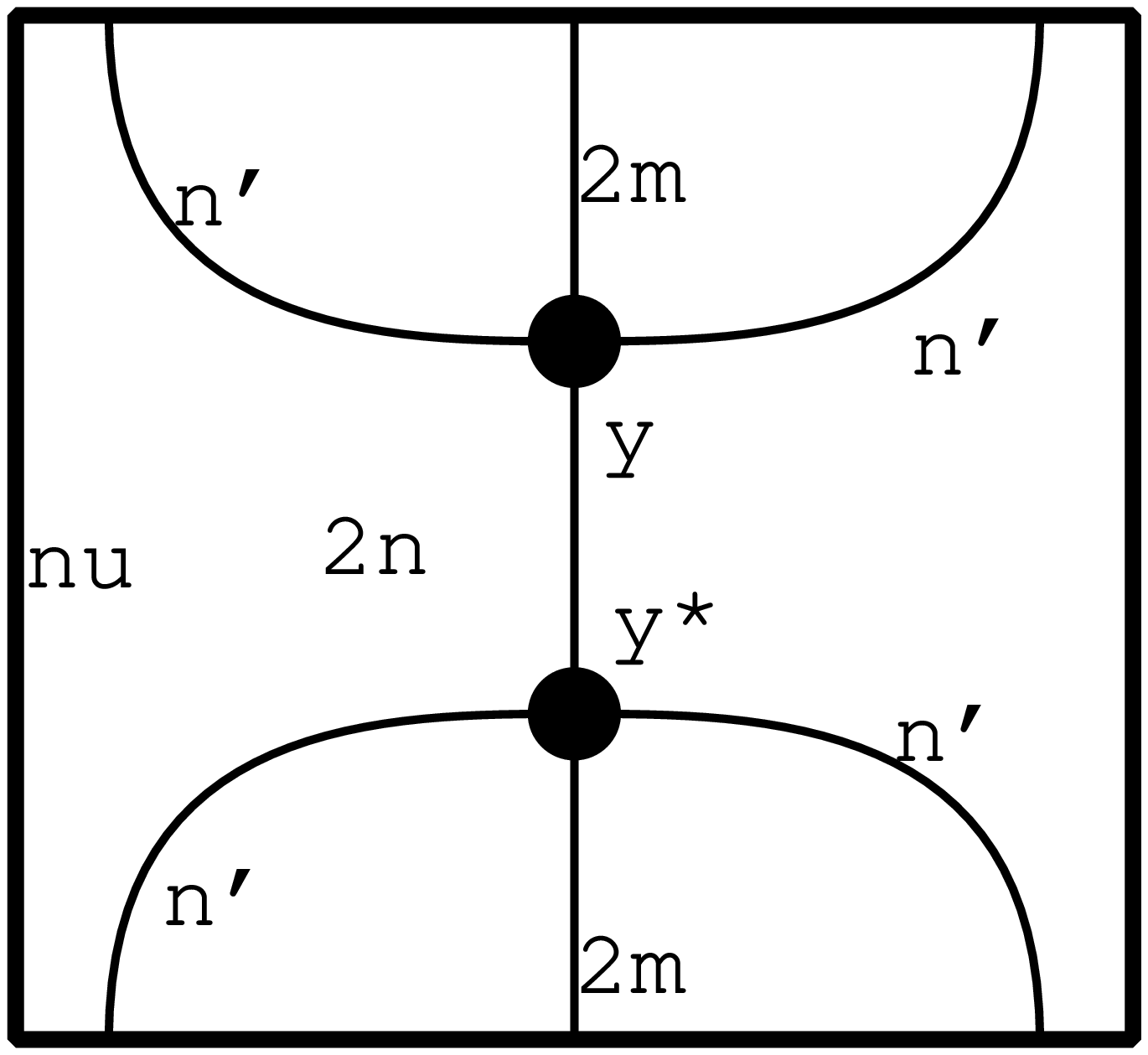}
} \in P_{\nu 2(m+n')}$ and $f : P_{\nu 2(m+n')} \ra \C$ is given by $\us{i,j\in I}{\sum} P_{ \!\!\!\!\!\!
\psfrag{2n}{$2n$}
\psfrag{2n'}{$2n'$}
\psfrag{n'}{$n'$}
\psfrag{2k}{$2k$}
\psfrag{2l}{$2l$}
\psfrag{xi}{$x_i$}
\psfrag{xj*}{$x^*_j$}
\psfrag{nu}{$\nu$}
\psfrag{y*}{$y^*$}
\psfrag{c1}{$c_{2n'} (\xi_j , \xi_i)$}
\psfrag{c2}{$c_{2n'} (\zeta_j , \zeta_i)$}
\psfrag{2m}{$2m$}
\includegraphics[scale=0.15]{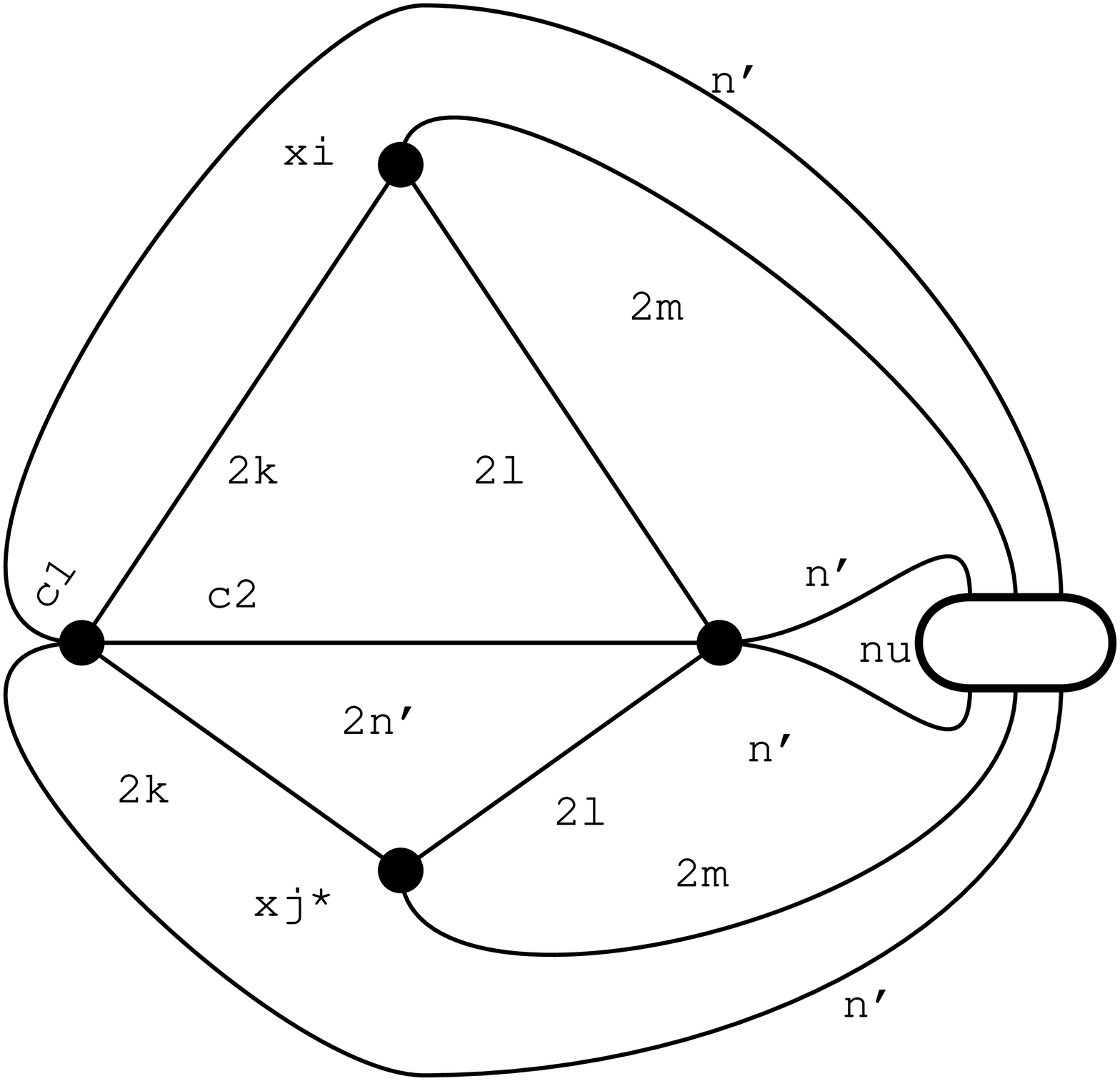}
}$ 
\vskip 2mm
We will first prove that $f$ is positive.
Let $t$ be a positive element of $P_{\nu 2(m+n')}$ and $s$ be its positive square root.
Note that $f(t) = \us{i,j \in I}{\sum} P_{ \!\!\!\!\!\!
\psfrag{2n}{$2(m+n')$}
\psfrag{2n'}{$2n'$}
\psfrag{n'}{$n'$}
\psfrag{2k}{$2k$}
\psfrag{2l}{$2l$}
\psfrag{xi}{$x_i$}
\psfrag{xj*}{$x^*_j$}
\psfrag{y}{$s$}
\psfrag{c1}{$c_{2n'} (\xi_j , \xi_i)$}
\psfrag{c2}{$c_{2n'} (\zeta_j , \zeta_i)$}
\psfrag{2m}{$2m$}
\includegraphics[scale=0.15]{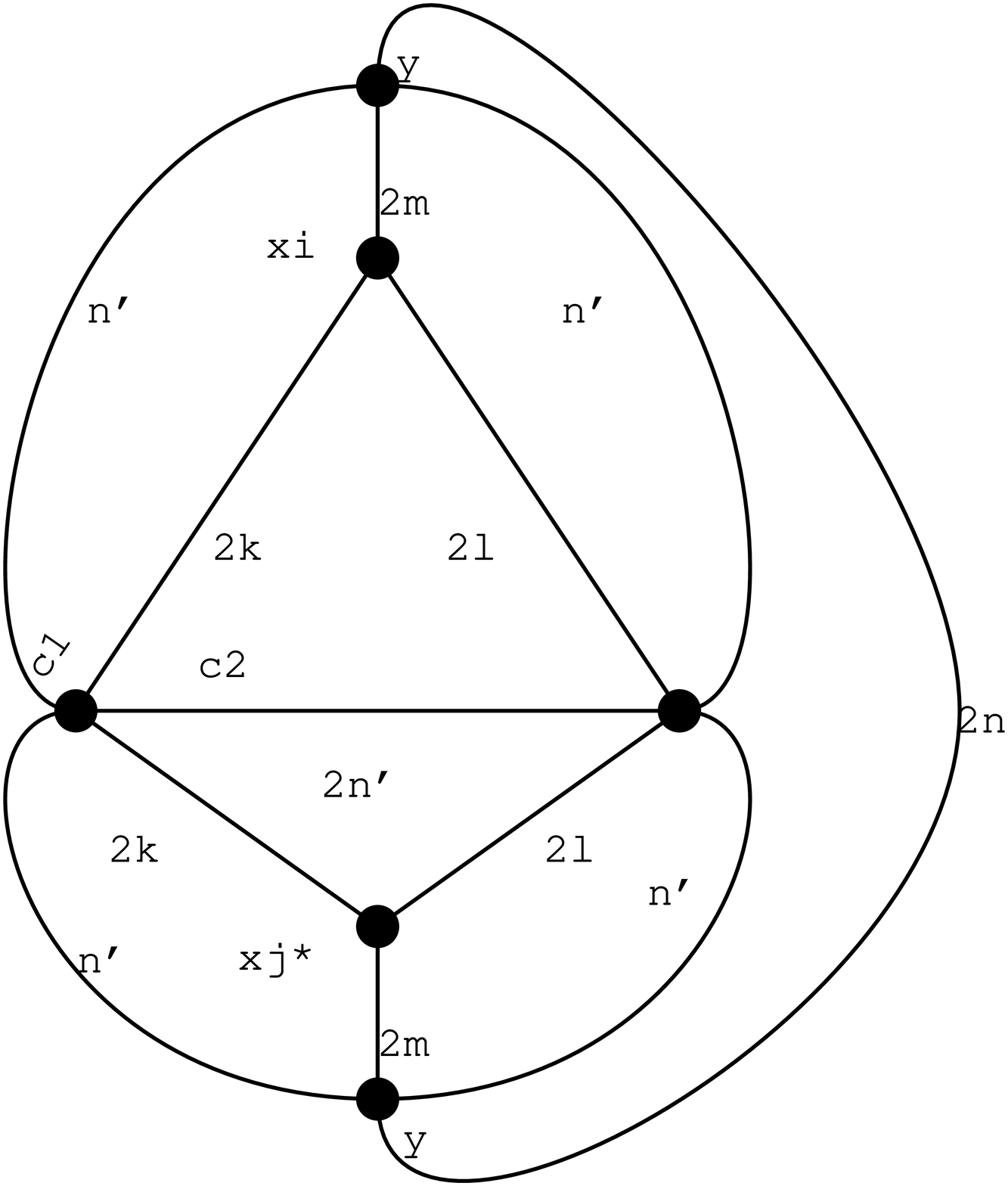}
}\!\!\!\!\!\!\!\!\!\!\!\!\! = \us{i, j \in I}{\sum} \left\lab \left( \xi_j \os{b' \circ a_j}{\otimes} \zeta_j \right) , \left( \xi_i \os{b' \circ a_i}{\otimes} \zeta_i \right) \right\rab$ where $a_i = \psi^0_{\eta (k+l) , \eta m} (x_i)$ for all $i\in I$ and $b' := \psi^{n'}_{\eta m , \nu (m+n')} (P_{
\psfrag{2n}{$2(m+n')$}
\psfrag{2n'}{$2n'$}
\psfrag{n'}{$n'$}
\psfrag{nu}{$\nu$}
\psfrag{s}{$s$}
\psfrag{xi}{$x_i$}
\psfrag{xj*}{$x^*_j$}
\psfrag{y}{$s$}
\psfrag{c1}{$c_{2n'} (\xi_j , \xi_i)$}
\psfrag{c2}{$c_{2n'} (\zeta_j , \zeta_i)$}
\psfrag{2m}{$2m$}
\includegraphics[scale=0.15]{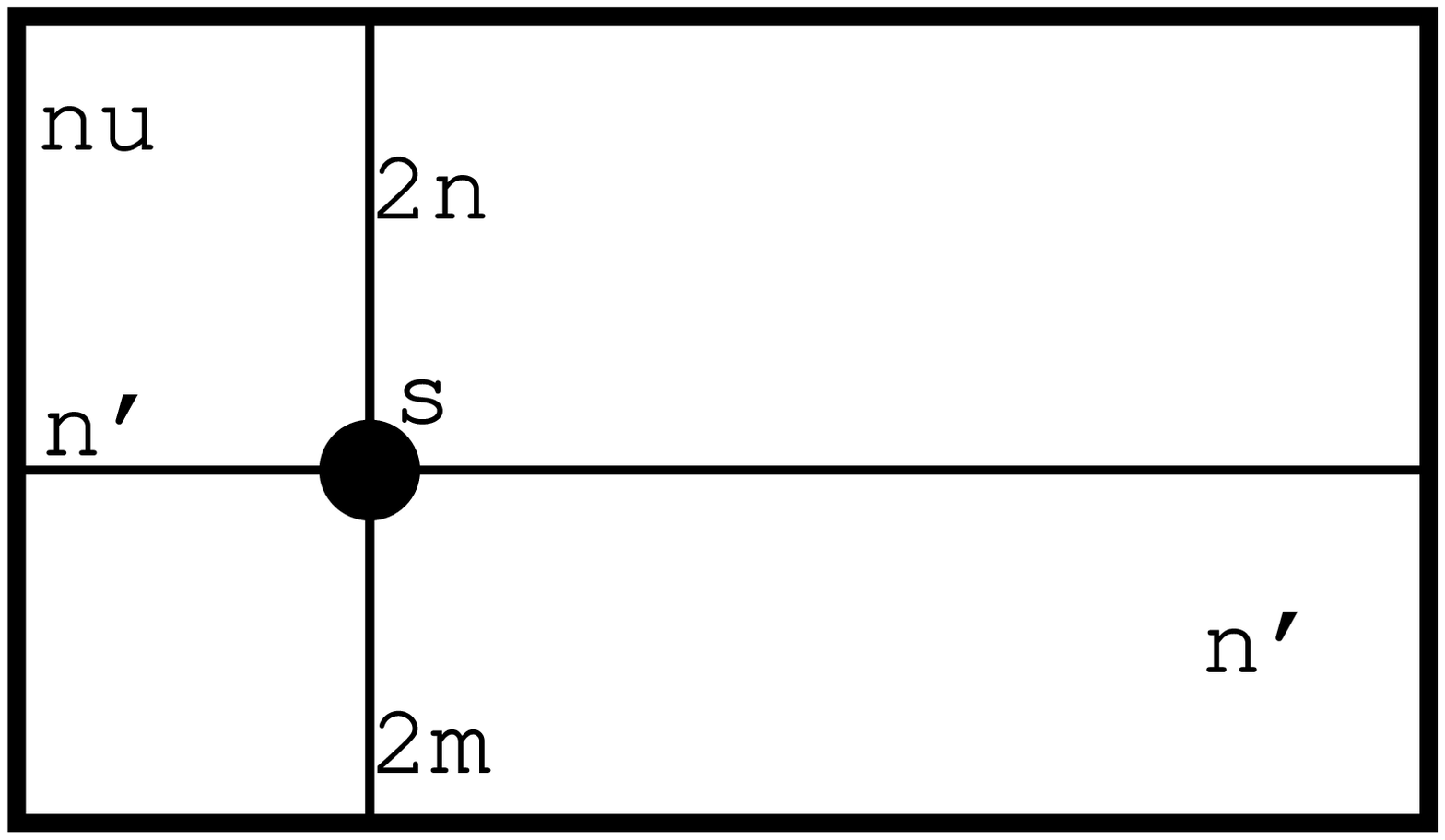}
})$.
By Proposition \ref{ipsd}, we get $f(t) \geq 0$.
Using positivity of the linear functional $f$ on the unital $C^*$-algebra $P_{\nu 2(m+n')}$, we get
\[
\lab b \cdot u , b \cdot u \rab = f(\tilde y) \leq \norm{\tilde y} f(1_{P_{\nu 2(m+n')}}) = \norm{\tilde y} \us{i,j\in I}{\sum} P_{ \!\!\!\!\!\!
\psfrag{2n}{$2n$}
\psfrag{2n'}{$2n'$}
\psfrag{n'}{$n'$}
\psfrag{2k}{$2k$}
\psfrag{2l}{$2l$}
\psfrag{xi}{$x_i$}
\psfrag{xj*}{$x^*_j$}
\psfrag{nu}{$\nu$}
\psfrag{y*}{$y^*$}
\psfrag{c1}{$c_{2n'} (\xi_j , \xi_i)$}
\psfrag{c2}{$c_{2n'} (\zeta_j , \zeta_i)$}
\psfrag{2m}{$2m$}
\includegraphics[scale=0.15]{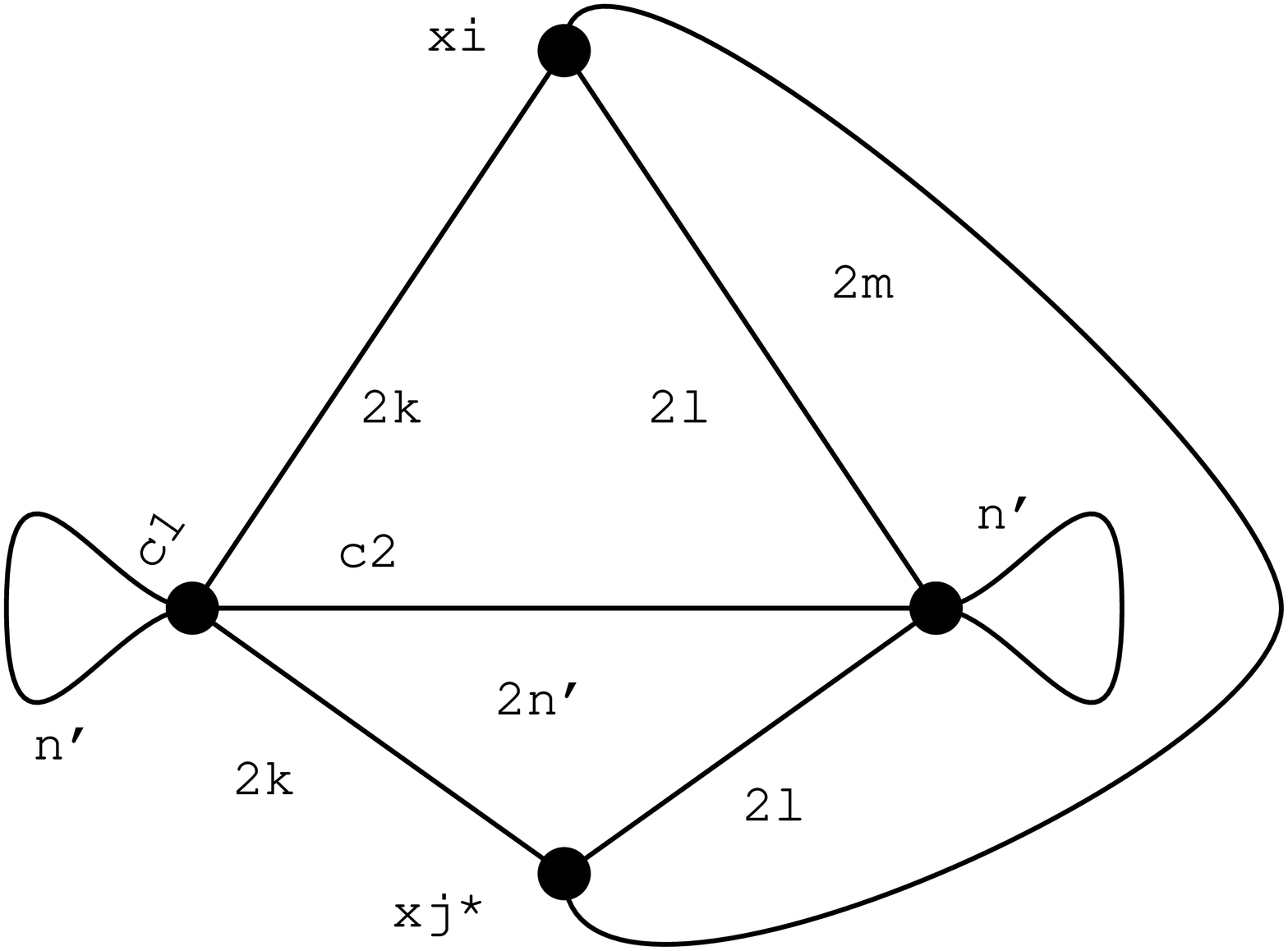}
} = \delta^{n'} \norm{\tilde y} \lab u , u \rab
\]
The second equality is acheived by collapsing the  left and the right caps (of $n'$ strings each) against each other which is made possible by $c(\xi_j , \xi_i) \in CC_{\eta k , \eta k}$ and $c(\zeta_j , \zeta_i) \in CC_{\eta l , \eta l}$; this produces a multiplicative factor of $\delta ^{n'}$; finally, one has to use the definition of $\lab \cdot , \cdot \rab$ to show that the remaining thing is indeed $\lab u , u \rab$.
Now, $\tilde y$ and $n'$ are completely independent of the choice of $u$ in $U_{\eta m}$.
Hence, we have the required inequality.
\end{proof}
Thus, the action $\tilde U_{\eta m} \ni \xi \os{a}{\btimes} \zeta \os {\tilde U_b} \longmapsto \xi \os{b\circ a}{\btimes} \zeta \in \tilde U_{\nu n}$ is well-defined and bounded for all $b \in AP_{\eta m , \nu n}$.
It is almost immediate that the action preserves $*$.
Hence, $\tilde U$ is a bounded $*$-affine $P$-module.
\vskip 2mm
We will use the symbol $V \btimes W$ to denote the completion of $\tilde U$, which then becomes a Hilbert affine $P$-module.
We would like to show that $\btimes$ extends to a functor but before that, we have to specify our category.
The obvious thing to consider would be Hilbert affine $P$-modules as objects and a morphism $f : V \ra W$ would be a natural transformation such that $f_{\vlon k} : V_{\vlon k} \ra W_{\vlon k}$ is bounded for all $\vlon k \in \t{Col}$.
However, from Proposition \ref{fXgbdd}, it will be apparent that the fusion of morphisms might not satisfy the extra condition of boundedness.
The reason is that $\{\norm{f_{\vlon k}}\}_{\vlon k \in \t{Col}}$ might not be bounded.
It is easy to see that $\norm{f_{+k}} = \norm{f_{-k}}$ and $\norm{f_{\pm (k-1)}} \leq \norm{f_{\pm k}}$  for all $k \in \N$.
Set $\norm{f} := \us{k \in \N}{\t{sup}} \norm{f_{\pm k}}$ which may very well be $\infty$.
With little effort, one can show that $\norm{f} < \infty$ if both $V$ and $W$ have finite supports; on the other hand, finiteness of the support might not be preserved under $\btimes$. Keeping this in mind, we define our ideal category in the following way.
\begin{defn}
$\mcal H APM$ will denote the category whose objects are Hilbert affine $P$-modules and morphisms are natural transformations $f : V \ra W$ having $\norm{f} < \infty$.
\end{defn}
The canonical dense subset $\tilde U_{\eta m}$ of $(V \btimes W)_{\eta m}$ will be denoted by $(V\btimes W)^o_{\eta m}$.
Let $f:V \ra V'$ and $g:W \ra W'$ be morphisms of Hilbert affine $P$-modules.
From Lemma \ref{ccprop} (vi), it easily follows $\left\lab \left(\xi' \os {a'} \btimes \zeta'\right) , \left( f(\xi) \os {a} \btimes g(\zeta) \right) \right\rab = \left\lab \left(f^*(\xi') \os {a'} \btimes g^*(\zeta') \right) , \left(\xi \os {a} \btimes \zeta \right) \right\rab$ for all $\xi \in V_{\vlon k}$, $\zeta \in W_{\vlon l}$, $a \in AP_{\vlon (k+l) , \eta m}$, $\xi' \in V_{\vlon' k'}$, $\zeta' \in W_{\vlon' l'}$, $a' \in AP_{\vlon' (k'+l') , \eta m}$.
This gives a well-defined linear map $(V\btimes W)^o \ni (\xi \os a \btimes \zeta) \os{f\btimes g} \longmapsto (f(\xi) \os a \btimes g(\zeta)) \in (V' \btimes W')^o$.
\begin{prop}\label{fXgbdd}
$f\btimes g$ is norm bounded.
\end{prop}
\begin{proof}
Consider an element $u = \us{i \in I}{\sum} \xi_i \os{x_i}\btimes \zeta_i$ in $(V\btimes W)^o_{\eta m}$ where $I$ is a finite set, $\xi_i \in V_{\eta k}$, $\zeta_i \in W_{\eta l}$, $x_i \in P_{\eta (k+l+m)}$ for $i\in I$.
Set $u_i := \xi_i \os{x_i}\btimes \zeta_i$, $\xi'_i := f(\xi_i)$, $\zeta'_i := g(\zeta_i)$, $u'_i := (f \btimes g) u_i = \xi'_i \os{x_i}\btimes \zeta'_i$ for $i \in I$.
Note that $\lab u'_j , u'_i \rab = P_{\!\!\!\!\!\!\!\!\!\!\!\!\!
\psfrag{xi}{$x_i$}
\psfrag{x*j}{$x^*_j$}
\psfrag{2m}{$2m$}
\psfrag{2k}{$2k$}
\psfrag{2l}{$2l$}
\psfrag{c1}{$c_0 (\xi'_j , \xi'_i)$}
\psfrag{c2}{$c_0 (\zeta'_j , \zeta'_i)$}
\psfrag{eta}{$\eta$}
\includegraphics[scale=0.15]{figures/fusion/ujui.eps}
} = P_{\!\!\!\!
\psfrag{2k}{$2k$}
\psfrag{c1}{$c_0 (\xi'_j , \xi'_i)$}
\psfrag{tij}{$t_{i,j}$}
\includegraphics[scale=0.15]{figures/fusion/ujui2.eps}
}$ where $t_{i,j} := P_{
\psfrag{xi}{$x_i$}
\psfrag{x*j}{$x^*_j$}
\psfrag{2m}{$2m$}
\psfrag{2k}{$2k$}
\psfrag{2l}{$2l$}
\psfrag{c1}{$c_0 (\xi_j , \xi_i)$}
\psfrag{c2}{$c_0 (\zeta'_j , \zeta'_i)$}
\psfrag{eta}{$\eta$}
\includegraphics[scale=0.15]{figures/fusion/fustij.eps}
}\in P_{\eta 2k}$.
Now, by Lemma \ref{ccprop} (iv) and complete positivity of the conditional expectation from $P_{\eta 2(k+m)}$ to $P_{\eta 2k}$, we may conclude that $t := \us{i,j \in I}{\sum} E_{i,j} \otimes t_{i,j}$ is positive in $M_I \otimes P_{\eta 2k}$.
Let $s := \us{i,j \in I}{\sum} E_{i,j} \otimes s_{i,j}$ be the positive square root of $t$.
If $\rho : P_{\eta 2k} \ra P_{\eta 2k}$ denotes the action of the $180$ degrees rotation tangle (from $\eta 2k$ to $\eta 2k$), that is, the marked points shift by $2k$ places, and $a_{i,j} = \psi^0_{\eta k, \eta k} (\rho (s_{i,j}))$ for $i,j \in I$, then 
\begin{align*}
\norm{(f\btimes g) u}^2 & = \us{i,j \in I}{\sum} \lab u'_j , u'_i \rab = \us{j' \in I}{\sum} \; \us{i,j \in I}{\sum} P_{\!\!\!\!\!\!\!\!
\psfrag{s}{$s_{i,j'}$}
\psfrag{s*}{$s^*_{j,j'}$}
\psfrag{2k}{$2k$}
\psfrag{c1}{$c_0 (\xi'_j , \xi'_i)$}
\includegraphics[scale=0.15]{figures/fusion/ujuis.eps}
} = \us{j' \in I}\sum \norm{\us{i \in I}\sum V'_{a_{i,j'} } (\xi'_i)}^2 = \us{j' \in I}\sum \norm{\us{i \in I}\sum V'_{a_{i,j'} } f(\xi_i)}^2\\
& = \us{j' \in I}\sum \norm{f\left( \us{i \in I}\sum V'_{a_{i,j'} } (\xi_i) \right) }^2 \leq \norm{f}^2 \us{j' \in I}\sum \norm{\us{i \in I}\sum V'_{a_{i,j'} } (\xi_i) }^2 = 
\norm{f}^2 \us{i,j \in I}{\sum} P_{\!\!\!\!\!\!\!\!\!\!\!\!\!
\psfrag{xi}{$x_i$}
\psfrag{x*j}{$x^*_j$}
\psfrag{2m}{$2m$}
\psfrag{2k}{$2k$}
\psfrag{2l}{$2l$}
\psfrag{c1}{$c_0 (\xi_j , \xi_i)$}
\psfrag{c2}{$c_0 (\zeta'_j , \zeta'_i)$}
\psfrag{eta}{$\eta$}
\includegraphics[scale=0.15]{figures/fusion/ujui.eps}
}.
\end{align*}
So, we have proved $\norm{\us{i \in I}\sum \left( f(\xi_i) \os{x_i}\btimes g(\zeta_i) \right)} \leq \norm{f} \norm{\us{i \in I}\sum \left( \xi_i \os{x_i}\btimes g(\zeta_i) \right)}$.
Applying similar technique on the second variable of $\btimes$, we will be able to show $\norm{\us{i \in I}\sum \left( \xi_i \os{x_i}\btimes g(\zeta_i) \right)} \leq \norm{g} \norm{\us{i \in I}\sum \left( \xi_i \os{x_i}\btimes \zeta_i \right)}$. Thus, $\norm{ \left. (f \btimes g) \right|_{(V\btimes W)_{\eta m}} } \leq \norm{f} \norm{g}$ which is independent of $m$.
\end{proof}
Composition of morphisms and identity morphisms are trivially preserved by $\btimes$.
Hence, $\btimes : \mcal H APM \times \mcal H APM \ra \mcal H APM$ is indeed a functor.
$\mcal HAPM$ does not appear to be strict with $\btimes$ as the tensor.
The associativity constraint is given by:
\[
\left[ (U\btimes V) \btimes W\right]^0_{\eta n'} \ni \left( \left(\xi \os{x}{\btimes} \zeta\right) \os{y}{\btimes} \omega \right) = \left( \left( \xi \os{1_{P_{\eta 2(k+l)}}}\btimes \zeta \right) \os{s}{\btimes} \omega \right) \mapsto \left( \xi \os{s}{\btimes} \left( \zeta \os{1_{P_{\eta 2(l+n)}}}\btimes \omega \right) \right) \in \left[ U \btimes \left( V \btimes W \right) \right]^0_{\eta n'}
\]
where $s := P_{\,
\psfrag{x}{$x$}
\psfrag{y}{$y$}
\psfrag{2k}{$2k$}
\psfrag{2l}{$2l$}
\psfrag{2m}{$2m$}
\psfrag{2n}{$2n$}
\psfrag{2n'}{$2n'$}
\psfrag{eta}{$\eta$}
\includegraphics[scale=0.15]{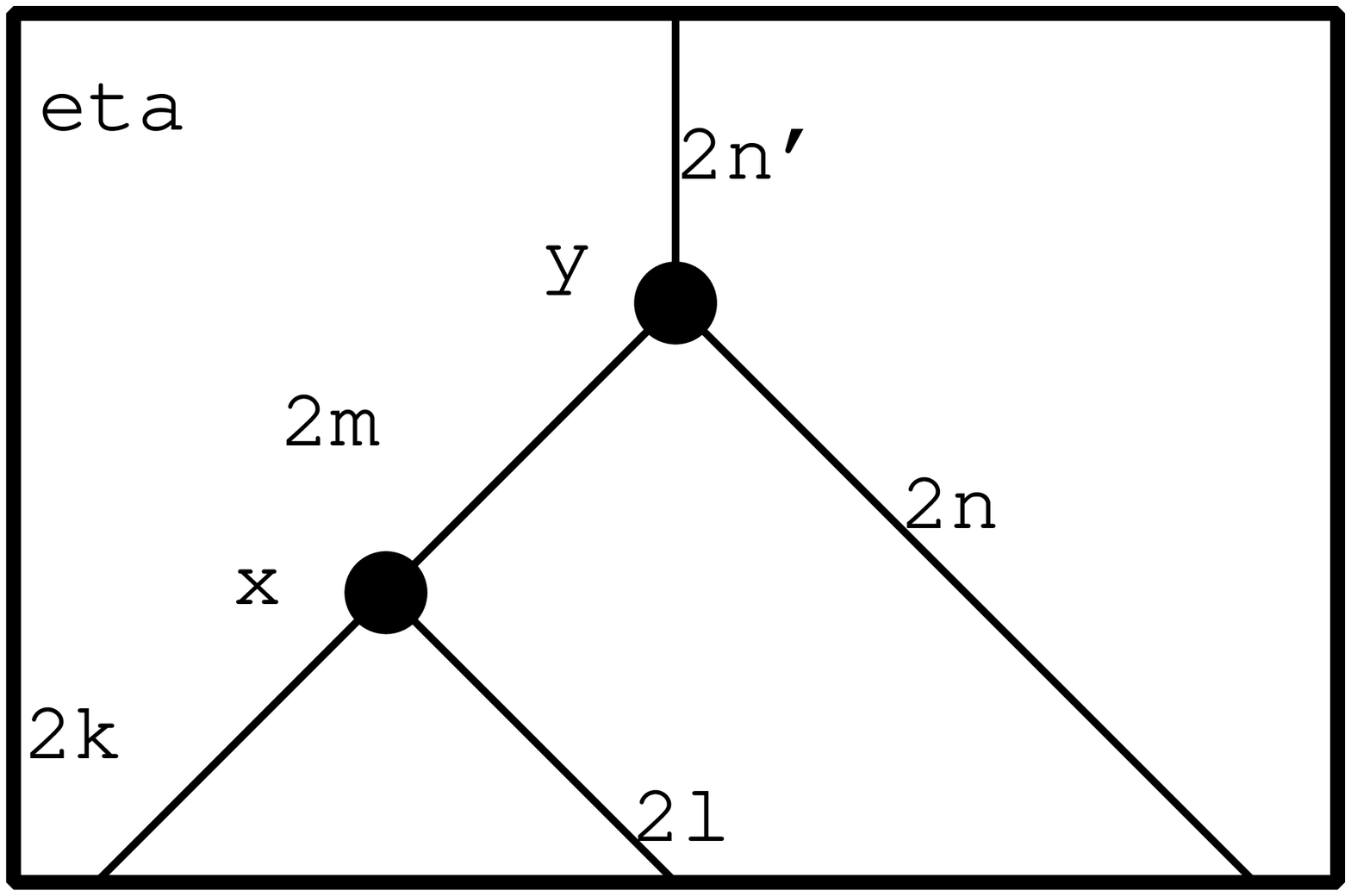}
}$ for $\xi \in U_{\eta k}$, $\zeta \in V_{\eta l}$, $\omega \in W_{\eta n}$, $x \in P_{\eta (k+l+m)}$, $y\in P_{\eta (m+n+n')}$.
It is easy to verify that the above map is an isometry and also has an inverse, namely,
\[
\left[ U \btimes \left( V \btimes W \right) \right]^0_{\eta n'} \ni \left(\xi \os{y}{\btimes} \left( \zeta \os{x}{\btimes} \omega \right) \right) = \left( \xi \os{t}{\btimes} \left(\zeta  \os{1_{P_{\eta 2(l+m)}}}\btimes \omega \right) \right) \mapsto \left( \left( \xi \os{1_{P_{\eta 2(k+l)}}}\btimes \zeta \right) \os{t}{\btimes} \omega \right) \in \left[ (U\btimes V) \btimes W\right]^0_{\eta n'}
\]
where $t := P_{\,
\psfrag{x}{$x$}
\psfrag{y}{$y$}
\psfrag{2k}{$2k$}
\psfrag{2l}{$2l$}
\psfrag{2m}{$2m$}
\psfrag{2n}{$2n$}
\psfrag{2n'}{$2n'$}
\psfrag{eta}{$\eta$}
\includegraphics[scale=0.15]{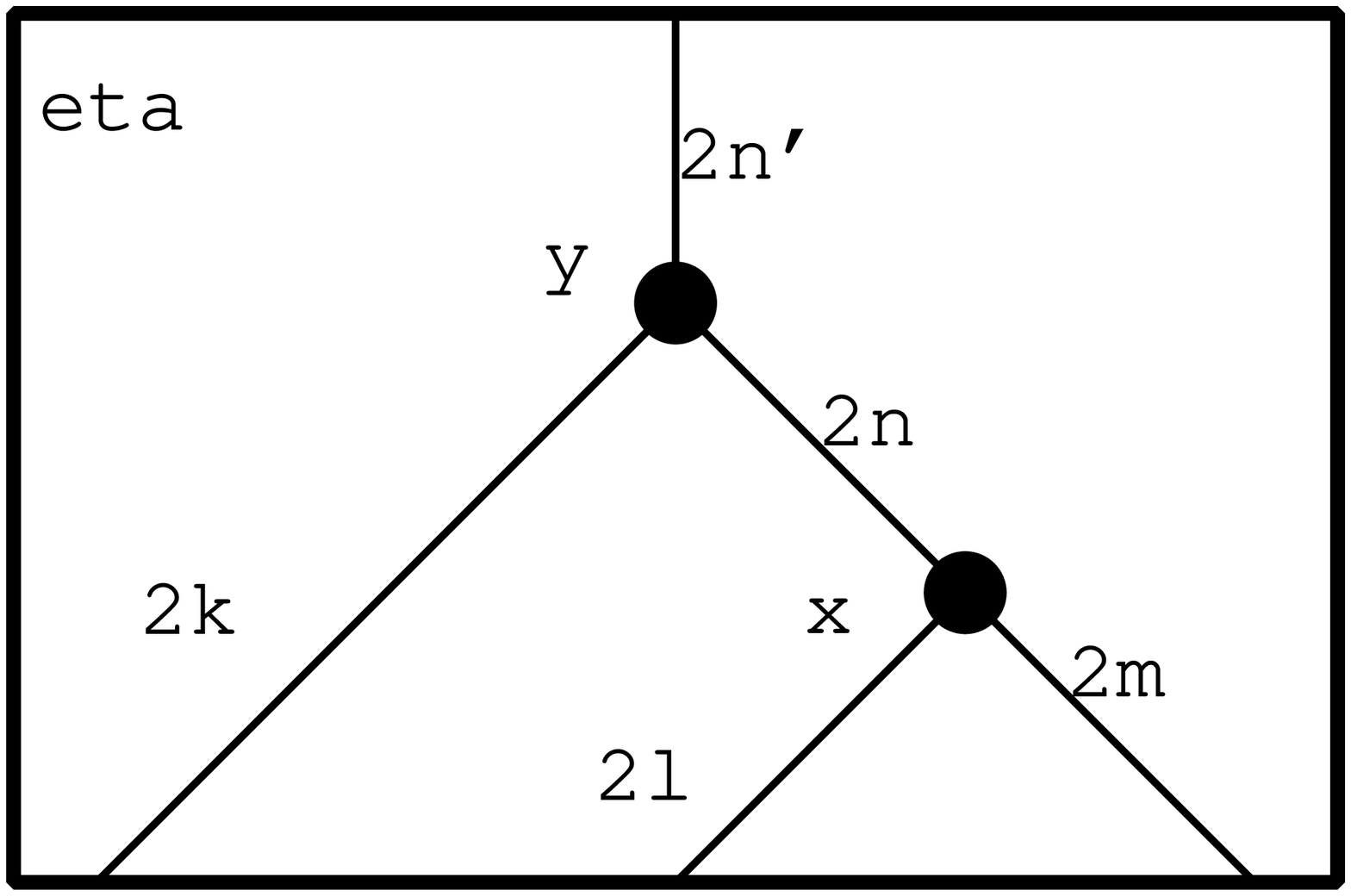}
}$ for $\xi \in U_{\eta k}$, $\zeta \in V_{\eta l}$, $\omega \in W_{\eta m}$, $x \in P_{\eta (l+m+n)}$, $y\in P_{\eta (k+n+n')}$; thus, it is a unitary between the dense subspaces.
Linearity of the above map with respect to the action of affine morphisms, follows from the following lemma whose proof is completely routine.

\begin{lem}\label{fusCCip}
The $\t{CC}$-valued inner product of $V\btimes W$ is given by:
\[
c_{n'} \left( \xi_1 \os{a_1}\btimes \zeta_1 , \xi_2 \os{a_2}\btimes \zeta_2 \right) =P_{\,
\psfrag{x1}{$x^*_1$}
\psfrag{x2}{$x_2$}
\psfrag{}{$$}
\psfrag{2k1}{$2k_1$}
\psfrag{2k2}{$2k_2$}
\psfrag{2l1}{$2l_1$}
\psfrag{2l2}{$2l_2$}
\psfrag{2m1}{$2m_1$}
\psfrag{2m2}{$2m_2$}
\psfrag{n1}{$n_1$}
\psfrag{n2}{$n_2$}
\psfrag{n'}{$n'$}
\psfrag{n'n1n2}{$n'+n_1+n_2$}
\psfrag{cv}{$c_{n'+n_1+n_2} (\xi_1,\xi_2)$}
\psfrag{cw}{$c_{n'+n_1+n_2} (\zeta_1,\zeta_2)$}
\psfrag{eta}{$\eta_2$}
\includegraphics[scale=0.15]{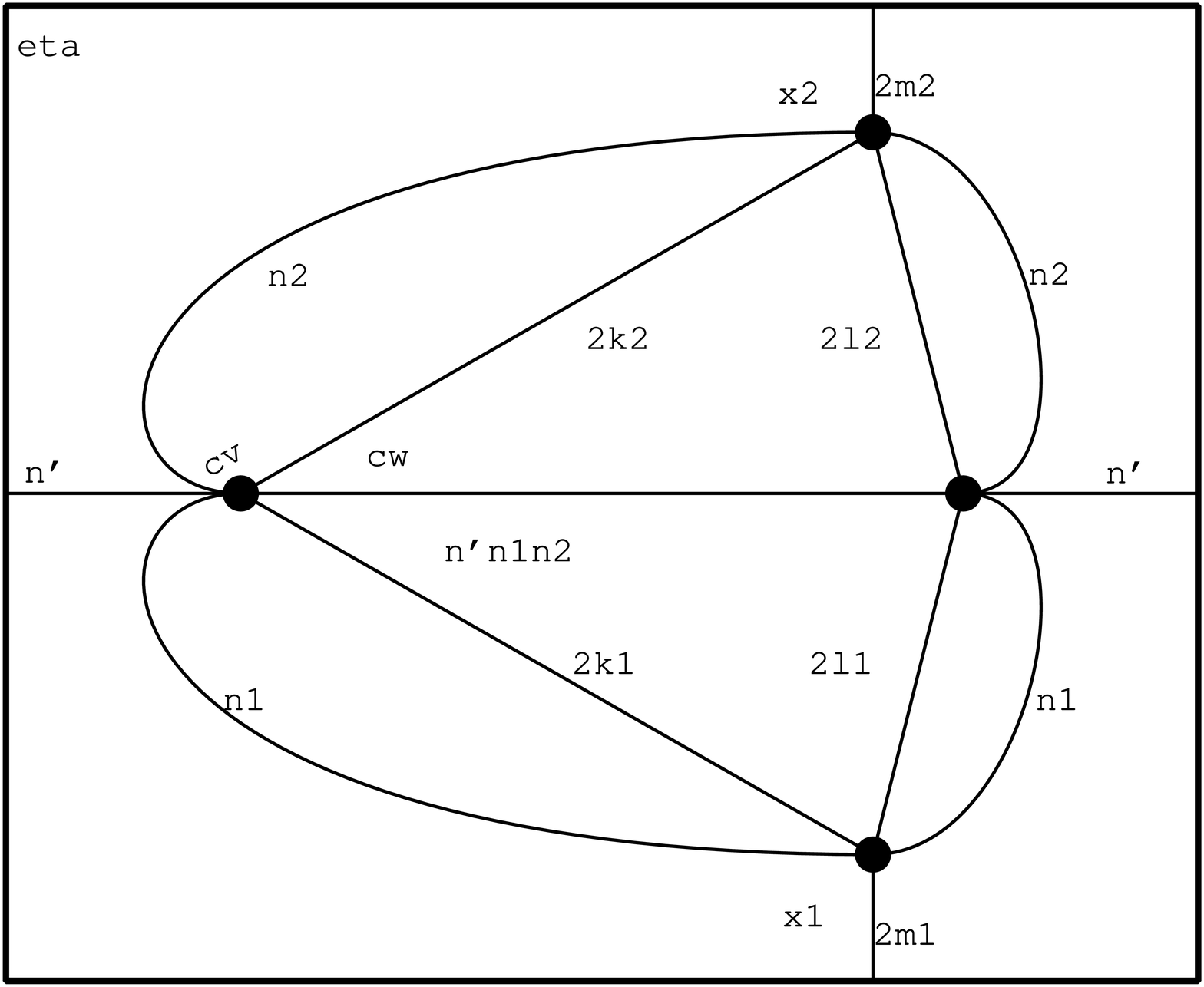}
}
\]
where $\xi_i \in V_{\vlon_i k_i}$, $\zeta_i \in W_{\vlon_i l_i}$, $a_i = \psi^{n_i}_{\vlon_i (k_i + l_i) , \eta_i m_i} (x_i) \in AP_{\vlon_i (k_i + l_i) , \eta_i m_i}$ for $i = 1, 2$, and $n' \in \N_{\eta_1, \eta_2}$.
\end{lem}
\noindent The identity object in $\mcal HAPM$ with respect to $\btimes$ is the planar algebra $P$ itself as a Hilbert affine $P$-module.
$P$ will also be referred as the {\em trivial Hilbert affine $P$-module} where the Hilbert space at level $\vlon k$ is the space $P_{\vlon k}$ with the inner product coming from the action of trace tangles (without any normalization),  and the action of an affine morphism is obtained from the action of a corresponding affine tangle treated as a regular semi-labelled tangle.
The left and the right unit constraints (which are also unitaries) are the following:
\begin{align*}
[P \btimes V]^0_{\eta m} \ni (\xi \os{a}{\btimes} \zeta)& \longmapsto \left( a \circ \psi^0_{\vlon l , \vlon (k+l)} \left( P_{\,
\psfrag{xi}{$\xi$}
\psfrag{y}{$y$}
\psfrag{2k}{$2k$}
\psfrag{2l}{$2l$}
\psfrag{e}{$\vlon$}
\includegraphics[scale=0.15]{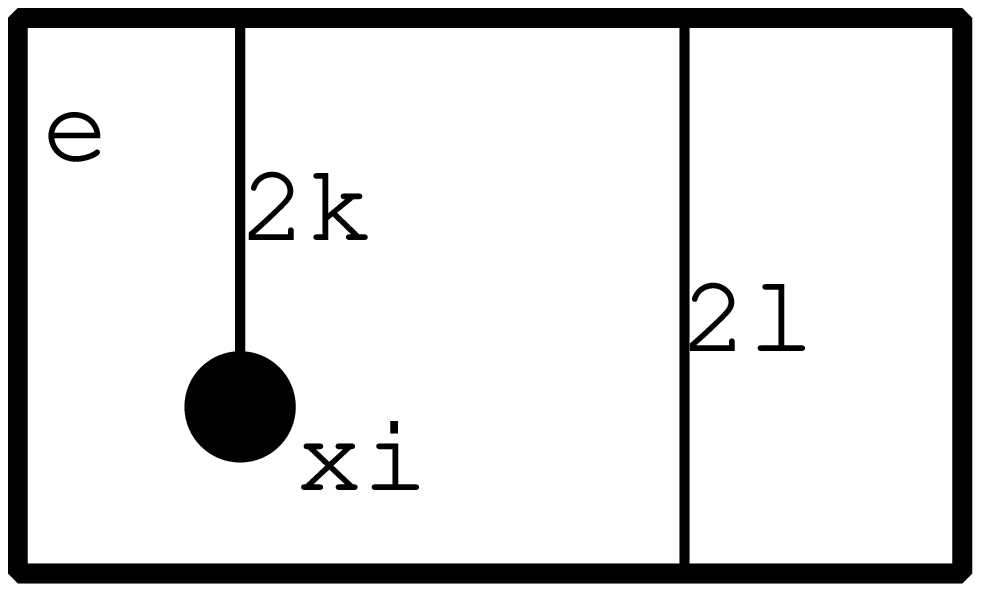}
} \right) \right) \cdot \zeta \in V_{\eta m} \t{, and}\\
[V \btimes P]^0_{\eta m} \ni (\zeta \os{a}{\btimes} \xi)& \longmapsto \left( a \circ \psi^0_{\vlon l , \vlon (l+k)} \left( P_{\,
\psfrag{xi}{$\xi$}
\psfrag{y}{$y$}
\psfrag{2k}{$2k$}
\psfrag{2l}{$2l$}
\psfrag{e}{$\vlon$}
\includegraphics[scale=0.15]{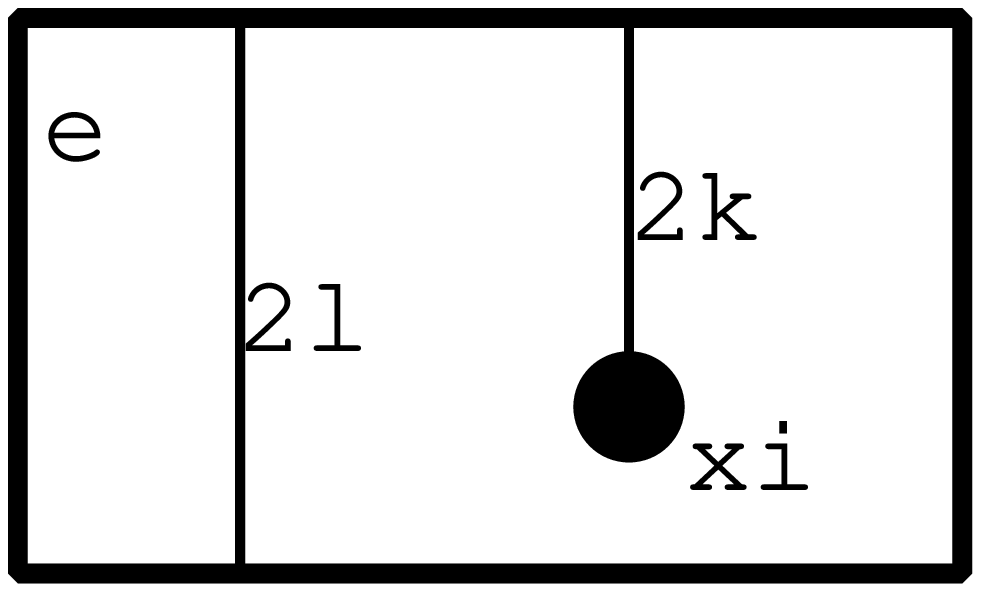}
} \right) \right) \cdot \zeta \in V_{\eta m}
\end{align*}
for $\xi \in P_{\vlon k}$, $\zeta \in V_{\vlon l}$, $a \in AP_{\vlon (k+l), \eta m}$.
\vskip 2mm
We will now define the {\em contragradient} of a Hilbert affine $P$-module.
For this, we introduce an {\em antipode functor} $S:AP \ra AP$ given by $\t{ob} (AP) \ni \vlon k \os{S}{\longmapsto} \vlon k \in \t{ob} (AP)$ and
\[
AP_{\vlon k, \eta l} \ni \psi^m_{\vlon k, \eta l} (x) \os{S}{\longmapsto} \psi^m_{\eta l , \vlon k} (P_{\,
\psfrag{x}{$x$}
\psfrag{2k}{$2k$}
\psfrag{2l}{$2l$}
\psfrag{m}{$m$}
\psfrag{e}{$\vlon$}
\includegraphics[scale=0.15]{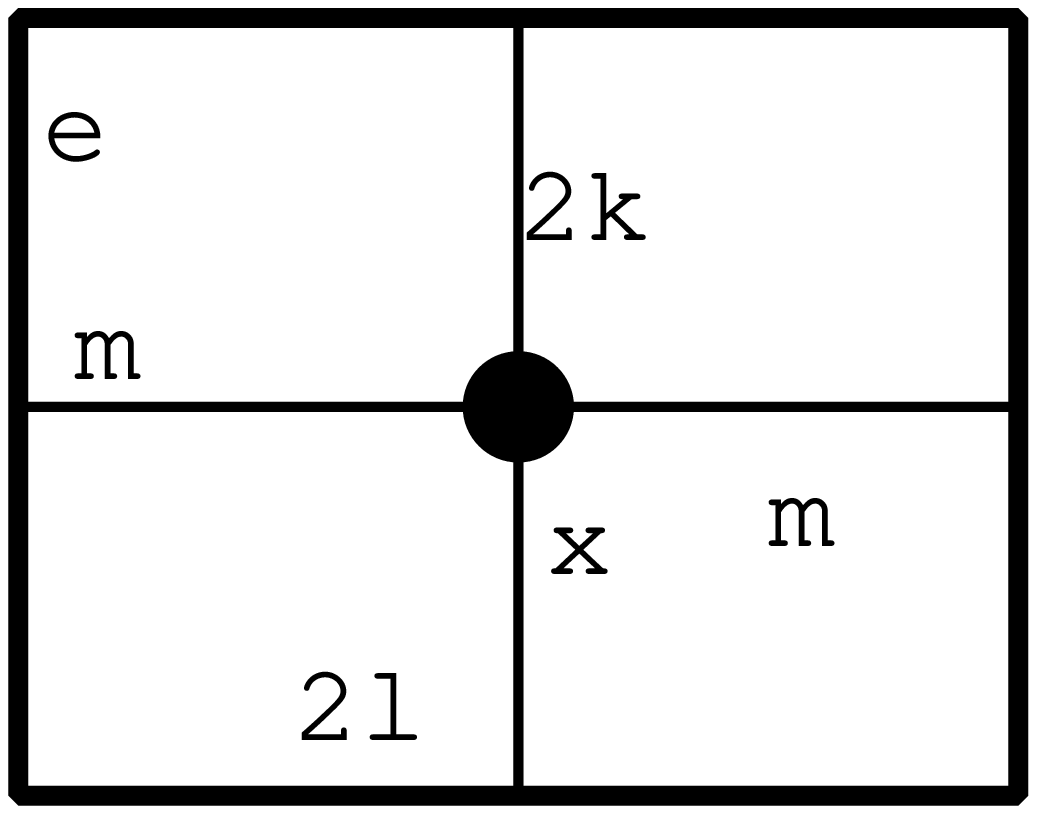}
}) \in AP_{\eta k , \vlon k} \t{ for all } m \in \N_{\vlon ,\eta} \t{ and } x \in P_{\eta (k+l+m)}.
\]
Well-definedness of $S$ at the level of morphisms, follows from Remark \ref{apmap}.
Note that $S$ is a contravariant, involutive, $\C$-linear functor; also, $S \circ * = * \circ S$.
\begin{defn}
Given a Hilbert affine $P$-module $V$, we define its {\em contragradient} (denoted by $\ol{V}$) by:\\
(i) $\ol{V}_{\vlon k} := \{\ol{\xi} : \xi \in V_{\vlon k} \}$ is the conjugate Hilbert space with $V_{\vlon k} \ni \xi \mapsto \ol{\xi} \in \ol{V}_{\vlon k}$ as the conjugate-linear unitary, and\\
(ii) $\ol{V}_a (\ol{\xi}) := \ol{V_{S(a^*)} (\xi )}$ for all $a\in AP_{\vlon k, \eta l}$, $\xi \in V_{\vlon k}$, $\vlon k , \eta l \in \t{Col}$.
\end{defn}
A natural question to ask right after the above definition, is whether the contragredient gives us a duality in $\mcal HAPM$.
Clearly, $\ol{\ol{V}}$ is isometrically isomorphic to $V$.
However, to have a full-fledged duality, one needs evaluation and coevaluation maps which might not exist (as norm bounded maps) in general.
In Section \ref{conj}, we will see that in a particular full subcategory of $\mcal HAPM$, the congradient indeed gives a structure of duality.
\vskip 2mm
We end this section by exhibiting a braiding on $\mcal HAPM$ which will be implemented by the {\em braiding affine morphism} defined in Figure \ref{affbraid}.
\begin{figure}[h]
\psfrag{e}{$\vlon$}
\psfrag{lhs}{$b_{\vlon,k,l} =$}
\psfrag{lhsinv}{$b^{-1}_{\vlon,k,l} =$}
\psfrag{rhs}{$\in AP_{\vlon (k+l) , \vlon (k+l)}$}
\psfrag{}{$$}
\psfrag{}{$$}
\psfrag{}{$$}
\psfrag{}{$$}
\psfrag{2k}{$2k$}
\psfrag{2l}{$2l$}
\psfrag{m}{$m$}
\psfrag{e}{$\vlon$}
\includegraphics[scale=0.15]{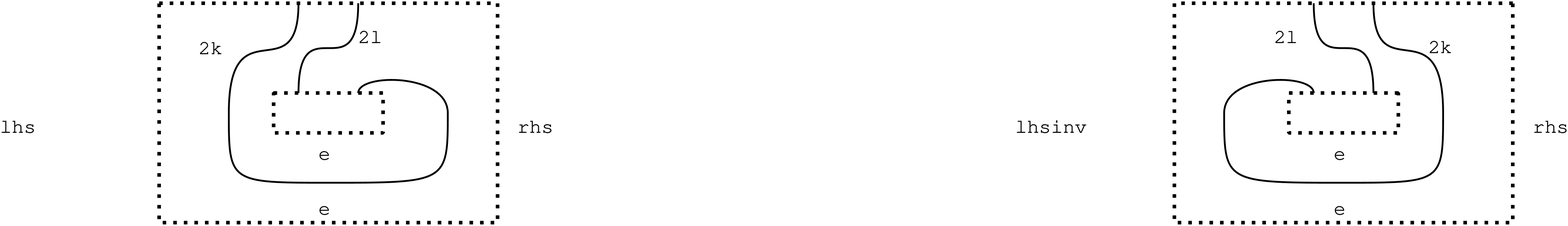}
\caption{Braiding affine morphism.}
\label{affbraid}
\end{figure}
\begin{prop}\label{braiding}
There exists a braiding $c$ on $\mcal H APM$ given by
\[
\left[ V \btimes W \right]^0_{\vlon m} \ni (\xi \os a \btimes \zeta ) \os{c_{V,W}} \longmapsto (\zeta \os {a \circ b_{\vlon,k,l}} \btimes V_{r_{\vlon k}} \, \xi ) \in \left[ W \btimes V \right]^0_{\vlon m}
\]
where $V$ and $W$ are two Hilbert affine $P$-modules, $a \in AP_{\vlon (k+l) , \vlon m}$, $\xi \in V_{\vlon k}$, $\zeta \in W_{\vlon l}$ and $r_{\vlon k}$ is the affine morphism
\psfrag{e}{$\vlon$}
\psfrag{lhs}{$b_{\vlon,k,l} =$}
\psfrag{lhsinv}{$b^{-1}_{\vlon,k,l} =$}
\psfrag{rhs}{$\in AP_{\vlon (k+l) , \vlon (k+l)}$}
\psfrag{}{$$}
\psfrag{}{$$}
\psfrag{}{$$}
\psfrag{}{$$}
\psfrag{2k}{$2k$}
\psfrag{2l}{$2l$}
\psfrag{m}{$m$}
\psfrag{e}{$\vlon$}
\includegraphics[scale=0.15]{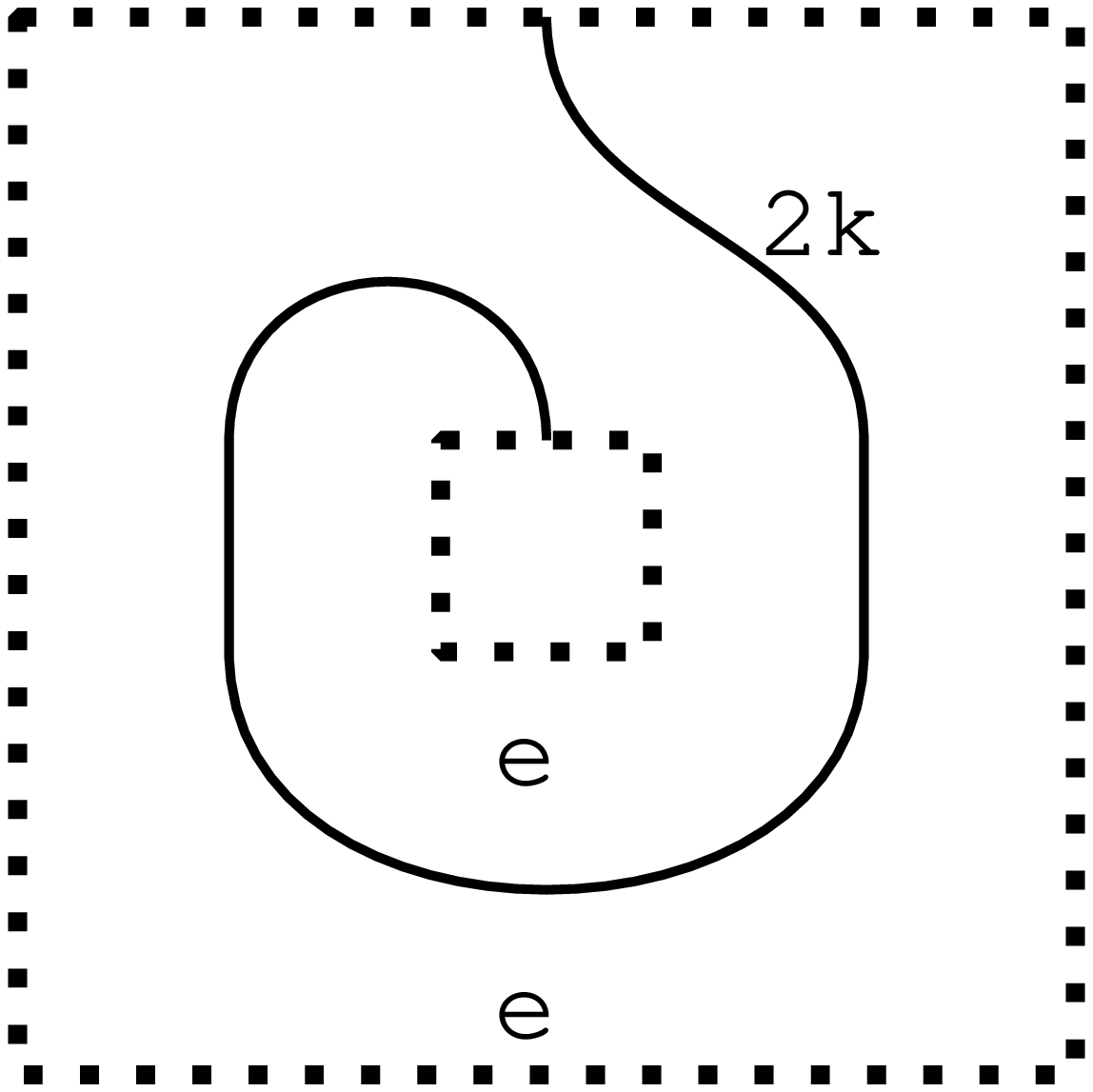}
$\in AP_{\vlon k, \vlon k}$.
\end{prop}
\begin{proof}
We begin with showing that the above is unitary, that is, (a) inner product preserving and (b) surjective.
Once (a) is proved, (b) easily follows from the fact that  $[V \btimes W]^0$ is dense in $V \btimes W$.

Let $\xi_i \in V_{\vlon k_i}$, $\zeta_i \in W_{\vlon l_i}$, $x_i \in P_{\vlon(k_i+l_i+m+n_i)}$, $a_i = \psi^{n_i}_{\vlon (k_i+l_i), \vlon m} (x_i)$ for $i= 1 , 2$, and $k=k_1+k_2$, $n=n_1+n_2$. Using (a) the formula of $CC$-valued inner product of fusion of affine modules obtained in Lemma \ref{fusCCip}, and (b) Lemma \ref{ccprop} (iii), (iv),  we can express the scalar $\left\lab (\zeta_1 \os {a_1 \circ b_{\vlon,k_1,l_1}} \btimes V_{r_{\vlon k_1}} \, \xi_1 ) , (\zeta_2 \os {a_2 \circ b_{\vlon,k_2,l_2}} \btimes V_{r_{\vlon k_2}} \, \xi_2 ) \right\rab $ as:
\[
P_{\!\!\!\!\!\!\!\!\!\!\!\!\!\!\!\!\!\!\!\!\!\!\!\!\!
\psfrag{e}{$\vlon$}
\psfrag{cv}{$c_{4k+n}(\xi_1,\xi_2\!)$}
\psfrag{cw}{$c_{2k + n}(\zeta_1,\zeta_2)$}
\psfrag{4kn1n2}{$2k+n$}
\psfrag{}{$$}
\psfrag{}{$$}
\psfrag{x1}{$x^*_1$}
\psfrag{x2}{$x_2$}
\psfrag{2k1}{$2k_1$}
\psfrag{2k2}{$2k_2$}
\psfrag{2l1}{$2l_1$}
\psfrag{2l2}{$2l_2$}
\psfrag{2m}{$2m$}
\psfrag{n1}{$n_1$}
\psfrag{n2}{$n_2$}
\includegraphics[scale=0.15]{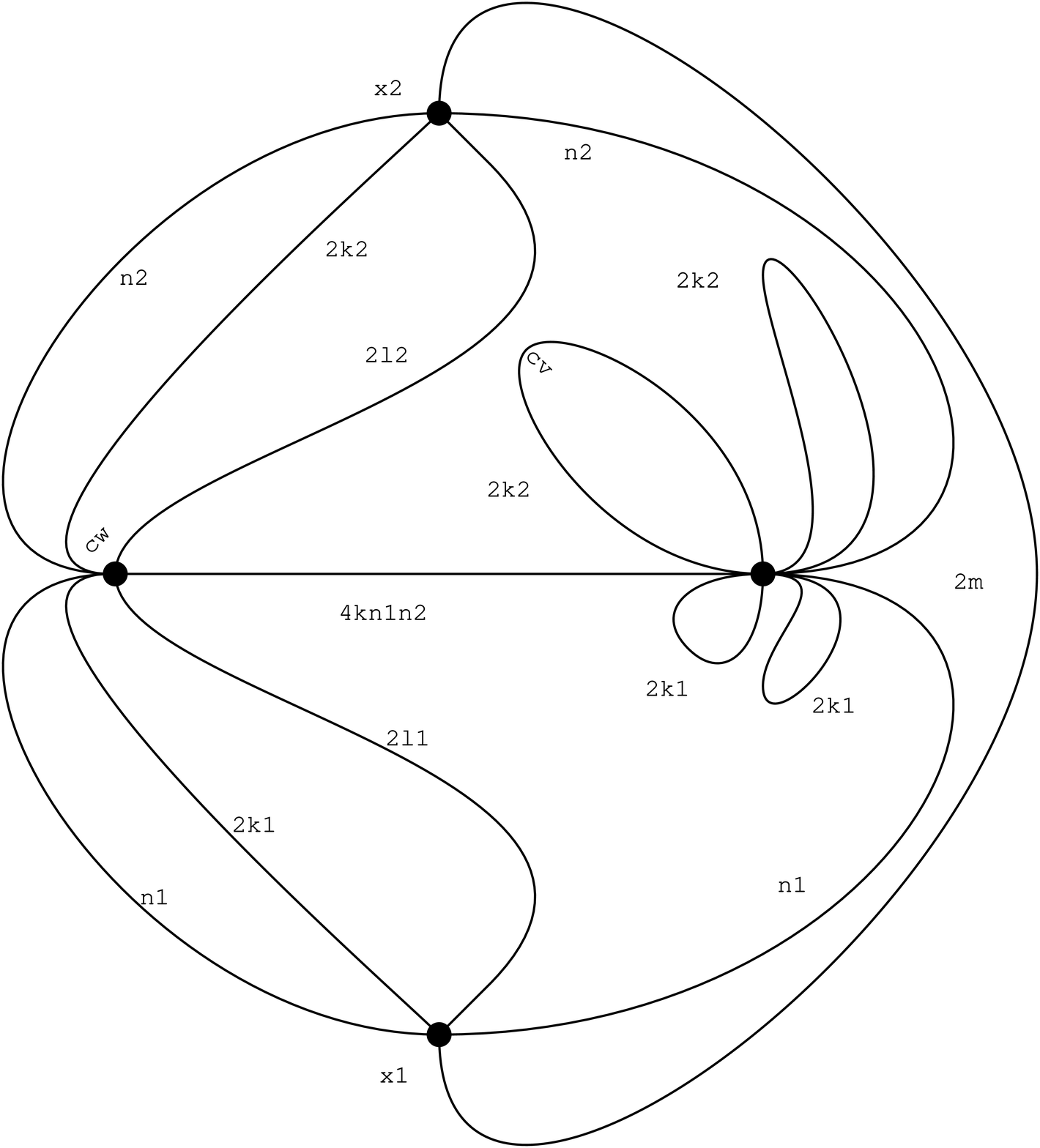}
}
=\;\;\;\;\;\;\;\;\;\;\;\;\;\;\;\;\;\;P_{\!\!\!\!\!\!\!\!\!\!\!\!\!\!\!\!\!\!\!\!\!\!\!\!\!
\psfrag{e}{$\vlon$}
\psfrag{cv}{$c_n(\xi_1,\xi_2)$}
\psfrag{cw}{$c_{2k + n}(\zeta_1,\zeta_2)$}
\psfrag{4kn1n2}{$2k_1+2k_2+n_1+n_2$}
\psfrag{n1n2}{$n$}
\psfrag{}{$$}
\psfrag{}{$$}
\psfrag{x1}{$x^*_1$}
\psfrag{x2}{$x_2$}
\psfrag{2k1}{$2k_1$}
\psfrag{2k2}{$2k_2$}
\psfrag{2l1}{$2l_1$}
\psfrag{2l2}{$2l_2$}
\psfrag{2m}{$2m$}
\psfrag{n1}{$n_1$}
\psfrag{n2}{$n_2$}
\includegraphics[scale=0.15]{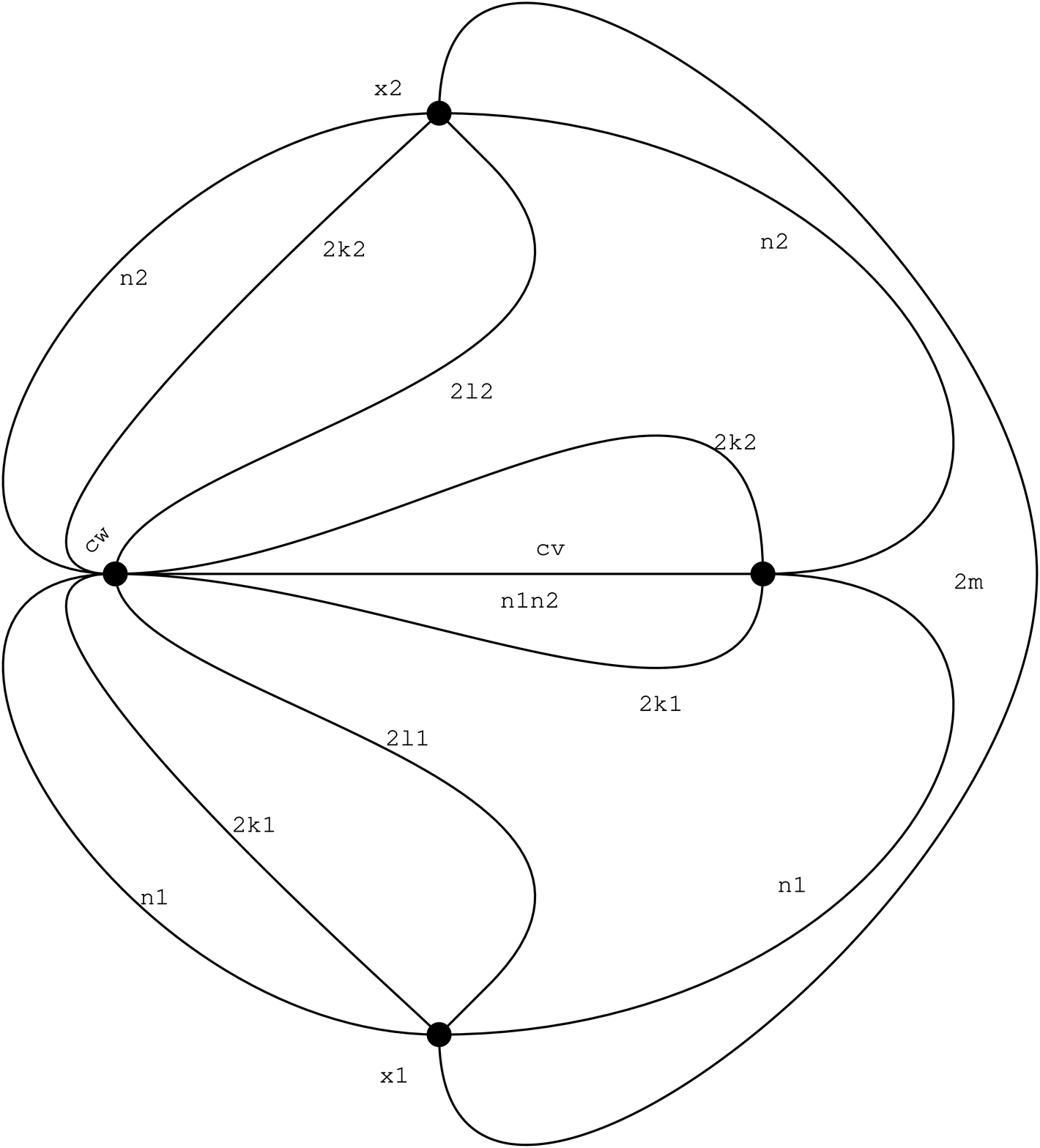}
}
\]
$= \;\;\;\;\;\;\;\;\;\;\;\;\;\;\;\;\;\;P_{\!\!\!\!\!\!\!\!\!\!\!\!\!\!\!\!\!\!\!\!\!\!\!\!\!
\psfrag{e}{$\vlon$}
\psfrag{cv}{$c_n(\xi_1,\xi_2)$}
\psfrag{cw}{$c_n(\zeta_1,\zeta_2)$}
\psfrag{4kn1n2}{$2k_1+2k_2+n_1+n_2$}
\psfrag{n}{$n$}
\psfrag{}{$$}
\psfrag{}{$$}
\psfrag{x1}{$x^*_1$}
\psfrag{x2}{$x_2$}
\psfrag{2k1}{$2k_1$}
\psfrag{2k2}{$2k_2$}
\psfrag{2l1}{$2l_1$}
\psfrag{2l2}{$2l_2$}
\psfrag{2m}{$2m$}
\psfrag{n1}{$n_1$}
\psfrag{n2}{$n_2$}
\includegraphics[scale=0.15]{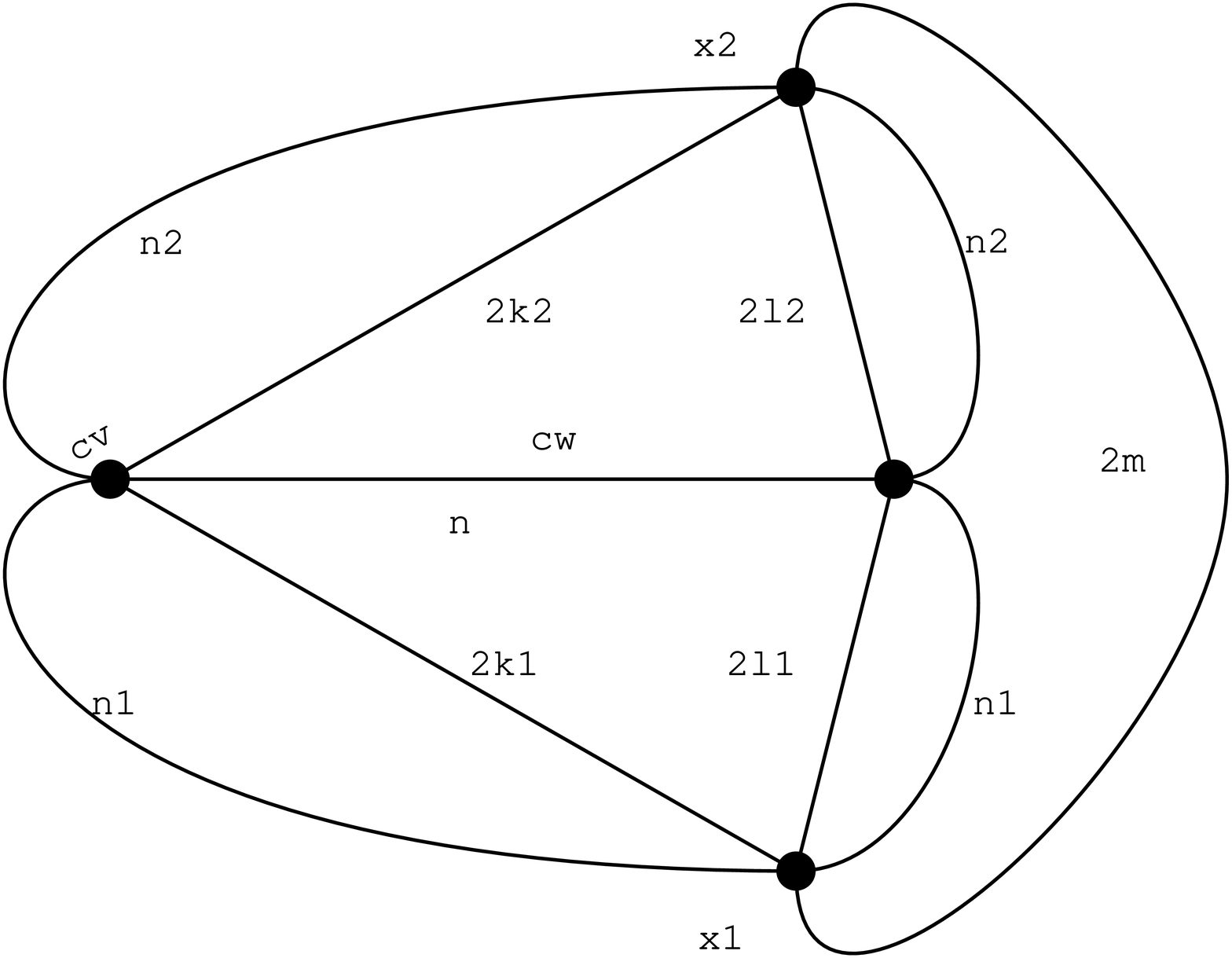}
}
= \left\lab (\xi_1 \os {a_1} \btimes \zeta_1 ) , (\xi_2 \os {a_2} \btimes \zeta_2 ) \right\rab$.

\noindent Thus, $c_{V,W}$ is a unitary.
Naturality of $c_{V,W}$ in $V$ and $W$, is straight-forward.

It remains to show (a)  $\left( \t{id}_V \btimes c_{U,W} \right) \circ \left( c_{U,V} \btimes \t{id}_W \right) = c_{U,V\btimes W}$ and (b) $\left( c_{U,W} \btimes \t{id}_W \right) \circ \left( \t{id}_U \btimes c_{V,W} \right) = c_{U \btimes V , W}$ where we hide the associativity constraints.
In order to show this, it will be useful to view $c_{U,V}$ in a slightly different way.
Note that $c_{V,W} \left(\xi \os a \btimes \zeta \right) = (W \btimes V)_a \left( \zeta \os {b_{\vlon ,k,l}} \btimes V_{r_{\vlon k}} \xi \right) $ for $a \in AP_{\vlon (k+l) , \vlon m}$, $\xi \in V_{\vlon k}$, $\zeta \in W_{\vlon l}$.
Recall the `inclusion' affine morphism $u_{\vlon k , n} = 
\psfrag{2k}{$2k$}
\psfrag{n}{$n$}
\psfrag{e}{$\vlon$}
\psfrag{eta}{$\eta$}
\includegraphics[scale=0.15]{figures/fusion/uincl.eps}
\in AP_{\vlon k,\eta (k+n)}$ appearing in the proof of Lemma \ref{fuswloglem}.
Following the steps in the proof of Lemma \ref{fuswloglem}, we get
\begin{align*}
\zeta \os {b_{\vlon ,k,l}} \btimes V_{r_{\vlon k}} \xi
= W_{u_{\vlon l , 2k}} \zeta \os {\!\!\!\!P_{
\psfrag{2k}{$2k$}
\psfrag{2l}{$2l$}
\psfrag{e}{$\vlon$}
\psfrag{eta}{$\eta$}
\includegraphics[scale=0.15]{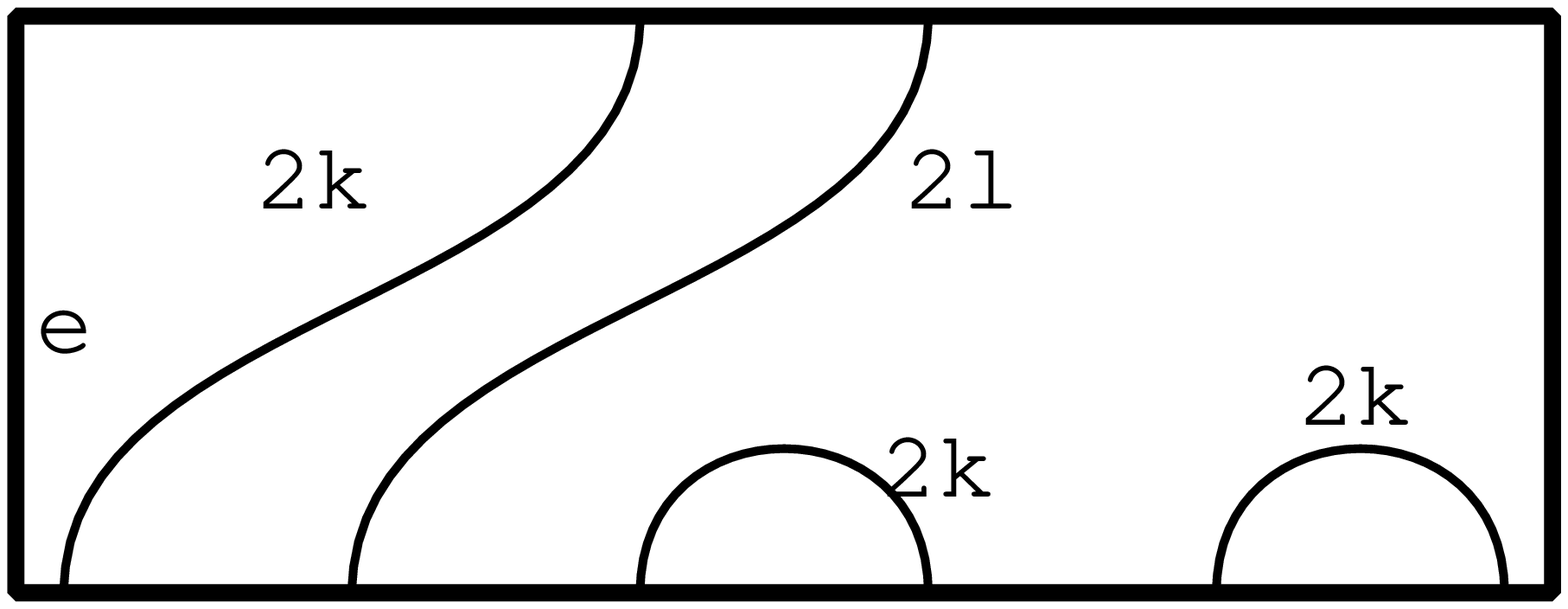}
}} \btimes V_{u_{\vlon k , 2k} \circ r_{\vlon k}} \xi
= & \; W_{u_{\vlon l , 2k}} \zeta \os {\!\!\!\!\!\!\!\!\!\!\!\!\!\!\!\!\!\!\!\!\!\!\!\!\!P_{
\psfrag{2k}{$2k$}
\psfrag{2l}{$2l$}
\psfrag{e}{$\vlon$}
\psfrag{eta}{$\eta$}
\includegraphics[scale=0.15]{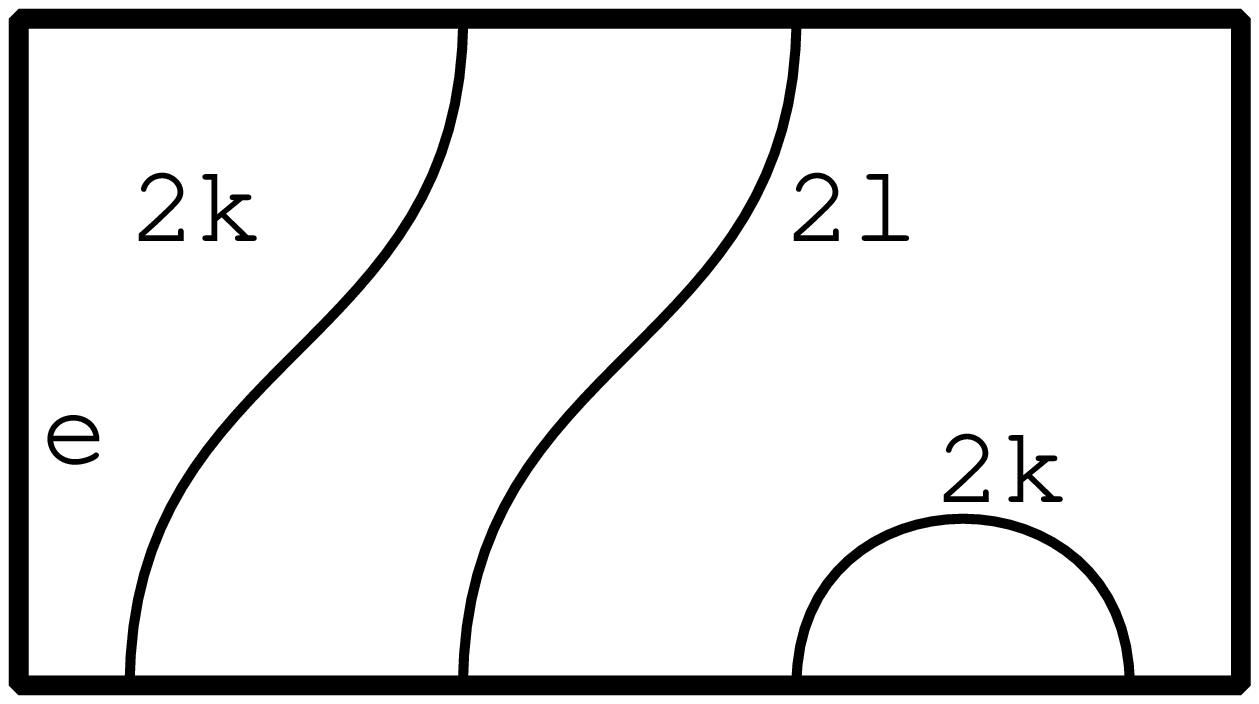}
}} \btimes V_{\psi^0_{\vlon 3k , \vlon k} (1_{P_{\vlon 4k}}) \circ u_{\vlon k , 2k} \circ r_{\vlon k}} \xi\\
= & \; W_{u_{\vlon l , 2k}} \zeta \os {\!\!\!\!\!\!\!\!\!\!\!\!\!\!\!\!\!\!\!\!\!\!\!\!\!P_{
\psfrag{2k}{$2k$}
\psfrag{2l}{$2l$}
\psfrag{e}{$\vlon$}
\psfrag{eta}{$\eta$}
\includegraphics[scale=0.15]{figures/fusion/b2.eps}
}}  \btimes \xi.
\end{align*}

Thus,
\[
c_{V,W} \left(\xi \os a \btimes \zeta \right)
= W_{u_{\vlon l , 2k}} \zeta \os {\tilde a } \btimes \xi \t{ where } \tilde a = a \circ \psi^0_{\vlon (3k+l) , \vlon (k+l)} (1_{P_{\vlon (4k+2l)}}).
\]

In the rest of the proof, for simplicity of calculations, we will hide the associativity constraints and replace them by identities.
Using this reformulation of $c_{V,W}$, equation (b) becomes almost immediate when we evaluate both sides on elements (in a dense subset of $(U\btimes (V \btimes W))_{\vlon n}$) of the form  $\omega \os x  \btimes \left( \xi \os{1_{P_{\vlon 2(l+m)}}} \btimes \zeta \right) = \left( \omega  \os{1_{P_{\vlon 2(k+l)}}} \btimes \xi  \right)\os x \btimes  \zeta $ for $\omega \in U_{\vlon k}$, $ \xi \in V_{\vlon l}$, $\zeta \in W_{\vlon m}$, $x \in P_{\vlon (k+l+m+n)}$ (as could be seen from the discussion of associativity constraint preceeding Lemma \ref{fusCCip}). 
For equation (a), observe that
\begin{align*}
\left( \omega  \os{1_{P_{\vlon 2(k+l)}}} \btimes \xi  \right) \os{x}{\scalebox{1}[1]{$\btimes$}}  \zeta
\os {c_{U,V} \btimes \t{id}_W} \longmapsto
& \; \scalebox{2}[4]{(}V_{u_{\vlon l , 2k}} \xi \os {\!\!\!\!\!\!\!\!\!\!\!\!\!\!\!\!\!\!\!\!\!\!\!\!\!P_{
\psfrag{2k}{$2k$}
\psfrag{2l}{$2l$}
\psfrag{e}{$\vlon$}
\psfrag{eta}{$\eta$}
\includegraphics[scale=0.15]{figures/fusion/b2.eps}
}}  \btimes \omega \scalebox{2}[4]{)} \os{x}{\scalebox{1}[1]{$\btimes$}} \zeta\\
= & \;\;\;\; \os {\!\!\!\!\!\!\!\!\!\!\!\!\!\!P_{
\psfrag{x}{$x$}
\psfrag{2k}{$2k$}
\psfrag{2l}{$2l$}
\psfrag{2m}{$2m$}
\psfrag{2n}{$2n$}
\psfrag{e}{$\vlon$}
\includegraphics[scale=0.15]{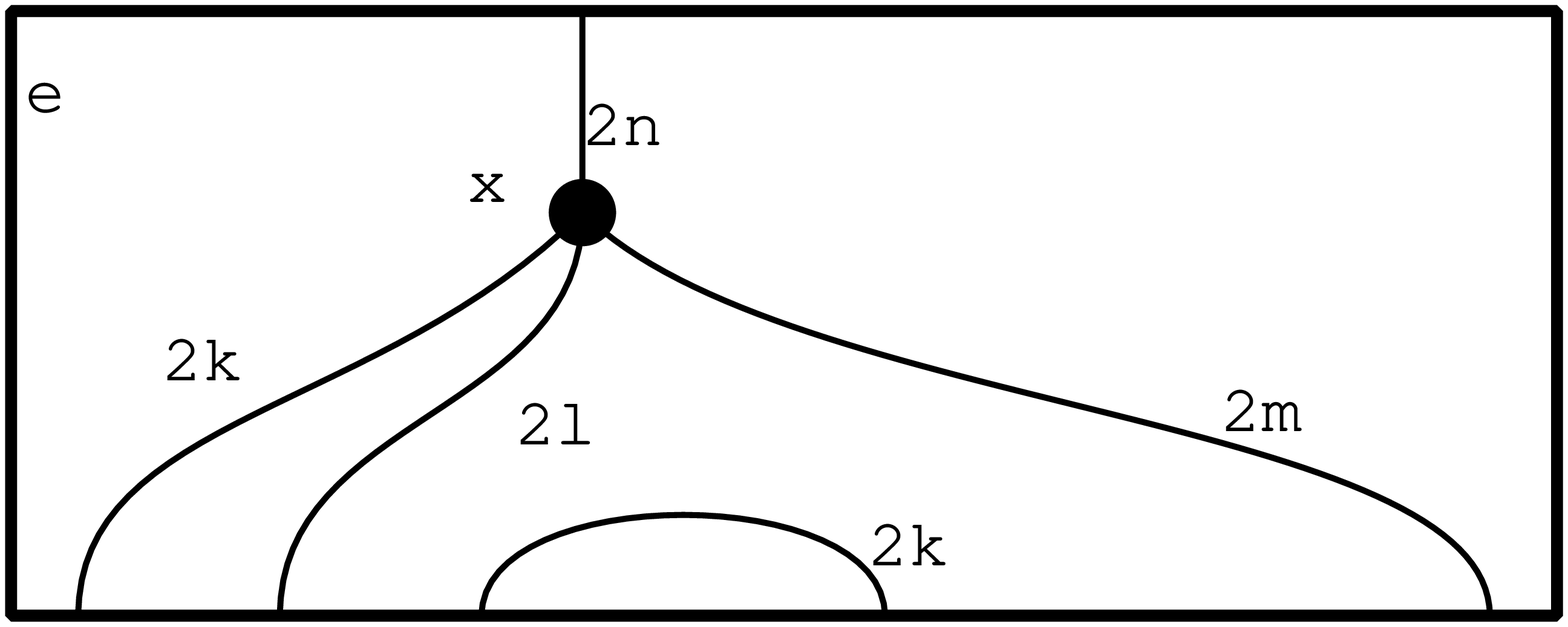}
}}
{V_{u_{\vlon l , 2k}} \xi\,  \scalebox{1.5}[3]{$\btimes$}\!\! \left(\omega \!\! \os {1_{P_{\vlon 2(k+m)}}} \btimes\!\! \zeta \right)}
\os {\t{id}_V \btimes c_{U,W}} \longmapsto
\os{\!\!\!\!\!\!\!\!\!\!\!\!\!\!\!P_{
\psfrag{x}{$x$}
\psfrag{2k}{$2k$}
\psfrag{2l}{$2l$}
\psfrag{2m}{$2m$}
\psfrag{2n}{$2n$}
\psfrag{e}{$\vlon$}
\includegraphics[scale=0.15]{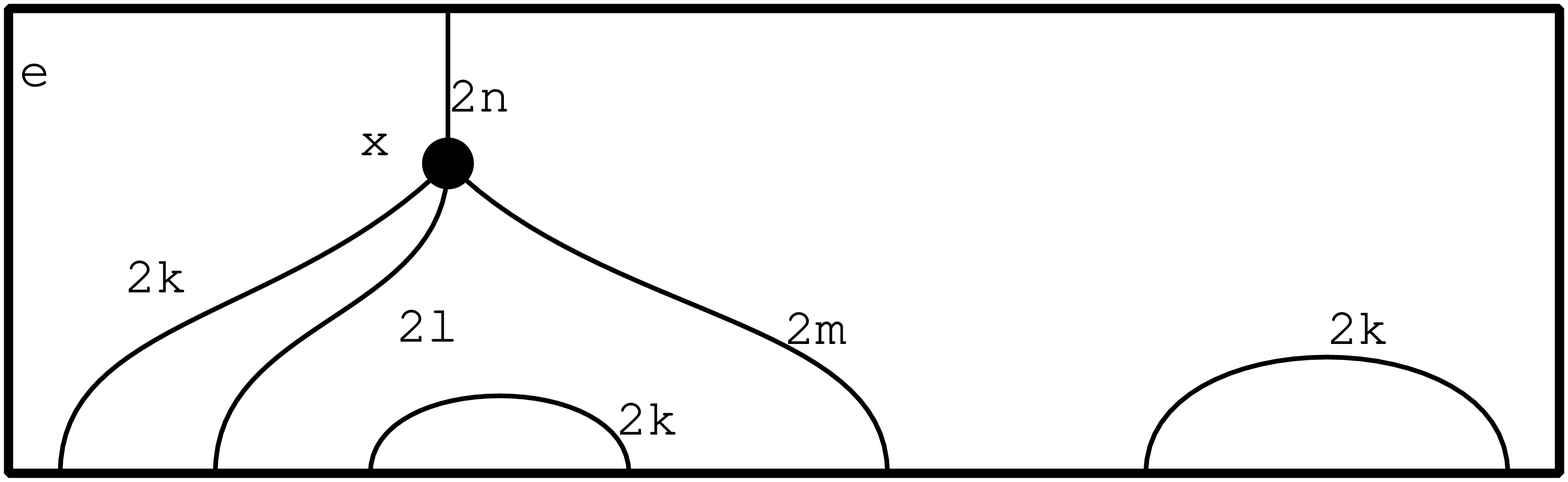}
}}
{V_{u_{\vlon l , 2k}} \xi\;\;  \scalebox{1.5}[3]{$\btimes$}\!\! \left(W_{u_{\vlon m , 2k}} \zeta \os {1_{P_{\vlon 2(3k+m)}}} \btimes \omega \right)}.
\end{align*}
After applying appropriate associativity constraint, the resultant can be written as:
\[
\os {\;\;\;\;\;\;\;P_{
\psfrag{x}{$x$}
\psfrag{2k}{$2k$}
\psfrag{2l}{$2l$}
\psfrag{2m}{$2m$}
\psfrag{2n}{$2n$}
\psfrag{e}{$\vlon$}
\includegraphics[scale=0.15]{figures/fusion/b4.eps}
}}
{\left( V_{u_{\vlon l , 2k}} \xi \os {1_{P_{\vlon 2(4k+l+m)}}} \btimes  W_{u_{\vlon m , 2k}} \zeta \right)\!\! \scalebox{1.5}[3]{$\btimes$} \omega }
\;\; = \os {\;\;\;\;P_{
\psfrag{P}{$^P$}
\psfrag{x}{$x$}
\psfrag{2k}{$2k$}
\psfrag{2l}{$2l$}
\psfrag{2m}{$2m$}
\psfrag{2n}{$2n$}
\psfrag{e}{$\vlon$}
\includegraphics[scale=0.15]{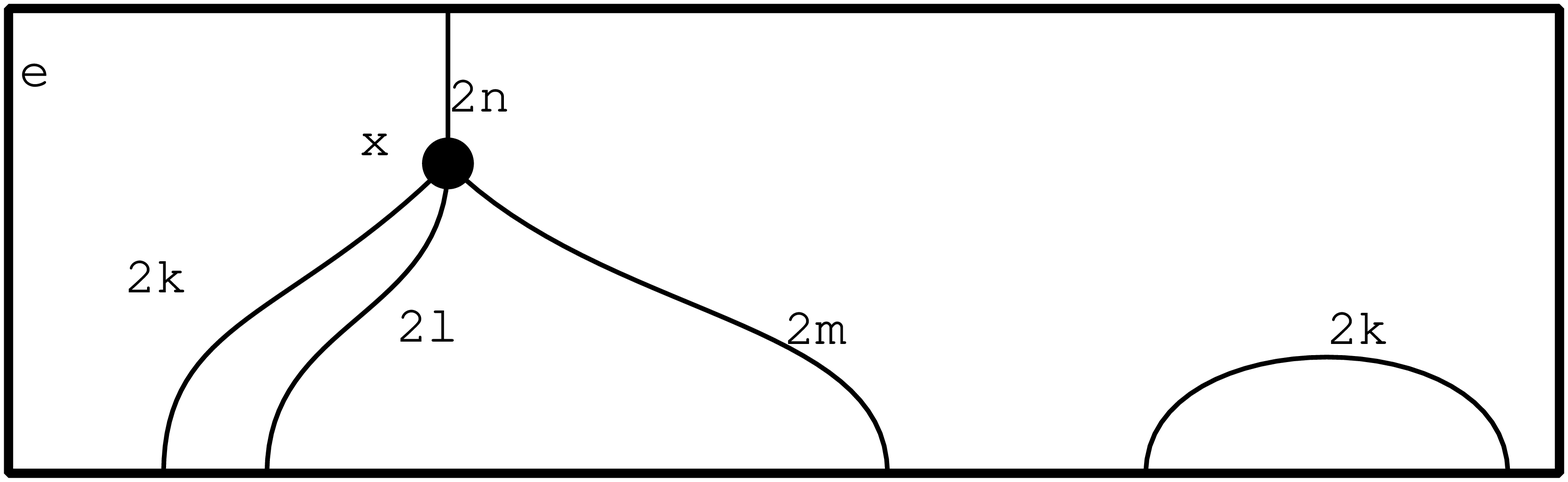}
}}
{\scalebox{2}[4]{(} \os {P_{
\psfrag{P}{$^P$}
\psfrag{x}{$x$}
\psfrag{2k}{$2k$}
\psfrag{2l}{$2l$}
\psfrag{2m}{$2m$}
\psfrag{2n}{$2n$}
\psfrag{e}{$\vlon$}
\includegraphics[scale=0.15]{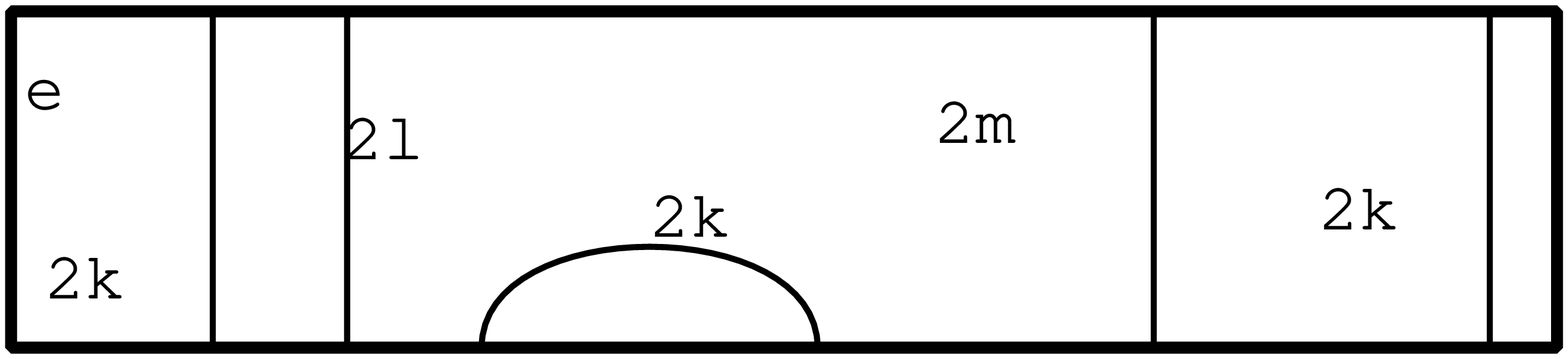}
}}
{V_{u_{\vlon l , 2k}} \xi \; \btimes \;\;\; W_{u_{\vlon m , 2k}} \zeta} \scalebox{2}[4]{)} \! \scalebox{2}[4]{$ \btimes$} \omega
}.
\]
Again, appealing to the proof of Lemma \ref{fuswloglem}, it is easy to verify that
\[
\os {\!\!\!\!\!\!\!\!\!\!P_{
\psfrag{2k}{$2k$}
\psfrag{2l}{$2l$}
\psfrag{2m}{$2m$}
\psfrag{2n}{$2n$}
\psfrag{e}{$\vlon$}
\includegraphics[scale=0.15]{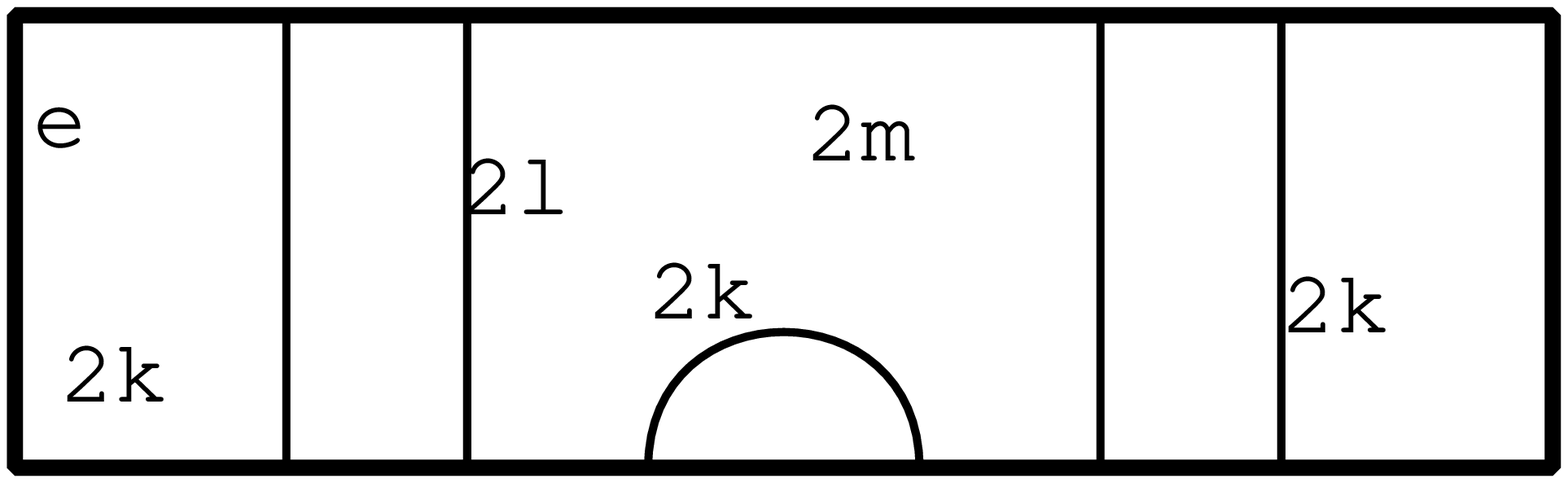}
}}
{V_{u_{\vlon l , 2k}} \xi  \btimes W_{u_{\vlon m , 2k}} \zeta} = (V\btimes W)_{u_{\vlon (l+m), 2k}} \left( \xi \os {1_{P_{\vlon 2(l+m)}}}\btimes \zeta \right).
\]
So, we finally get
\begin{align*}
\left( \omega  \os{1_{P_{\vlon 2(k+l)}}} \btimes \xi  \right)\os x \btimes  \zeta
\os {\left( \t{id}_V \btimes c_{U,W} \right) \circ \left( c_{U,V} \btimes \t{id}_W \right)} \longmapsto & \os{\;\;\;\;\;\;\;\;\;\;\;\;\;\;\;\;\;\;\;\;\;\;\;\;\;\;\;\;\;\;\;\;\;\;\;\;\;P_{
\psfrag{x}{$x$}
\psfrag{2}{$2(k+l+m)$}
\psfrag{2k}{$2k$}
\psfrag{2m}{$2m$}
\psfrag{2n}{$2n$}
\psfrag{e}{$\vlon$}
\includegraphics[scale=0.15]{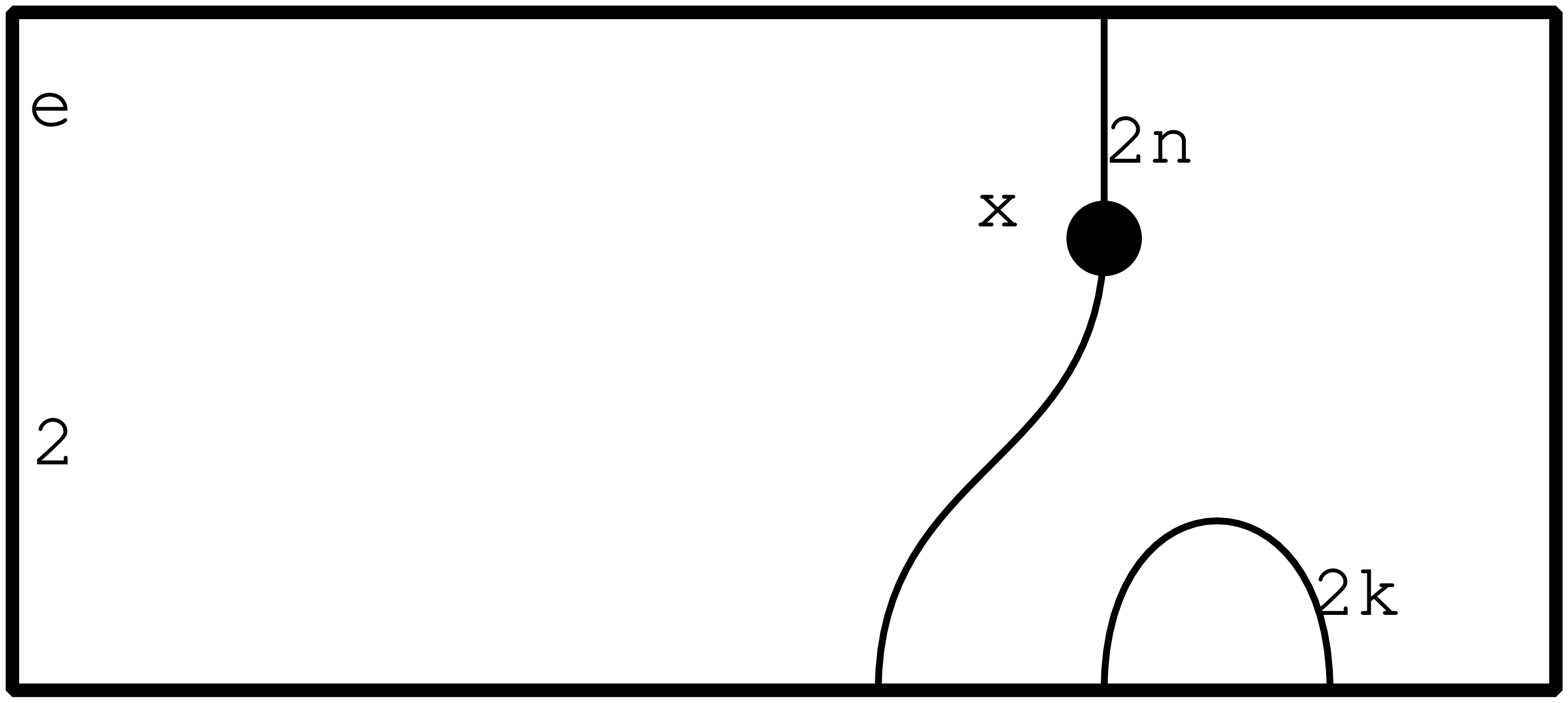}
}} {(V\btimes W)_{u_{\vlon (l+m), 2k}} \left( \xi \os {1_{P_{\vlon 2(l+m)}}}\btimes \zeta \right) \scalebox{1.5}[3]{$\btimes$} \omega}\\
= & \; c_{U,V \btimes W} \left( \omega \os x  \btimes \left( \xi \os{1_{P_{\vlon 2(l+m)}}} \btimes \zeta \right) \right). 
\end{align*}
\end{proof}

\section{Walker's conjecture}\label{conj}
We will first describe the Drinfeld center of the $N$-$N$-bimodule category appearing in the standard invariant of an extremal finite index subfactor $N \subset M$.
Thereafter, we will build a functor from this center to the locally finite Hilbert affine $P$-modules where $P$ is the subfactor planar algebra associated to $N \subset M$.
After restriction of the co-domain with some finiteness criterion, this functor becomes an equivalence.

We begin with setting up some notations.
Let $f \us{m}{\odot} g := P_{
\psfrag{f}{$f$}
\psfrag{g}{$g$}
\psfrag{2k-m}{$2k-m$}
\psfrag{2l-m}{$2l-m$}
\psfrag{m}{$m$}
\psfrag{e}{$\vlon$}
\includegraphics[scale=0.15]{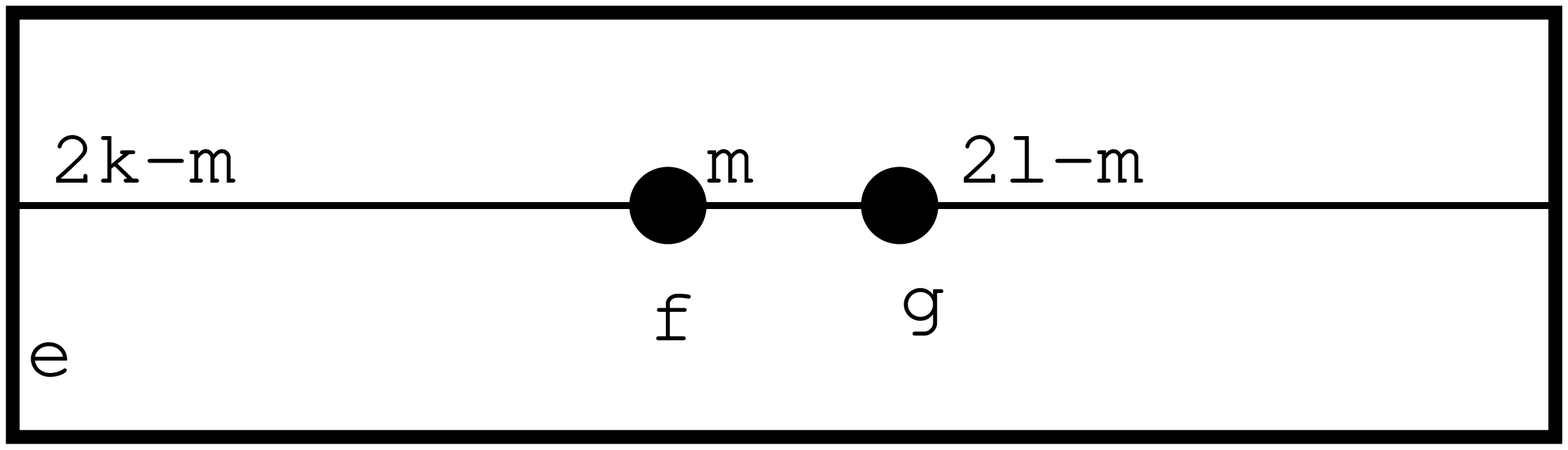}
} \in P_{\vlon (k+l-m)}$ for $f \in P_{\vlon k}$, $g\in P_{\vlon l}$, and  $m \leq \t{min}\{k,l\}$.
Instead of the $N$-$N$-bimodule category, we will work with an equivalent category $\mcal C$ defined by:

\noindent {\em Object:} $\t{ob} (\mcal C) := \{ p \in \mscr P (P_{+2k} ) : k \in \N_0\}$,

\noindent {\em Morphisms:} $\mcal C (p,q) := \left\{ f \in P_{+(k+l)} : f \us{2k}{\odot} p = f = q \us{2l}{\odot} f \right\}$ where $p \in \mscr P (P_{+2k})$, $q \in \mscr P (P_{+2l})$,

\noindent {\em Composition of morphisms:} $f \circ g := f \us{2l}{\odot} g \in \mcal C (p,r)$ where $f \in \mcal C(q,r)$, $g \in \mcal C(p,q)$, $p \in \mscr P (P_{+2k})$, $q \in \mscr P (P_{+2l})$, $r \in \mscr P (P_{+2m})$,

\noindent {\em Identity morphism:} $p$ is the identity of $\mcal C (p,p)$ where $p \in \mscr P (P_{+2k})$.

Note that $*$-structure on $P$ induces a $*$ structure on $\mcal C$.
We will always assume $\delta := \sqrt{ [M:N]} >1$ so that direct sums exists in $\mcal C$ and thereby avoiding any extra linearization of $\mcal C$. Thus, $\mcal C$ becomes a semisimple $\C$-linear $*$-category.

We consider the monoidal structure on $\mcal C$ given by the obvious functor $\otimes : \mcal C \times \mcal C \ra \mcal C$ ($p \otimes q := P_{
\psfrag{p}{$p$}
\psfrag{q}{$q$}
\psfrag{2k}{$2k$}
\psfrag{2l}{$2k$}
\psfrag{2m}{$2l$}
\psfrag{2n}{$2l$}
\psfrag{+}{$+$}
\includegraphics[scale=0.15]{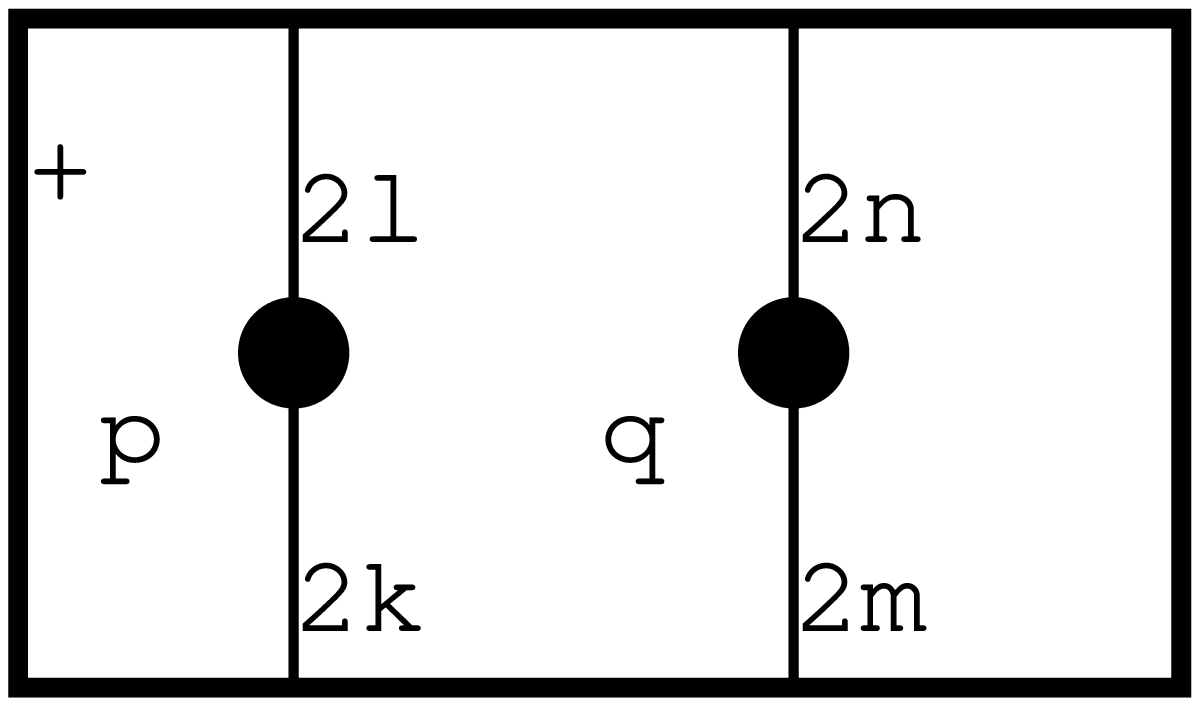}
} \in \mscr P (P_{+2(k+l)})$ and $f \otimes g := P_{
\psfrag{p}{$f$}
\psfrag{q}{$g$}
\psfrag{2k}{$2k$}
\psfrag{2l}{$2l$}
\psfrag{2m}{$2m$}
\psfrag{2n}{$2n$}
\psfrag{+}{$+$}
\includegraphics[scale=0.15]{figures/z2ap/poq.eps}
} \in \mcal C (p \otimes q, r\otimes s)$ for $p \in \mscr P (P_{+2k}),q \in \mscr P (P_{+2l}),r \in \mscr P (P_{+2m}), s \in \mscr P (P_{+2n}), f \in \mcal C(p,q), g \in \mcal C(r,s)$) where the unit object is $1_{\mcal C} := 1_{P_{+0}}$, and the associative and unit constraints are trivial (and thereby, $\mcal C$ becomes strict).

We now proceed towards the center of $\mcal C$, which we denote by $\mcal {ZC}$.
(See any standard textbook (for instance, \cite{Kas}) for the general definition.)
Objects of $\mcal {ZC}$ are pairs $(p,c)$'s for $p \in \t{ob} (\mcal C)$ and natural isomorphism $c : \cdot \otimes p \ra p \otimes \cdot$ satisfying $c_{q \otimes r} = (c_q \otimes 1_r) \circ (1_q \otimes c_r)$ for all $q,r \in \t{ob} (\mcal C)$ and $c_{1_{P_{+0}}} = p$.
Since we are working with a $*$-category, it will be relevant to consider unitary commutativity constraints  (see \cite{Mug}); henceforth, all our commutativity constraints will be natural unitaries.
{(In fact, M\"{u}ger had considered the unitary Drinfeld center in \cite[$\S 6$]{Mug} and had proved that for a fusion category, the unitary Drinfeld center is  equivalent to the usual Drinfeld center.)
}Such a $c$ will be referred as {\em unitary commutativity constraint}. Morphism $f \in \mcal {ZC} ((p,c),(q,d))$ is an element  $f \in \mcal C (p,q)$ which is compatible with $c$ and $d$, that is,
\begin{equation}\label{fcd}
(f \otimes 1_r) \circ c_r = d_r \circ (1_r \otimes f) \t{ for all } r \in \t{ob}(\mcal C).
\end{equation}
\begin{rem}\label{cprop}
Using naturality, we can rebuild any commutativity constraint $c : \cdot \otimes p \ra p \otimes \cdot$ only from the information $c_{2m} := c_{1_{+2m}}$ for all $m \in \N_0$ where $1_{+2m}$ is the identity of $P_{+2m}$; in this way, we get an element  $\{ c_{2m} \}_{m \in \N_0}$ of $CC_{+k,+k}$ (defined in Proposition \ref{cma}).
Further, by the relation $c_{q \otimes r} = (c_q \otimes 1_r) \circ (1_q \otimes c_r)$ for all $q,r \in \t{ob} (\mcal C)$, it is enough to know $c_2$.
In particular, $c_{2(l+m)} = P_{
\psfrag{cl}{$c_{2l}$}
\psfrag{cm}{$c_{2m}$}
\psfrag{+}{$+$}
\psfrag{2k}{$2k$}
\psfrag{2l}{$2l$}
\psfrag{2m}{$2m$}
\includegraphics[scale=0.15]{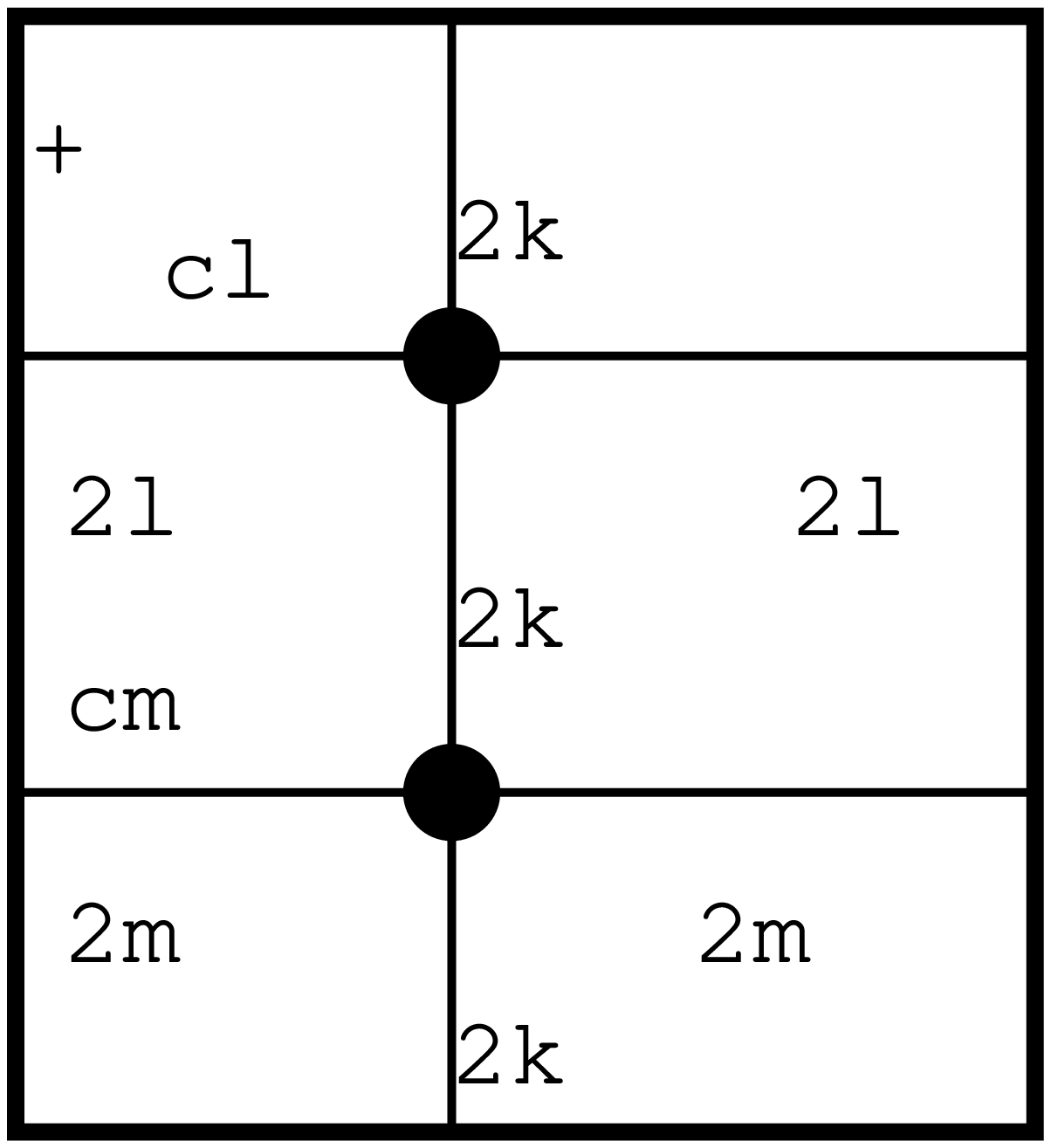}
}$ for all $l,m \in \N_0$; this relation will be referred as {\em grouping relation}.
Using (i) grouping relation for $l=m$, (ii) unitarity of $c_{2l} \in \mcal C (1_{+2l} \otimes p , p \otimes 1_{+2l})$ and (iii) $c_0 = p$, we get $c_{2l} = P_{
\psfrag{cl}{$c^*_{2l}$}
\psfrag{+}{$+$}
\psfrag{2k}{$2k$}
\psfrag{2l}{$2l$}
\includegraphics[scale=0.15]{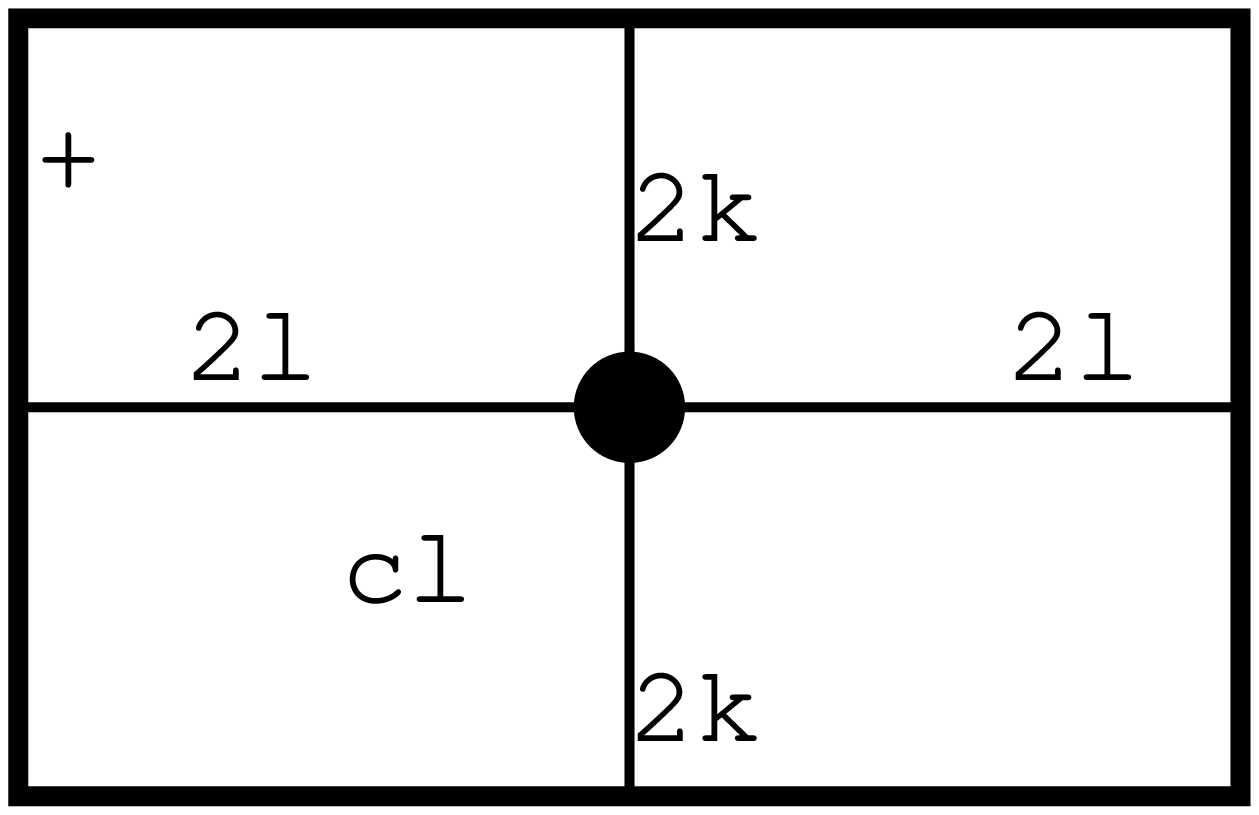}
}$ (referred as {\em $*$-relation}) for $l\in \N_0$.
{Conversely, for any $d \in CC_{+l , +l}$ satisfying the grouping and $*$-relations, $d_0 \in \mscr P (P_{+2l})$ and $d$ extends to a unique unitary commutativity constraint $d : \cdot \otimes d_0 \ra d_0 \otimes \cdot$.}
\end{rem}
We now look at the tensor structure on $\mcal {ZC}$. If $(p,c),(q,d) \in \t{ob} (\mcal{ZC})$ where $p\in \mscr P (P_{+2k})$ and $q\in \mscr P (P_{+2l})$, then set $e_{2m} := P_{
\psfrag{c}{$c_{2m}$}
\psfrag{d}{$d_{2m}$}
\psfrag{+}{$+$}
\psfrag{2k}{$2k$}
\psfrag{2l}{$2l$}
\psfrag{2m}{$2m$}
\psfrag{2n}{$2n$}
\includegraphics[scale=0.15]{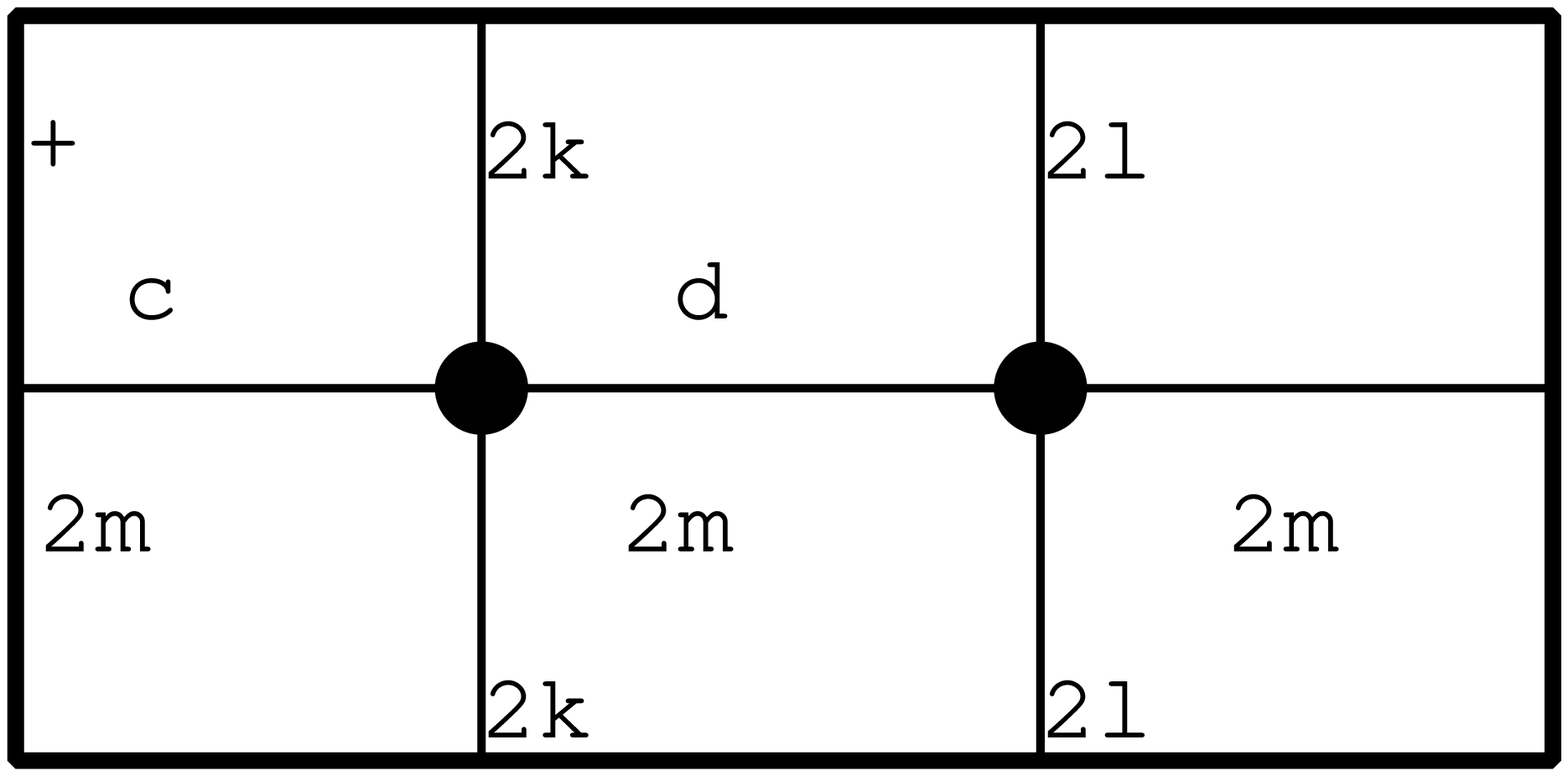}
}$.
Note that $e_{2m}$'s give rise to a commutativity constraint $e : \cdot \otimes (p\otimes q) \ra (p\otimes q) \otimes \cdot$ and thereby, $(p\otimes q , e) \in \t{ob} (\mcal{ZC})$.
\label{ZCtensor} Define $(p,c) \otimes (q,d) := (p\otimes q, e)$.
In this way, $\otimes$ gives a bifunctor on $\mcal{ZC}$.
The identity object of $\mcal{ZC}$ is  $(1_{\mcal C}, i)$ where $i_{2k} := 1_{P_{+2k}}$ for all $k\in \N_0$.
The associativity and unit constraints will be trivial for $\otimes$.
So, $\mcal{ZC}$ is strict.

Fix $(p,c) \in \t{ob} (\mcal {ZC})$.
Suppose $p \in \mscr P (P_{+2k})$.
We will construct a $+$-affine $P$-module $V$.

\noindent {\em The vector spaces:} Set 
\comments{
$V_{-0} := \t{Range } P_{
\psfrag{p}{$p$}
\psfrag{+}{$+$}
\psfrag{-}{$-$}
\psfrag{2k}{$2k$}
\psfrag{2k-2}{$2k\!-\!2$}
\includegraphics[scale=0.15]{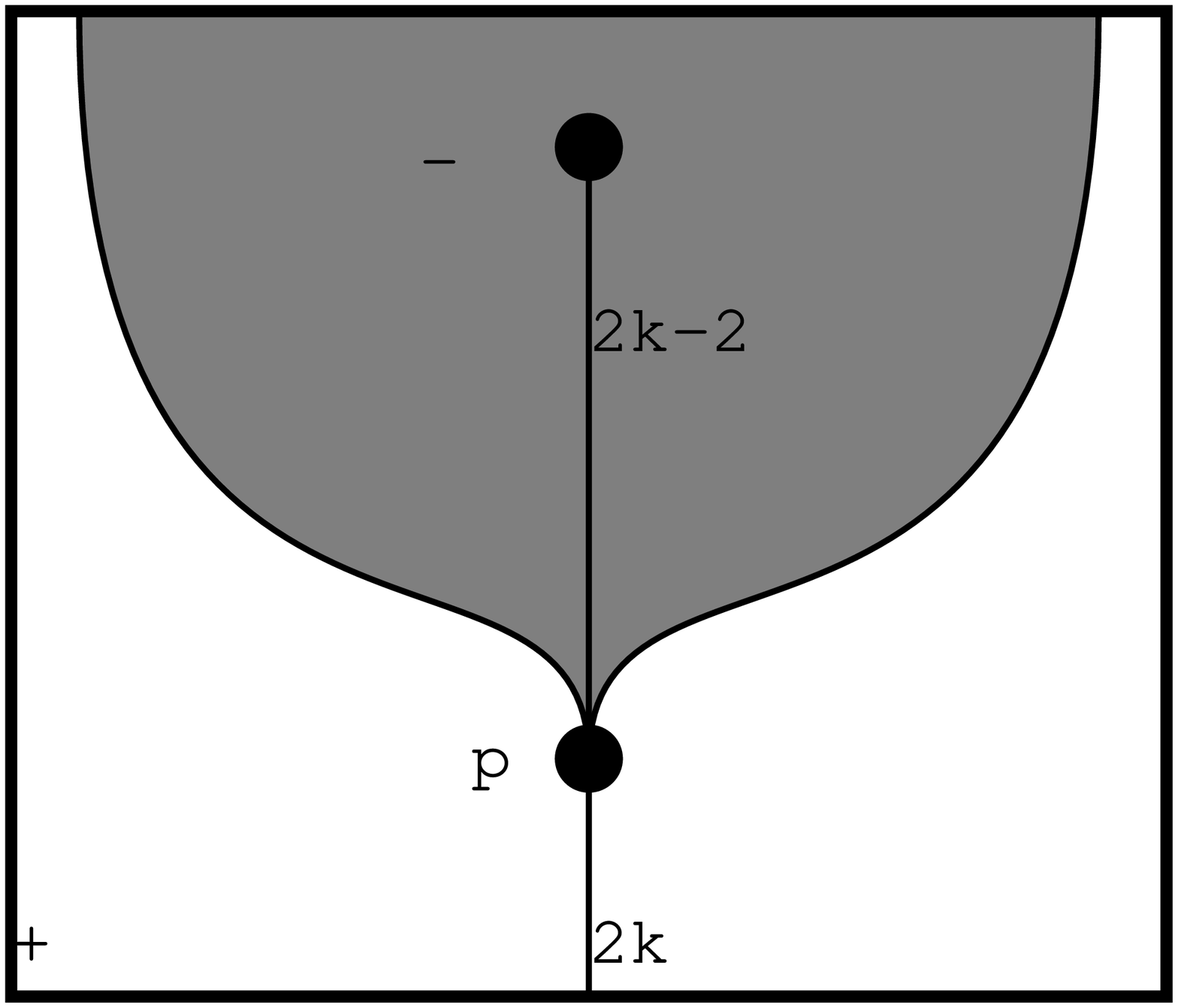}
}$
}
$V_{-0} := \t{span}\{ \us{l\in \N}{\cup} \t{Range } P_{
\psfrag{c}{$c_{2l}$}
\psfrag{+}{$+$}
\psfrag{-}{$-$}
\psfrag{2k}{$2k$}
\psfrag{2l-1}{$2l-1$}
\includegraphics[scale=0.15]{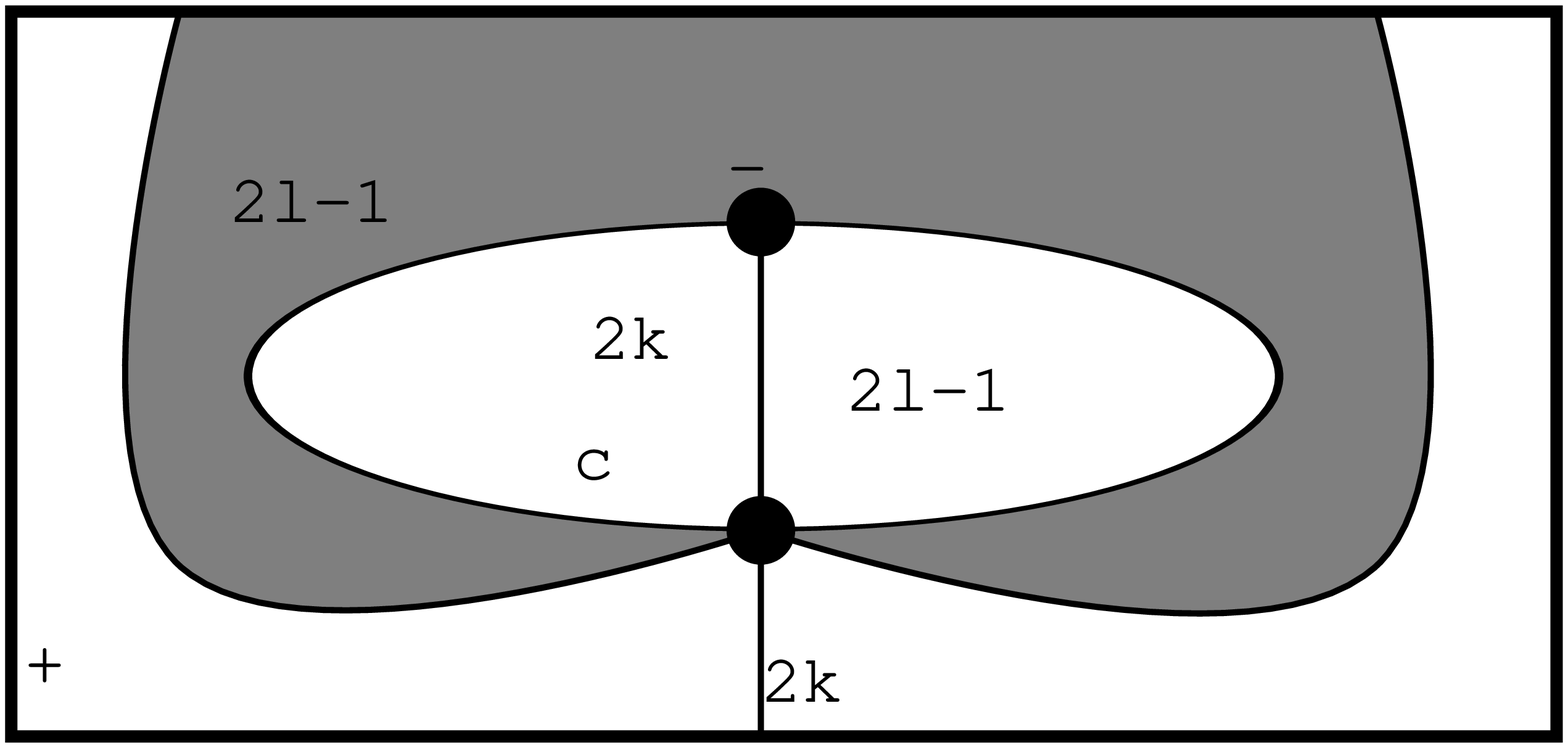}
}\}$
and $V_{+ l} := P_{+(k+l)} \us {2k} \odot p$ for all $l \in \N_0$.
Note that $V_{-0} \subset V_{+1}$; moreover, decomposing
using group relation of $\{c_{2m}\}_{m\in \N_0}$, it is enough to consider $l=1$ in the definition of $V_{-0}$, that is, $V_{-0} := \t{Range } P_{
\psfrag{c}{$\;c_{2}$}
\psfrag{+}{$+$}
\psfrag{-}{$-$}
\psfrag{2k}{$2k$}
\psfrag{2l-1}{}
\includegraphics[scale=0.15]{figures/z2ap/v-0.eps}
}$.

In order to define the action of affine morphism on these vector spaces, we will now go over some prerequisites.
\begin{lem}
For all $a = \psi^{2n}_{+l,+m} (x) \in AP_{+l,+m}$, $v \in V_{+l}$, $l,m \in \N_0$, the element $a \tr v := P_{
\psfrag{c}{$c_{2n}$}
\psfrag{+}{$+$}
\psfrag{x}{$x$}
\psfrag{v}{$v$}
\psfrag{2k}{$2k$}
\psfrag{2l}{$2l$}
\psfrag{2m}{$2m$}
\psfrag{2n}{$2n$}
\includegraphics[scale=0.15]{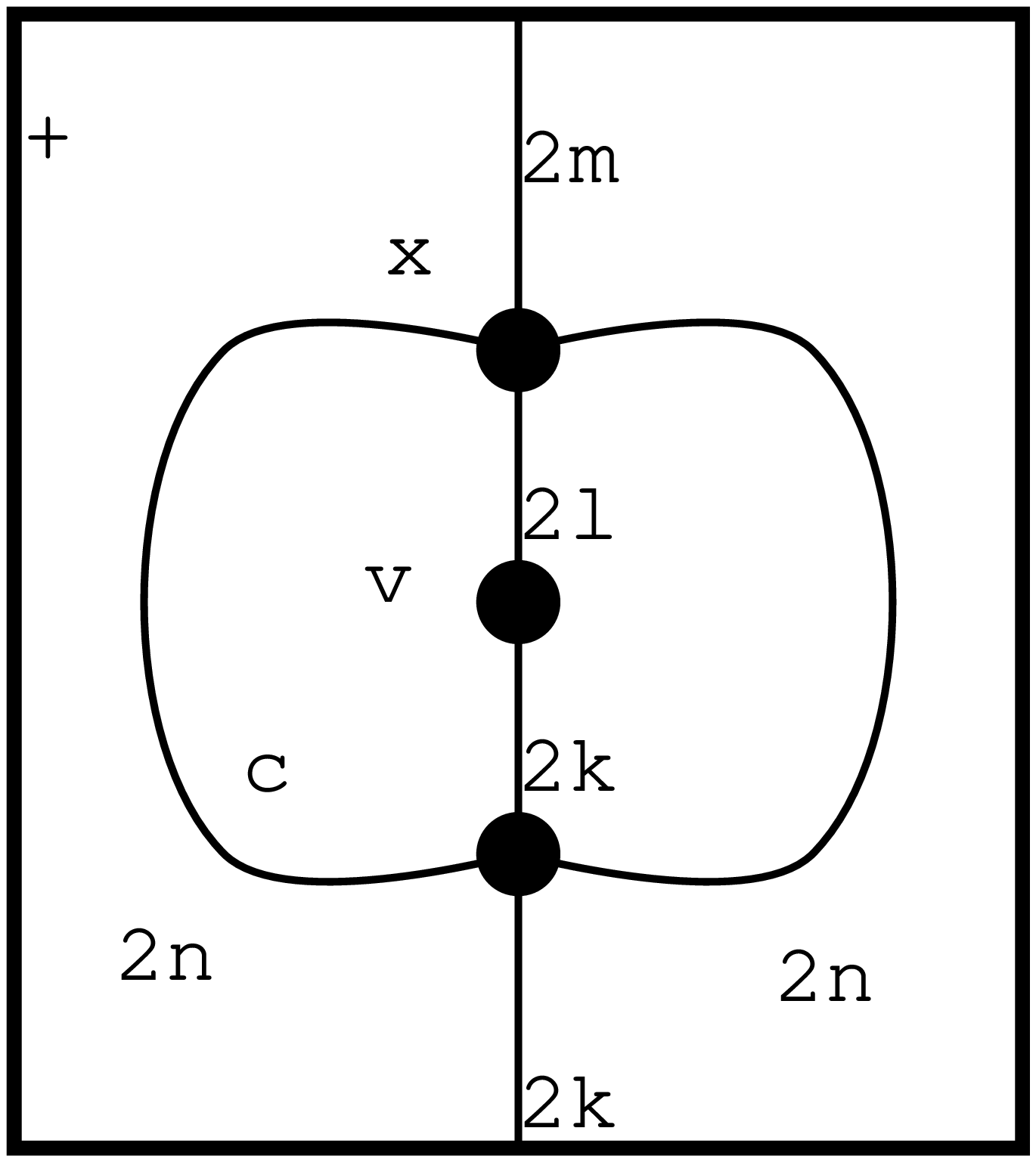}
}$ in $V_{+m}$ is independent of the choice of $n \in \N_0$ and $x \in P_{l+m+2n}$. 
\end{lem}
\begin{proof}
Consider the map $\mcal S_{+l ,+m} \supset \mcal T_{+(l+m+2n)} (P) \ni T \os{\vphi}{\mapsto} P_{
\psfrag{c}{$c_{2n}$}
\psfrag{+}{$+$}
\psfrag{x}{$\!\!\!\!P_T$}
\psfrag{v}{$v$}
\psfrag{2k}{$2k$}
\psfrag{2l}{$2l$}
\psfrag{2m}{$2m$}
\psfrag{2n}{$2n$}
\includegraphics[scale=0.15]{figures/z2ap/a.v.eps}
} \in V_{+m}$.
Naturality of $c$ implies that $\vphi$ is invariant under $\sim$.
Thus, by Remark \ref{apmap}, we have a well defined linear map from $\mcal A_{+l, +m} (P) \ni A \os{\tilde \vphi}{\mapsto} P_{
\psfrag{c}{$c_{2n}$}
\psfrag{+}{$+$}
\psfrag{x}{$\!\!\!\!P_X$}
\psfrag{v}{$v$}
\psfrag{2k}{$2k$}
\psfrag{2l}{$2l$}
\psfrag{2m}{$2m$}
\psfrag{2n}{$2n$}
\includegraphics[scale=0.15]{figures/z2ap/a.v.eps}
} \in V_{+m}$ where $A = \Psi^{2n}_{+l,+m} (X)$ for some $n \in \N_0$, $X \in \mcal P_{l+m+2n} (P)$.
Also, $\tilde \vphi (A) = 0$ whenever $P_X = 0$.
Hence, we have the required result.
\end{proof}
Let $d$ \label{d} denote the affine tangle
\psfrag{+}{$+$}
\psfrag{-}{$-$}
\includegraphics[scale=0.15]{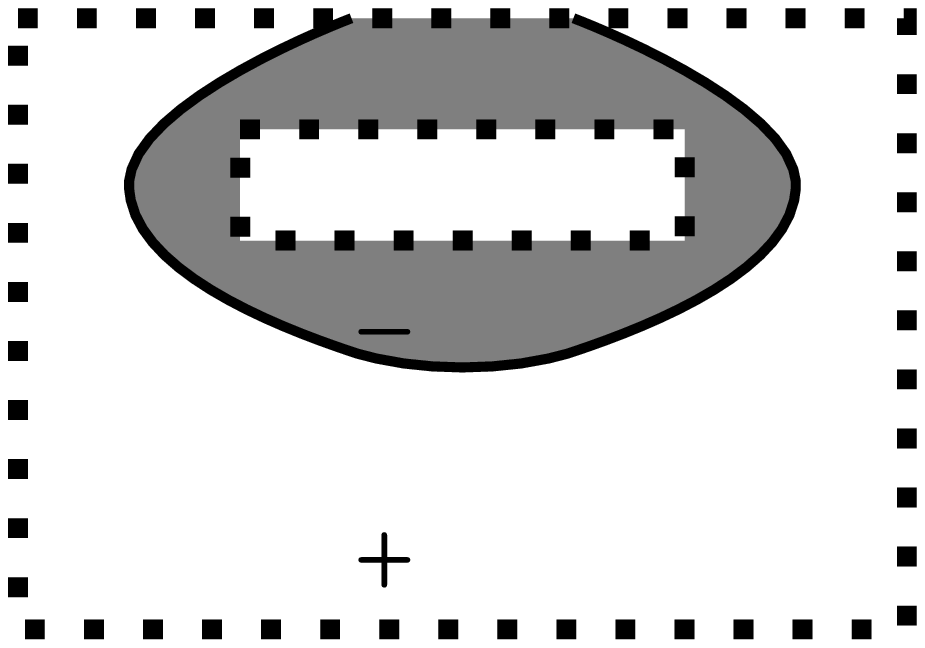}
$: -0 \ra +1$.

\noindent {\em Action of affine morphisms:}
\begin{align*}
(a) & \; AP_{+l,+m} \times V_{+l} \ni (a,v) \mapsto V_a(v) := a \tr v \in V_{+m} \t{ for all } l,m \in \N_0,\\
(b) & \; AP_{+l,-0} \times V_{+l} \ni (a,v) \mapsto V_a(v) := \delta^{-1/2} (d \circ a) \tr v \in V_{-0} \t{ for all } l \in \N_0,\\
(c) & \; AP_{-0,+l} \times V_{-0} \ni (a,v) \mapsto V_a(v) := \delta^{-1/2} (a \circ d^*) \tr v \in V_{+l} \t{ for all } l \in \N_0,\\
(d) & \; AP_{-0,-0} \times V_{-0} \ni (a,v) \mapsto V_a(v) := \delta^{-1} (d \circ a \circ d^*) \tr v \in V_{-0}.
\end{align*}
In (b), (c) and (d), we use the fact $V_{-0} \subset V_{+1}$.
\begin{lem}
The action preserves composition.
\end{lem}
\begin{proof}
Using the {grouping relation given in}  Remark \ref{cprop}(i), it is easy to verify, $V_{a \circ b} (v) = V_a (V_b (v))$ for $a \in AP_{+m,+n}$, $b\in AP_{+l,+m}$ and $v \in V_{+l}$ where $l,m,n \in \N_0$.
Similar arguments would work for other cases of composition over $+$-signed colors.

We now consider a composition of over a $-$-signed color.
Suppose $a \in AP_{-0,+m}$, $b \in AP_{+l , -0}$ for $l,m\in \N_0$ and $v \in V_{+l}$. Observe that $V_a( V_b (v)) = \delta^{-1} (a \circ d^* \circ d \circ b) \tr v = (a \circ b) \tr v = V_{a \circ b} (v)$.
Other instances of composition over $-$-signed colors can be deduced in the same way.
\end{proof}
Thus, $V$ becomes a locally finite affine $P$-module.
We equip $V_{+l}$ with the inner product $\lab \cdot , \cdot \rab_V$, induced by the picture trace on $P_{+(k+l)}$ and $V_{-0}$ with that on $P_{+(k+1)}$.
\begin{lem}
The action preserves $*$.
\end{lem}
\begin{proof}
For $a = \psi^{2n}_{+l,+m} (x) \in AP_{+l,+m}$, $v \in V_{+l}$, $w\in V_{+m}$, note that
\[
\lab w , a \tr v \rab_V = P_{
\psfrag{c}{$c^*_{2n}$}
\psfrag{+}{$+$}
\psfrag{x}{$x$}
\psfrag{v}{$v$}
\psfrag{w}{$w^*$}
\psfrag{2k}{$2k$}
\psfrag{2l}{$2l$}
\psfrag{2m}{$2m$}
\psfrag{2n}{$2n$}
\includegraphics[scale=0.15]{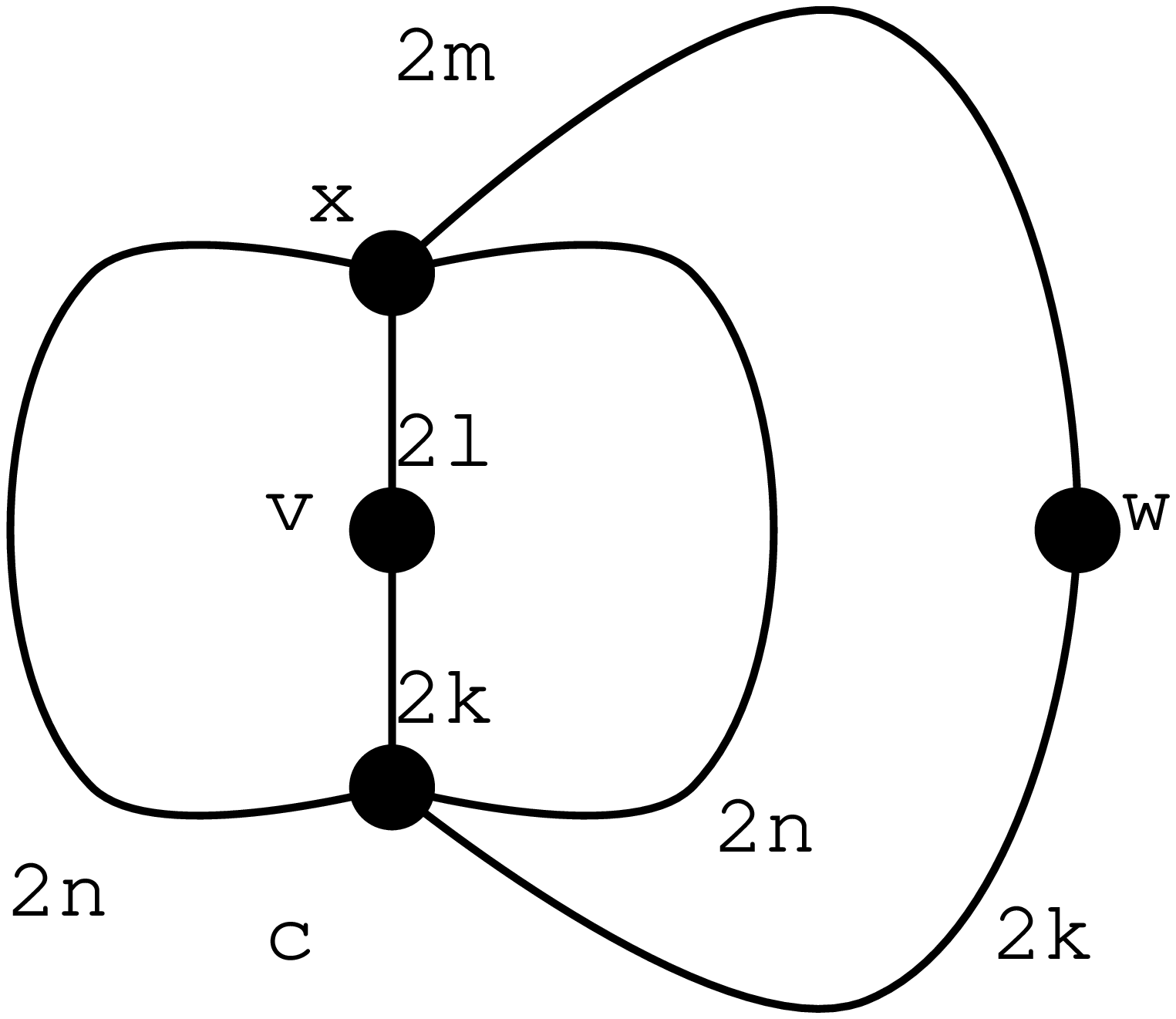}
} = P_{
\psfrag{c}{$c^*_{2n}$}
\psfrag{+}{$+$}
\psfrag{x}{$x$}
\psfrag{v}{$v$}
\psfrag{w}{$w^*$}
\psfrag{2k}{$2k$}
\psfrag{2l}{$2l$}
\psfrag{2m}{$2m$}
\psfrag{2n}{$2n$}
\includegraphics[scale=0.15]{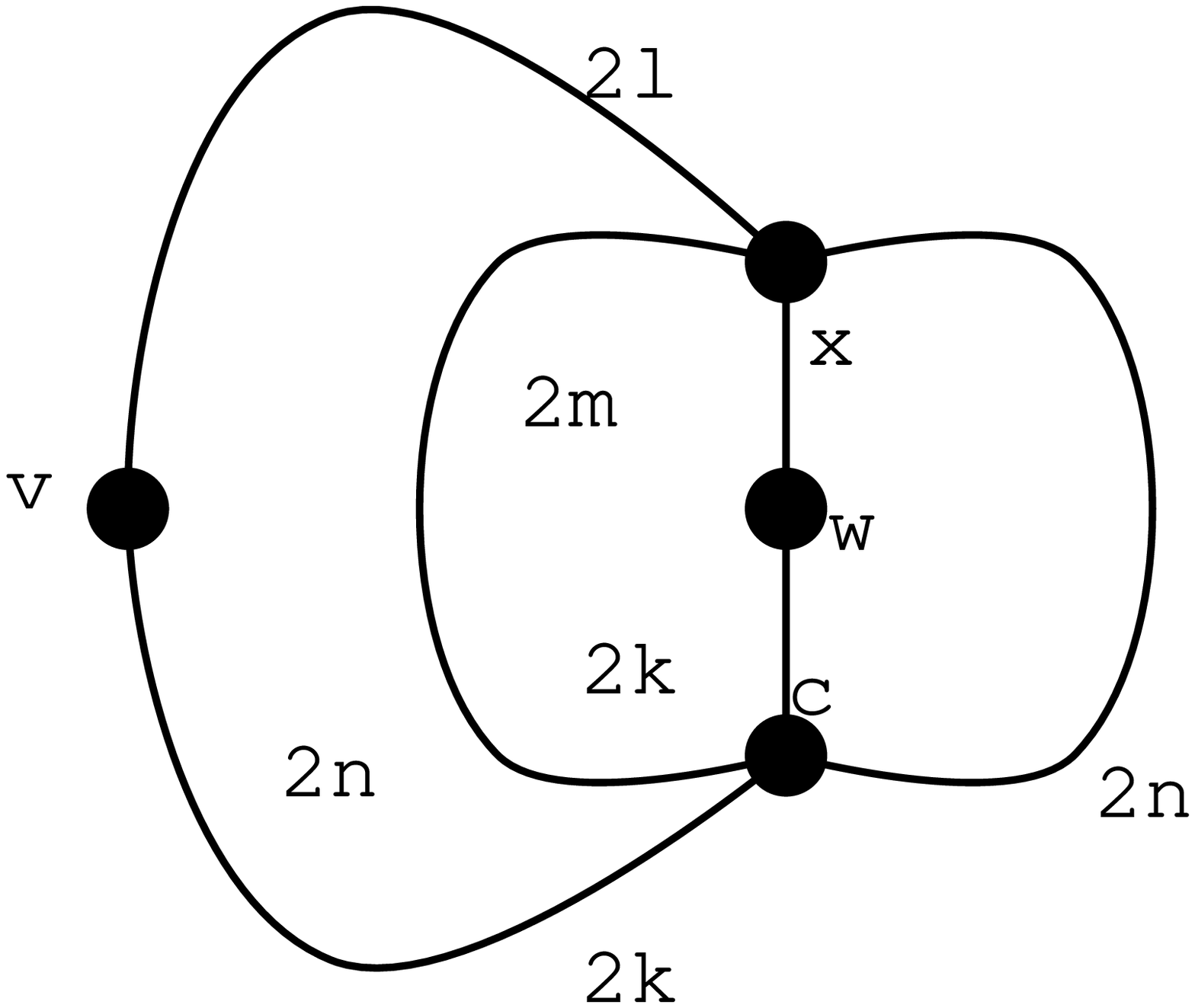}
} = \lab \psi^{2n}_{+m,+l} (x^*) \tr w , v \rab_V = \lab a^* \tr w , v \rab_V
\]
where the first equality follows from the definition of $\lab \cdot , \cdot \rab_V$ and the second identity in Remark \ref{cprop}, and in the second equality, using sphericality of $P$. This proves $\lab w , a \tr v \rab_V = \lab a^* \tr w , v \rab_V$ for both input and output colors of $a$ being $+$-signed.

We will only show one case of the input and output colors of $a$ being $-$-signed and other can be deduced in the same method. For $a \in AP_{+l ,-0}$, we get
\[
\lab w , a \tr v \rab_V = \delta^{-1/2} \lab w , (d \circ a) \tr v \rab_V = \delta^{-1/2} \lab (a^* \circ d^*) \tr w , v \rab_V = \lab a^* \tr w , v \rab_V.
\]
\end{proof}
This makes $V$ into a locally finite Hilbert $+$-affine $P$-module.
Henceforth, we will denote $V$ by $V(p,c)$.
\begin{rem}\label{ccipvpc}
{The $CC$-valued inner products of $V(p,c)$ are given by}:
\begin{align*}
& c_{2n}(w , v) = P_{
\psfrag{c}{$c_{2n}$}
\psfrag{+}{$+$}
\psfrag{v}{$v$}
\psfrag{w}{$w^*$}
\psfrag{2k}{$2k$}
\psfrag{2l}{$2l$}
\psfrag{2m}{$2m$}
\psfrag{2n}{$2n$}
\includegraphics[scale=0.15]{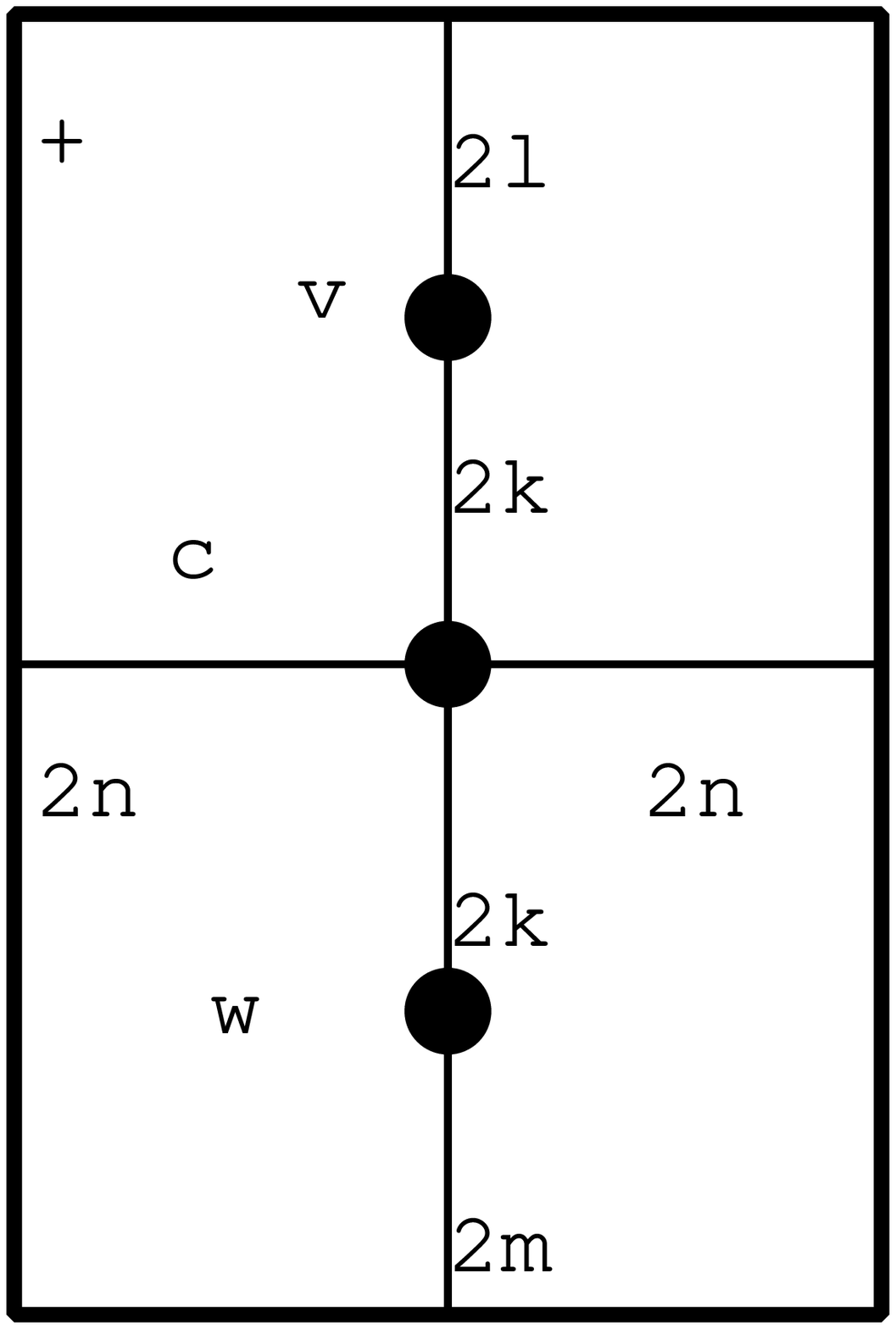}
}, && c_{2n-1}(w , x) = \delta^{ - \frac 1 2} P_{
\psfrag{c}{$c_{2n}$}
\psfrag{+}{$-$}
\psfrag{x}{$x$}
\psfrag{w}{$w^*$}
\psfrag{2k}{$2k$}
\psfrag{2l}{$2l$}
\psfrag{2m}{$2m$}
\psfrag{2n}{$2n\!-\!1$}
\includegraphics[scale=0.15]{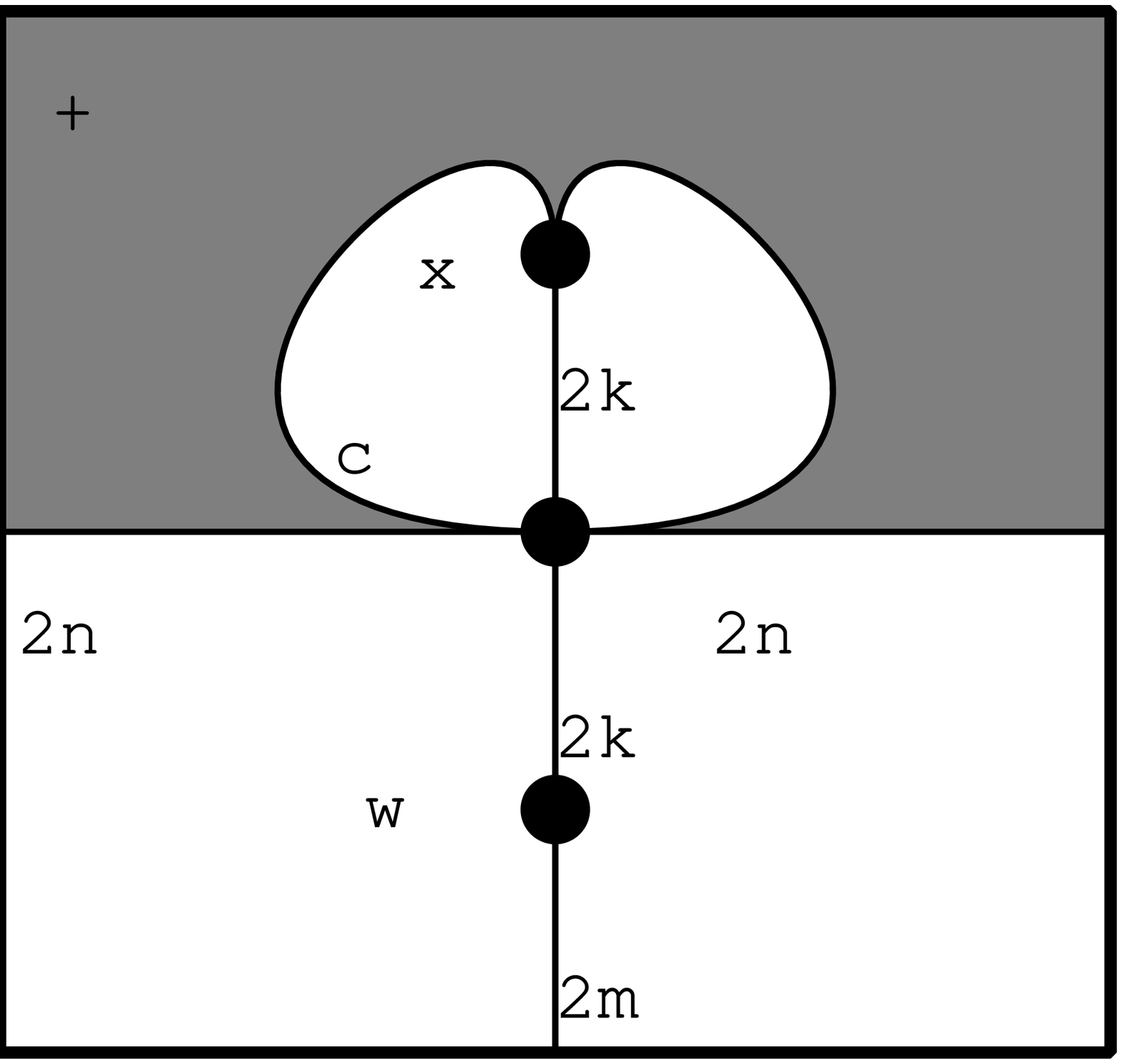}
},\\
& c_{2n-1}(y , v) = \delta^{ - \frac 1 2} P_{
\psfrag{c}{$c_{2n}$}
\psfrag{+}{$+$}
\psfrag{v}{$v$}
\psfrag{y}{$y^*$}
\psfrag{2k}{$2k$}
\psfrag{2l}{$2l$}
\psfrag{2m}{$2m$}
\psfrag{2n}{$2n\!-\!1$}
\includegraphics[scale=0.15]{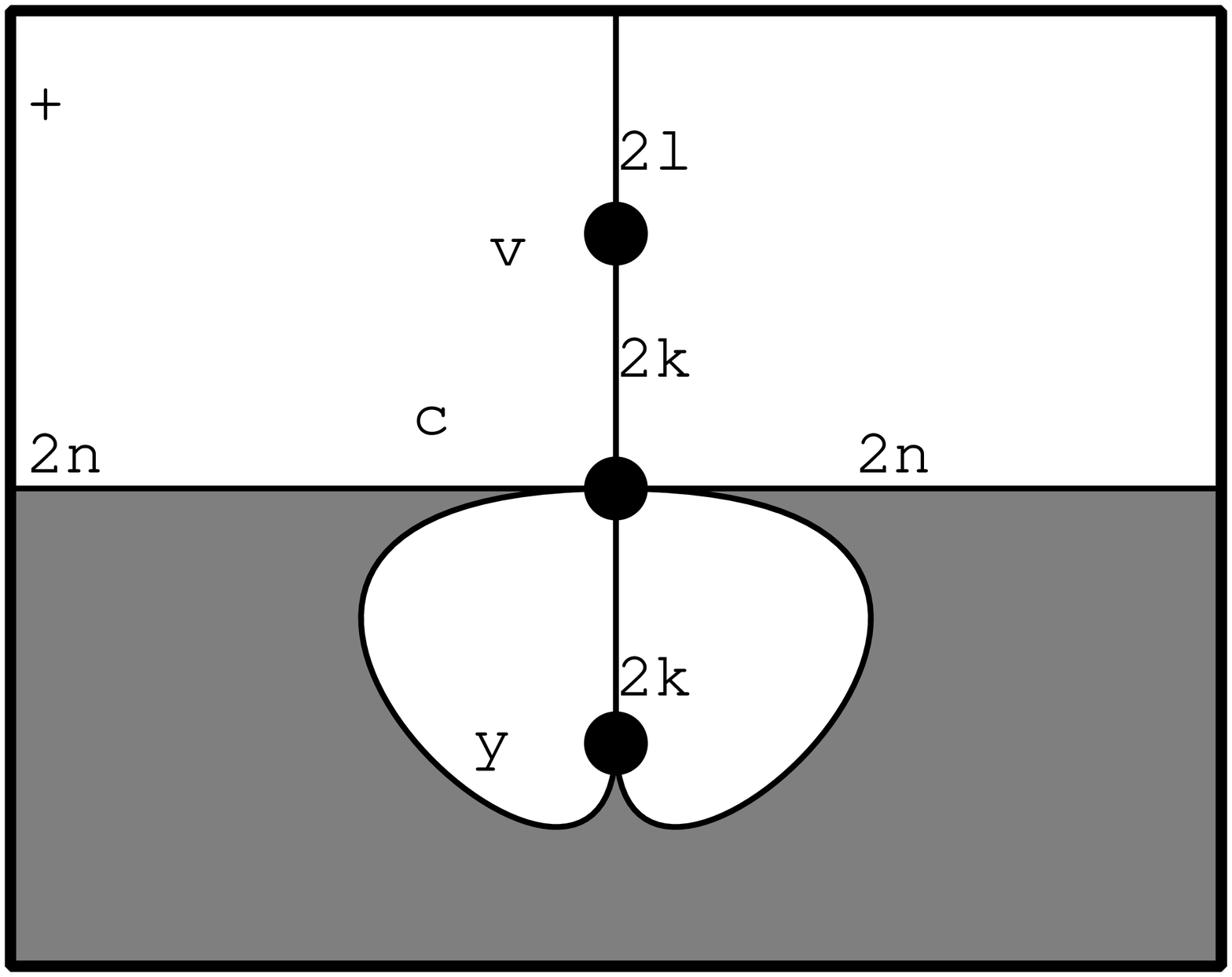}
}, && c_{2n-2}(y , x) = \delta^{ - 1} P_{
\psfrag{c}{$c_{2n}$}
\psfrag{+}{$-$}
\psfrag{x}{$x$}
\psfrag{y}{$y^*$}
\psfrag{2k}{$2k$}
\psfrag{2l}{$2l$}
\psfrag{2m}{$2m$}
\psfrag{2n}{$2n\!-\!2$}
\includegraphics[scale=0.15]{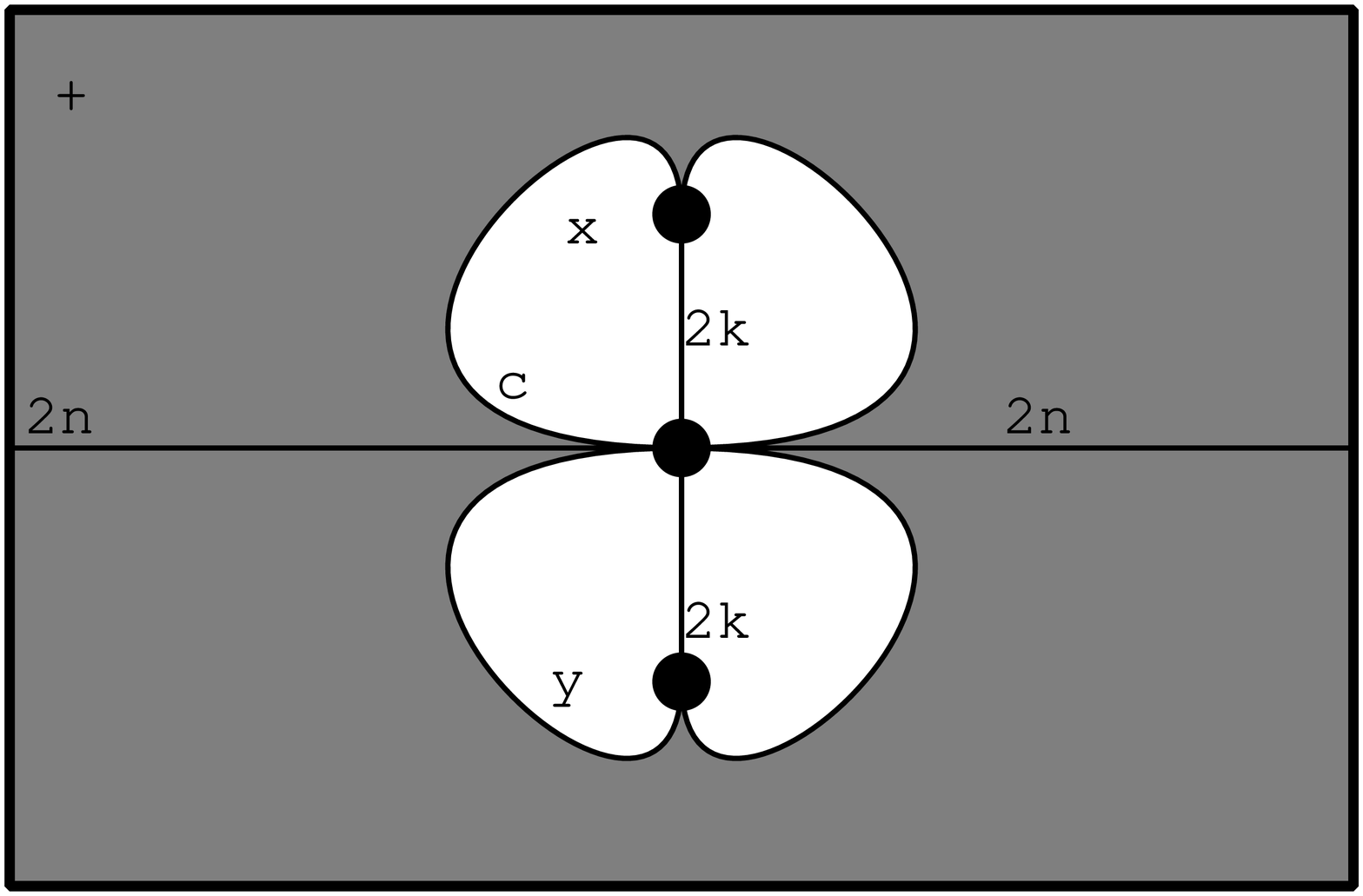}
},
\end{align*}
for all $v\in V_{+2l} (p,c)$, $w \in V_{+2m} (p,c)$, $x,y \in V_{-0} (p,c)$.
\end{rem}
\begin{lem}\label{fusbasis}
Suppose $(p,c) , (q,d) \in\t{ob}(\mcal{ZC})$.
If $(r,e) := (p,c) \otimes (q,d)$ (as defined on Page \pageref{ZCtensor}) and $W:= V(p,c) \btimes V(q,d)$, then
$
W_{+m} = \{ p \os x \btimes q : x \in V_{+m} (r,e) \}$  and $ W_{-0} =\{p \os a \btimes q : a \circ \psi^0_{+(k+l),+(k+l)} (r) = a \in \t{Range } \psi^1_{+(k+l) , -0} \}.
$
Thus, $W$ becomes locally finite.
\end{lem}
\begin{proof}
Consider $\xi \os a \btimes \zeta \in W_{+m}$.
Without loss of generality, (using Lemma \ref{fuswloglem}) we may assume $\xi \in V_{+k'}(p,c)$, $\zeta \in V_{+l'} (q,d)$ and $a = \psi^{0}_{+(k'+l') , +m} (x)$.
Set $x' := P_{
\psfrag{c}{$c_{2n}$}
\psfrag{d}{$d_{2n}$}
\psfrag{+}{$+$}
\psfrag{x}{$x$}
\psfrag{xi}{$\xi$}
\psfrag{zeta}{$\zeta$}
\psfrag{2k}{$2k$}
\psfrag{2k'}{$2k'$}
\psfrag{2l}{$2l$}
\psfrag{2l'}{$2l'$}
\psfrag{2m}{$2m$}
\psfrag{2n}{$2n$}
\includegraphics[scale=0.15]{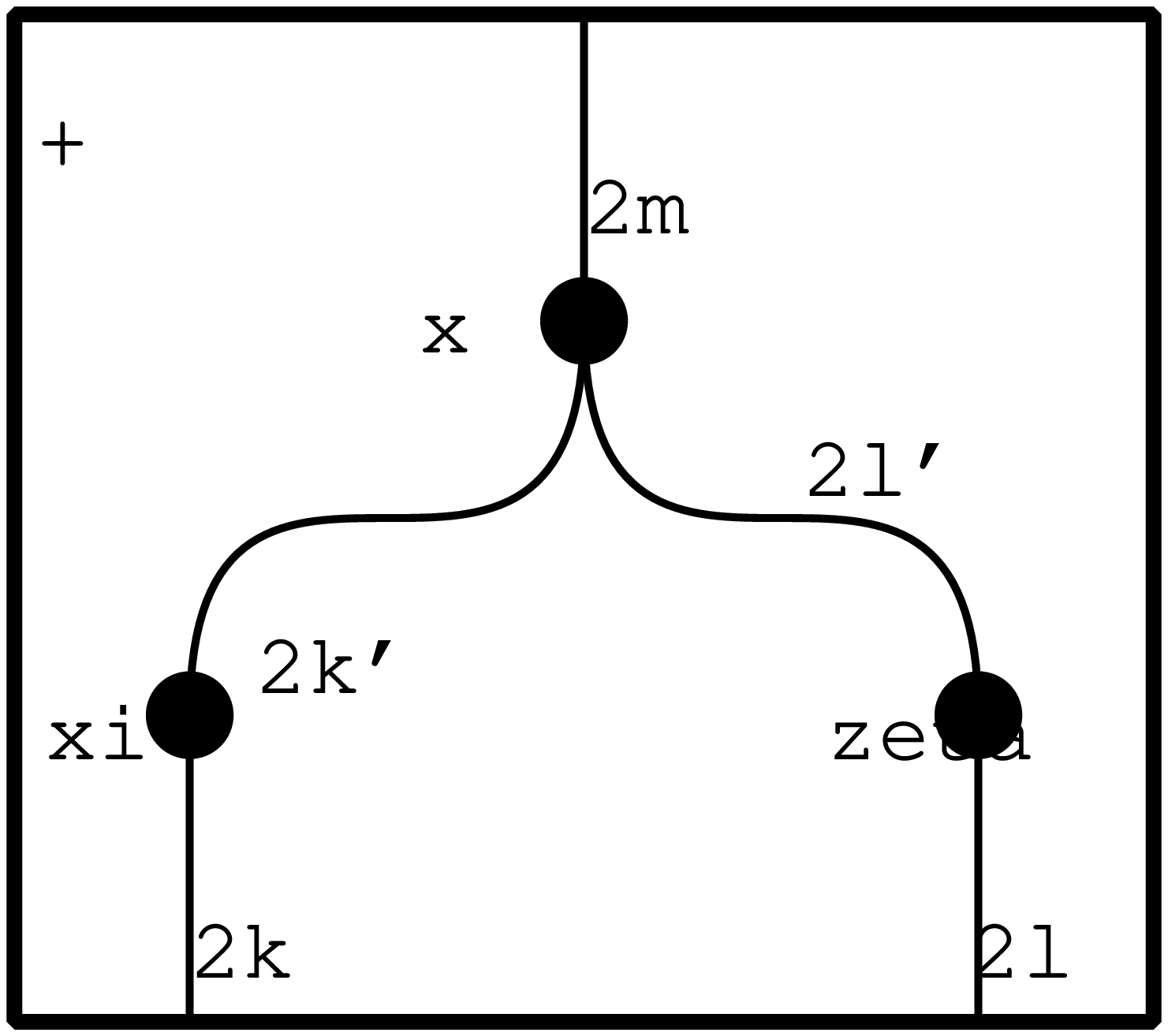}
} \in V_{+m} (r,e)$.
Now, for all $\xi_1 \in V_{+k_1} (p,c)$, $\zeta_1 \in V_{+l_1} (q,d)$, $x_1 = P_{+(k_1+l_1 +m)}$, $k_1,l_1 \in \N_0$, we have
\[
\left\lab \xi_1 \os {x_1} \btimes \zeta_1 , \xi \os a \btimes \zeta \right\rab = P_{\!\!\!\!\!\!\!
\psfrag{c}{$c_{2n\!+\!2n_1}$}
\psfrag{d}{$d_{2n\!+\!2n_1}$}
\psfrag{+}{$+$}
\psfrag{x}{$x$}
\psfrag{x1}{$x^*_1$}
\psfrag{xi}{$\xi$}
\psfrag{xi1}{$\xi^*_1$}
\psfrag{zeta}{$\zeta$}
\psfrag{zeta1}{$\zeta^*_1$}
\psfrag{2k}{$2k$}
\psfrag{2k1}{$2k_1$}
\psfrag{2k'}{$2k'$}
\psfrag{2l}{$2l$}
\psfrag{2l1}{$2l_1$}
\psfrag{2l'}{$2l'$}
\psfrag{2m}{$2m$}
\psfrag{2n}{$2n$}
\psfrag{2n1}{$2n_1$}
\psfrag{2n+2n1}{$2n\!+\!2n_1$}
\includegraphics[scale=0.15]{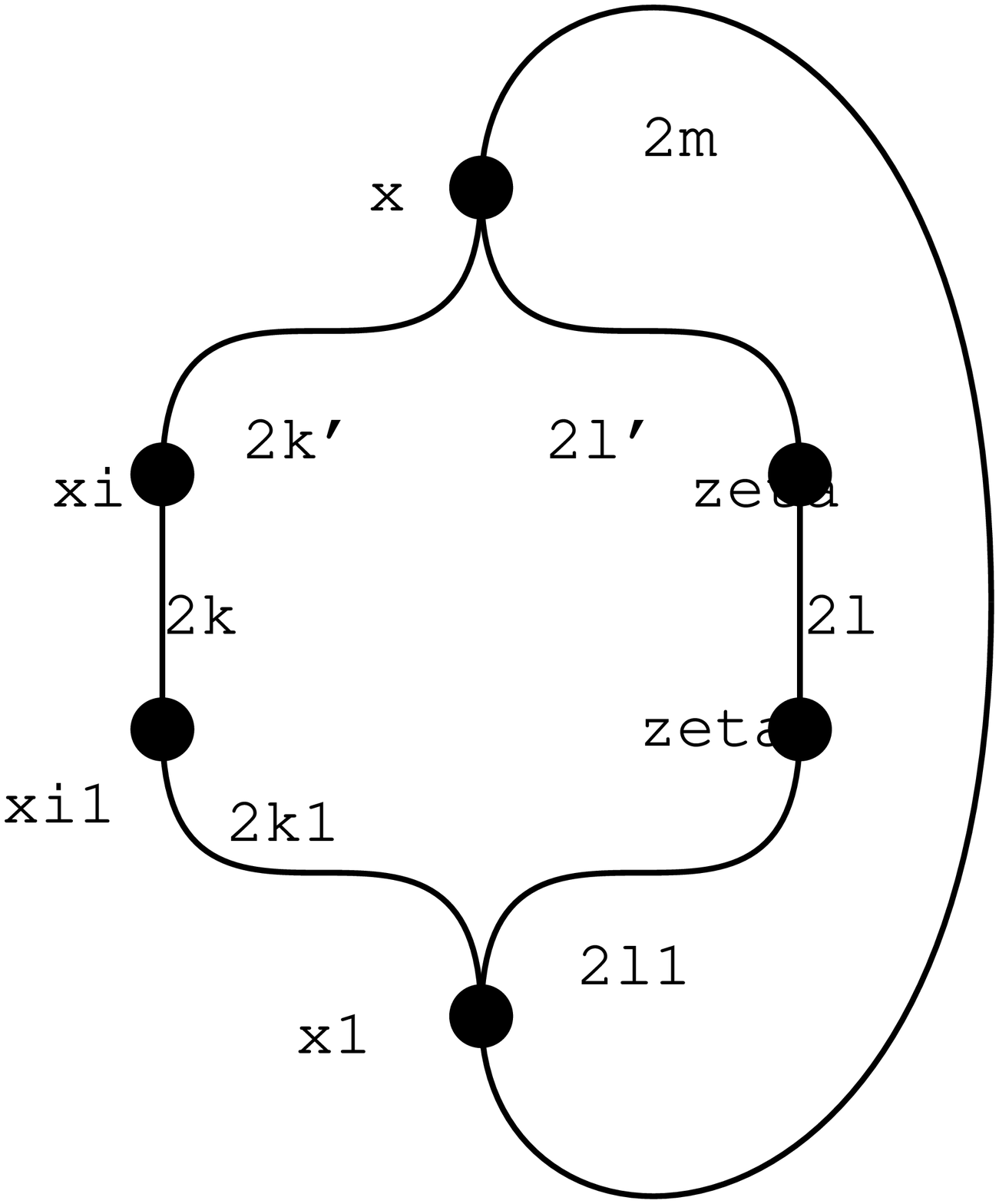}
} = \left\lab \xi_1 \os {x_1} \btimes \zeta_1 , p \os {x'} \btimes q \right\rab
\]
where we use Lemma \ref{fusCCip} and Remark \ref{ccipvpc} to acheive the equalities.
Thus, $\xi \os a \btimes \zeta = p \os {x'} \btimes q$.

We now prove the second part.
We use the affine morphism $d \in AP_{-0,+1}$ (defined on Page \pageref{d}) to get $W_{-0} = W_{d^* \circ d} (W_{-0}) = W_{d^*} (W_{+1}) = \left\{ \left( p \os{d^* \circ b} \btimes q\right) : b = \psi^0_{+(k+l),+1} (x) \t{ for } x \in V_{+1} (r,e) \right\}$.
Clearly, $d^* \circ b \in \t{Range } \psi^1_{+(k+l),-0}$, and using $x \in V_{+1}(r,e) = \left(P_{+(k+l+1)} \us {2k+2l} \odot r\right)$, we get the relation $d^* \circ b \circ \psi^0_{+(k+l),+(k+l)} (r) = d^* \circ b$.
\end{proof}
Related to these modules $V(p,c)$'s, we introduce the following
definition. 
\begin{defn}\label{finPsupp}
A Hilbert affine $P$-module $V$ is said to have {\em finite $P$-support} if there exists $k \in \N_0$ such that $V_{\pm l} = \overline{\t{span } V_{AP^{=0}_{\pm k, \pm l}} (V_{\pm k})}$ for all $l \in \N_0$.
\end{defn}
The condition in the above definition is only relevant when $l > k$ because the equality trivially holds for $l \leq k$.
\begin{rem}\label{findepfinPsupp}
If $\t{depth} (P) < \infty$, then all Hilbert affine $P$-modules have finite $P$-support.
\end{rem}
Theorem 6.11 in \cite{Gho}, tells us that all Hilbert affine $P$-modules can have weight at most half the depth, and thereby, must have finite support when $P$ is finite depth.
However, finite $P$-support is stronger than finite support.
For this, we look at the proof of \cite[Proposition 6.8]{Gho}, {from which we get the identity $AP_{\vlon l , \vlon l} = AP^{=0}_{\vlon k, \vlon l} \circ AP_{\vlon l , \vlon k}$ for all $k > \frac 1 2 \t{depth}(P)$, $l \in \N_0$, $\vlon = \pm$.
So, if $V$ is a Hilbert affine $P$-module, then $V_{\vlon l} = V_{AP^{=0}_{\vlon k , \vlon l}} (V_{\vlon k})$.}

We consider the full subcategory $\mcal D$ of $\mcal HAPM$ whose objects are locally finite and have finite $P$-support.
\begin{rem}
  $\mcal D$ is closed under taking $\btimes$, and hence, it is a braided tensor *-category.
\end{rem}
\begin{lem}
{ $V(p,c)$ has finite $P$-support, that is, $\t{Ind } V(p,c) \in \t{ob} (\mcal D)$ where $\t{Ind}$ is the induction of Hilbert $+$-affine $P$-modules as defined in the proof of Remark \ref{eaff}.}
\end{lem}
\begin{proof}
Let us  denote $\t{Ind } V(p,c)$ by $V$.
Suppose $p\in \mscr P (P_{+2k})$.
Note that $V_{+l} = P_{+(k+l)} \us{2k} \odot p = V_{AP^{=0}_{+k,+l}} (p) = V_{ AP^{=0}_{+k,+l}} (V_{+k})$ for all $l \in \N_0$.

Next, we consider the space $V_{-l}$ for $l>0$.
By definition, $V_{-l} = V_{+l}$ as a vector space.
Fix $l > k+1$ and {$y \in V_{-l} = \left(P_{+(k+l)} \us {2k} \odot p\right)$.}
Set $x:=P_{
\psfrag{-}{$-$}
\psfrag{2l-1}{$2l\!-\!1$}
\psfrag{y}{$y$}
\psfrag{c}{$c_2$}
\psfrag{2k}{$2k$}
\includegraphics[scale=0.15]{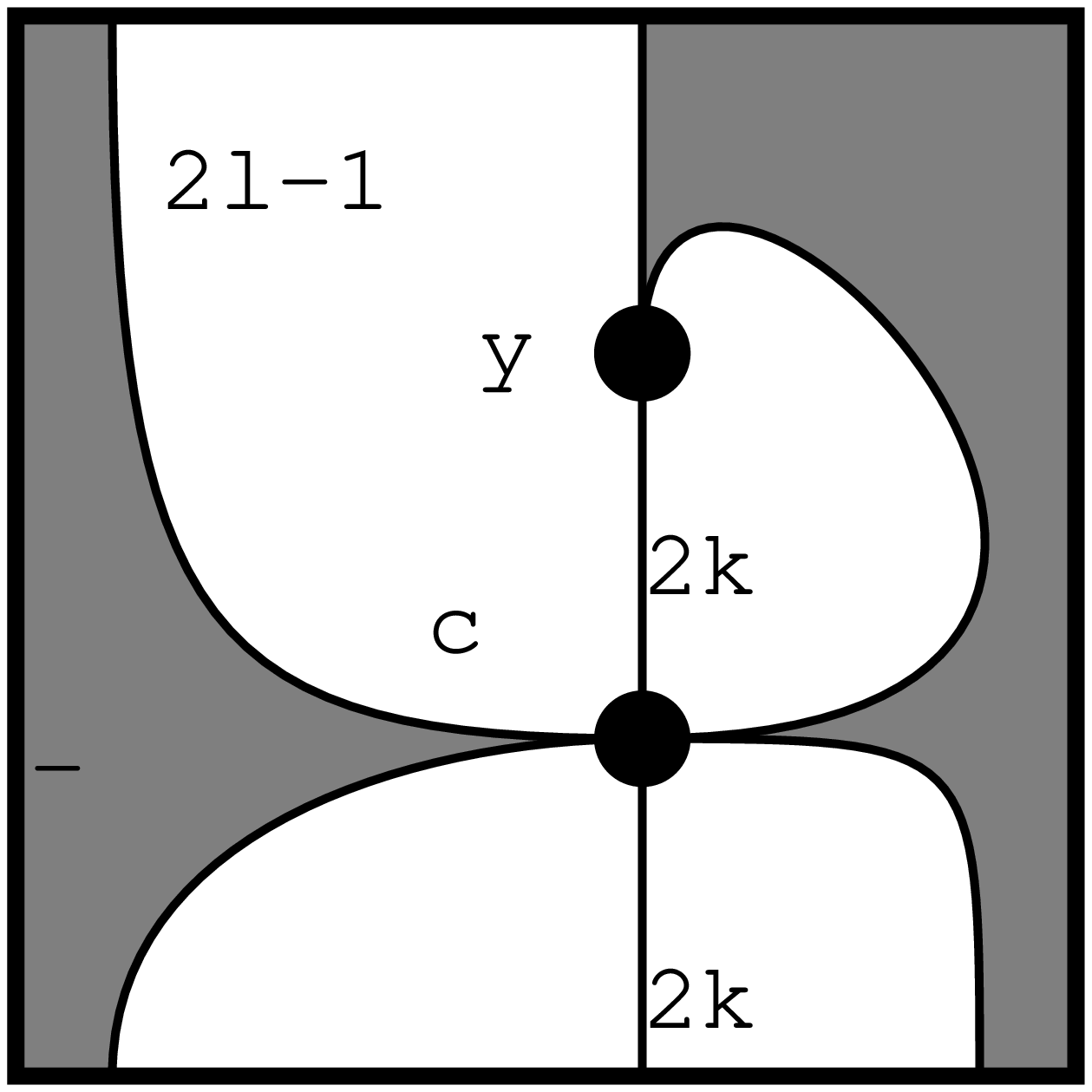}
}\in P_{-(k+l+1)}$ and $v := P_{
\psfrag{+}{$+$}
\psfrag{p}{$p$}
\psfrag{2k}{$2k$}
\includegraphics[scale=0.15]{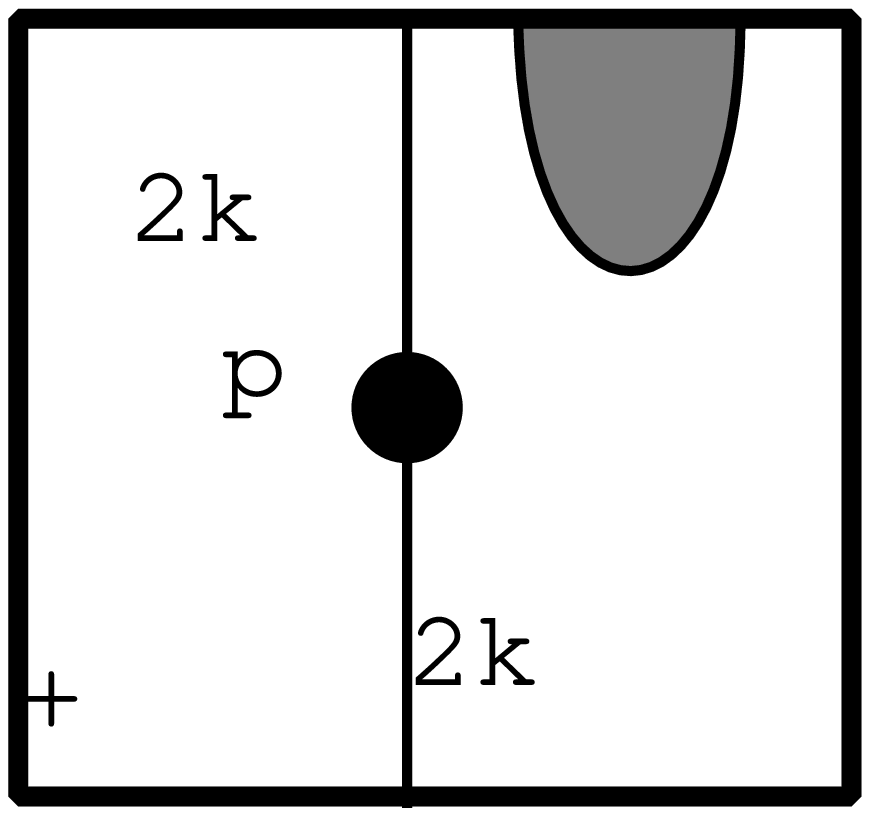}
} \in V_{-(k+1)}$.
If $b = \psi^0_{-(k+1),-l} (x) \in AP^{=0}_{-(k+1),-l}$, then $V_{b} (v) = V_{AR_{-l} \circ b \circ (AR_{-(k+1)})^{-1}} (v)$.
Observe that $AR_{-l} \circ b \circ (AR_{-(k+1)})^{-1} = \psi^2_{+(k+1),+l} (P_{
\psfrag{-}{$+$}
\psfrag{2l-1}{$2l\!-\!1$}
\psfrag{y}{$y$}
\psfrag{c}{$c_2$}
\psfrag{2k}{$2k$}
\psfrag{2}{$2$}
\psfrag{}{$$}
\includegraphics[scale=0.15]{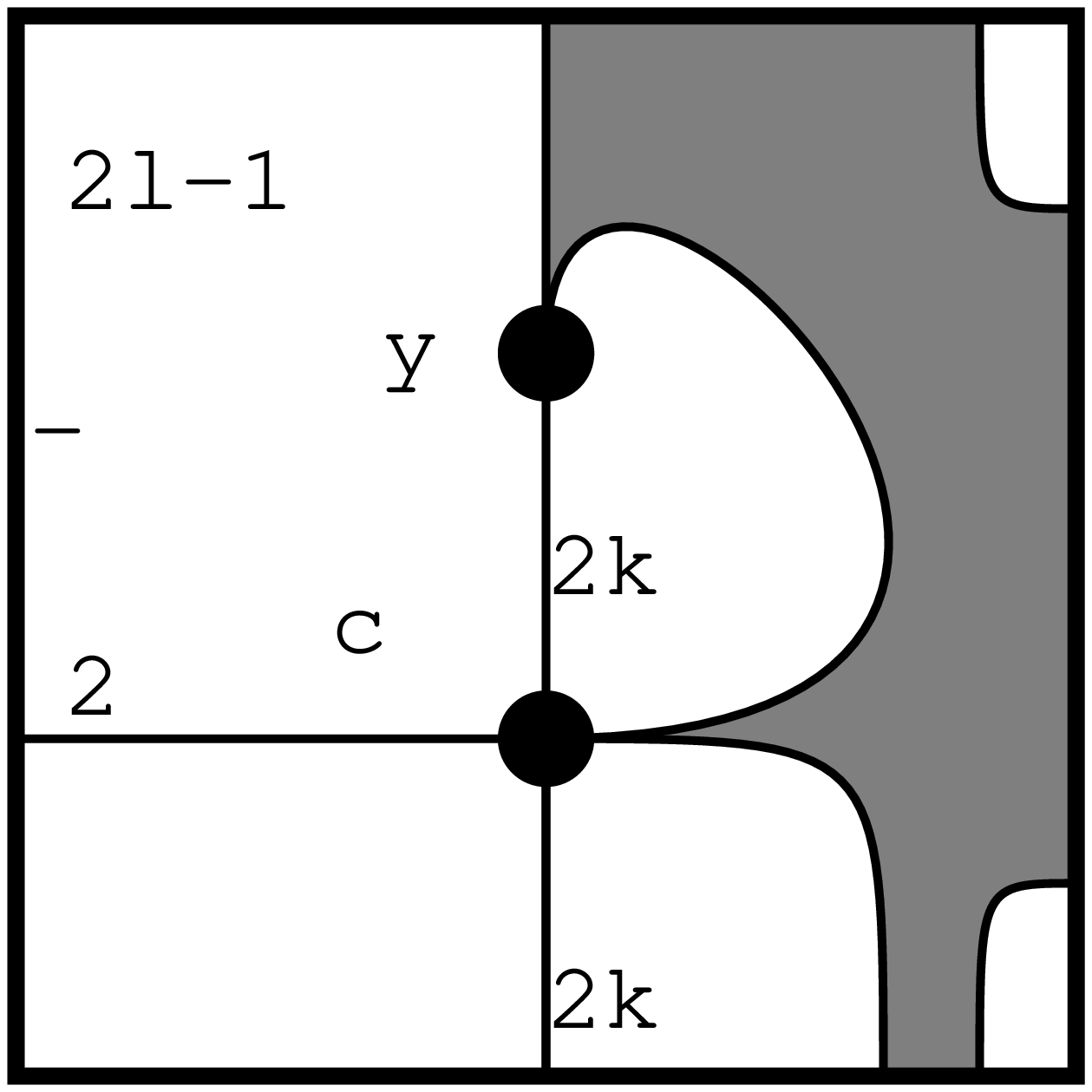}
})$.
So, $V_b (v) = P_{
\psfrag{-}{$+$}
\psfrag{2l-1}{$2l\!-\!1$}
\psfrag{y}{$y$}
\psfrag{c}{$c_2$}
\psfrag{2k}{$2k$}
\psfrag{2}{$2$}
\psfrag{}{$$}
\includegraphics[scale=0.15]{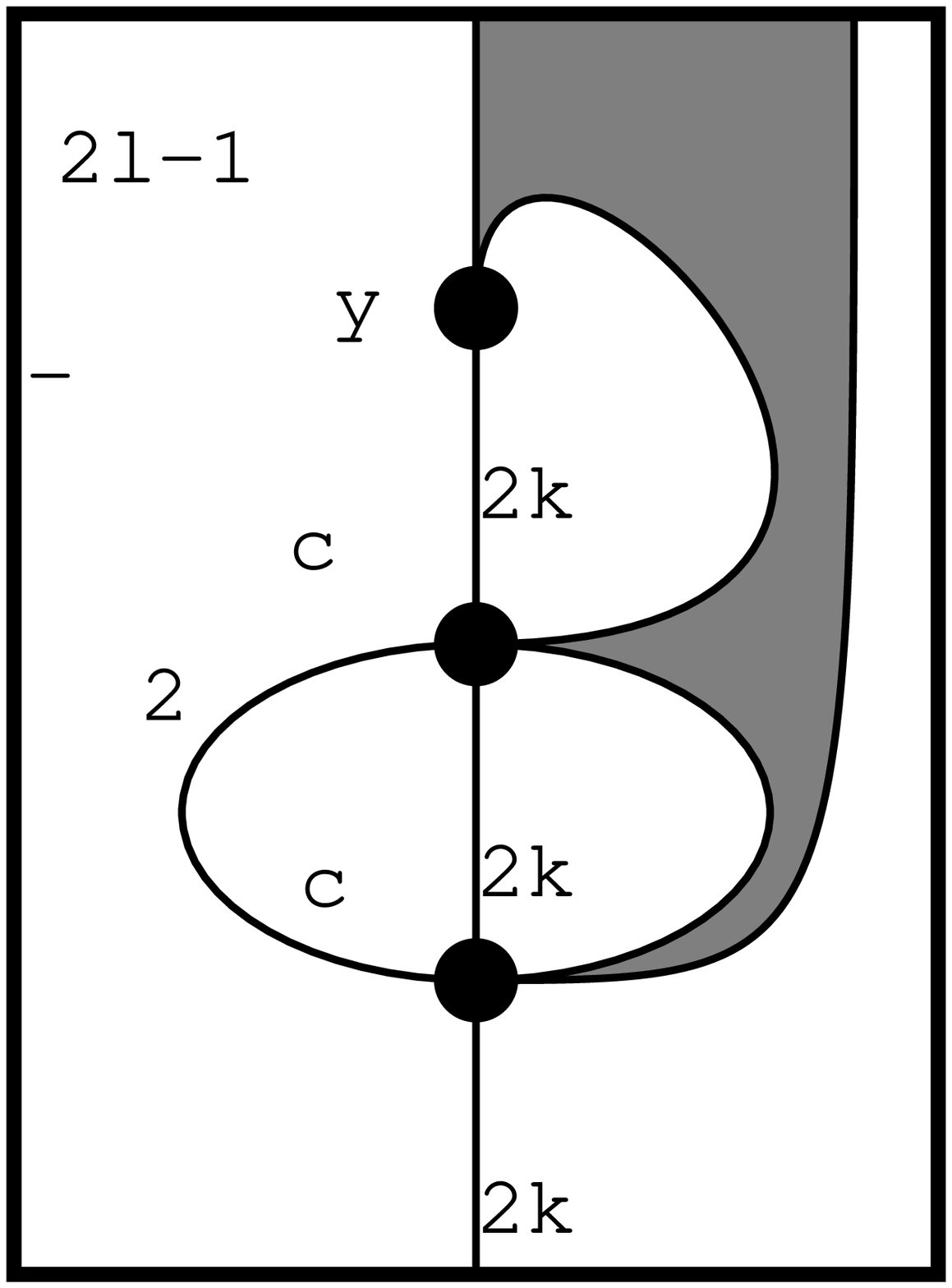}
} = \delta y$  using the grouping relation of $c$ in the last equality.
Thus, $V_{-l} = V_{AP^{=0}_{-(k+1), - l}} (V_{-(k+1)})$ for all $l > k+1$ and thereby, for all $l \in \N_0$.
Hence, $V_{-l} = V_{AP^{=0}_{\vlon (k+1), \vlon l}} (V_{\vlon (k+1)})$ for all $l \in \N_0$, $\vlon = \pm$.
\end{proof}

\begin{thm}
{$\t{ob}(\mcal {ZC} ) \ni (p,c) {\longmapsto} Ind\, V(p,c) \in \t{ob} (\mcal D)$ extends to a contravariant, fully faithful, $\C$-linear, monoidal $*$-functor from $ \mcal {ZC}$ to $ \mcal D$.}
\end{thm}

\begin{proof}
{By Remark \ref{eaff}, it is enough to show that $\t{ob}(\mcal {ZC} ) \ni (p,c) \os{V}{\longmapsto}  V(p,c)$ extends to a contravariant, fully faithful, $\C$-linear, monoidal $*$-functor.}

 We begin with defining $V$ in the level of morphisms.
Let $f:(p,c)  \ra (q,d)$ be a morphism in $\mcal {ZC}$ where $p \in \mscr P (P_{+2k})$ and $q \in \mscr P (P_{+2l})$.
Define $V(f)$ by $V_{+m} (q,d) \ni v \os {V(f)} \longmapsto v \us {2l} \odot f \in V_{+m} (p,c)$ for all $m \in \N_0$.
Note that by the relation \ref{fcd} and the definition of $V_{-0} (q,d)$ and $V_{-0} (p,c)$, we get $V(f) (V_{-0} (q,d)) \subset V_{-0} (p,c)$ where $V_{-0} (q,d)$ and $V_{-0} (p,c)$ are considered as subsets of $V_{+1} (q,d)$ and $V_{+1} (p,c)$ respectively; so, we define $V_{-0} (q,d) \os {V(f)} \lra V_{-0} (p,c)$ as the restriction of $V_{+1} (q,d) \os {V(f)} \lra V_{+1} (p,c)$.
Again, relation \ref{fcd} easily implies that $V(f)$ preserves the action of those affine morphisms whose input and output colors are both $+$-signed; then, one can use this fact to show that $V(f)$ preserves action of all affine morphisms, that is, $V(f) \circ V_a (q,d) = V_a (p,c) \circ V(f)$ for all $a\in AP_{\vlon k,\eta l}$, $\vlon k, \eta l \in \{-0\}\cup \{+n\}_{n \in \N_0}$.
$V(f \circ g) = V(g) \circ V(f)$ and $V(f^*) = (V(f))^*$ follow almost immediately.
Thus, $V$ is a contravariant, $\C$-linear $*$-functor.

To check $V$ is faithful, note that $V_{+2l} (q,d) \ni q \os{V(f)}{\longmapsto} f \in V_{+2l} (p,c)$ which is nonzero if and only if $0 \neq f:(p,c)\ra (q,d)$.

Now, consider a morphism $g: V(q,d) \ra V(p,c)$ in $\mcal D$.
Set $f := g (q) \in V_{+l} (p,c) = P_{+(k+l)} \us{2k} \odot p$ where $q$ is viewed as an element of $V_{+l} (q,d) = P_{+2l} \us{2l} \odot q$.
Clearly, $f = f \us{2k} \odot p$.
Using the $V(p,c)$- and $V(q,d)$-actions of affine morphisms, and $g \in \t{Mor} (\mcal D)$, we get $v \us {2l} \odot f = V_{\psi^0_{+l,+l} (v)} (p,c) \left( g (q) \right) = g \left( V_{\psi^0_{+l,+l} (v)} (q,d) (q) \right) = g(v)$ for all $v \in V_{+l} (q,d)$.
Taking $v = q$, we get $f = q \us{2l} \odot f$ implying $f \in \mcal C(p,q)$; moreover, $\left( \left. \cdot \us{2l} \odot f \right) \right|_{V_{+l} (q,d)} = g: V_{+l} (q,d) \ra V_{+l} (p,c)$.
It remains to show that $f$ is compatible with $c$ and $d$. For $m \in \N_0$, $x \in P_{+ (k+l +2m)}$, observe that
\[
P_{\!\!
\psfrag{c}{$c_{2m}$}
\psfrag{+}{$+$}
\psfrag{f}{$f$}
\psfrag{x}{$x$}
\psfrag{2k}{$2k$}
\psfrag{2l}{$2l$}
\psfrag{2m}{$2m$}
\includegraphics[scale=0.15]{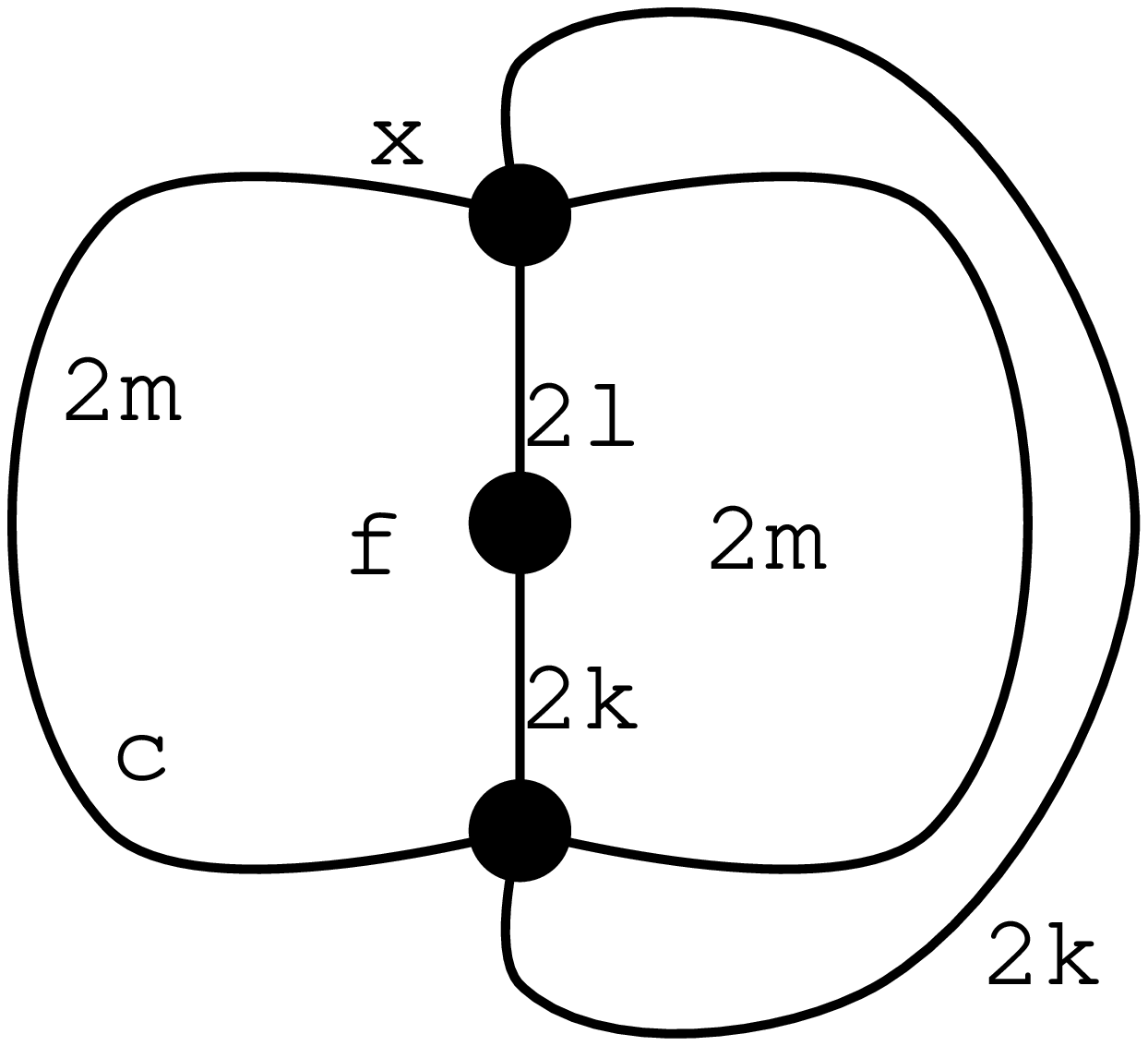}
} \!\!\!\! = \lab p , V_a (p,c) (f) \rab_{V (p,c)} = \lab p , g \left( V_a (q,d) (q) \right) \rab_{V (p,c)} = \lab p , \left[ V_a (q,d) (q) \right] \us{2l} \odot f \rab_{V(p,c)} = P_{\!\!
\psfrag{c}{$d_{2m}$}
\psfrag{+}{$+$}
\psfrag{f}{$f$}
\psfrag{x}{$x$}
\psfrag{2k}{$2k$}
\psfrag{2l}{$2l$}
\psfrag{2m}{$2m$}
\includegraphics[scale=0.15]{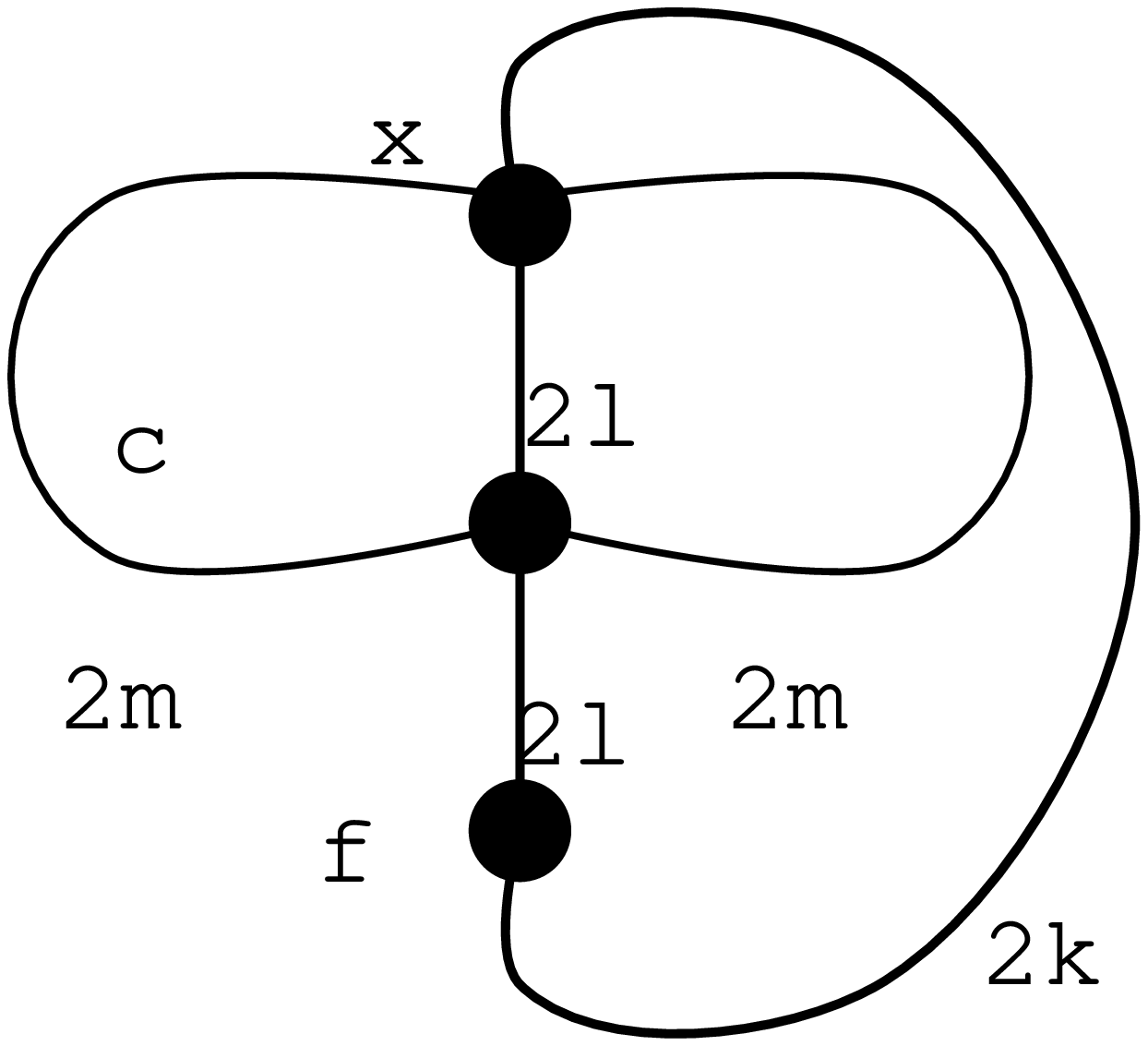}
}
\]
where $a:= \psi^{2m}_{+l , +k} (x)$.
By the non-degeneracy of the action of the trace tangle
, we get $d_{2m} \circ (1_{+2m} \otimes f) = (f \otimes 1_{+2m}) \circ c_{2m}$ for all $m \in \N_0$ where $f$ is viewed as an element of $\mcal C (p,q)$.
So, $f \in \mcal{ZC} ((p,c) , (q,d))$. Since $\left. g \right|_{ V_{+l} (q,d) } = \left. V(f) \right|_{ V_{+l} (q,d) }$ and $V(q,d) = \left[ V_{+l} (q,d) \right]$, therefore $V(f) = g$.

It remains to show that $V$ is a monoidal functor.
Suppose $(p,c) , (q,d) \in \t{ob} (\mcal {ZC})$ with $p \in \mscr P (P_{+2k})$ and $q \in \mscr P (P_{+2l})$; set $(r,e) := (p,c) \otimes (q,d)$ (as defined in Page \pageref{ZCtensor}), $k' =k+l$, $U := V(p,c)$, $V:= V(q,d)$ and $W := U\btimes V$.
So, $r = p\otimes q \in \mscr P (P_{+2k'})$.
Define the map $f: V(r,e) \ra W$ by
\begin{align*}
V_{+ m} (r,e) \ni x & \os{f_{+m}} \longmapsto p \os{x}{\btimes} q \in W_{+m} \t{ for } m\in \N_0 \t{ and}\\
V_{-0} (r,e) \ni P_{
\psfrag{c}{$\;\,e_{2}$}
\psfrag{+}{$+$}
\psfrag{-}{$x$}
\psfrag{2k}{$2k'$}
\psfrag{2l-1}{}
\includegraphics[scale=0.15]{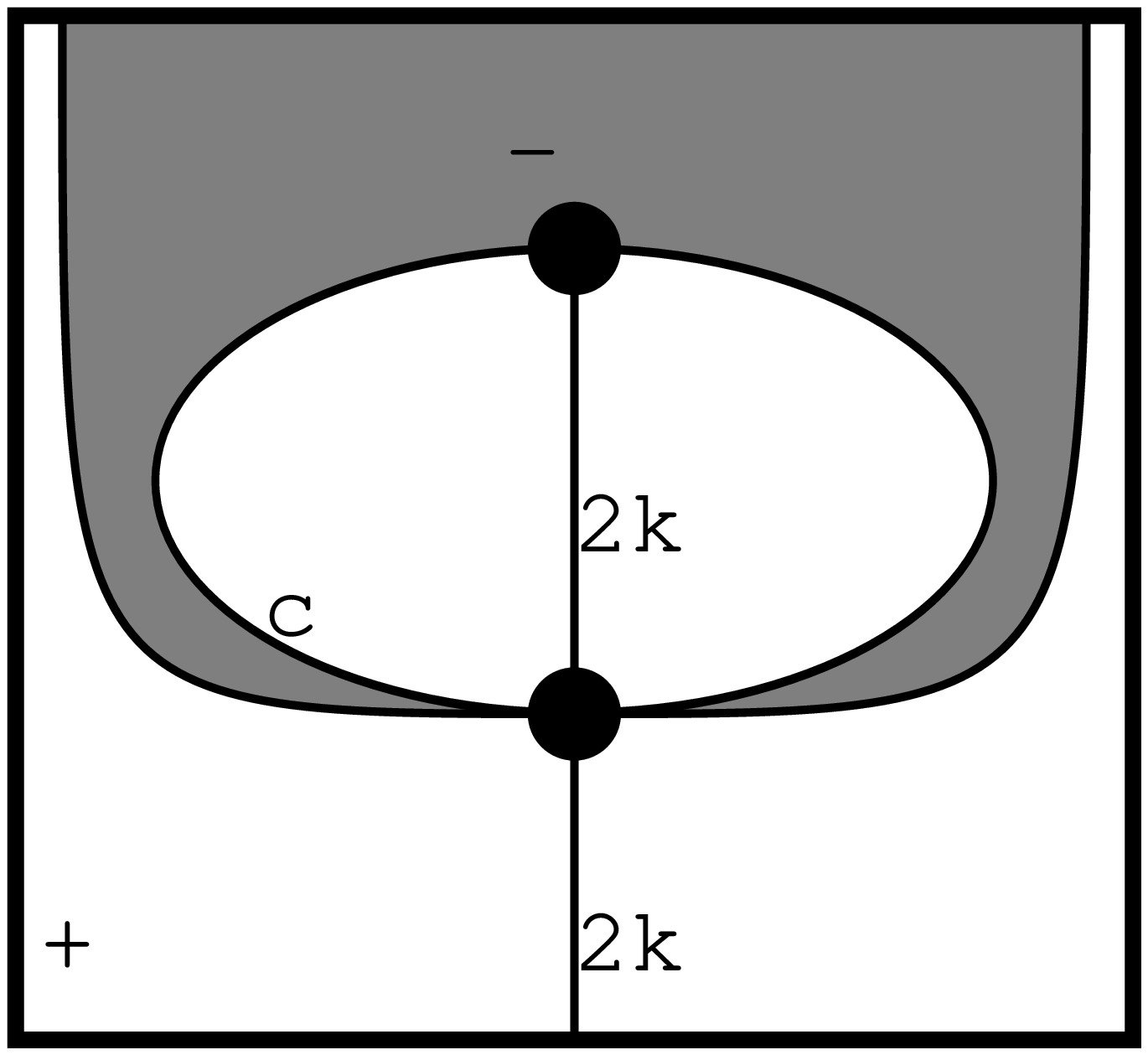}
} & \os{f_{-0}} \longmapsto \delta^{\frac 1 2} \left( p \os{\psi^{1}_{+k',-0} (x)}{\btimes} q \right) \in W_{-0}.
\end{align*}
We will check that $f_{-0}$ is well-defined and $f$ is an isometric morphism between Hilbert affine $P$-modules (and hence, an isometric isomorphism since $f$ is surjective by Lemma \ref{fusbasis}).
For this, it will be enough to show that $f$ preserves the $CC$-valued inner product, that is, $c_{n'} (w,v) = c_{n'} (f(w) , f(v))$ for all $w \in V_{\vlon m} (r,e)$, $v \in V_{\eta n} (r,e)$ where $\vlon m , \eta n \in \{-0\} \cup \{+s:s\in \N_0\}$ and $n' \in \N_{\vlon , \eta}$.
We will prove this only for two cases, namely, $(\vlon ,\eta) = (+,-)$ and $(\vlon ,\eta) = (-,-)$; the case $(\vlon ,\eta) = (+,+)$ is very easy, and the case $(\vlon ,\eta) = (-,+)$ can be established using arguments similar to the proof of the case $(\vlon ,\eta) = (+,-)$.
Set $k' := k+l$.

\noindent {\bf Case 1:}
Suppose $w \in V_{+m}(r,e)$ and $v = P_{
\psfrag{c}{$\;\,e_{2}$}
\psfrag{+}{$+$}
\psfrag{-}{$x$}
\psfrag{2k}{$2k'$}
\psfrag{2l-1}{}
\includegraphics[scale=0.15]{figures/z2ap/f-0.eps}
} \in V_{-0} (r,e)$.
So,
\[
c_{2n'-1} (v,w) = \delta^{ - \frac 1 2} P_{
\psfrag{e2}{$e_{2}$}
\psfrag{e2n}{$e_{2n'}$}
\psfrag{-}{$-$}
\psfrag{x}{$x$}
\psfrag{w}{$w^*$}
\psfrag{2k'}{$2k'$}
\psfrag{2n}{$2m$}
\psfrag{2n'-1}{$2n'\!-\!1$}
\includegraphics[scale=0.15]{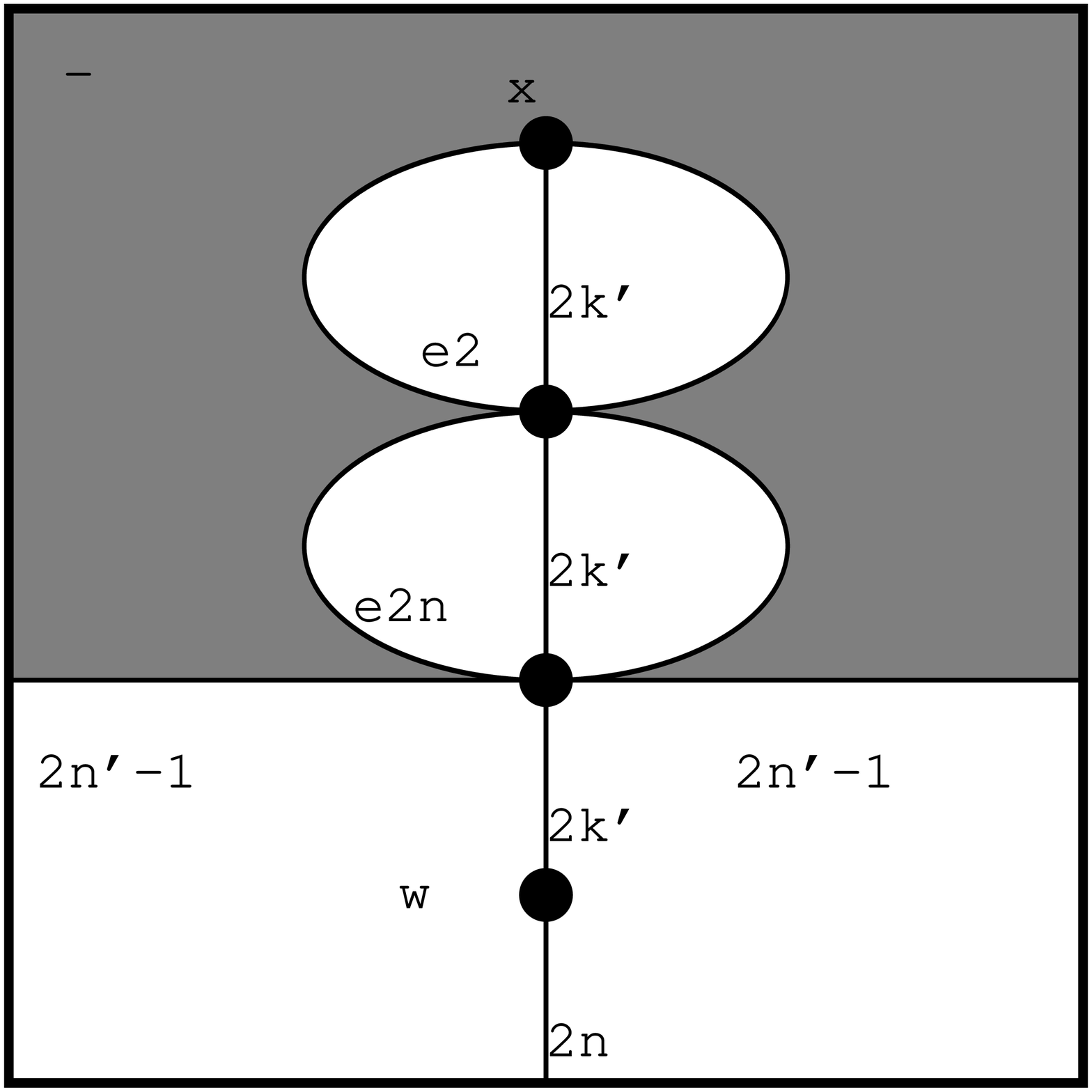}
} = \delta^{\frac 1 2} P_{
\psfrag{e2n}{$e_{2n'}$}
\psfrag{-}{$-$}
\psfrag{x}{$x$}
\psfrag{w}{$w^*$}
\psfrag{2k'}{$2k'$}
\psfrag{2n}{$2m$}
\psfrag{2n'-1}{$2n'\!-\!1$}
\includegraphics[scale=0.15]{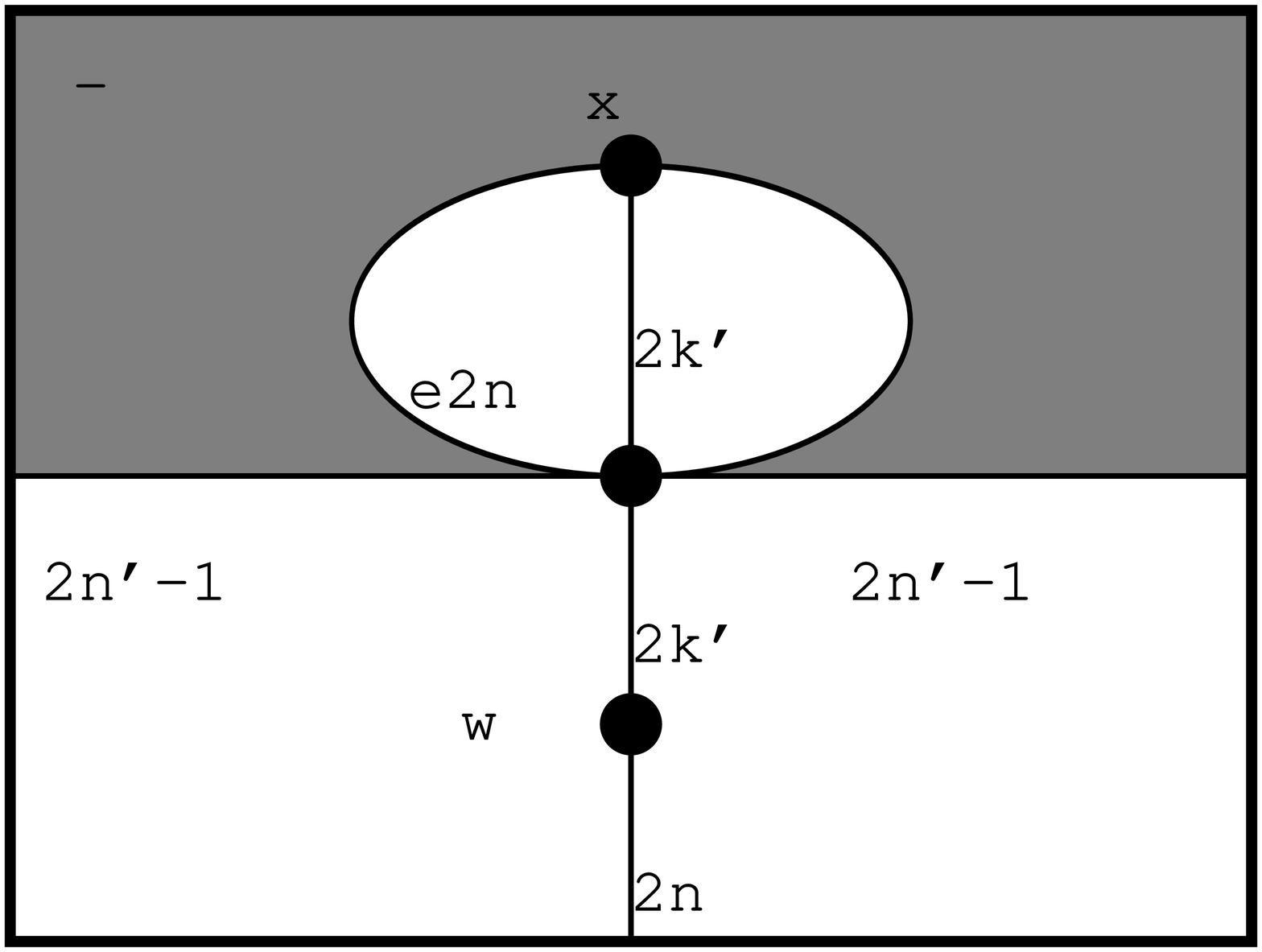}
} = c_{2n'-1} (f(v),f(w))
\]
where the first equality follows from Remark \ref{ccipvpc}, the second comes from the grouping relation of $\{e_{2m}\}_{m \in \N_0} \in CC_{+k',+k'}$, and for the last one, we apply Lemma \ref{fusCCip} and use the dependence of $\{e_{2m}\}_{m \in \N_0}$ on $\{c_{2m}\}_{m \in \N_0}$ and $\{d_{2m}\}_{m \in \N_0}$.

\noindent {\bf Case 2:}
Suppose $w = P_{
\psfrag{c}{$\;\,e_{2}$}
\psfrag{+}{$+$}
\psfrag{-}{$\!\!y$}
\psfrag{2k}{$2k'$}
\psfrag{2l-1}{}
\includegraphics[scale=0.15]{figures/z2ap/f-0.eps}
} \in V_{-0} (r,e)$ and $v = P_{
\psfrag{c}{$\;\,e_{2}$}
\psfrag{+}{$+$}
\psfrag{-}{$x$}
\psfrag{2k}{$2k'$}
\psfrag{2l-1}{}
\includegraphics[scale=0.15]{figures/z2ap/f-0.eps}
} \in V_{-0} (r,e)$.
So,
\[
c_{2n'-2} (v,w) = \delta^{-1} P_{
\psfrag{e2}{$e_{2}$}
\psfrag{e2*}{$e^*_{2}$}
\psfrag{e2n}{$e_{2n'}$}
\psfrag{-}{$-$}
\psfrag{x}{$x$}
\psfrag{y}{$y^*$}
\psfrag{2k'}{$2k'$}
\psfrag{2n'-2}{$2n'\!-\!2$}
\includegraphics[scale=0.15]{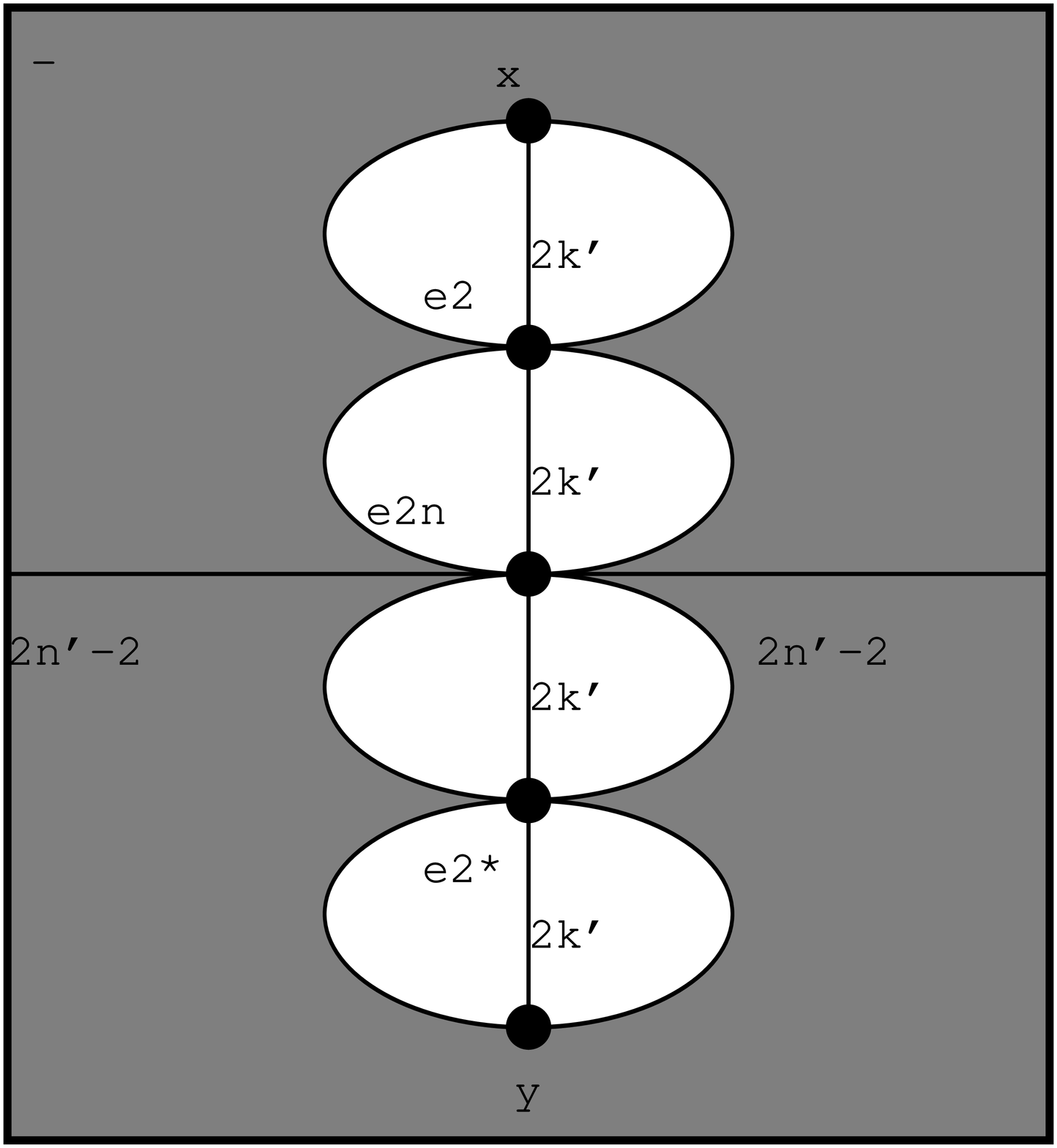}
} = \delta P_{
\psfrag{e2n}{$e_{2n'}$}
\psfrag{-}{$-$}
\psfrag{x}{$x$}
\psfrag{y}{$y^*$}
\psfrag{2k'}{$2k'$}
\psfrag{2n'-2}{$2n'\!-\!2$}
\includegraphics[scale=0.15]{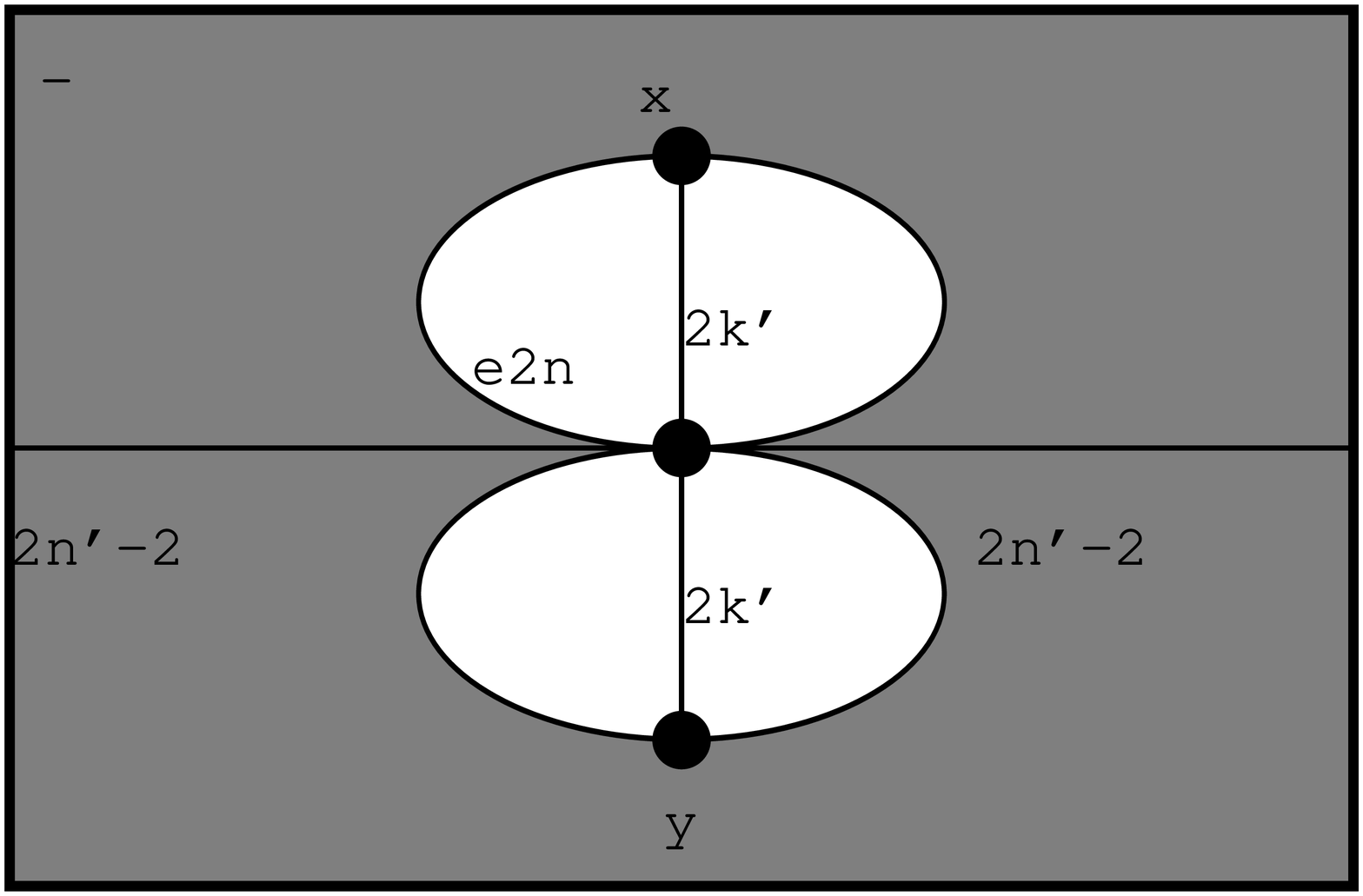}
} = c_{2n'-2} (f(v),f(w))
\]
where, just like Case 1, the first equality follows from Remark \ref{ccipvpc}, the second comes from the grouping relation and the $*$-relation of $\{e_{2m}\}_{m \in \N_0} \in CC_{+k',+k'}$, and for the last one, we apply Lemma \ref{fusCCip} and use the dependence of $\{e_{2m}\}_{m \in \N_0}$ on $\{c_{2m}\}_{m \in \N_0}$ and $\{d_{2m}\}_{m \in \N_0}$.

Next, we show that $V$ preserves the identity objects.
Note that $V_{\vlon k} (1_{\mcal C} , i) = \left\{
\begin{tabular}{ll}
$P_{+k}$, & if $\vlon = +$,\\
$\C 1_{P_{+1}}$, & if $\vlon k = -0$.
\end{tabular}
\right.$
On the other hand, the unit object $1_{\mcal D}$ of $\mcal D$ is the planar algebra $P$ itself viewed as a Hilbert affine $P$-module.
Define $\vphi : V(1_{\mcal C} , i) \ra P$ by $\vphi |_{V_{+k} (1_{\mcal C} , i)} = \t{id}_{P_{+k}}$ and $V_{-0} (1_{\mcal C} , i) \ni 1_{P_{+1}} \os \vphi \mapsto \sqrt{\delta} \; 1_{P_{-0}} \in P_{-0}$.
Clearly, $\vphi$ is an isometric isomorphism of Hilbert affine $P$-modules.

Commutativity of compatibility diagrams involving $f$, $\vphi$ and associativity and unit constraints are completely routine to check and is left as an exercise.

This completes the proof.
\end{proof}
We will now proceed towards proving essential surjectivity of the functor $V$.
Consider a locally finite, Hilbert affine $P$-module $W$ having finite $P$-support.
So, there exists $l \in \N$ such that $W_{+m} = \t{span } AP^{=0}_{+l,+m} (W_{+l})$ for all $m \in \N_0$.
One can view $W_{+l}$ as a finite dimensional left $P_{+2l}$-module by restricting the action of $AP_{+l,+l}$ to its subalgebra $P_{+2l}$ (given by the inclusion $\psi^0_{+l,+l}$).
Before proceeding further, we set up some notations and state a useful fact.

\noindent {\bf Notation:} For any $p \in \mscr P(P_{+2k})$, $V_{+l} (p) := \left[ P_{+(k+l)} \us{2k} \odot p \right]$ is the left $P_{+2l}$-module where (i) the inner product is given by action of the trace tangle and (ii) $P_{+2l} \times V_{+l} (p) \ni (x,v) \mapsto x \us{2l} \odot v \in V_{+l} (p)$ gives the $P_{+2l}$-action.
\begin{rem}\label{vlp}
For any finite dimensional Hilbert space $V$ with an action of the  finite dimensional  $C^*$-algebra $P_{+2l}$, there exists $k \geq l$ and $p \in \mscr P(P_{+2k})$ such that

(a)  $V$ is isometrically isomorphic (as a $P_{+2l}$-module) to  $V_{+l} (p)$, and
(b) $V_{+l} (q) \neq \{0\}$ for all $0 \neq q \in \mscr P (P_{+2k})$ such that $p \geq q$.
\end{rem}
This remark is acheived by considering a finite decomposition $V \cong \us{1\leq i \leq n}\oplus V^{\oplus n_i}_i$ as $P_{+2l}$-modules where $V_i$'s are mutually non-isomorphic irreducible $P_{+2l}$-modules.
There exists a set of centrally orthogonal set of minimal projections $\{q_i : 1 \leq i \leq n\}$ in $P_{+2l}$ such that $V_i \cong V_{+l}(q_i) = (P_{+2l})q_i$ as $P_{+2l}$-modules for all $1\leq i \leq n$.
If needed, we move to a higher level $k > l$ such that size of the direct summand corresponding to the minimal projection $p_i := q_i (e_{2l+1} \ldots e_{2k-1})$ in $P_{+2k}$, is bigger than $n_i$ for all $1 \leq i \leq n$.
Clearly, $V_{+l} (p_i) = P_{+(k+l)} \us{2k}\odot p_i \cong P_{+2l} \us{2l}\odot q_i = V_{+l}(q_i) \cong V_i$, and also there exists a projection $p$ in $P_{+2k}$ which can be decomposed as $p=\ous{n} \sum {i = 1} \; \ous{n_i} \sum {j=1} p_{i,j}$ such that $p_{i,j}$'s are minimal projections, $p_{i,j}$ and $p_{i',j'}$ have orthogonal central support whenever $i \neq i'$, and $p_{i,j}$ and $p_i$ are always Murray-von Neumann equivalent.
It is then straight forward to verify that this $k$ and $p$ satisfies the conditions (a) and (b) of the remark.

Applying the above remark for $V = W_{+l}$, there exists $k \geq l$ and $p \in \mscr P (P_{+2k})$ satisfying conditions (a) and (b); let $\vphi_{+l}: W_{+l} \ra V_{+l} (p)$ be the $P_{+2l}$-linear isometric isomorphism.
Using the finiteness of $P$-support of $W$, we consider $W_{+m} \ni \psi^0_{+l, +m} (x) \cdot w \os{\vphi_{+m}} \longmapsto x \us{2l} \odot \vphi_{+l} (w) \in V_{+m} (p)$ where $x \in P_{+(l+m)}$ and $w \in W_{+l}$.
Well-definedness, injectivity and $P_{+2m}$-linearity of $\vphi_{+m}$, follows from $\lab \psi^0_{+l,+m} (x) \cdot v , \psi^0_{+l,+m} (y) \cdot w \rab = \lab v , \left[ \left( \psi^0_{+l,+m} (x) \right)^* \circ\psi^0_{+l,+m} (y) \right] \cdot w \rab = \lab v , \psi^0_{+l,+l} (x^* \us{2m} \odot y) \cdot w \rab = \lab \vphi_{+l} (v) , (x^* \us{2m} \odot y) \us{2l} \odot \vphi_{+l} (w) \rab = \lab x \us{2l} \odot \vphi_{+l} (v) , y \us{2l} \odot \vphi_{+l} (w) \rab$ for all $x , y \in P_{+(l+m)}$, $v,w \in W_{+l}$.
In fact, $\vphi_{+m}$'s are isometries.

Surjectivity of $\vphi_{+m}$ is immediate for $m \leq l$.
Now, suppose $m > l$.
We will check $V_{+m} (p) = \t{span } P_{+(l+m)} \us{2l} \odot V_{+l} (p)$ which will then imply surjectivity of $\vphi_{+m}$.
Let $\{ u_\alpha \}_\alpha$ denote an orthonormal basis of $V_{+l} (p)$ and set $z := \us \alpha \sum u^*_\alpha \us{2l} \odot u_\alpha \in \left(P_{+2k} \right)_p$.
$z$ is clearly a positive element of $P_{+2l}$.
Also, $z$ must be invertible in $\left(P_{+2k} \right)_p$; if not, there exists $0 \neq q \in \mscr P \left( \left(P_{+2k} \right)_p \right)$ such that $qz^{1/2} = 0$, that is,
$
0 =  P_{TR^l_{+2k}} (qz) = \us \alpha \sum \norm{ u_\alpha \us{2k} \odot q}^2_{V_{+l} (p)}
$
and thereby, we get $0 = V_{+l} (p) \us{2k} \odot q = V_{+l} (q)$ which is a contradiction to condition (b) of Remark \ref{vlp} (where $TR^l_{+2k} : +2k \ra +0$ is the left trace tangle).
Let $v$ be the inverse of $z$ in $\left(P_{+2k} \right)_p$.
So, we have the identity $p = \us \alpha \sum v \us{2k} \odot u^*_\alpha \us{2l} \odot u_\alpha$.
This identity implies $V_{+m} (p) = \t{span } P_{+(l+m)} \us{2l} \odot V_{+l} (p)$.

We now try to show that there exists a unitary commutativity constraint $c : \cdot \otimes p \ra p \otimes \cdot$ such that $(p,c)$ becomes an object of $\mcal {ZC}$.
For this, we first induce the action of affine morphisms on $V_{+m} (p)$'s via the maps $\vphi_{+m}$'s.
Consider the element $c \in CC_{+k,+k}$ coming from the linear functional $AP_{+k, +k} \ni a \mapsto \lab p , a \cdot p \rab_{V_{+k} (p)} \in \C$ which is described in the proof of Proposition \ref{cma}, that is, $\lab p , \psi^{2m}_{+k,+k} (x) \cdot p \rab_{V_{+k} (p)} = P_{
\psfrag{c}{$c_{2m}$}
\psfrag{x}{$x$}
\psfrag{2k}{$2k$}
\psfrag{2l}{$2k$}
\psfrag{2m}{$2m$}
\includegraphics[scale=0.15]{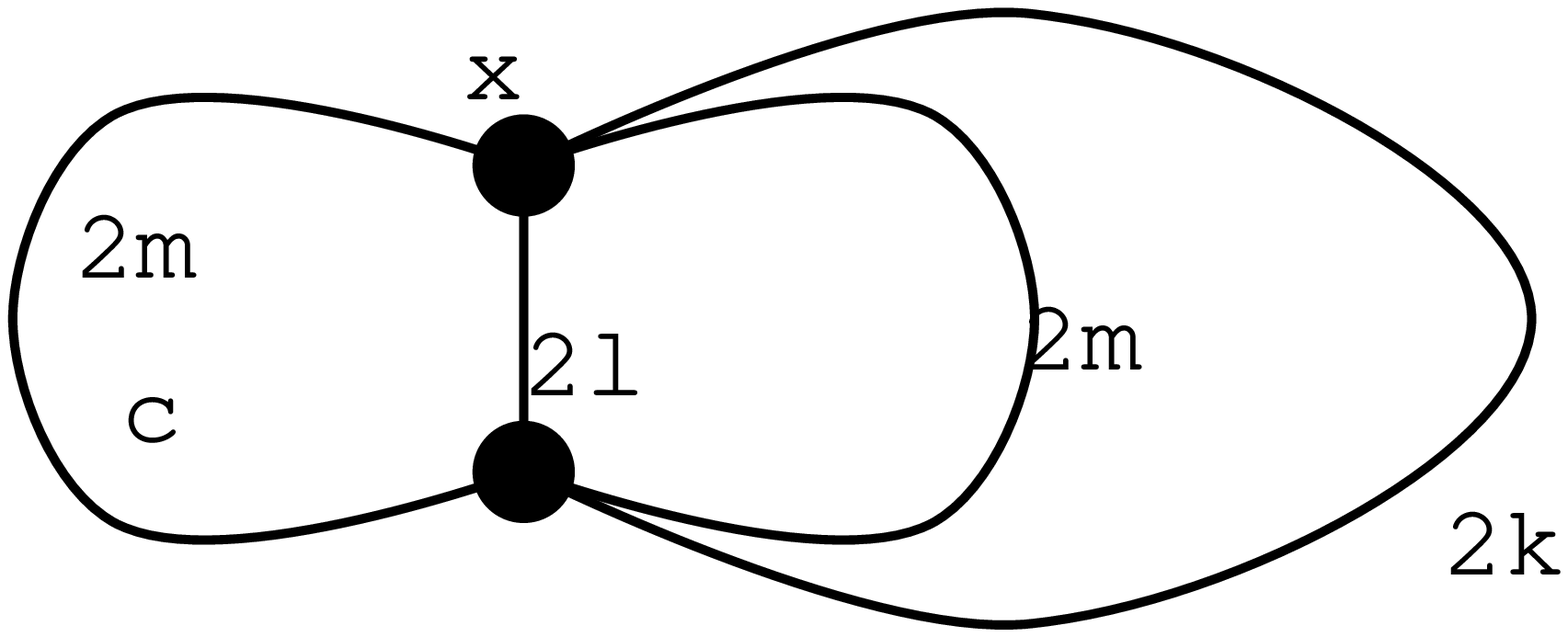}
}$ for all $m \in \N_0$, $x \in P_{+2(k+m)}$.
This then implies
\begin{equation}\label{vwip}
\lab v , \psi^{2s}_{+m,+n} (x) \cdot w \rab_{V_{+n} (p)} =  \lab p , \left[\psi^0_{+n,+k} (v^*) \circ \psi^{2m}_{+k,+k} (x) \circ \psi^0_{+k,+m} (w) \right] \cdot p \rab_{V_{+k} (p)} = P_{
\psfrag{c}{$c_{2s}$}
\psfrag{x}{$x$}
\psfrag{v}{$v^*$}
\psfrag{w}{$w$}
\psfrag{2k}{$2k$}
\psfrag{2n}{$2n$}
\psfrag{2m}{$2m$}
\psfrag{2s}{$2s$}
\includegraphics[scale=0.18]{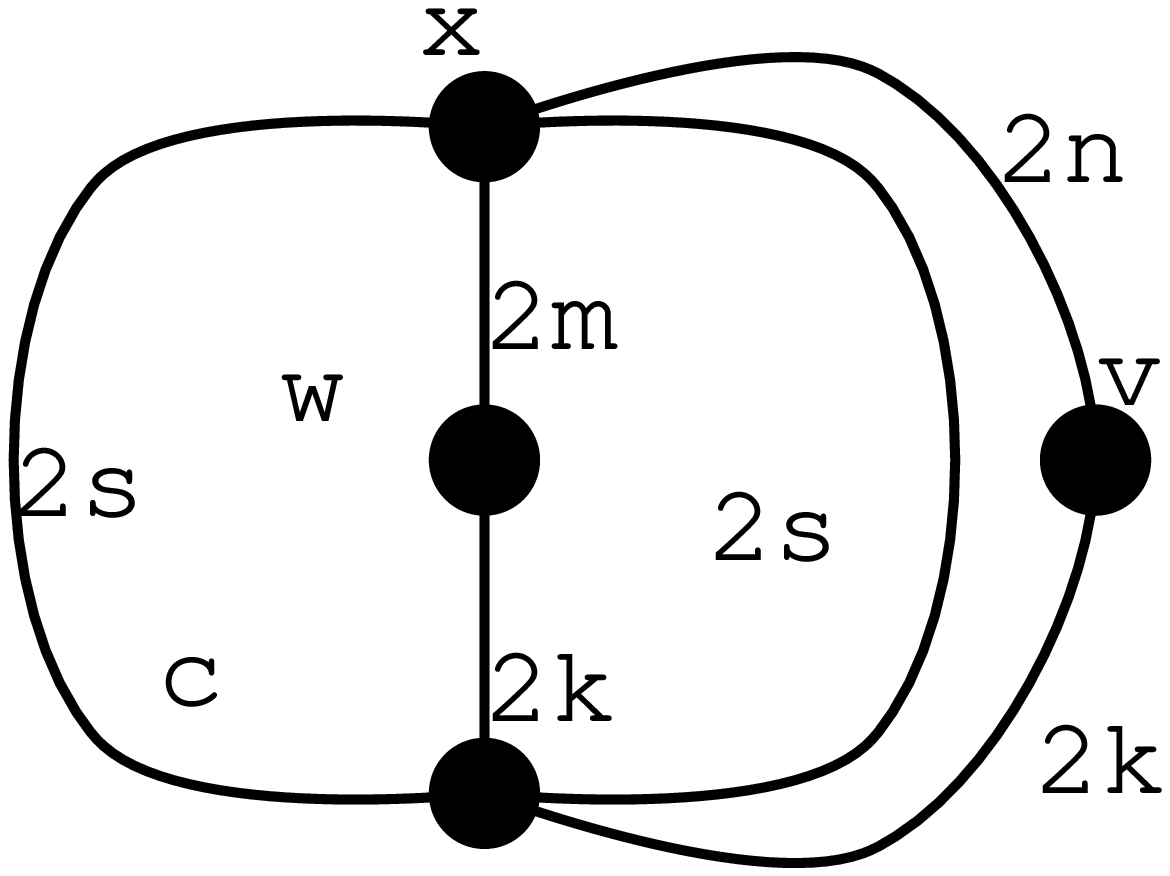}
}
\end{equation}
for all $m,n,s \in \N_0$, $x \in P_{+(m+n+2s)}$, $w \in V_{+m}(p)$, $v \in V_{+n}(p)$. 

By setting $n=m=k$, $v = w = p$ in Equation \ref{vwip}, and using non-degeneracy of the action of trace tangle and the freedom of the choice of $x$, we get $c_{2s} = \left [ p \otimes 1_{+2s} \right] c_{2s} \left [  1_{+2s} \otimes p \right]$.
Thus, $c_{2s} \in \mcal C (1_{+2s} \otimes p , p \otimes 1_{+2s} )$.

Similarly, using the freedom of the choice of $v$ in Equation \ref{vwip} and nondegeneracy of the action of trace tangle, we may conclude that the action has the following form:
\begin{equation}\label{a.w}
a \cdot w = P_{
\psfrag{c}{$c_{2s}$}
\psfrag{x}{$x$}
\psfrag{+}{$+$}
\psfrag{v}{$v^*$}
\psfrag{w}{$w$}
\psfrag{2k}{$2k$}
\psfrag{2n}{$2n$}
\psfrag{2m}{$2m$}
\psfrag{2s}{$2s$}
\includegraphics[scale=0.15]{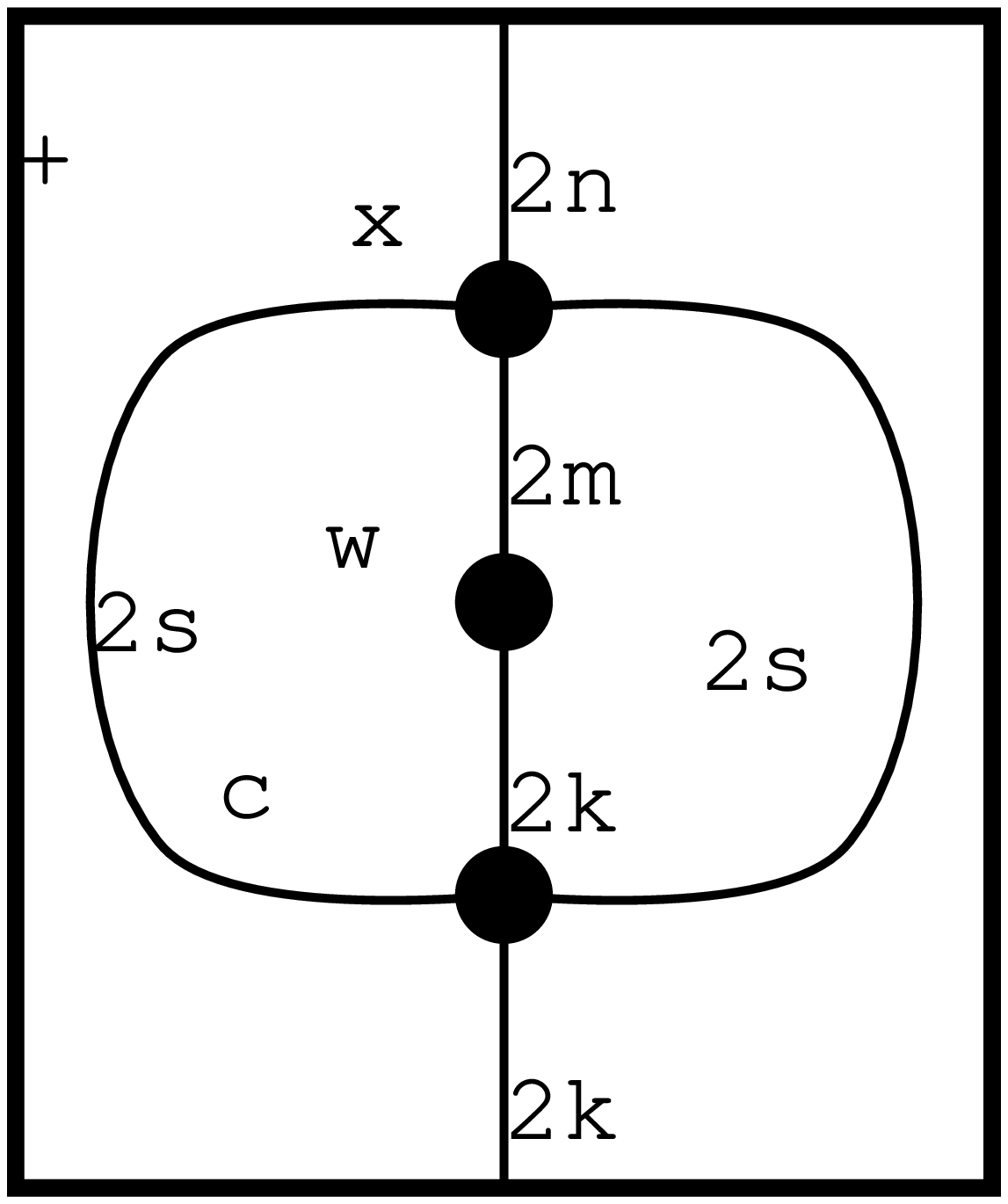}
} \t{ for all } m,n, s \in \N_0, w \in V_{+m} (p), x \in P_{+(m+n+2s)} \t{ s.t. } a = \psi^{2s}_{+m,+n} (x).
\end{equation}

We need to show that $c$ satisfies the grouping relation (in Remark \ref{cprop}).
\comments{
$c_{2(s+t)} = P_{
\psfrag{cl}{$c_{2t}$}
\psfrag{cm}{$c_{2s}$}
\psfrag{+}{$+$}
\psfrag{2k}{$2k$}
\psfrag{2l}{$2t$}
\psfrag{2m}{$2s$}
\includegraphics[scale=0.15]{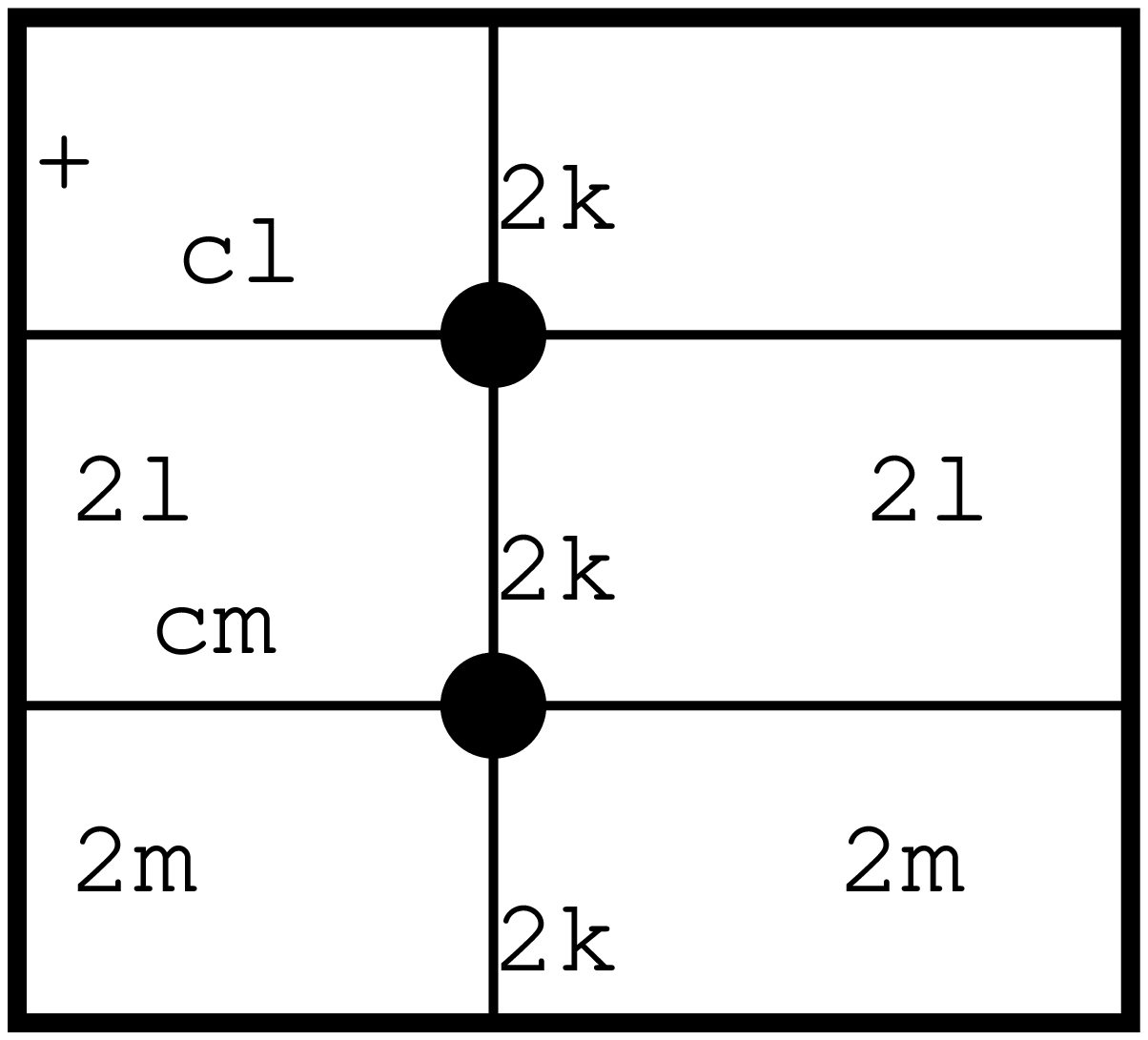}
}$ for all $s,t \in \N_0$.
}
For this, set $a := \psi^{2s}_{+(k+2t) , + (k + 2(s+t) ) } (1_{+(k+2(s+t) )})$ and $b := \psi^{2t}_{+k , +(k+2t)} (1_{+(k+2s)})$.
The description of the action of affine morphisms in Equation \ref{a.w}, implies
\[
P_{
\psfrag{cs}{$c_{2s}$}
\psfrag{ct}{$c_{2t}$}
\psfrag{x}{$x$}
\psfrag{+}{$+$}
\psfrag{2k}{$2k$}
\psfrag{2s}{$2s$}
\psfrag{2t}{$2t$}
\includegraphics[scale=0.15]{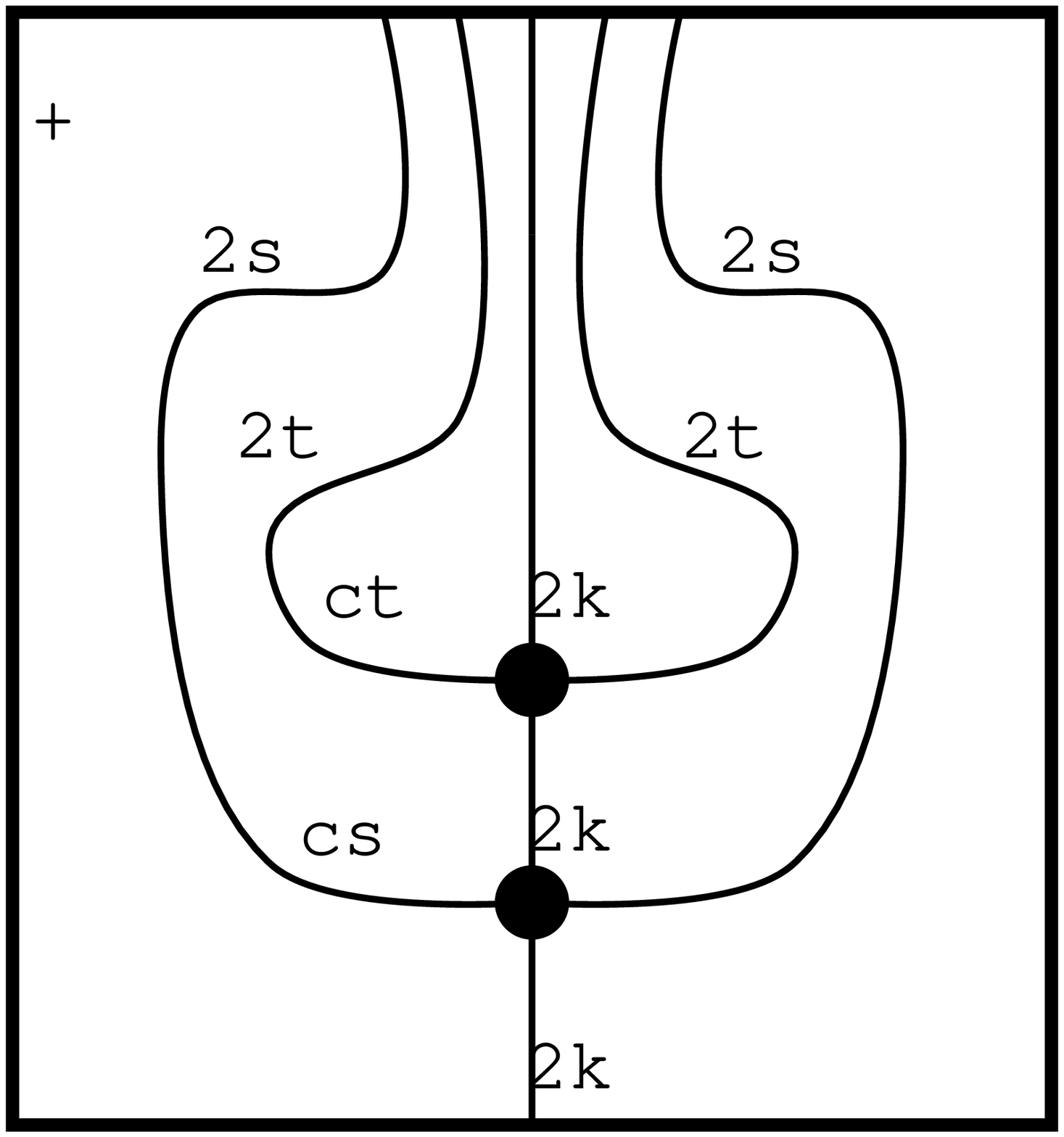}
} = a \cdot (b \cdot p) = (a \circ b) \cdot p = \left[ \psi^{2(s+t)}_{+k,+k} (1_{+2(k+s+t)}) \right] \cdot p = P_{
\psfrag{c}{$c_{2(s+t)}$}
\psfrag{x}{$x$}
\psfrag{+}{$+$}
\psfrag{2k}{$2k$}
\psfrag{s}{$2(s+t)$}
\includegraphics[scale=0.15]{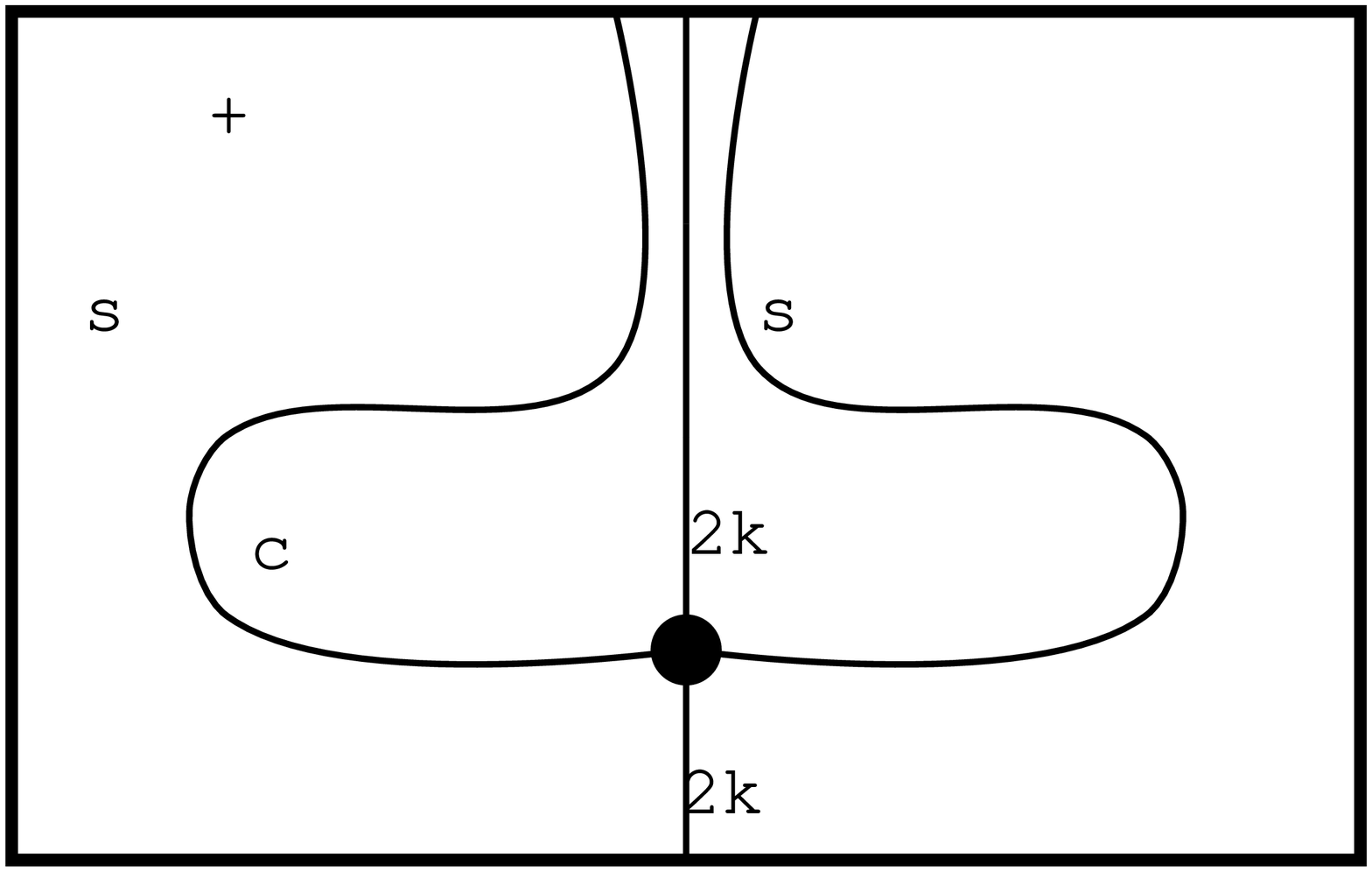}
}
\]
which gives the required equation for grouping relation.

For the $*$-relation (in Remark \ref{cprop}), we use the action (in Equation \ref{a.w}) to obtain
\[
P_{
\psfrag{c}{$c_{2m}$}
\psfrag{x}{$x$}
\psfrag{2k}{$2k$}
\psfrag{2l}{$2k$}
\psfrag{2m}{$2m$}
\includegraphics[scale=0.15]{figures/z2ap/ppip.eps}
} = \lab p , \psi^{2m}_{+k,+k} (x) \cdot p \rab_{V_{+k} (p)} = \lab \left( \psi^{2m}_{+k,+k} (x) \right)^* \cdot p , p \rab_{V_{+k} (p)} = P_{
\psfrag{c}{$c^*_{2m}$}
\psfrag{x}{$x$}
\psfrag{2k}{$2k$}
\psfrag{2l}{$2k$}
\psfrag{2m}{$2m$}
\includegraphics[scale=0.15]{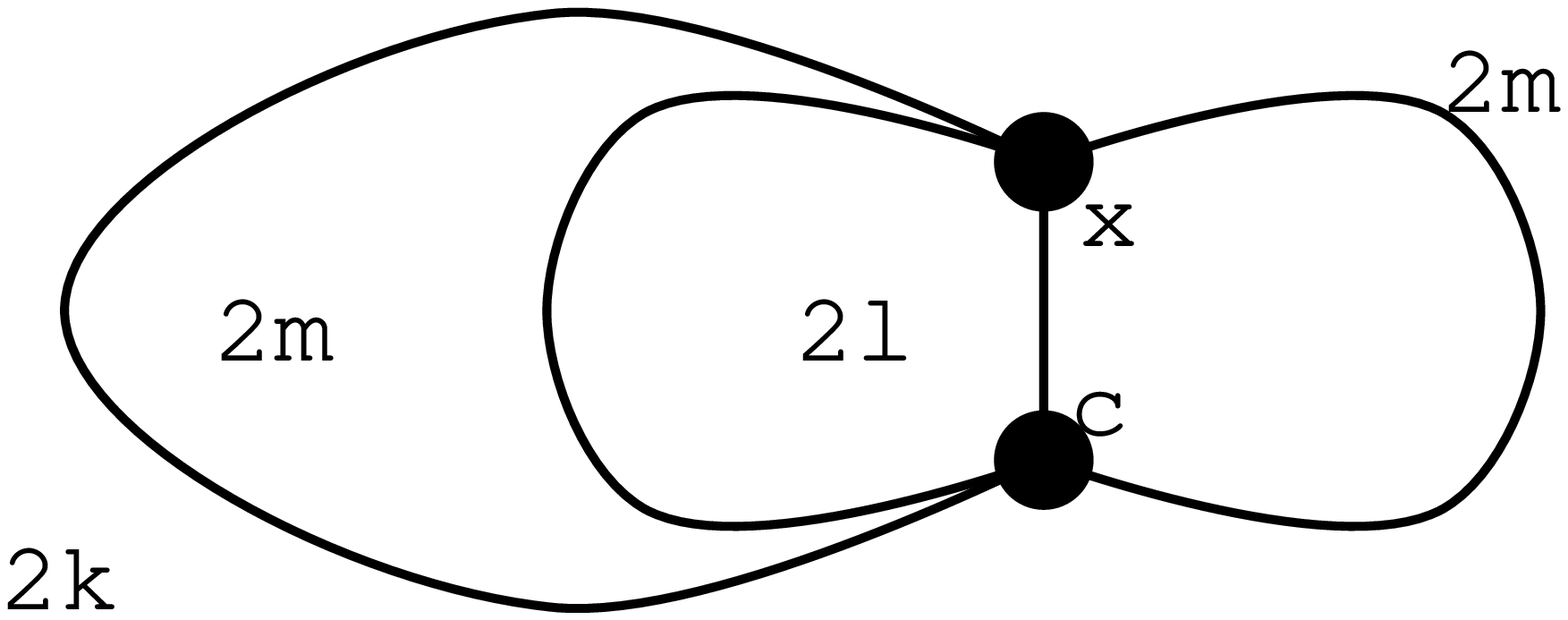}
}
\]
for all $m \in \N_0$, $x \in P_{+2(k+m)}$.
Then, we apply sphericality on the right side and use non-degeneracy of the trace tangle to get the required relation.

$c_0 = p$ follows from Equation \ref{a.w} by setting $s=0$, $m=n=k$, $w = p$ and $x = 1_{+2k}$.
Hence, by Remark \ref{cprop}, $c$ is a unitary commutativity constraint and as a result, $(p,c) \in \t{ob} (\mcal {ZC})$.
From the action of the tangles in Equation \ref{a.w}, we also proved that $\vphi_{+k} : W_{+k} \ra V_{+k} (p) = V_{+k} (p, c)$ is an $AP_{+k,+k}$-linear isometric isomorphism.
Since support of both $W$ and $V(p,c)$ can be atmost $k$, therefore by Remark \ref{affmodmor}, $W$ must be isometrically isomorphic to $V(p,c)$ as affine $P$-modules.
This shows $V$ is essentially surjective.
Thus, we have proved the following theorem.

\begin{thm}\label{genconj}
Let $N \subset M$ be a finite index extremal subfactor, $P$ be its
 subfactor planar algebra and $\mcal{C}$ be the category of $N$-$N$-bimodules generated by $_NL^2(M)_M$. Then, the Drinfeld center $\mcal {ZC}$ is contravariantly equivalent to the category of locally finite Hilbert affine $P$-modules with finite $P$-support as braided tensor categories, in particular, with notations as above,   $V: \mcal {ZC} \ra \mcal D$ gives such an equivalence.
\end{thm}
\begin{rem}
Theorem \ref{genconj} proves Walker's conjecture in the affirmative using Remark \ref{findepfinPsupp}.
\end{rem}

\subsection*{Acknowledgements}
The authors would like to thank Masaki Izumi, Vaughan Jones and Scott Morrison for  fruitful discussions on several occassions.

\bibliographystyle{alpha}


\end{document}